\documentclass[oneside,english,reqno]{amsart}

\usepackage{xparse}
\let\realItem\item % save a copy of the original item
\makeatletter
\NewDocumentCommand\myItem{ o }{%
   \IfNoValueTF{#1}%
      {\realItem}% add an item
      {\realItem[#1]\def\@currentlabel{#1}}% add an item and update label
}
\makeatother

\usepackage{enumitem}    
\setlist[enumerate]{
    before=\let\item\myItem,       % use \myItem in enumerate
    label=\textnormal{(\arabic*)}, % format the label
    widest=(2')                    % set the widest label
}

\usepackage{amsmath,amsthm,amssymb,bbm}
\usepackage{mathrsfs}
\usepackage{mathtools}

\usepackage{tikz}
\usepackage{graphicx}
\usepackage{calculator}

\theoremstyle{plain}
\newtheorem{theorem}{Theorem}[section]
\newtheorem{proposition}[theorem]{Proposition}
\newtheorem{lemma}[theorem]{Lemma}
\newtheorem{corollary}[theorem]{Corollary}

\theoremstyle{definition}
\newtheorem{definition}[theorem]{Definition}

\newtheorem{notation}[theorem]{Notation}

\newtheorem{assumption}[theorem]{Assumption}

\newtheorem{example}[theorem]{Example}

\newtheorem{remark}[theorem]{Remark}

\newtheorem*{ack}{Acknowledgments}

\newenvironment{myproof}[2] {\paragraph{{\it Proof of} {#1} {#2}:}}{\hfill$\square$}

\newcommand{\equaldist}{\stackrel{\lower0.5pt\hbox{$\scriptstyle\mathrm d$}}=}

\newcommand{\N}{\mathbb{N}}
\newcommand{\cA}{\mathscr{A}}

\newcommand{\DeclareAutoPairedDelimiter}[3]{%
  \expandafter\DeclarePairedDelimiter\csname Auto\string#1\endcsname{#2}{#3}%
  \begingroup\edef\x{\endgroup
    \noexpand\DeclareRobustCommand{\noexpand#1}{%
      \expandafter\noexpand\csname Auto\string#1\endcsname*}}%
  \x}
\DeclareAutoPairedDelimiter{\norm}{\lVert}{\rVert}
\DeclareAutoPairedDelimiter{\ip}{\langle}{\rangle}
\newcommand{\vertiii}[1]{{\left\vert\kern-0.25ex\left\vert\kern-0.25ex\left\vert #1 
    \right\vert\kern-0.25ex\right\vert\kern-0.25ex\right\vert}}
\newcommand{\numberthis}{\addtocounter{equation}{1}\tag{\theequation}}

\newcommand{\ds}{d}
\newcommand{\E}{\mathbb{E}}
\newcommand{\e}{\varepsilon}
\newcommand{\C}{\mathbb{C}}
\newcommand{\R}{\mathbb{R}}
\newcommand{\MNC}{\mathrm{M}_N(\C)}
\newcommand{\MkC}{\mathrm{M}_k(\C)}
\newcommand{\GL}{\mathrm{GL}}
\newcommand{\into}{\int_0^{\boldsymbol{\cdot}}}
\DeclareMathOperator{\tr}{tr}

\numberwithin{equation}{section}

\newcommand{\ignore}[1]{}
\usepackage{todonotes}
\usepackage{enumitem}
\usepackage{wasysym}

\usepackage[linktocpage,colorlinks,linkcolor=blue,anchorcolor=blue,citecolor=blue,urlcolor=black,pagebackref]{hyperref}

\renewcommand*{\backref}[1]{\ifx#1\relax \else Page #1 \fi}
\renewcommand*{\backrefalt}[4]{
  \ifcase #1 \footnotesize{(Not cited.)}
  \or        \footnotesize{(Cited on page~#2.)}
  \else      \footnotesize{(Cited on pages~#2.)}
  \fi
}

\begin{document}

\title{Matrix Random Walks and the Lima Bean Law}

\author{Bruce K. Driver\textsuperscript{1}}
\address{\textsuperscript{1}Department of Mathematics, UC San Diego La Jolla, CA 92093-0112}
\email{bdriver@ucsd.edu}

\author{Brian C. Hall\textsuperscript{2}}
\address{\textsuperscript{2}Department of Mathematics, University of Notre Dame, Notre Dame, IN 46556 USA}
\email{bhall@nd.edu}

\author{Ching-Wei Ho\textsuperscript{3}}
\address{\textsuperscript{3}Institute of Mathematics, Academia Sinica, Taipei 10617, Taiwan}
\email{chwho@gate.sinica.edu.tw}
\thanks{\textsuperscript{3}Supported in part by NSTC grant 114-2115-M-001-005-MY3}

\author{Todd Kemp\textsuperscript{4}}
\address{\textsuperscript{4}Department of Mathematics, UC San Diego, La Jolla, CA 92093-0112}
\email{tkemp@ucsd.edu}
\thanks{\textsuperscript{4}Supported in part by NSF Grants DMS-2055340 and DMS-2400246}

\author{Yuriy Nemish\textsuperscript{5}}
\address{\textsuperscript{5}Chicago, IL}
\email{yuriy.nemish@gmail.com}

\author{Evangelos A. Nikitopoulos\textsuperscript{6}}
\address{\textsuperscript{6}Department of Mathematics, University of Michigan,  530 Church Street, Ann Arbor, MI 48109-1043}
\email{enikitop@umich.edu}

\author{F\'elix Parraud\textsuperscript{7}}
\address{\textsuperscript{7}Department of Mathematics, Queen's university, 48 University Ave, Kingston, 
ON, Canada,  K7L 3N8}
\email{felix.parraud@gmail.com}

\begin{abstract} A matrix random walk is a stochastic process of the form  $B_k = (I+A_1)\cdots(I+A_k)$ where $A_j$ are independent ``step'' matrices in $\MNC$.  With the right entry-covariance, a rescaled matrix random walk converges to Brownian motion $B(t)$ on a matrix Lie group.  In this paper, we study the eigenvalues of such rescaled matrix random walks, as $N\to\infty$ and $k\to\infty$.

The standard Brownian motion $W(t)$ on $\MNC$ has independent Gaussian entries at each $t$. It is bi-invariant: mutiplying on the left or right by a unitary does not change the distribution.  We prove that the empirical eigenvalue distribution of any matrix random walk $B_k$ with bi-invariant steps $A_j$ and initial distribution converges (for fixed $k$ as $N\to\infty$) to a probability measure on $\C$: the Brown measure of the free probability $\ast$-distribution limit $b_k$ of the random walk.  If the steps $A_j$ are identically distributed with normalized Hilbert--Schmidt norm $\|A_j\|_2 = t$, the limit law of eigenvalues is supported on a compact ``lima bean'' shaped region.  We explicitly compute the limit measure and region, and characterize their phase transitions as $t$ evolves.

We prove that the Brown measure of $b_k$ converges as $k\to\infty$, to the Brown measure of the free multiplicative Brownian motion, assuming only that the steps are bi-invariant and normalized in Hilbert--Schmidt norm.  Thus the Brownian motion is the universal limit of rescaled matrix random walks, under very general assumptions on the distribution of steps.\end{abstract}

\maketitle

\tableofcontents

\section{Introduction}

This paper principally concerns random matrices of the form
\begin{equation} \label{eq.Bk} B_k(t) = (I+(\textstyle{\frac{t}{k}})^{1/2}A_1)(I+(\textstyle{\frac{t}{k}})^{1/2}A_2)\cdots(I+(\textstyle{\frac{t}{k}})^{1/2}A_k) \end{equation}
where $(A_k)_{k\in\mathbb{N}}$ are independent matrices in $\MNC$ whose distribution is $\mathrm{U}(N)$ bi-invariant:
\begin{equation} \label{eq.bi-inv.intro} UA_jV^\ast \equaldist A_j \qquad \text{for any }U,V\in\mathrm{U}(N). \end{equation}
Equivalently, in the singular value decomposition $A_j = U_j T_j V_j^\ast$, the left and right singular vector matrices $U_j$ and $V_j$ each have the Haar distribution on $\mathrm{U}(N)$, and all three matrices $T_j,U_j,V_j$ are independent.  (While the $t\ge 0$ parameter could be absorbed into $A_j$, it will be convenient to normalize $A_j$ in Hilbert--Schmidt norm and evolve the scaling separately.)

Under some assumptions on $A_1 = A_1^N$ as $N\to\infty$, Guionnet, Krishnapur, and Zeitouni proved in \cite{GuionnetKZ-single-ring} that the empirical law of eigenvalues of $A_1^N$ converges to a smooth rotationally-invariant distribution supported on an annulus in $\mathbb{C}$; see also \cite{HoZhong2025} for local limits and deformations of this ``Single Ring Theorem'' result.  %More recently, \cite{Cook} have studied product ensembles $(A_1)^k$ and $A_1\cdots A_k$ ({\color{blue} check for context}) ...

Provided that the empirical singular value distribution, i.e.\ the empirical law of the diagonal entries of $T_j = T_j^N$, converges as $N\to\infty$, the multi-matrix ensemble $(A_1,\ldots,A_k)$ converges as $N\to\infty$ in $\ast$-distribution to $(a_1,\ldots,a_k)$ in a tracial noncommutative probability space, where the $a_j$ are freely independent $\mathscr{R}$-diagonal random variables (cf.\ Propositions \ref{prop.asymp.free.Haar.GUE} and \ref{prop.bi-inv}); hence $B_k(t)$ converges in $\ast$-distribution as $N\to\infty$ to 
\begin{equation} \label{def.bk} b_k(t) = (1+(\textstyle{\frac{t}{k}})^{1/2}a_1)(1+(\textstyle{\frac{t}{k}})^{1/2}a_2)\cdots(1+(\textstyle{\frac{t}{k}})^{1/2}a_k). \end{equation}
Convergence in $\ast$-distribution does not, however, imply convergence of the empirical eigenvalue distribution, as all the necessary powerful literature on the Circular Law \cite{Bordenave-Chafai-circular} and the Single Ring Theorem attest.

Our first main Theorem \ref{thm.conv.ESD} show that, under mild assumptions on the bi-invariant step distributions $A_j$, the empirical eigenvalue distribution of the random walk $B_k(t)$ converges as $N\to\infty$ to to the Brown measure (Section \ref{sect.Brown.measure}) of the operator $b_k(t)$.  What's more, when that steps $A_j$ are identically distributed, we show in main Theorems \ref{thm.Brown.Measure.1}, \ref{thm.Brown.Measure.2}, and \ref{thm.Brown.Measure.3} that the limit Brown measure is supported on a compact region $\overline{\Sigma_k(a,t)}$ which is well-understood (and is often shaped like a lima bean). Aside from some possible structural point masses and potential mass on the boundary that dissipate with large $t$ and $k$, the Brown measure has a real analytic density.

When the step distribution $A_1$ has the same covariance as an i.i.d.\ Gaussian matrix (a Ginibre ensemble) in $\MNC$, for fixed $N$, $B_k(t)$ converges in distribution as $k\to\infty$ to the law of the standard Brownian motion $B(t)$ on the Lie group $\GL(N,\mathbb{C})$ at time $t$.  This was proved in \cite{Berger}; we prove strong $L^p$ bounds and almost sure convergence below in Theorem \ref{main.thm.1.MoBettaWZ}, along with the corresponding free probability limit in Theorem \ref{thm.free.MoBettaWZ}.

Finally, in main Theorem \ref{thm.lima.bean}, we prove that the Brown measure of $b_k(t)$ converges to the Brown measure of the free multiplicative Brownian motion $b(t)$ introduced in \cite{Kemp2016} as $k\to\infty$, in a strong sense: local uniform convergence of densities (sometimes called {\em superconvergence}, cf.\ \cite{BercVoic1995}).  We refer to this limit distribution as {\bf the Lima Bean Law}.  The limit is universal over all bi-invariant step distributions with fixed Hilbert--Shmidt norm; in particular, even the covariance need not match the Brownian motion for the eigenvalues to converge to the Lima Bean Law.

\bigskip

\begin{figure}[h!]
  \includegraphics[scale=0.62]{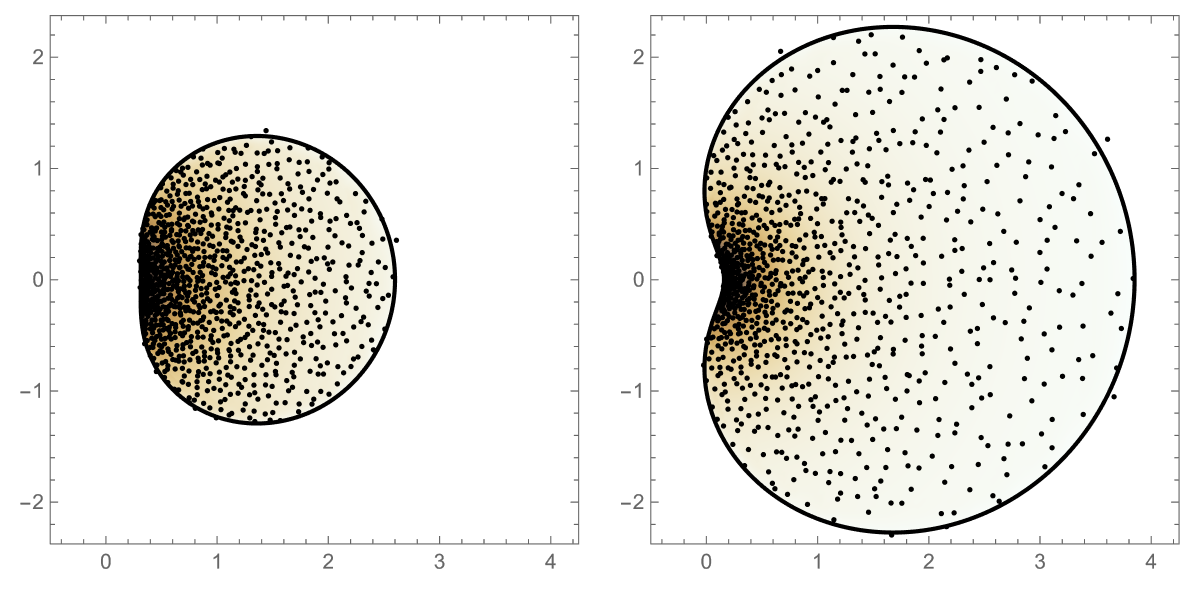}
  \caption{The Brown measure of the free random walk $b_k(t)$ with $k=6$ and $t=1$ (left) and $t=2$ (right).  Also shown are $1000$ eigenvalues of the associated $B^N_k(t)$.  Here the step distribution is circular / Ginibre, corresponding directly to the Brownian motion.
  \label{fig.1}}
\end{figure}

\subsection{Motivation} \hfill

\medskip

Brownian motion is ubiquitous throughout probability theory and its applications.  Indeed, it is the central limit object in the class of stochastic processes on Euclidean spaces.  This fact is made precise by Donsker's theorem, sometimes called Donsker's invariance principle, which states the following.

\begin{theorem}[Donsker, \cite{Donsker}] \label{thm.Donsker}
Let $(X_n)_{n=1}^\infty$ be i.i.d.\ $L^2$ random vectors in $\mathbb{R}^d$ with mean $\mathbb{E}[X_n]=\mathbf{0}$ and covariance $\mathbb{E}[X_nX_n^\top]=I$.  Let $(S(t))_{t\ge 0}$ be the piecewise affine process defined by $S(0)=\mathbf{0}$ and $S(n) = X_1+\cdots+X_n$ for $n\in\mathbb{N}$:
\[ S(t) = X_1+\cdots + X_{\lfloor t\rfloor} + (t-\lfloor t\rfloor)X_{\lfloor t\rfloor +1}. \]
Then the diffusion rescaled sequence $W_k(t) = S(kt)/\sqrt{k}$ converges pathwise in distribution (i.e.\ weakly on $C([0,T],\mathbb{R}^d)$) to the standard Brownian motion on $\mathbb{R}^d$ as $k\to\infty$.
\end{theorem}
\noindent Donsker's theorem and its quantitative variants are the main tools used to approximate and simulate Brownian paths as diffusion rescaled random walks.

Beyond Euclidean space, Brownian motion and its heat kernel marginal distribution remains a central object in many geometric contexts.  On any Riemannian manifold, Brownian motion is the Markov process whose generator is the Laplacian.  It can also be constructed from the Euclidean Brownian motion in the tangent space at the starting point, ``rolled'' onto the manifold via a (Stratonovich) SDE that generalizes the ODE describing rolling without slipping in classical mechanics.  

Specializing to the context of a Lie group $G$, fixing an inner product on the Lie algebra $\mathfrak{g}$ (the tangent space at the identity in the group) yields a left-invariant Riemannian metric on the group, translating vectors at any point back to the identity by the differential of the group multiplication map.  The Brownian motion on $G$ is the Riemannian Brownian motion with respect to this metric.  The rolling map takes a particularly simple form for a matrix Lie group $G\subset \MNC$.  Fix an orthonormal basis $\beta$ for the Lie algebra $\mathfrak{g}\subset \mathfrak{gl}_N = \MNC$, and let $\{W_\xi\}_{\xi\in\beta}$ be a family of i.i.d.\ standard Brownian motions on $\mathbb{R}$.
Set
\begin{equation} \label{eq:Lie.BM} W(t) = \sum_{\xi\in\beta} W_\xi(t)\, \xi, \qquad t \geq 0,
\end{equation}
which is a version of the Brownian motion on the Hilbert space $\mathfrak{g}$.

\begin{definition} \label{def.BM.on.G} The {\bf Brownian motion $(B(t))_{t\ge 0}$ on the group $G$} (for the chosen left-invariant metric) is the strong solution of the Stratonovich SDE
\begin{equation} \label{eq.BMonG} \ds B(t) = B(t)\circ \ds W(t), \qquad B(0) = I. \end{equation}
In terms of It\^o calculus, \eqref{eq.BMonG} may be written as
\[ \ds B(t) = B(t)\,\ds W(t) + \frac12B(t)\Xi\,\ds t \]
where $\Xi = \sum_{\xi\in\beta} \xi^2$ (cf.\ \cite[p.\ 116]{McKean}).
\end{definition}

The principal Donsker-like approximation theorem for Brownian motion on Riemannian manifolds is the Wong--Zakai theorem \cite{WongZakai1,WongZakai2}.  In the context of a matrix Lie group, it takes an elegant form:
Simply replace the Brownian motion $W_t$ in $\mathfrak{g}$ with its interpolated random walk approximation, converting \eqref{eq.BMonG} into an ODE.
To make the statement precise, let $T > 0$ and $\Pi = \{0 = t_0 < \cdots < t_k = T\}$ be a partition of $[0,T]$.

\begin{notation}\label{nota.part}
Fix $t \in \Pi$.
\begin{enumerate}[font=\normalfont,label=(\roman*)]
    \item $t_-$ is the member of $\Pi$ immediately to the left of $t$;
    that is, $0_- = (t_0)_- \coloneqq t_0 = 0$, and $(t_j)_- \coloneqq t_{j-1}$ for all $j=1,\ldots,k$.
    \item $\Delta t \coloneqq t-t_-$, and $|\Pi| \coloneqq \max\left\{ \Delta s : s \in \Pi\right\}$ is the mesh of $\Pi$.
    \item If $f$ is a function from $[0,T]$ to some vector space, then $\Delta_t f \coloneqq f(t) - f(t_-)$.
    \item $\displaystyle \ell_{\Pi}(s) \coloneqq \sum_{r \in \Pi} r_-\,1_{(r_-,r]}(s)$ and $\displaystyle  r_{\Pi}(s) \coloneqq \sum_{r \in \Pi} r\,1_{(r_-,r]}(s)$ for all $s \geq 0$.
\end{enumerate}
\end{notation}

Define the piecewise affine approximation $W_{\Pi}$ of the process $\left.W\right|_{[0,T]}$ by
\begin{equation} \label{eq.Donsker.t-} W_\Pi(t) = W(\ell_{\Pi}(t)) + (t-\ell_{\Pi}(t))\frac{W(r_{\Pi}(t))-W(\ell_{\Pi}(t))}{r_{\Pi}(t)-\ell_{\Pi}(t)}, \qquad 0 \leq t \leq T. \end{equation}
This process is differentiable at $t$ for all $t\notin\Pi$, with derivatives
\[ \frac{\ds W_\Pi(t)}{\ds t} = \frac{W(t_j)-W(t_{j-1})}{t_j-t_{j-1}}, \qquad t_{j-1}<t<t_j. \]
The Wong--Zakai approximation $B_\Pi$ to $\left.B\right|_{[0,T]}$ relative to the partition $\Pi$ is the solution to the ODE
\begin{equation} \label{eq.WZ.intro} \frac{\ds B_\Pi(t)}{\ds t} = B_\Pi(t)\frac{\ds W_\Pi(t)}{\ds t} \; \text{ a.e.}, \qquad B_\Pi(0) = I. \end{equation}
Equation \eqref{eq.WZ.intro} is easily solved explicitly:
\[ B_\Pi(t) = \prod_{s \in \Pi} \exp\left( \frac{s \wedge t -s_- \wedge t}{\Delta s}\Delta_sW\right) , \qquad 0 \leq t \leq T. \]
More transparently, if $t_{j-1} \leq t \leq t_j$, then
\[
B_{\Pi}(t) = \exp\left(\Delta_{t_1}W\right) \exp\left(\Delta_{t_2}W\right)\cdots \exp\left(\Delta_{t_{j-1}}W\right) \exp\left(\frac{t-t_{j-1}}{\Delta t_j}\Delta_{t_j} W\right).
\]
\begin{theorem}[Wong--Zakai, \cite{WongZakai1,WongZakai2}] If $(\Pi_k)_{k\in\mathbb{N}}$ is a sequence of partitions of $[0,T]$ such that $|\Pi_k| \to 0$ as $k \to \infty$, then the sequence $(B_{\Pi_k})_{k \in \N}$ of processes converges weakly in $C([0,T],G)$ to the Brownian motion $B$ on $G$. \end{theorem}
In particular, if $\Pi_k$ is the uniform partition with $\Delta t_j = T/k$ for all $j=1,\ldots,k$, then the increments can be realized as $\Delta_{t_j} W = (T/k)^{1/2}X_j$ where $X_1,\ldots,X_k$ are independent and distributed as $W(1)$.
Thus,
\begin{equation} \label{eq.WZ.exp} B_{\Pi_k}(T) \equaldist \exp\left(\sqrt{\frac{T}{k}}X_1\right)\exp\left(\sqrt{\frac{T}{k}}X_2\right)\cdots \exp\left(\sqrt{\frac{T}{k}}X_k\right). \end{equation}
Equation \eqref{eq.WZ.exp} is an appealing formulation of a random walk on $G$:
Each term $(T/k)^{1/2}X_j$ is in the Lie algebra, so its exponential is in the group.  Setting this aside for computational purposes, we can simplify:
\[ \exp\left(\sqrt{\frac{T}{k}}X_j\right) = I + \sqrt{\frac{T}{k}}X_j + \frac{T}{2k}X_j^2 + o\left(\frac{1}{k}\right), \]
and taking a product of $k$ such exponentials, the terms of order $o(1/k)$ are negligible to the limit as $k\to\infty$.

\begin{definition} Let $\{A_j,C_j\}_{j\in\mathbb{N}}$ be i.i.d.\ random matrices in $\MNC$.  Let $t\ge 0$.  The {\bf matrix random walk} with step distribution $(A_1,C_1)$ is the sequence of random matrices
\begin{equation} \label{eq:random-walk-Berger} B_k(t) = \prod_{j=1}^k \left(I+\sqrt{\frac{t}{k}}A_j+\frac{t}{k}C_j\right), \qquad k \in \N. \end{equation}
In the special case $C_1=0$, we refer to $A_1$ as the step distribution.
\end{definition}

\begin{remark}
We abuse notation slightly here.  Really, the ``random walk'' should be the process evolving with $k$ but without the $O(k^{-1/2})$ rescaling.
More precisely, $k\mapsto B_k(t)$ is a {\em rescaled} matrix random walk.
\end{remark}

\begin{theorem}[Berger, \cite{Berger}] \label{thm.Berger} Let $\mathbb{E}[A_1] = 0$, and suppose $A_1$ has the same covariance as the ``flat'' Brownian motion $W$, from \eqref{eq:Lie.BM}, at time $1$:
\[ \mathbb{E}\left[[A_1]_{ij}[A_1]_{mn}\right] = \mathbb{E}\left[[W(1)]_{ij}[W(1)]_{mn}\right] \qquad 1\le i,j,m,n\le N. \]
If $t \geq 0$, then the matrix random walk $(B_k(t))_{k \in \N}$ from \eqref{eq:random-walk-Berger} converges in distribution to $\Phi(t)$, where $\Phi$ is the strong solution of the It\^o SDE
\[ \ds \Phi(t) = \Phi(t)\,\ds W(t) + \Phi(t)\mathbb{E}[C_1]\,\ds t, \qquad \Phi(0) = I. \]
In particular, $B_k(t)$ converges in distribution to the Brownian motion $B(t)$ on $G$ if and only if $\mathbb{E}[A_1]=0$, $A_1$ has the same covariance as $W(1)$, and $\mathbb{E}[C_1] = \frac12\Xi = \frac12\sum_{\xi\in\beta} \xi^2$.
\end{theorem}

\begin{remark} Since the inner product on $\mathfrak{g}$ can be chosen with impunity, the only real condition on the collection $\beta\subset \MNC$ is that it is linearly independent.  In fact, linear independence is not required for Berger's result, only that $\beta$ contains at most $N^2$ matrices. \end{remark}

\subsection{Main results} 

\subsubsection{Random walk approximations to Brownian motion} \hfill

\medskip

Berger's Theorem \ref{thm.Berger} and related results in \cite{Watkins} for matrix random walk approximations of Brownian motion on Lie groups are generally proved using Markov semigroup methods.  Though powerful, they are less adept at producing quantitative bounds for rates of convergence.  Our first results concern the special case of Theorem \ref{thm.Berger} in which the step distribution comes from the ``flat'' Brownian motion itself (not just matching covariance).  In this case, the matrix random walk can be realized as a solution of a kind of SDE.

\begin{lemma} \label{lem.simple.SDE} Let $\Pi$ be a partition of $[0,T]$ (and cf.\ Notation \ref{nota.part}).  Let $(W(t))_{t\ge 0}$ and $\Xi$ be as in Definition \ref{def.BM.on.G} and \eqref{eq:Lie.BM}.
The stochastic integral equation
\begin{equation} \label{eq.simple.SDE}
X(t) = I + \int_0^t X(\ell_{\Pi}(s))\, \ds W(s)+ \frac12\int_0^t X(\ell_{\Pi}(s))\Xi\,\ds s, \qquad 0 \leq t \leq T.
\end{equation}
has unique(-up-to-indstinguishability) solution
\begin{equation} \label{eq.Bpi.soln}
X(t) = B_\Pi(t) \coloneqq \prod_{s \in \Pi} \left(I+ W(s \wedge t)-W(s_- \wedge t) + \frac12\Xi (s \wedge t -s_- \wedge t)\right).
\end{equation}
\end{lemma}

\begin{remark}
Note well that $B_{\Pi}$ in Lemma \ref{lem.simple.SDE} is \emph{not} the same as in \eqref{eq.WZ.intro} ff.
\end{remark}
Lemma \ref{lem.simple.SDE} is easily verified by direct computation.

Note that the increments $\{\Delta_s W = W(s) - W(s_-) \colon s \in \Pi\}$ are independent, and
\[ \Delta_s W  \equaldist W\big(\sqrt{\Delta s}\big) \equaldist \sqrt{\Delta s}\,W(1). \]
Hence, if $\Pi = \Pi_k$ is the uniform partition of $[0,T]$ with $k$ steps of length $T/k$, then \eqref{eq.simple.SDE} shows that $B_{\Pi_k}(T)$ is distributed as a matrix random walk $B_k(T)$, as in \eqref{eq:random-walk-Berger}, with step distribution $A_1 = W(1)$.

Lemma \ref{lem.simple.SDE} allows the use of tools from stochastic calculus to analyze the matrix random walk.  We use this to prove $L^p$ convergence for all finite $p$ and a.s.\ convergence for fine enough partitions.

In the following, for $A\in \MNC$, $|A|^2 = \frac1N\mathrm{Tr}(A^\ast A)$ denotes the (squared) normalized Hilbert--Schmidt norm.

\begin{theorem} \label{main.thm.1.MoBettaWZ} For any partition $\Pi$ of $[0,T]$, let $B_\Pi$ denote the process in \eqref{eq.Bpi.soln}, and let $B$ be the Brownian motion in \eqref{eq.BMonG} associated to the ``flat'' Brownian motion $W$ in \eqref{eq:Lie.BM} determined by the set $\beta\subset \MNC$. 
If $2 \leq p<\infty$, there is a constant $K = K(N,\beta,T,p)<\infty$ such that
\[ \mathbb{E}\left[\sup_{0\le t\le T} |B_\Pi(t)-B(t)|^p\right] \le K |\Pi|^{\frac{p}{2}}. \]
\end{theorem}

Theorem \ref{main.thm.1.MoBettaWZ} is proved in Section \ref{sec.main.thm.1.MoBettaWZ}.  An immediate corollary is that the matrix random walk $B_k(t)$ with step distribution $W(1)$ converges to the Brownian motion $B(t)$ in $L^p$ for every $p<\infty$.  What's more:

\begin{corollary} \label{cor.BC} If $(\Pi_k)_{k\in\mathbb{N}}$ is a sequence of partitions of $[0,T]$ such that $\sum_k |\Pi_k|^\alpha <\infty$ for some $\alpha>0$ (in particular, if $|\Pi_k|$ decays polynomially), then the sequence $(B_{\Pi_k})_{k \in \N}$ of processes converges to the Brownian motion $B$ pathwise on $[0,T]$ a.s.\ as $k\to\infty$. \end{corollary}

\begin{proof} Fix $\epsilon>0$, and let
\[ E_k(\epsilon) \coloneqq \left\{ \sup_{0\le t\le T} |B_{\Pi_k}(t)-B(t)| \ge \epsilon \right\}. \]
Let $p>\min\{2\alpha,2\}$.  By Markov's inequality and Theorem \ref{main.thm.1.MoBettaWZ},
\[\mathbb{P}(E_k(\epsilon)) \le \frac{1}{\epsilon^p}\mathbb{E}\left[\sup_{0\le t\le T}|B_{\Pi_k}(t)-B(t)|^p\right] \le \frac{K}{\epsilon^p}|\Pi_k|^{\frac{p}{2}}.\]
Since $\sum_k |\Pi_k|^\alpha <\infty$, it follows that $|\Pi_k|\to 0$; hence, for large $k$, $|\Pi_k|^{p/2} \le |\Pi_k|^\alpha$ by the choice $p>2\alpha$.  Therefore, $\sum_k \mathbb{P}(E_k(\epsilon)) \le K\epsilon^{-p} \sum_k |\Pi_k|^{p/2} < \infty$.

Hence, by the Borel--Cantelli lemma, $\mathbb{P}(E_k(\epsilon) \text{ for infinitely many }k)=0$; i.e.\ there is some random $N$ so that, with probability $1$, $|B_{\Pi_k}(t)-B(t)|<\epsilon$ for all $t\in[0,T]$ and all $k\ge N$.  This concludes the proof.
\end{proof}

Let us specialize Theorem \ref{main.thm.1.MoBettaWZ} to the case of the standard Brownian motion on $\GL(N,\mathbb{C})$ --- meaning that $\beta=\beta_N$ is an orthonormal bases for the whole matrix space $\MNC$.  In service of the large-$N$ limit that principally concerns us in this paper, the inner product used to define $\beta_N$ is chosen to be scaled up:
\[
\langle A,B\rangle = N\operatorname{Re}\operatorname{Tr}(B^\ast A), \qquad A,B \in M_{N \times N}(\mathbb{C}).
\]
(Hence, the associated norm is actually $N|\cdot|$; this is of no concern, as the inner product used to define $\beta_N$ need not relate to the one used to measure the convergence in Theorem \ref{main.thm.1.MoBettaWZ}.)  Of note is the fact that, for this inner product,
\begin{equation} \label{eq.Xi=0}
\Xi = \sum_{\xi\in\beta_N} \xi^2 = 0
\end{equation}
for any orthonormal basis $\beta_N$; see \cite{Kemp2016}.

The best estimates we prove for the constant $K(N,\beta_N,T,p)$ in Theorem \ref{main.thm.1.MoBettaWZ} diverge with $N$ for each $p$ and $T$.  Nevertheless, both $B_\Pi$ and $B$ have large-$N$ limits in the free probability sense: The ``flat'' Brownian motion $W = W^N$ (which for this $\beta_N$ is a scaled Ginibre ensemble for each time) converges in finite dimensional $\ast$-distributions (cf.\ Definition \ref{def.conv.fd*d}) to a {\bf circular Brownian motion} $w$, and the $\GL(N,\mathbb{C})$ Brownian motion $B = B^N$ converges in finite dimensional $\ast$-distributions to the {\bf free multiplicative Brownian motion} $b$ in a $W^\ast$-probability space, as proved in \cite{Kemp2016}.  (See Section \ref{sect.fbm} for a discussion of the relevant free probability constructions.)

Following the same basic outline as the proofs of Lemma \ref{lem.simple.SDE} and Theorem \ref{main.thm.1.MoBettaWZ}, we have the following mirror results in the free world, which strengthen $L^p$ convergence for $p < \infty$ to $L^\infty$ convergence.

\begin{lemma}\label{lem.free.simple.SDE}
Fix a $W^*$-probability space $(\mathscr{A},\varphi)$ and a circular Brownian motion motion $w = (w(t))_{t \geq 0}$ in $\cA$.
Let $\Pi$ be a partition of $[0,T]$.
The free stochastic integral equation
\[
x(t) = 1 + \int_0^t x(\ell_{\Pi}(s))\,\ds w(s), \qquad 0 \leq t \leq T,
\]
has unique solution
\[
x(t) = b_{\Pi}(t) \coloneqq \prod_{s \in \Pi}\left(1 + w(s \wedge t) - w(s_- \wedge t)\right).
\]
In particular, if $\Pi_k$ is the uniform partition on $[0,T]$ with $k$ steps, then
\[ b_{\Pi_k}(T) \equaldist \left(1+\sqrt{\frac{T}{k}}a_1\right)\left(1+\sqrt{\frac{T}{k}}a_2\right)\cdots\left(1+\sqrt{\frac{T}{k}}a_k\right), \]
where $a_1,\ldots,a_k$ are freely independent from each other and all have the same $\ast$-distribution as $w(1)$.
\end{lemma}

\begin{theorem} \label{thm.free.MoBettaWZ}
Retain the setting and notation from Lemma \ref{lem.free.simple.SDE}.
Let $b$ be the free multiplicative Brownian motion \eqref{eq.free.mult.BM}, defined as the unique solution to the free stochastic differential equation
\[
\ds b(t) = b(t)\,\ds w(t), \qquad b(0) = 1.
\]
There is a constant $\overline{K} = \overline{K}(T)<\infty$ such that
\[ \sup_{0 \leq t \leq T}\|b_\Pi(t) - b(t)\|_{\infty} \le \overline{K} |\Pi|^{\frac12}. \]
\end{theorem}

\begin{remark}
Since $\|\cdot\|_p\le\|\cdot\|_\infty$, taking $p$th powers in the preceding equation yields the same $L^p$ estimates as in Theorem \ref{main.thm.1.MoBettaWZ}, with $K$ replaced by $\overline{K}^p$.
\end{remark}

As before, Lemma \ref{lem.free.simple.SDE} is a simple calculation.
Theorem \ref{thm.free.MoBettaWZ} is proved in Section \ref{sec.thm.free.MoBettaWZ}.  Convergence in $L^\infty$, or just convergence in $L^p$ for all finite $p$, with a bound that is locally uniform in time implies convergence in finite-dimensional $\ast$-distributions, cf.\ Proposition \ref{prop.Lp>D*}.  Hence, we have proved the following more natural conclusion.

\begin{corollary} \label{cor.NCFDD} If $(\Pi_k)_{k\in\mathbb{N}}$ is a sequence of partitions of $[0,T]$ such that $|\Pi_k|\to 0$ as $k\to\infty$, then $(b_{\Pi_k})_{k \in \N}$ converges in finite-dimensional $\ast$-distributions to $b|_{[0,T]}$. \end{corollary}

\subsubsection{Convergence of Eigenvalues\label{sect.intro.eigenvalues}} \hfill

\medskip

For any matrix $A\in \MNC$, denote by $\mu_A$ the {\bf empirical spectral distribution} (ESD)
\begin{equation} \label{eq.def.ESD} \mu_A = \frac1N\sum_{j=1}^N \delta_{\lambda_j(A)} \end{equation}
where $\{\lambda_j(A)\}_{j=1}^N$ are the eigenvalues of $A$ counted with multiplicity.  If $A$ is a random matrix, $\mu_A$ is a random probability measure on $\mathbb{C}$.

For sequences of random measures $(\mu_n)_{n\in\mathbb{N}}$, the standard notions of convergence are {\em weak convergence almost surely or in probability}.  That is: $\mu_n\rightharpoonup\mu$ if
\[ \int_{\mathbb{C}} f\,d\mu_n\to \int_{\mathbb{C}} f\,d\mu \qquad \forall\,f\in C_b(\mathbb{C}) \]
a.s. or in probability.

Note that the eigenvalues of a matrix $A$ are the zeroes of the characteristic polynomial $\chi_A(\lambda) = \det(A-\lambda I)$, whose coefficients are polynomials in the entries of $A$.  The zeroes of any polynomial are continuous functions of its coefficients; hence, there is a continuous map $A\mapsto (\lambda_1(A),\ldots,\lambda_N(A))$ defining the eigenvalues, cf.\ \cite[Theorem 5.2]{Kato1982}.  It therefore follows from  Corollary \ref{cor.BC} that the ESD of the matrix random walk $B_k(t)$ converges weakly a.s.\ to the ESD of the Brownian motion $B(t)$.  Indeed, appealing to \eqref{eq.def.ESD}, for $f\in C_b(\mathbb{C})$
\[ \int_{\mathbb{C}} f\,d\mu_{B_k(t)} = \frac1N\sum_{j=1}^N f(\lambda_j(B_k(t))) \to \frac1N\sum_{j=1}^N f(\lambda_j(B(t))) = \int_{\mathbb{C}} f\,d\mu_{B(t)} \]
a.s.\ as $k\to\infty$, where the convergence follows from the facts that $B_k(t)\to B(t)$ a.s. and that the functions  $f\circ\lambda_j\colon \MNC\to\mathbb{C}$ are continuous.

The preceding discussion may be misleading if it gives the impression that convergence of eigenvalues is generally easy.  While this may be true for convergence of matrices {\em in a fixed dimension $N$}, our principal interest is in limits of the measures $\mu_{B_k(t)}$ for matrix random walks $B_k(t) = B_k^N(t)$ {\em as $N\to\infty$}.  Let us introduce and re-emphasize notation for the main theorem.

\begin{notation} \label{notation.Bk->bk}
We fix the following notation to be used throughout what follows.
\begin{itemize}
    \item Let $A^N_j = U^N_j T^N_j V^{N\ast}_j$, $1\le j\le k$ be unitarily bi-invariant random matrices (see \eqref{eq.bi-inv.intro}), all independent.  Denote by $\nu^N_j$ the $\mathrm{ESD}$ of $T^N_j$.
    \item Let $U^N_0$ be a unitary random matrix independent from $A^N_1,\ldots,A^N_k$, and let $T^N_0$ be a positive random matrix for which $U^N_0 = \exp(iT^N_0)$. Denote by  $\nu^N_0$ the $\mathrm{ESD}$ of $T^N_0$.
    \item As in \eqref{eq.Bk}, for $t>0$ denote by $B_k^N(t)$ the random matrix
    \[ B_k^N(t) = \prod_{j=1}^k \left(I+\sqrt{\frac{t}{k}}A_j^N\right). \]
\end{itemize}
\end{notation}

\begin{remark} \begin{enumerate}
    \item The characterization of unitarily bi-invariant random matrices in terms of their singular value decomposition can be found below in Proposition \ref{prop.bi-inv}.
    \item There is no loss in generality in choosing a positive matrix $T^N_0$ for the unitary $U_0^N$; the eigenvalues can always be chosen in $[0,2\pi)$ or any other interval of length $2\pi$.
    \item Note that the time parameter $t>0$ could be absorbed into the step matrices $A_j^N$; we will find it useful later to normalize the $A_j^N$ in Hilbert--Schmidt norm, and study the overall evolution in time.
\end{enumerate}   
\end{remark}

We are primarily interested in this paper in the normalized matrix random walk with steps $A_j^N$ and initial distribution $U_0^N$: I.e. $U_0^N B_k^N(t)$.  Studying the eigenvalues of this highly non-normal random matrix ensemble has been challenging; on the other hand, its $\ast$-distribution (cf.\ Section \ref{sec.*-distribution}, i.e.\ traces of polynomials in the matrix and its adjoint, is well-understood.

\begin{lemma} \label{lem.Bk->bk} Suppose for each $j\in\{0,1,\ldots, k\}$ the empirical measure $\nu_j^N$ (cf.\ Notation \ref{notation.Bk->bk}) converges weakly a.s.\ to a probability measure $\nu_j$.  There is a $W^\ast$-probability space (cf.\ Section \ref{sect.*-prob}) containing operators $a_1,\ldots,a_k$ and $u_0$, all freely independent, so that, for all $t>0$, $U^N_0 B^N_k(t)$ converges in $\ast$-distribution to
\begin{equation} \label{eq.ubk} u_0 b_k(t) = u_0\prod_{j=1}^k \left(1+\sqrt{\frac{t}{k}}a_j\right). \end{equation}
The operators $a_j$ are $\mathscr{R}$-diagonal (cf.\ Section \ref{sec.R-diag}) determined by the conditions that $|a_j| = \sqrt{a_j^\ast a_j}$ has distribution $\nu_j$; the unitary operator $u_0$ has the form $u_0=\exp(i\vartheta_0)$ for some selfadjoint operator $\vartheta_0$ with distribution $\nu_0$.
\end{lemma}

Lemma \ref{lem.Bk->bk} is proved in Section \ref{sec.R-diag}, following a general discussion of $\mathscr{R}$-diagonal elements and limits of bi-invariant ensembles.  Convergence in $\ast$-distribution does not generally imply convergence of the $\mathrm{ESD}$, as Example \ref{ex.Jordan.block} shows; but it does point to what the limit distribution of eigenvalues {\em ought} to be: the {\bf Brown measure} of the $\ast$-distribution limit, cf.\ Definition \ref{def.Brown.measure}.

As detailed in Sections \ref{sect.Brown.measure} and \ref{sect.SRT}, proving convergence to the Brown measure requires fine control of small singular values of of $B_k^N(t)$ as $N$ grows.  To guarantee this fine control, stronger assumptions than just weak convergence of singular values are required.

\begin{assumption} \label{assump.sing.conv} Let $\nu^N$ denote a sequence of empirical measures of (random) points $0\le \sigma_{\min}^N=:\sigma_1^N\le\cdots\le\sigma_N^N:=\sigma_{\max}^N$.
%Let $T=T^N$ be a selfadjoint random matrix ensemble with singular values $\sigma_{\min}(T) = \sigma_1(T)\le \cdots\le\sigma_N(T) = \sigma_{\max}(T)$. (Note that $\sigma_{\max}(T) = \|T\|$ while $\sigma_{\min}(T) = \|T^{-1}\|^{-1}$ or $=0$ if $T$ is not invertible.) Let $\nu^N$ denote the empirical measure of singular values
\[ \nu^N = \frac{1}{N}\sum_{j=1}^N \delta_{\sigma_j^N} \]
Let $\nu$ be a deterministic measure on $[0,\infty)$.  Our assumptions are:
\begin{enumerate}
    \item[(RW1)] \label{assump.SV1} There is some $M<\infty$ so that $\mathbb{P}\{\sigma_{\max}^N\le M\}=1$ for all large $N$.
    \item[(RW2)] \label{assump.SV2} There are $\alpha,\beta,\gamma>0$ so that, for all $\epsilon>0$,
    \[ \mathbb{P}\{\sigma_{\min}^N\le\epsilon\} \le \gamma \epsilon^{\alpha} N^{\beta}. \]
    \item[(RW3)] \label{assump.SV3} There are $\Upsilon,\delta>0$ so that the Wasserstein distance $\mathcal{W}_1$ between $\nu^N$ and $\nu$ satisfies $\mathbb{P}\{\mathcal{W}_1(\nu^N,\nu)\le \Upsilon N^{-\delta}\}=1$ for all large $N$.
   \item[(RW4)] \label{assump.SV4} The measure $\nu$ is absolutely continuous with a density $\rho$ which satisfies
   \[ \int_0^\infty t\rho(t)^3\,dt < \infty. \]
\end{enumerate}
\end{assumption}

We will apply Assumption \ref{assump.sing.conv} to the singular values $\sigma_j^N = \sigma_j(A^N_j)$ of the step matrices in the random walk \eqref{eq.Bk}.  Assumptions \ref{assump.SV1} matches Assumption \ref{SRT.1} from the Single Ring Theorem, and holds true for the classic random matrix ensembles such as the Ginibre Ensemble (cf.\ Definition \ref{def.classic.ensembles}). Assumption \ref{assump.SV2} is equivalent to the condition that $\sigma_{\min}(T^N) \gtrsim N^{-\alpha'}$ with probability $\gtrsim 1-N^{-\beta'}$ for some $\alpha',\beta'>0$, while Assumption \ref{assump.SV3} gives a natural quantitvative speed of convergence of the singular value distribution, measured in terms of the ubiquitous Wasserstein distance (see Section \ref{section.Wegner} for definition and details).  Finally, Assumption \ref{assump.SV4} is a mild regularity condition.  The case $\nu=\delta_1$ is to allow for the case that the steps $A_j^N$ themselves are (asymptotically) Haar Unitary.  The alternative density requirements allows tools from free entropy theory to come into play.

Under these conditions, we can upgrading of Lemma \ref{lem.Bk->bk} from convergence in $\ast$-distribution to the much more difficult convergence of $\mathrm{ESD}$.

\begin{theorem} \label{thm.conv.ESD} For $t>0$, let %$U_0^N$,
$A_j^N$, $B_k^N(t)$, and $\nu_j^N$ for $0\le j\le k$ be as in Notation \ref{notation.Bk->bk}.  Suppose that $\nu_j^N$ converges weakly a.s.\ to a probability measure $\nu_j$.  Suppose further that the measures $\nu_j^N$ and $\nu_j$ each satisfy \ref{assump.SV1}, \ref{assump.SV2}, and \ref{assump.SV3} of Assumption \ref{assump.sing.conv}.  Finally, suppose that {\bf one} of the following holds:
\begin{itemize}
    \item \ref{assump.SV4} of Assumption \ref{assump.sing.conv} holds; or
    \item Each $A^N_j$ is a polynomial in independent $\mathrm{GUE}$s; or
    \item Each $A^N_j$ is a polynomial in Haar Unitary Ensembles.
\end{itemize}

Let %$u_0$ and
$b_k(t)$ denote the $\ast$-distribution limits from \eqref{eq.ubk}.  Then the $\mathrm{ESD}$ of %$U_0^N 
$B_k^N(t)$ converges weakly in probability to the Brown measure of %$u_0 
$b_k(t)$.
\end{theorem}

Our proof of Theorem \ref{thm.conv.ESD} has two parts:
\begin{itemize}
    \item Polynomial lower bounds on the decay rate of the smallest singular value of $B_k^N(t)-zI$ for almost every $z\in\mathbb{C}$.
    \item Wegner estimates on anti-concentration of moderately small singular values of $B_k^N(t)$.
\end{itemize}
In broad terms, this follows the method of proof of the Circular Law that is nicely summarized in \cite{Bordenave-Chafai-circular} and used in \cite{GuionnetKZ-single-ring,GuionnetZeitouni}, and more recently in \cite{Cook} where it was used to prove convergence of the $\mathrm{ESD}$ for quadratic polynomials in independent $\mathrm{GUE}$ ensembles.

Control of the smallest shifted singular value, i.e.\ control of the psuedospectrum (cf.\ Section \ref{sect.Brown.measure}), is achieved using Rudelson--Vershynin's very powerful Theorem \ref{thm.RV} for arbitrary perturbations of Haar Unitary Ensembles.  Wegner estimates in this case (for more general selfadjoint polynomials in bi-invariant matrices) are proved using results from \cite{Parraud2022}, coupled with methods from the theory of free entropy and free information.  The proof can be found in Section \ref{sect.proof.conv.eig}.

\subsubsection{Linearization and Freeness} \hfill

\medskip

Theorem \ref{thm.conv.ESD} identifies, in principle, the large-$N$ limit of the empirical distribution of eigenvalues of the matrix random walk $U_0^N B_k^N(t)$: it is the Brown measure of the free random walk $u_0b_k(t)$ in \eqref{eq.ubk}.  Actually computing the Brown measure is another difficult task, which forms the third leg of our main results.

Dispensing with the initial condition $u_0$ for the moment, a useful technique here is {\em Linearization}.  Given random variables $x_1,\ldots,x_k$ in a $W^\ast$-probability space $(\mathscr{A},\varphi)$ We relate the product $x_1\cdots x_k$ to the matrix-valued random variable
\[ X = \begin{bmatrix}
0 & x_1 & 0 & \cdots & 0 & 0 \\
0 & 0 & x_2 & \cdots & 0 & 0 \\
\vdots & \vdots & \vdots & \ddots & \vdots & \vdots \\
0 & 0 & 0 & \cdots & x_{k-2} & 0 \\
0 & 0 & 0 & \cdots & 0 & x_{k-1} \\
x_k & 0 & 0 & \cdots & 0 & 0
\end{bmatrix}.
\]
The reason is that $X^k$ is diagonal, with all diagonal entries given as cyclic permutations of the product $x_1\cdots x_k$.  Provided $x_1,\ldots,x_k$ are freely independent, it then follows that the Brown measure $\mu_{x_1\cdots x_k}$, computed with respect to the state $\varphi$ on $\mathscr{A}$, is the push-forward of the Brown measure $\mu_X$ by the map $z\mapsto z^k$; here $\mu_X$ is computed with respect to the state $X\mapsto \varphi[\mathrm{tr}(X)]$ where $\mathrm{tr} = \frac{1}{k}\mathrm{Tr}$ is the normalized trace.  This is proved in Proposition \ref{prop.linearization}.

\begin{remark} If $x_1,\ldots,x_k$ are in $\MNC$, random or not, then the eigenvalues of $x_1\cdots x_k$ are the same as the eigenvalues of $X^k$, each with multiplicity amplified by $k$; in the finite-dimensional matrix case, the freeness requirement in \ref{prop.linearization} is not necessary. \end{remark}

We may view the above matrix $X$ as an {\em operator-valued} random variable, in $\mathscr{A}\otimes\MkC$.  Using the matrix unit basis $E_{i,j}\in\MkC$ (where the $(\ell,m)$-entry of $E_{i,j}$ is $\delta_{i\ell}\delta_{jm}$), we can express $B$ as
\[ X = \sum_{j=1}^k x_j\otimes E_{j,j+1} \quad \text{where }k+1\equiv 1. \]
(Here we are using the standard identification of the Kronecker product of matrices with the tensor product of endomorphisms on vector spaces.)

For our setup, $x_j = 1+(\frac{t}{k})^{1/2} a_j$, which means the associated linearized matrix $X$ decomposes as a sum.  Using the unitary bi-invariance of the $\mathscr{R}$-diagonal steps $a_j$, we may incorporate the unitary initial condition into the linearization by noting that
\begin{equation} \label{eq.ubk.linearized} u_0 b_k(t) = u_0\prod_{j=1}^k \left(1+\sqrt{\frac{t}{k}}a_j\right) \equaldist \prod_{j=1}^k \left(u_0^{1/k}+\sqrt{\frac{t}{k}}a_j\right) \end{equation}
where $u_0^{1/k}$ is any fixed unitary whose $k$th power is $u_0$.  This is proved in Theorem \ref{thm.unitary.distribute}.  As such, we define
\begin{equation}
     \label{eq.Z.A.def.intro}
     Z = \sum_{j=1}^k u_0^{1/k}\otimes E_{j,j+1} \quad\textrm{and}\quad A = \sum_{j=1}^k a_j\otimes E_{j,j+1}.
\end{equation}
Then the linearization of \eqref{eq.ubk.linearized} yields the following.

\begin{proposition} \label{prop.AZ} The Brown measure of $u_0b_k(t)$ is the push-forward of the Brown measure of $Z+\sqrt{t/k}\,A$ under $z\mapsto z^k$. \end{proposition}

Provided the separate $\ast$-distributions of $Z$ and $A$ in $(\mathscr{A}\otimes\MkC,\varphi\otimes\mathrm{tr})$ can be computed, the distribution of their sum may yet remain mysterious.  Although the $\mathscr{A}$-components $u_0^{1/k}$ and $a_1,\ldots,a_k$ forming $Z$ and $A$ are freely independent with respect to $\varphi$, that typically does not translate to (scalar) freeness of the operator-valued random variables.  (In general it leads to operator-valued freeness, a.k.a.\ freeness with amalgamation, cf.\ \cite[Chapter 9]{MingoSpeicherBook}, which is weaker.)  Fortunately, there is enough invariant structure in our setting to give both true scalar freeness and the explicit distribution of $A$.  The following theorem sets us on a path to use the analytic tools of free probability to compute this Brown measure.

\begin{theorem} \label{thm.linearized.freeness}
Let $k\in\mathbb{N}$, and let $u_0,a_1,\ldots,a_k$ be freely independent operators in a $W^\ast$-probability space $(\mathscr{A},\varphi)$, where $u_0$ is unitary and $a_1,\ldots,a_k$ are $\mathscr{R}$-diagonal (cf.\ Section \ref{sec.R-diag}).  Let $Z,A\in\mathscr{A}\otimes\MkC$ be defined as in \eqref{eq.Z.A.def.intro}.

Then $Z$ is unitary, and $A$ is $\mathscr{R}$-diagonal in $(\mathscr{A}\otimes\MkC,\varphi\otimes\mathrm{tr})$.  Moreover: if $a_1,\ldots,a_k$ all have the same distribution, then $A$ has the same distribution as $a_1$, and $Z$ and $A$ are freely independent in $(\mathscr{A}\otimes\MkC,\varphi\otimes\mathrm{tr})$.
\end{theorem}

Theorem \ref{thm.linearized.freeness} is proved in Section \ref{sect.freeness.ZA}, following a much more general development in Section \ref{sec.linearized.freeness} which is of independent interest.  Indeed, similar freeness and distributional properties hold for more general {\em circulant invariant block matrices}, whose $\ast$-distribution is invariant under the conjugation action of a fixed full-cycle permutation of $k$.  (For example: a matrix with all $k^2$ entries freely independent, and identically distributed along each circulant subdiagonal.) See Theorem \ref{thm.general.freeness} for a general theorem about freeness in this setting.

\subsubsection{Brown Measures of Free Random Walks} \hfill

\medskip

Using complex analytic tools from free probability theory, we analyze the Brown measure of the free random walk in Section \ref{sect.BM.computation.1}.  We do so using the characterization in Proposition \ref{prop.AZ} relating the free random walk to the operator-valued matrices $A,Z$ in \eqref{eq.Z.A.def.intro}, and their free probability properties as expressed in Theorem \ref{thm.linearized.freeness}, together with tools from \cite{BercoviciZhong2022Rdiag}.

\begin{theorem} \label{thm.Brown.Measure.1} Let $k\in\mathbb{N}$, $t>0$, and let $a,a_1,\ldots,a_k$ be freely independent, identically distributed $\mathscr{R}$-diagonal operators in a $W^\ast$-probability space $(\mathscr{A},\varphi)$, normalized so that $\|a\|_2=1$.  Let $u_0\in\mathscr{A}$ be a unitary operator freely independent from $\{a_1,\ldots,a_k\}$.  Let $Z,A\in\mathscr{A}\otimes\MkC$ be defined as in \eqref{eq.Z.A.def.intro}.  Define the open set
\begin{equation} \label{eq.Omega.t.def.intro} \Omega_k(t) := \left\{\lambda\in\mathbb{C}\colon \|(Z-\lambda I)^{-1}\|_2^2>\frac{k}{t}\;\;\&\;\; \lambda\in\mathbb{C}\colon |\lambda|^2 > \frac{t}{k\|a^{-1}\|_2^2}-1 \right\}\end{equation}
where $\|a^{-1}\|_2:=\infty$ if $a$ does not have an $L^2$ inverse.  Also define $S_k\subseteq\mathrm{spec}(Z)$ by
\begin{equation} \label{eq.Sk.def.intro} S_k := \{\lambda\in\mathbb{C}\colon \varphi(\ker(Z-\lambda I))+\varphi(\ker(a))\ge 1\} \end{equation}
where $\mathrm{ker}(X)$ denotes the orthogonal projection onto the kernel of the operator $X$.  Then the following hold true.
\begin{enumerate}
    \item The support of the Brown measure of $Z+\sqrt{t/k}\,A$ is equal to $\overline{\Omega_k(t)}$.
    \item $S_k\subset \Omega_k(t)$ for all $t>0$, is a finite set, and $S_k=\emptyset$ for all large $k$.
    \item In the decomposition $\mu_{\mathrm{c}} + \mu_{\mathrm{sc}}+\mu_{\mathrm{p}}$ of the Brown measure of $Z+\sqrt{t/k}\,A$,
    \begin{itemize}
    \item $\mu_{\mathrm{c}}$ has a real-analytic density on $\Omega_k(t)$.
    \item $\mu_{\mathrm{p}}$ is an atomic measure supported on $S_k$.
    \item $\mu_{\mathrm{sc}}$ is a singular continuous measure supported on $\partial\Omega_k(t)$.
\end{itemize}
\end{enumerate}
\end{theorem}

\begin{remark} The density of $\mu_c$ is never $0$ for $t>0$.  We do not know if the Brown measure can actually have mass on the boundary, i.e.\ it is possible that $\mu_{\mathrm{sc}}$ is always $0$; it is difficult to tell whether concentrated mass in the initial condition could result in mass on the boundary at positive times.
\end{remark}

From Proposition \ref{prop.AZ} and Theorem \ref{thm.Brown.Measure.1}, we derive the following characterization of the Brown measure of $u_0 b_k(t)$.

\begin{theorem} \label{thm.Brown.Measure.2} Let $k,t,a,a_1,\ldots,a_k,u_0$ be as in the statement of Theorem \ref{thm.Brown.Measure.1}.
Define the disk $D_k(a,t)$, the ``survival time'' $T_k(u_0,\,\cdot\,)\colon\mathbb{C}\to \mathbb{R}_+$, and the ``Lima Bean'' $\Sigma(k,u_0,t)$ as follows:
\begin{equation} \label{eq.lifetime.k} T_k(u_0,z) := \frac{k(|z|^{2/k}-1)}{|z|^2-1}\left(\int_{\mathbb{T}} \frac{\mu_{u_0}(d\xi)}{|\xi-z|^2}\right)^{-1} \end{equation}

\begin{equation} \label{eq.Dkat} D_k(a,t):= \left\{z\in\mathbb{C}\colon \vert z\vert^{2/k} < \frac{t}{k\|a^{-1}\|_2^2}-1\right\} \end{equation}

\begin{equation} \label{eq.Sigma.intro} \Sigma_k(u_0,t) := \left\{z\in\mathbb{C}\colon T_k(u_0,z)<t\right\}.
\end{equation}

Then the Brown measure of the free random walk $u_0b_k(t)$ has support equal to the closure $\overline{\Sigma_k(u_0,t)}\setminus D_k(a,t)$.  (Note that $\overline{D}_k(a,t)$ is empty if $t<k\|a^{-1}\|_2^2$.)  The Brown measure may have atoms in the finite set $S_k^\ast\subset \mathrm{spec}(u_0)$, the push-forward of $S_k$ \eqref{eq.Sk.def.intro} under $z\mapsto z^k$, which is contained in $\Sigma_k(u_0,t)\setminus D_k(a,t)$. On $\Sigma_k,u_0,t)\setminus (\overline{D_k(a,t)}\sqcup S_k^\ast)$, the Brown measure has a density $\rho_k(t,\cdot)$ which is real analytic except for a pole at $0$.  \end{theorem}

\begin{figure}[b!]  \includegraphics[scale=0.5]{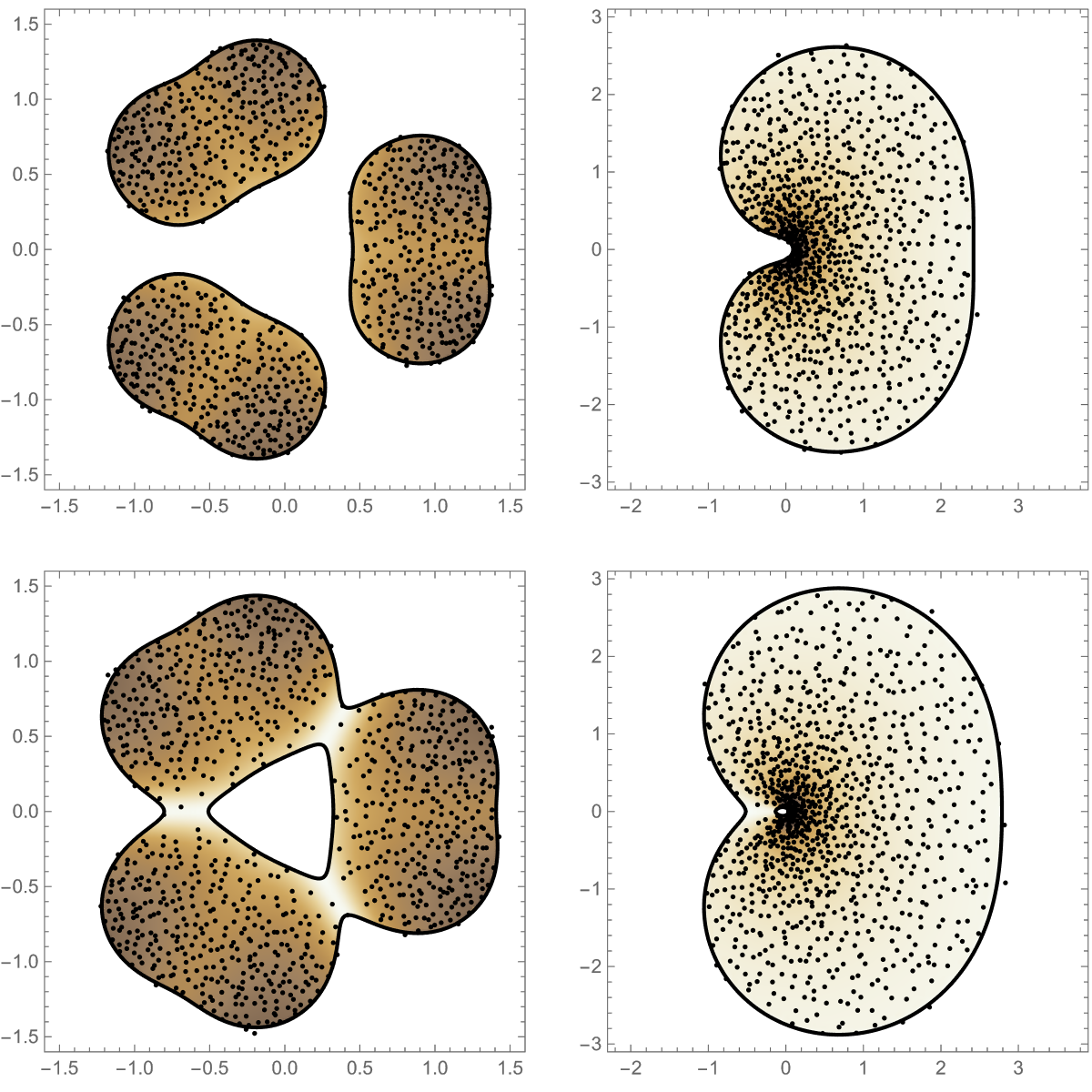}
  \caption{Two instances of the Brown measure of $Z+\sqrt{t/k}\,A$, alongside the Brown measure of the free random walk $u_0 b_k(t)$, with $k=3$ and $t=2.2$ (top) and $t=2.6$ (right).  Here the initial condition $u_0$ in $Z$ has atoms at $e^{\pm i\pi/3}$ with equal mass $\frac12$, and the step distribution is circular / Ginibre.  Each image also shows $1000$ eigenvalues of the assocciated matrix model.}
  \label{fig.6.1}
\end{figure}

The region $\Sigma_k(u_0,t)$ can be quite complicated, depending on the initial condition $u_0$.  In the special case $u_0=1$, in which case we denote $\Sigma_k(1,t):=\Sigma_k(t)$, we provide a precise geometric description of the region in polar coordinate, and characterize how it evolves over time.

\begin{theorem} \label{thm.Brown.Measure.3} For $k\in\mathbb{N}$, $r>0$, and $\theta\in(-\pi,\pi]$, let $T_k(r,\theta) = T_k(1;re^{i\theta})$, i.e.
\begin{align*} T_k(r,\theta) :=& \frac{k(r^{2/k}-1)}{r^2-1}(r^2-2r\cos\theta+1), \quad r\ne 1 \end{align*}
and $T_k(1,\theta):= 2(1-\cos\theta)$ so that $T_k$ is continuous.

For each $\theta$ there is a unique $r_k^{\min}(\theta)$ that minimizes the function $r\mapsto T_k(r,\theta)$, which is decreasing for $0<r<t_k^{\min}(\theta)$ and increasing for $r>r_k^{\min}(\theta)$.  For each $t>0$, there is a unique $r_k^+(t,\theta)>r_k^{\min}(\theta)$ so that $T_k(r_k^+(t,\theta),\theta) = t$.  There is at most one $r_k^-(t,\theta) < r_k^{\min}(\theta)$ such that $T_k(r_k^-(t,\theta),\theta)=t$; otherwise define $r_k^-(t,\theta)=0$.  Define
\begin{equation} \label{eq.tk*.def}
t_k^\ast(\theta) := T_k(r_k^{\min}(\theta),\theta), \qquad t_k^{\mathrm{c}}:= t_k^\ast(\pi).
\end{equation}
The function $\theta\mapsto t_k^\ast(\theta)$ is symmetric (depends only on $|\theta|$) and increasing for $0<\theta<\pi$.  The preimage $(t_k^\ast)^{-1}([0,t)):= I_k(t)$ is a symmetric interval in $(-\pi,\pi)$ for $t<t_k^{\mathrm{c}}$, and $t_k^{\mathrm{c}} <\min\{k,4\}$.

The domain $\Sigma_k(t)$ is characterized as follows:
\begin{equation} \label{eq.Sigma.kt.polar.def}
\Sigma_k(t) = \begin{cases} \left\{re^{i\theta}\colon \theta\in I_k(t) \; \& \; r_k^-(t,\theta)<r<r_k^+(t,\theta)\right\} & \text{if} \; t\le k, \\
\left\{re^{i\theta}\colon \theta\in(-\pi,\pi] \;\&\; r<r_k^+(t,\theta)\right\} & \text{if} \; t>k. \end{cases}
\end{equation}
In the interval $0<t<k$, $\Sigma_k(t)$ and its closure are topological disks when $t<t_k^\ast$.  At the ``collision time'' $t=t_k^\ast$, The two ``lobes'' of $\overline{\Sigma_k(t)}$ collide, so that the closure is an annulus; For $t_k^\ast<t<k$, $\Sigma_k(t)$ is also an open annulus, with center hole shrinking to $0$ at $t=k$ and then disappearing.

The following table summarizes the topological phase transitions that $\Sigma_k(t)$ and its closure undergo is $t$ increases.

\begin{table}[h!]
\centering
\caption{Topological phase transitions of the Brown measure support.}
\label{table.1}
    \begin{tabular}{ |c|c|c| }
    $t$ & $\Sigma_k(t)$ & $\overline{\Sigma_k(t)}$  \\     \hline
    $0<t<t_k^\ast$ & open disk & closed disk \\
    $t=t_k^\ast$ & open disk & closed annulus \\
    $t_k^\ast<t<k$ & open annulus & closed annulus \\
    $t=k$ & punctured disk & closed disk \\
    $t>k$ & open disk & closed disk
%    $t=k\|a^{-1}\|_2^2$ & punctured disk & closed disk \\
%    $t>k\|a^{-1}\|_2^2$ & open annulus & closed annulus
    \end{tabular}
\end{table}
The support of the Brown measure of $u_0 b_k(t)$ is equal to $\overline{\Sigma_k(t)}$ for $t<k\|a^{-1}\|_2^2$, and equals $\overline{\Sigma_k(t)}\setminus D_k(a,t)$ \eqref{eq.Dkat} for $t\ge k\|a^{-1}\|_2^2$.  Hence there is a further topological phase transition at $t=k\|a^{-1}\|_2^2\ge k$ back to an annulus in the case that the step distribution $a$ has an $L^2$ inverse.
\end{theorem}

\begin{remark} Due to our normalization $\|a\|_2=1$, if $a^{-1}\in L^2$ then by H\"older's inequality $1=\|1\|_1 = \|a\cdot a^{-1}\|_1 \le \|a\|_2\|a^{-1}\|_2 = 1\cdot\|a^{-1}\|_2$; hence $\|a^{-1}\|_2^2\ge 1$, consistent with the time ordering of the phase transitions.  If $a=u$ is Haar unitary then $\|a^{-1}\|_2 = \|a\|_2=1$ and the support of the Brown measure transitions immediately at $t=k$ from one annulus to another (becoming a punctured disk at $t=k$).  Otherwise, if $a$ is a non-unitary $\mathscr{R}$-diagonal operator satisfying $\|a\|_2=1$, then $\|a^{-1}\|_2>1$ (cf.\ \cite{HaagerupLarsen2000}), and  so $k<k\|a^{-1}\|_2^2$ and there are separate phase transitions at $t=k$ and later at $t=k\|a^{-1}\|_2^2$ \end{remark}

\begin{figure}[h!]
  \includegraphics[scale=0.6]{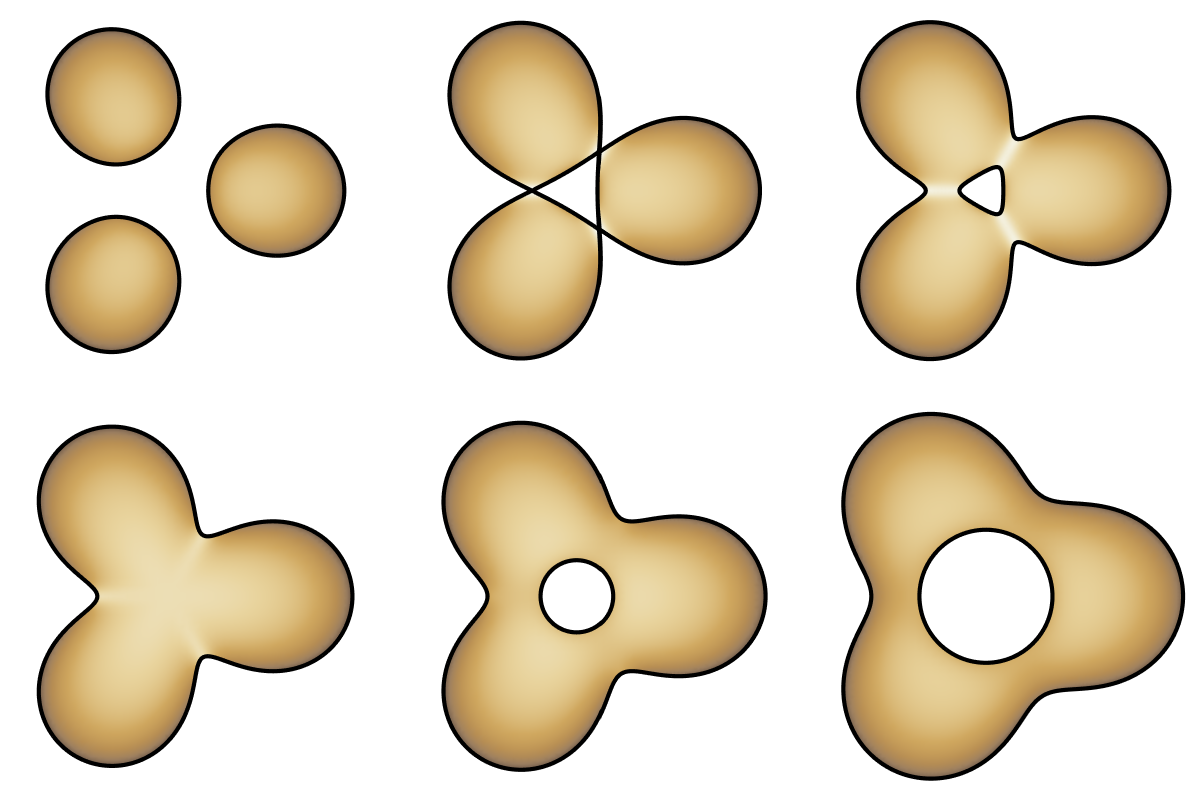}
  \caption{The Brown measure linearized model $Z+\sqrt{t/k}A$ at $t$ evolves through the phase transitions of its topology; here $k=3$, the initial condition is $u_0=1$, and the step distribution is Haar unitary.
  \label{fig.6.5}}
\end{figure}

Theorems \ref{thm.Brown.Measure.1}, \ref{thm.Brown.Measure.2}, and \ref{thm.Brown.Measure.3}, along with finer results on the density of the measure and evolution of its support domain, are proved in Section \ref{sect.BM.computation.1}.

\begin{figure}[h!]
  \includegraphics[scale=0.6]{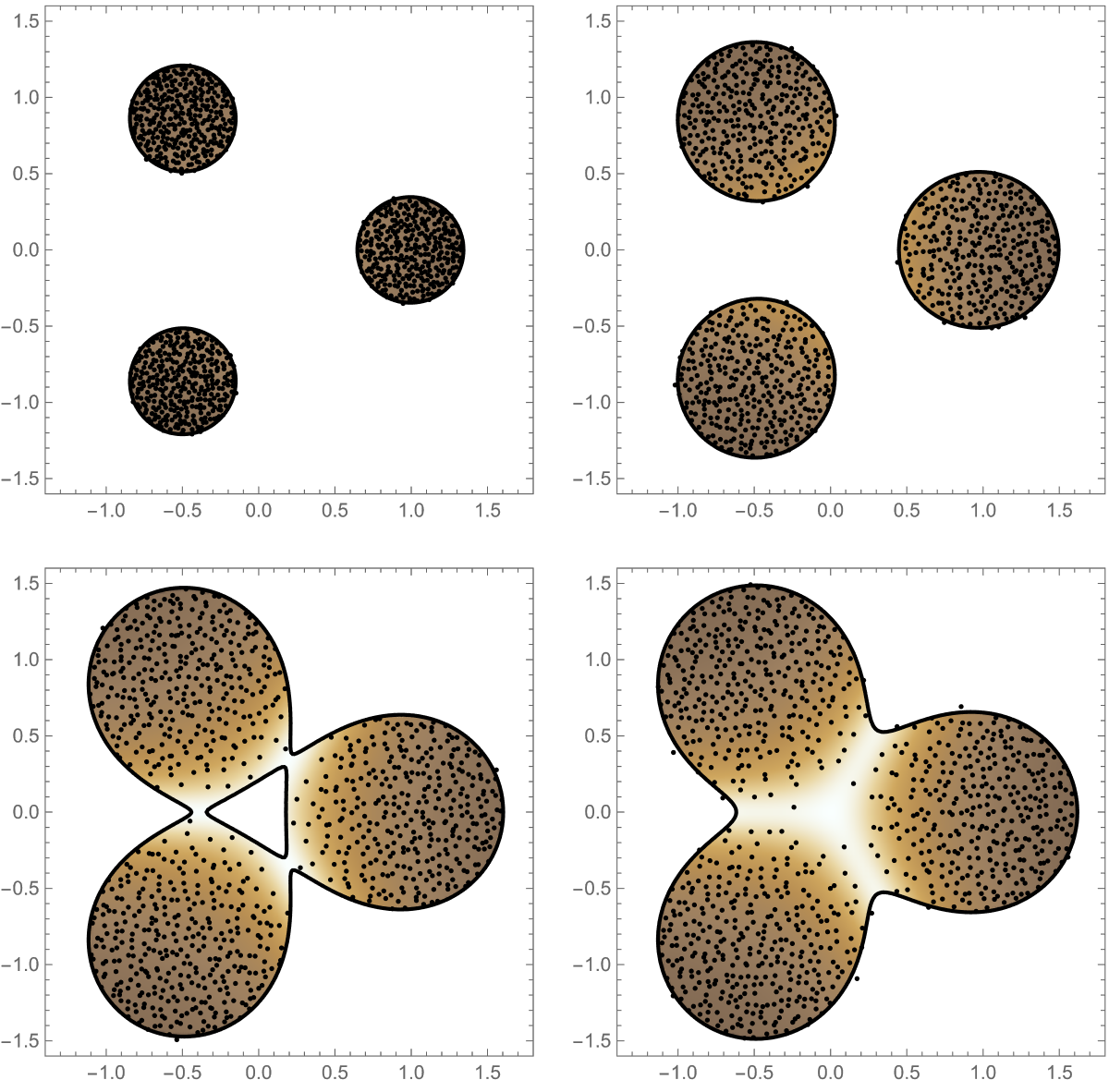}
  \caption{The evolution of the region $\Omega_k(t)$ for the linearized model $Z+\sqrt{t/k}A$ with $k=3$, initial condition $u_0=1$, circular / Ginibre steps.
  \label{fig.6.3}}
\end{figure}

\begin{figure}[h!]
  \includegraphics[scale=0.6]{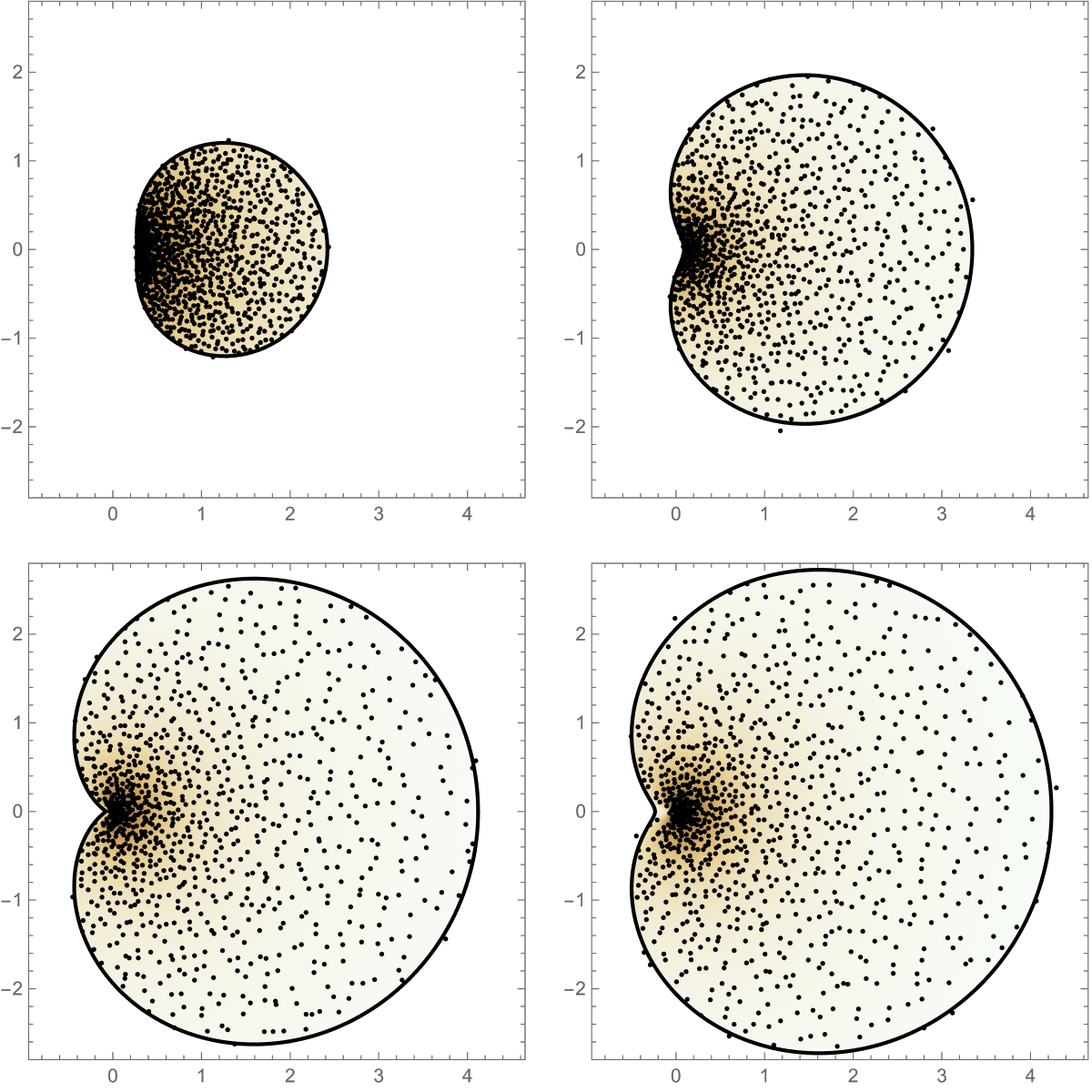}
  \caption{The evolution of the region $\Sigma_k(t)$ for the same parameter choices as in Figure \ref{fig.6.3}.
  \label{fig.6.4}}
\end{figure}

\subsubsection{The Lima Bean Law} \hfill

\medskip

The support domain $\Sigma_k(u_0,t)\setminus D_k(a,t)$ of the Brown measure of $u_0 b_k(t)$ is quite tractable, especially in the case that $u_0=1$, cf.\ Theorem \ref{thm.Brown.Measure.2} and \ref{thm.Brown.Measure.3}.  On the other hand, the density is rather more difficult to compute.  %In Proposition \ref{prop.add.Cauchy} we give an expression for the Cauchy transform of the Brown measure of $Z+\sqrt{t/k}\,A$, from which the density may be recovered from Stieltjes inversion; but that formula involves an implicit positive function $\eta_k(t,z)$, the square of the subordinator (cf.\ Theorem \ref{thm.subordination.general}) between the symmetrized laws of $|Z-z^{1/k}|$ and $\sqrt{t/k}\,|A|$, and this function is highly implicit in general.
In the special case $k=2$, with step distribution either circular or Haar unitary, we are able to compute explicit formulas for the densities; see Section \ref{section.explicit.formulas}.  For large $k$, explicit formulas are an insurmountable challenge.

Now, for fixed $t>0$, for all large $k$ the disk $D_k(a,t)$ is empty, and the support set of the Brown measure is $\overline{\Sigma_k(u_0,t)}$.  From \eqref{eq.lifetime.k} and \eqref{eq.Sigma.intro},  $\Sigma_k(u_0,t)$  is the set of $z\in\mathbb{C}$ where $T_k(u_0,z)<t$.  Note that
\[ \lim_{k\to\infty} k(|z|^{2/k}-1) = \log(|z|^2) \]
and hence
\[\lim_{k\to\infty}T_k(u_0,z) = \frac{\log(|z|^2)}{|z|^2-1}\left(\int_{\mathbb{T}} \frac{\mu_{u_0}(d\xi)}{|\xi-z|^2}\right)^{-1} = T_\infty(u_0,z) \]
where $T_\infty(u_0,\cdot)$ is defined in \eqref{eq.lifetime.k=infty}.  In particular, if $b(t)$ is the free multiplicative Brownian motion \eqref{eq.free.mult.BM}, Theorem \ref{thm.fmbm} shows that the support of the Brown measure of $u_0b(t)$ is $\overline{\Sigma_\infty(u_0,t)}$ where $\Sigma_\infty(u_0,t) = \{z\in\mathbb{C}\colon T(u_0,z)<t\}$.  This is strongly suggestive that at least the support of the Brown measure of $u_0b_k(t)$ converges to the support of the Brown measure of $u_0b(t)$ as $k\to\infty$.  Our final main theorem is that, in fact, the Brown measures converge, in a strong sense.

\begin{theorem}[Lima Bean Law] \label{thm.lima.bean} Let $k\in\mathbb{N}$ and $t>0$.  Let $u_0 b_k(t)$ be any free random walk \eqref{eq.ubk} with steps $a_j$ normalized with $\|a_j\|_2=1$, and let $u_0b(t)$ be a free multiplicative Brownian motion \eqref{eq.free.mult.BM}.  Then
\begin{enumerate}
    \item The domain $\overline{\Sigma_k(u_0,t)}$ converges to the domain $\overline{\Sigma_\infty(u_0,t)}$ in Hausdorff distance as $k\to\infty$
    \item The density $\rho_k(t,\cdot)$ of the Brown measure of $u_0b_k(t)$ converges to the density $\rho_\infty(t,\cdot)$ of the Brown measure of $u_0b(t)$ as $k\to\infty$, uniformly on compact subsets of $\Sigma_\infty(u_0,t)$.
    \item The Brown measure of $u_0b_k(t)$ converges weakly to the Brown measure of $u_0b(t)$ as $k\to\infty$.
\end{enumerate}
\end{theorem}

This theorem mirrors the finite-$N$ limit of Theorem \ref{main.thm.1.MoBettaWZ} and subsequent discussion at the beginning of Section \ref{sect.intro.eigenvalues}, which shows the distribution of eigenvalues of $B^N_k(t)$ converges to the distribution of eigenvalues of $B^N(t)$ as $k\to\infty$.  Theorem \ref{thm.lima.bean} shows that the large-$N$ limit eigenvalue distribution of $B^N_k(t)$ converges to the large-$N$ limit eigenvalue distribution of $B^N(t)$ as $k\to\infty$: and it is {\em far} more general.  In particular, the only assumption on the limit step distribution $a_j$ is bi-invariance (i.e.\ $\mathscr{R}$-diagonal) together with normalization $\|a_j\|_2=1$.  Whereas the steps $A_j^N$ of $B_k^N(t)$ must actually be Ginibre Ensembles in Theorem \ref{main.thm.1.MoBettaWZ}, or minimally must have the same entry-covariance as Ginibre for Berger's Theorem \ref{thm.Berger}, our main Theorem \ref{thm.lima.bean} shows that the Lima Bean distribution of the free multiplicative Brownian motion is the universal limit for {\em all} random walks with bi-invariant steps.

\section{Background}

In this extensive background section we outline most of the concepts and tools we will use throughout this paper, including the basic constructs of noncommutative probability, free probability theory, Brown measure, and stochastic calculus on Lie groups and more general noncommutative probability settings.

\subsection{Noncommutative Probability}

\subsubsection{$\ast$- and $W^\ast$-Probability Spaces\label{sect.*-prob}} \hfill

\medskip

A {\bf $\ast$-probability space} (sometimes called a noncommutative probability space) is a pair $(\mathscr{A},\varphi)$, where $\mathscr{A}$ is a unital $\ast$-algebra, i.e.\ possessing a multiplicative unit $1$ and an involution $a\mapsto a^\ast$ that is sesquilinear and satisfies $(a^\ast)^\ast = a$, and $\varphi\colon\mathscr{A}\to\mathbb{C}$ is a {\bf state}, i.e.\ a linear functional $\varphi\colon\mathscr{A}\to\mathbb{C}$ that is unital ($\varphi(1)=1$) and positive ($\varphi(a^\ast a)\ge 0$ for all $a\in\mathscr{A}$).  The state is {\bf faithful} if $\varphi(a^\ast a)=0$ implies $a=0$, in which case $\langle a,b\rangle := \varphi(b^\ast a)$ defines an inner-product on $\mathscr{A}$.  The state is {\bf tracial} if $\varphi(ab)=\varphi(ba)$ for all $a,b\in\mathscr{A}$ (which need not be true in general, as the algebra need not be commutative).

\begin{example} \label{ex.*-prob.spaces} Following are some typical examples of $\ast$-probability spaces.
\begin{enumerate}
    \item If $(\Omega,\mathscr{F},\mathbb{P})$ is a probabiliy space, then the pair $(L^\infty(\Omega,\mathscr{F},\mathbb{P}),\mathbb{E})$ is a $\ast$-probability space (that happens to be commutative).  Here the adjoint is complex conjugation, and so $f^\ast f = |f|^2$ represent all non-negative functions.  The state $\mathbb{E}$ is automatically tracial, and is faithful.  A variant of this space expands the algebra to $L^{\infty-}$, containing all random variables that have finite moments of all orders (but are not necessarily bounded).
    \item $(\MNC,\frac{1}{N}\mathrm{Tr})$ is a $\ast$-probability space, where the adjoint operation is conjugate transpose.  The normalized trace state is faithful and tracial.
    \item The previous two examples can be combined to realize random matrices (with bounded or $L^{\infty-}$ entries) in a $\ast$-probability space, where the combined state is $\varphi(A) = \frac{1}{N}\mathbb{E}\mathrm{Tr}[A]$.  Alternatively, just the normalized trace may be used, in which case the state $\varphi(A)$ is a random variable that satisfies all the desired properties a.s.
    \item If $H$ is a separable Hilbert space, the {\bf free} (or {\bf Boltzmann}) {\bf Fock space} $\mathscr{F}_0(H)$ is the completion of the tensor algebra $\bigoplus_{n\ge 0} H^{\otimes n}$ in the boosted inner product determined by
    \[ \langle h_1\otimes\cdots\otimes h_n,g_1\otimes\cdots\otimes g_m\rangle := \delta_{nm} \langle h_1,g_1\rangle_H\cdots\langle h_n,g_n\rangle_H. \]
    The ``$0$-particle space'' $H^{\otimes 0}$ is equal to the span $\mathbb{C}\Omega$ for some unit vector $\Omega$ called {\em vacuum}.  It defines the {\bf vacuum expectation state} $\tau_0$ on the algebra $\mathscr{B}(\mathscr{F}(H))$ of bounded operators on the Fock space:
    \[ \tau_0(x) = \langle x\Omega,\Omega\rangle_{\mathscr{F}_0(H)}. \]
    This state is neither faithful nor tracial in general.
    \item \label{ex.*-prob.spaces.Fock} On the free Fock space $\mathscr{F}_0(H)$, consider the {\bf left creation operators} $l(h)\colon\mathscr{F}_0(H)\to\mathscr{F}_0(H)$ given by $l(h)\psi = h\otimes\psi$ for $h\in H$.  The left creation operators are bounded, with norm $\|l(h)\|_{\mathrm{op}} = \|h\|_H$.  The adjoint $l(h)^\ast$ is easily computable: $l(h)^\ast\Omega = 0$ and on the ``$n$-particle space'' $H^{\otimes n}$ the action is $l(h)^\ast(g_1\otimes g_2\otimes\cdots\otimes g_n) = \langle g_1,h\rangle g_2\otimes\cdots\otimes g_n$, realizing the {\bf left annihilation operators}.  Ergo $\omega(h) = \ell(h)+\ell^\ast(h)$ is a selfadjoint bounded operator.  Fix an orthonormal basis $\{e_j\}_{j\in J}$ for $H$; the von Neumann subalgebra of $\mathscr{B}(\mathscr{F}_0(H))$ generated by $\{\omega(e_j)\}_{j\in J}$ is denoted $\Gamma_0(H)$, the {\bf free Gaussian factor}.  The vacuum expectation state $\tau_0$ is both faithful and tracial when restricted to $\Gamma_0(H)$, cf. \cite{BKS}.
\end{enumerate}
\end{example}
Each of the preceding examples is in fact a {\bf $W^\ast$-probability space}: a pair $(\mathscr{A},\varphi)$ where $\mathscr{A}\subseteq\mathscr{B}(H)$ is an algebra of operators on a Hilbert space closed in the $\sigma$-weak operator topology ($\sigma$-WOT), and the state is faithful, tracial, and also ``normal'', meaning $\sigma$-WOT continuous.  (Essentially, this continuity says that an appropriate version of the Dominated Convergence Theorem holds for the state.)

\subsubsection{Noncommutative $L^p$ Spaces}

\begin{notation} In a $W^\ast$-probability space $(\mathscr{A},\varphi)$, for $1\le p<\infty$, the {\bf $L^p$-norm} of an element $a\in\mathscr{A}$ is
\[ \|a\|_p = \|a\|_{L^p(\mathscr{A},\varphi)} := \varphi[|a|^p] = \varphi[(a^\ast a)^{p/2}]. \]
Although $a^\ast a \ne aa^\ast$ in general, since $\varphi$ is tracial it follows that $\|a^\ast\|_p = \varphi[(aa^\ast)^{p/2}] =\|a\|_p$.

The space $L^p(\mathscr{A},\varphi)$ is defined to be the $\|\cdot\|_p$-completion of $\mathscr{A}$.  (If $\mathscr{A}\subseteq\mathscr{B}(H)$, $L^p(\mathscr{A},\varphi)$ can be realized as a space of unbounded linear operators on $H$ affiliated with $\mathscr{A}$.)  Also, denote $L^\infty(\mathscr{A},\varphi):=\mathscr{A}$; this is sensible due to the fact (that holds true in the general setting of a faithful state) that $\|a\|_\infty:=\lim_{p\to\infty}\|a\|_p =\|a\|_{\mathscr{A}}=\|a\|_{\mathrm{op}}$, the operator norm of $a$ acting on $H$.
\end{notation}

The preceding notation is fully consistent with the classical setting: the abstract $L^p$-completion of $L^\infty$ coincides with $L^p$ over any measure space.  Moreover, as in the classical case, the trace $\varphi\colon\mathscr{A}\to\mathbb{C}$ extends uniquely to a bounded linear functional $L^1(\mathscr{A},\varphi)\to\mathbb{C}$ with operator norm $1$ (and so we also refer to this extension as $\varphi$).  In general, {\bf H\"older's Inequality} holds: if $1\le p\le \infty$ and $1/p+1/p'=1$,
\begin{equation} \label{eq.Holder} |\varphi[ab])| \le \|ab\|_1 \le \|a\|_p\|b\|_{p'} \quad a\in L^p(\mathscr{A},\varphi), b\in L^{p'}(\mathscr{A},\varphi). \end{equation}
(Here the first step is the triangle inequality $|\varphi[a]| \le \varphi(|a|) = \varphi[(a^\ast a)^{1/2}]$.)  The usual induction argument shows that, more generally,
\begin{equation} \label{eq.Holder.n} |\varphi[a_1\cdots a_n]| \le \|a_1\cdots a_n\|_p \le \|a_1\|_{p_1}\cdots \|a_n\|_{p_n} \end{equation}
whenever $p,p_1,\ldots,p_n\in[1,\infty]$ with $1/p_1+\cdots+1/p_n= 1/p$ and $a_j\in L^{p_j}(\mathscr{A},\varphi)$ for $1\le j\le n$.  By taking $a_2=\cdots=a_n=1$ in \eqref{eq.Holder.n} we see as usual that $\|a\|_p\le\|a\|_q$ whenever $1\le p\le q\le\infty$; hence we have the usual reverse inclusion of $L^p$ spaces $L^p(\mathscr{A},\varphi)\supseteq L^q(\mathscr{A},\varphi)$.  This also implies that the restriction of $\varphi$ from $L^1$ to $L^p$ is a continuous linear functional.

Note that the $L^p$ norm on $(\mathscr{A},\varphi)$ has a dual characterization: if $1\le p\le\infty$ and $p'$ is the conjugate exponent with $1/p+1/p'=1$ then for any $a\in\mathscr{A}$
\[ \|a\|_p = \sup\{|\varphi[ab]|\colon b\in\mathscr{A}, \|b\|_{p'}\le 1\}. \]
It follows that the map $a\mapsto \varphi[\,\cdot\, a]$ extends to an isomorphism from $L^p(\mathscr{A},\varphi)$ onto the dual space $L^{p'}(\mathscr{A},\varphi)^\ast$ when $1<p\le\infty$.  Of particular note is the case $p=p'=2$, where the self-duality is tied to the fact that $\langle a,b\rangle_{L^2}:=\varphi[b^\ast a]$ defines an inner product with respect to which $L^2(\mathscr{A},\varphi)$ is a Hilbert space (also known as the Gelfand--Naimark--Segal (GNS) space, where the original algebra $\mathscr{A}$ can be viewed as acting via left multiplication).  However, unlike the classical case, typically $L^\infty(A,\varphi)^\ast = \mathscr{A}^\ast \supsetneq L^1(\mathscr{A},\varphi)$.

\subsubsection{$\ast$-Distributions\label{sec.*-distribution}} \hfill

\medskip

Let $(\mathscr{A},\varphi)$ be a $\ast$-probability space.  The {\bf $\ast$-distribution} of a collection $\mathbf{a} = (a_i)_{i\in I}\subset\mathscr{A}$ is the set of all $\ast$-moments
\begin{equation} \label{eq.*-distribution} \varphi(a_{i_1}^{\varepsilon_1} a_{i_2}^{\varepsilon_2} \cdots a_{i_r}^{\varepsilon_r}), \qquad r\in\mathbb{N}, \quad i_1,\ldots,i_r\in I, \quad \varepsilon_1,\ldots,\varepsilon_r\in\{1,\ast\}. \end{equation}
It is sometimes construed as a linear functional $\varphi_{\mathbf{a}}\colon\mathscr{P}_I\to\mathbb{C}$, where $\mathscr{P}_I$ is the algebra of abstract $\ast$-polynomials, i.e.\ the span of all $\ast$-monomials (aka $\ast$-words) $X_{i_1}^{\varepsilon_1}\cdots X_{i_r}^{\varepsilon_r}$ in noncommuting variables $(X_i)_{i\in I}$, where the value of $\varphi_{\mathbf{a}}$ on such a $\ast$-monomial is precisely the scalar in \eqref{eq.*-distribution}.  This is to mirror the classical case where the distribution of a collection of random variables is identified as a probability measure on the state space which, via the Riesz Representation Theorem, is a collection of linear functionals on a function space indexed by the original random variables.  The $\ast$-distribution $\varphi_{\mathbf{a}}$ is sometimes called the (noncommutative) {\bf Law} of the (noncommutative) random vector $\mathbf{a}$.

The analogy between $\ast$-distribution and  classical probability distribution is more direct for a single normal operator $a$ (i.e.\ satisfying $aa^\ast = a^\ast a$) in a $W^\ast$-probability space.  In that setting, all $\ast$-monomials reduce to those of the form $a^n (a^\ast)^m$ for $m,n\in\mathbb{N}$, and the Spectral Theorem yields a projection valued measure $E^a$ on $\mathbb{C}$ for which
\[ a^n (a^\ast)^m = \int_{\mathbb{C}} \zeta^n\bar\zeta^m\,E^a(d\zeta). \]
As such, the (ordinary $[0,1]$-valued) probability measure $\mu_a = \varphi\circ E^a$ captures the entire $\ast$-distribution through the moment formula
\begin{equation} \label{eq.*-moments.measure} \varphi[a^n (a^\ast)^m] = \int_{\mathbb{C}} \zeta^n\bar\zeta^m\,\mu_a(d\zeta). \end{equation}
Since $a$ is a bounded operator, its spectrum $\mathrm{spec}(a)$ is a compact subset of $\mathbb{C}$.  The measure $\mu_a$, whose support equals $\mathrm{spec}(a)$, is called the {\bf spectral measure} of $a$.  As it is compactly-supported, $\mu_a$ is completely determined by \eqref{eq.*-moments.measure}.

Two classes of normal operators are of special note: {\bf selfadjoint} ($a=a^\ast$) and {\bf unitary} ($a^\ast a = aa^\ast = 1$).  In each case, the $\ast$-distribution reduces to moments of the form $\varphi(a^n)$ for $n\in\mathbb{N}$.  For unitary $u$, the spectral measure $\mu_u$ is a probability measure on the unit circle $\mathbb{T}\subset\mathbb{C}$.  For selfadjoint $x$, the spectral measure $\mu_x$ is a probability measure on $\mathbb{R}$; an important subcase is {\bf positive} (or more precisely non-negative) operators of the form $r=a^\ast a$ for some $a\in\mathscr{A}$.  Here $\mathrm{spec}(r)\subset [0,\infty)$.

\begin{remark} \label{rk.unbounded.spec.meas} If $(\mathscr{A},\varphi)$ is a $W^\ast$-probability space and $a\in L^1(\mathscr{A},\varphi)$ is a(n unbounded) normal operator, it still possesses a spectral measure.  The Spectral Theorem still assigned a projection-valued measure $E^a$ supported on the (non-compact) spectrum of $a$.  The operators in $L^1(\mathscr{A},\varphi)$ are affiliated with $\mathscr{A}$, and so all the spectral projections are in fact in $\mathscr{A}$; hence $\mu_a = \varphi\circ E^a$ still makes sense.  In this setting, \eqref{eq.*-moments.measure} holds true in the sense that both sides are equal when $a^n(a^\ast)^m\in L^1(\mathscr{A},\varphi)$ (i.e.\ if $a$ is actually in $L^{n+m}(\mathscr{A},\varphi)$) and both do not exist otherwise.  In general, what holds true and uniquely defines $\mu_a$ for normal $a\in L^1(\mathscr{A},\varphi)$ is
\[ \varphi[f(a)] = \int_{\mathrm{spec}(a)}f(\zeta)\,\mu_a(d\zeta) \quad \forall\, f\in C_c(\mathbb{C}). \]
\end{remark}

If $A\in \MNC$ is a normal (random) matrix, its spectral measure with respect to the normalized trace is its empirical spectral distribution ($\mathrm{ESD}$) defined in \eqref{eq.def.ESD}.

The $\ast$-distribution of a non-normal operator cannot be fully encoded in natural way as a measure on $\mathbb{C}$; however, there is an extension called {\em Brown measure} (also denoted $\mu_a$) which is central to this paper, and is discussed in detail below in Section \ref{sect.Brown.measure}.  In general, for a set $(a_i)_{i\in I}$ of size $\ge 2$, there is no reasonable notion of spectral measure if the $a_i$ do not commute.

\begin{definition} \label{def.conv.*-dist} Given $\ast$-probability spaces $(\mathscr{A},\varphi)$ and $(\mathscr{A}_N,\varphi_N)$ and collections $\mathbf{a}=(a_i)_{i\in I}\subset\mathscr{A}$ and $\mathbf{a}_N = (a_{N,i})_{i\in I}\subset\mathscr{A}_N$, we say {\bf $\mathbf{a}_N$ converges in $\ast$-distribution to $\mathbf{a}$}, $\mathbf{a}_N\rightharpoonup \mathbf{a}$, if the linear functionals $\varphi_{\mathbf{a}_N}$ converge pointwise to $\varphi_{\mathbf{a}}$ on $\mathscr{P}_I$.  That is to say: each $\ast$-moment \eqref{eq.*-distribution} of $\mathbf{a}_N$ converges to that same $\ast$-moment of $\mathbf{a}$.

In the case that each $\mathbf{a}_N=a_N$ is a single normal operator, $a_N\rightharpoonup a$ means that the moments of $\mu_{a_N}$ converge to the moments of $\mu_a$.  If the sequence $\mu_{a_N}$ is also known to be tight, then this is equivalent to weak convergence.  In particular, if there is a fixed compact set $K\subset\mathbb{C}$ that contains all the spectra $\mathrm{spec}(a_N)$, then convergence in $\ast$-distribution is equivalent to weak convergence, justifying the use of the same symbol $\rightharpoonup$.
\end{definition}

In (classical) probability theory, $L^p$ convergence is stronger than weak convergence of random variables.  The same holds true for $L^p$ convergence compared to $\ast$-convergence in a single $W^\ast$-probability space.

\begin{proposition} \label{prop.Lp>D*} Let $\mathbf{a}_k = (a_{i,k})_{i\in I}$ and $\mathbf{a} = (a_i)_{i\in I}$ be in a $W^\ast$-probability space $(\mathscr{A},\varphi)$.  If $a_{i,k}\to a_i$ in $L^p(\mathscr{A},\varphi)$ for all $p\in[1,\infty)$ and all $i\in I$
\[ \|a_{i,k}-a_i\|_p \to 0 \quad \text{as} \; k\to\infty \]
Then $\mathbf{a}_k$ converges to $\mathbf{a}$ in $\ast$-distribution.
\end{proposition}

\begin{remark} In the special case of $L^\infty$ of a classical probability space, the statement is that if the components of a bounded random vector each converge to limits in $L^p$ for every $p$ then the {\em joint} moments, and ergo joint distribution, of the random vector converges to the limit.  This is not as surprising as it appears at first blush: convergence of components separately in all $L^p$ spaces implies that any finite linear combination of components converges to the relevant linear combination of the limits, which is enough information to specify the limit joint distribution by the Portmanteau theorem. (More technical conditions are required to move from convergence of moments to convergence in distribution beyond the bounded setting.)
\end{remark}

\begin{proof} For a given $\ast$-monomial $F((X_i)_{i\in I}) = X_{i_1}^{\varepsilon_1}\cdots X_{i_n}^{\varepsilon_n}$,
\begin{align*} \delta_k(F):= \varphi[F(\mathbf{a}_k)]-\varphi[F(\mathbf{a})] =& \varphi[a_{i_1,k}^{\varepsilon_1}\cdots a_{i_n,k}^{\varepsilon_n} - a_{i_1}^{\varepsilon_1}\cdots a_{i_n}^{\varepsilon_n}] \\
=& \varphi[u_1u_2\cdots u_n - v_1v_2\cdots v_n]
\end{align*}
where, for convenience, we have denoted the terms $a_{i_j,k}^{\varepsilon_j} = u_j$ and $a_{i_j}^{\varepsilon_j} = v_j$ for $1\le j\le n$.  We expand this difference as a telescoping sum
\[ u_1u_2\cdots u_n - v_1v_2\cdots v_n = \sum_{j=1}^n u_1\cdots u_{j-1}(u_j-v_j)v_{j+1}\cdots v_n. \]
Hence, by H\"older's inequality,
\begin{align*} |\delta_k(F)| &= \left|\sum_{j=1}^n \varphi[u_1\cdots u_{j-1}(u_j-v_j)v_{j+1}\cdots v_n]\right| \\
&\le \sum_{j=1}^n \|u_1\|_n\cdots \|u_{j-1}\|_n \|u_j-v_j\|_n \|v_{j+1}\|_n\cdots \|v_n\|_n.
\end{align*}
By assumption $\|a_{i_j,k}-a_{i_j}\|_n\to 0$ as $k\to\infty$; hence, for some $k_j\in\mathbb{N}$ it follows that, for $k\ge k_j$, $\|u_j\|_\ell = \|a_{i_j,k}^{\varepsilon_j}\|_n = \|a_{i_j,k}\|_n \le \|a_{i_j}\|_n+1 = \|a_{i_j}^{\varepsilon_j}\|_n+1 = \|v_j\|_n+1$.  Taking $\rho_n = \max\{\|v_j\|_n\colon 1\le j\le n\}$, we therefore have for all $k\ge\max\{k_1,\ldots,k_n\}$
\[ |\delta_k(F)| \le (\rho_n+1)^{n-1}\sum_{j=1}^n \|u_j-v_j\|_{n}. \]
But, again, by assumption $\|u_j-v_j\|_n = \|(a_{i_j,k}-a_{i_j})^{\varepsilon_j}\|_n= \|a_{i_j,k}-a_{i_j}\|_n\to 0$ as $k\to\infty$ for each $j$, and hence $\delta_k(F)\to 0$.  Thus $\varphi_{\mathbf{a}_k}(F)\to \varphi_{\mathbf{a}}(F)$ for all $F\in\mathscr{P}_I$, as desired.
\end{proof}

In the context of (random) matrices, the $\ast$-distribution records {\em bulk} information; convergence in $\ast$-distribution is robust to many perturbations, in particular those of ``low'' rank.  The following lemma states and proves a precise result which is folklore for general matrices, a version of which is proved in \cite{Bai-Silverstein} in the case of Hermitian random matrices.  We give it here for an single matrix ensemble, but the same proof yields a comparable result for a multi-matrix ensemble.

\begin{lemma} \label{lem.perturb.*-dist} Let $(A_N)_{N\in\mathbb{N}}$ be matrices in $(\MNC,\frac1N\mathrm{Tr})$ that converges in $\ast$-distribution to an element $a$ in a $\ast$-probability space.  Let $E_N\in \MNC$ satisfy $\max\{\|A_N\|,\|E_N\|\}\le n(N)$ and $\mathrm{rank}(E_N)\le r(N)$, where $\|\cdot\|$ denotes the operator norm on $\MNC$ and $r(N)\cdot n(N)^\ell = o(N)$ for all $\ell\in\mathbb{N}$.  Then $A_N' = A_N+E_N$ also converges in $\ast$-distribution to $a$.  \end{lemma}

\begin{proof} Fix a $\ast$-monomial $F(X) = X^{\varepsilon_1} X^{\varepsilon_2}\cdots X^{\varepsilon_\ell}$ for $\varepsilon_1,\varepsilon_2,\ldots,\varepsilon_\ell\in\{1,\ast\}$.  For notation's sake, relabel $A_N = E_N^0$ and $E_N = E_N^1$.  Then
\[ F(A_N') = p(E_N^0+E_N^1) = \sum_{\eta\in \{0,1\}^\ell} (E_N^{\eta_1})^{\varepsilon_1} (E_N^{\eta_2})^{\varepsilon_2} \cdots (E_N^{\eta_\ell})^{\varepsilon_\ell}. \]
The term in the sum with $\eta=(0,0,\ldots,0)$ is $F(A_N)$.  Each of the other $2^{\ell}-1$ terms is a product with at least one $E_N$ or $E_N^\ast$, a matrix of rank $\le r(N)$.  Since the rank of a product of matrices is bounded above by the minim rank of any term, it follows that every remaining term has rank $\le r(N)$.  Morevover, the norm of each term is $\|(E_N^{\eta_1})^{\varepsilon_1} (E_N^{\eta_2})^{\varepsilon_2} \cdots (E_N^{\eta_\ell})^{\varepsilon_\ell}\| \le \|(E_N^{\eta_1})^{\varepsilon_1} \|\|(E_N^{\eta_2})^{\varepsilon_2}\| \cdots \|(E_N^{\eta_\ell})^{\varepsilon_\ell}\| \le n(N)^\ell$.

Thus, $F(A_N')-F(A_N)$ is a sum of $2^{\ell}-1$ matrices each of rank $\le r(N)$ and norm $\le n(N)^\ell$; by subadditivity of norm and rank, we see that $F(A_N')-F(A_N) = D_N$ with $\mathrm{rank}(D_N)\le r(N)\cdot(2^\ell-1)$ and $\|D_N\|\le n(N)^\ell\cdot(2^{\ell}-1)$.  Now, the rank is the number of nonzero eigenvalues, and the norm is the largest magnitude eigenvalue; hence the trace, which is the sum of the eigenvalues, satisfies
\[ |\mathrm{Tr}(D_N)| \le r(N)\cdot n(N)^\ell\cdot (2^\ell-1)^2. \]
Hence, by the assumption of the lemma,
\[ \left|\frac{1}{N}\mathrm{Tr}(p(A_N'))-\frac{1}{N}\mathrm{Tr}(p(A_N))\right| = \left|\frac{1}{N}\mathrm{Tr}(D_N)\right| = o(1). \]
Ergo, $A_N'$ and $A_N$ have the same asymptotic $\ast$-moments.
\end{proof}

\subsection{Free Probability and Asymptotic Random Matrix Theory} In this section we introduce the notion of {\em free independence} analogous to statistical independence in the classical setting, introduce tools and consequences of freeness, and discuss connections to random matrix theory.

\subsubsection{Free Independence and Free Cumulants} \hfill

\medskip

In a $\ast$-probability space $(\mathscr{A},\varphi)$, unital $\ast$-subalgebras $(\mathscr{A}_i)_{i\in I}$ of $\mathscr{A}$ (not necessarily closed under $\ast$)  are {\bf freely independent} or {\bf free} if products of alternating centered moments are centered: $\varphi[a_1\cdots a_n]=0$ whenever $\varphi[a_1]=\cdots=\varphi[a_n]=0$ and $a_1\in\mathscr{A}_{i_1},\ldots,a_n\in\mathscr{A}_{i_n}$ with $i_1\ne i_2,i_2\ne i_3,\ldots,i_{n-1}\ne i_n$.  Collections $\mathbf{a}_1,\mathbf{a}_2,\ldots$ are called freely independent if the unital subalgebras they generate are freely independent.

\begin{remark} \begin{enumerate}
\item Freeness is hereditary: if $(\mathscr{A}_i)_{i\in I}$ are freely independent and $\mathscr{B}_i\subseteq\mathscr{A}_i$ then $(\mathscr{B}_i)_{i\in I}$ are freely independent.  In a $W^\ast$-probability space, it would be natural to define freeness of random variables as freeness of the $W^\ast$-algebras they generate; by heredity and continuity of the state, this is equivalent to the above definition.
\item Typically the definition of freeness is state for subalgebras that are not necessarily closed under $\ast$.  In that case, what we call free independence above is a stronger notion called {\bf $\ast$-freeness}.  We will have no use for the weaker non-$\ast$ version of freeness in this paper, so we simply use freeness to encompass $\ast$-freeness.
\end{enumerate} \end{remark}

Freeness is a moment factorization property.  For example: if $a,b\in(\mathscr{A},\varphi)$ are freely independent, then $a^\circ:=a-\varphi[a]$ and $b^\circ:=b-\varphi[b]$ are in the unital algebras generated by $a$ and $b$ respectively and are centered; hence, freeness implies that $a^\circ b^\circ$ is centered, and so
\begin{align*} 0 = \varphi[a^\circ b^\circ] = \varphi[(a-\varphi[a])(b-\varphi[b])] &= \varphi\left[ab - \varphi[a]b-\varphi[b]a +\varphi[a]\varphi[b]\right] \\
&= \varphi[ab]-\varphi[a]\varphi[b]-\varphi[a]\varphi[b]+\varphi[a]\varphi[b] \\
&= \varphi[ab]-\varphi[a]\varphi[b].
\end{align*}
Thus, we have the familiar covariance condition: if $a,b$ are free then $\varphi[ab]=\varphi[a]\varphi[b]$.  Similar calculations show that $\varphi[aba] = \varphi[a^2]\varphi[b]$, matching classical moment factorization for independent random variables; but with products of order 4 or higher, the noncommutativity shows up: if $\{a_1,a_2\}$ is free from $\{b_1,b_2\}$, then
\[ \varphi[a_1b_1a_2b_2] = \varphi[a_1a_2]\varphi[b_1]\varphi[b_2] + \varphi[a_1]\varphi[a_2]\varphi[b_1b_2] - \varphi[a_1]\varphi[a_2]\varphi[b_1]\varphi[b_2]. \]

To organize and compute with freeness and moments, Speicher introduced the theory of {\bf free cumulants} that echoes the classical cumulants from probability and statistics.

\begin{definition} Let $\mathrm{P}(I)$ denote the collection of all {\bf set partitions} of a finite set $I$: i.e.\ $\pi\in\mathrm{P}(I)$ is a collection of disjoint nonempty subsets (called {\bf blocks}) of $I$ whose union is $I$.  If $I = \{1,2,\ldots,n\}$ we denote $\mathrm{P}(I) = \mathrm{P}(n)$. \end{definition}

Given a $\ast$-probability space $(\mathscr{A},\varphi)$, any collection $\{\xi_n\}_{n\in\mathbb{N}}$ of multi-linear functionals $\xi_n\colon \mathscr{A}^n\to\mathbb{C}$ extends to a family $\{\xi_{\pi}\colon\pi\in\mathrm{P}(n)\}$ for each $n\in\mathbb{N}$, defined as follows: If $\pi = \{B_1,\ldots,B_r\}$ then $\xi_\pi := \varphi_{B_1}\cdots \varphi_{B_r}$, where if $B = \{i_1<i_2<\cdots<i_\ell\}$ then $\xi_B[a_1,\ldots,a_n]:= \xi_{\ell}[a_{i_1},\ldots,a_{i_\ell}]$.  For example: if $\pi = \{\{1,3,4\},\{2,5\},\{6\}\}\in\mathrm{P}(6)$ then $\xi_\pi[a_1,\ldots,a_6] = \xi_3[a_1,a_3,a_4]\xi_2[a_2,a_5]\xi_1[a_6]$.  Note then that $\xi_n = \xi_{\pi_n}$ where $\pi_n$ is the single block partition $\{1,\ldots,n\}$.

The state $\varphi$ induces multi-linear functionals $\varphi_n$ by multiplication inside:
\[ \varphi_n[a_1,\ldots,a_n] = \varphi[a_1\cdots a_n]. \]
We define the {\bf classical cumulants} $\{c_n\colon n\in\mathbb{N}\}$ and their multiplicative extension $\{c_\pi\colon\pi\in\mathrm{P}(n)\}$ inductively by the moment-cumulant formula
\begin{equation} \label{e.moment.cumulant.c} \varphi[a_1\cdots a_n] = \sum_{\pi\in\mathrm{P}(n)} c_\pi[a_1,\ldots,a_n]. \end{equation}
For example:\ $\mathrm{P}(1)$ has a single partition $\vert=\{1\}$, while $\mathrm{P}(2) = \{\left\lvert\;\right\rvert,\sqcup\}$ where $\left\lvert\;\right\rvert = \{\{1\},\{2\}\}$ while $\sqcup = \{1,2\}$.
Hence \eqref{e.moment.cumulant.c} in the case $n=1$ says that $\varphi[a] = c_1[a]$ (the first cumulant is the mean), and \eqref{e.moment.cumulant.c} i nthe case $n=2$ says
\[ \varphi[ab] = c_{\left\lvert\;\right\rvert}[a,b] + c_\sqcup[a,b] = c_1[a]c_1[b] + c_2[a,b] = \varphi[a]\varphi[b] + c_2[a,b] \]
which implies that $c_2[a,b] = \varphi[ab]-\varphi[a]\varphi[b]$ is the covariance.  Higher cumulants have names and important roles in statistics: $c_3$ is called skewness, $c_4$ is kurtosis.

Classical cumulants connect with classical probability (through the method of moments) in two ways.  For a bounded random variable $X$, the exponential generating function $C_X(z) = \sum_{n=0}^\infty c_n[X,\ldots,X] z^n/n!$ satisfies $\exp(C_X(z)) = \mathbb{E}[\exp(zX)]$ -- i.e.\ $C_X$ is the log of the moment generating function (sometimes called the cumulant generating function).  Second, and more important to us: bounded random variables $(X_i)_{i\in I}$ are independent if and only if {\em all their mixed classical cumulants vanish}; that is, for any $n\in\mathbb{N}$ and indices $i_1,\ldots,i_n\in\mathbb{N}$, $c_n[X_{i_1},\ldots,X_{i_n}]=0$ whenever the indices $i_1,\ldots,i_n$ are not all equal.

All of these concepts can be ported to the free world, by restricting all set partitions to be {\bf noncrossing}.

\begin{definition} Let $I$ be a finite ordered set.  A {\bf crossing} in a partition $\pi\in\mathrm{P}(n)$ is a pair distinct blocks $B_1,B_2$ in $\pi$ with indices $i_1<i_2<j_1<j_2$ in $I$ satisfying $i_1,j_1\in B_1$ and $i_2,j_2\in B_2$.  If a partition is {\bf noncrossing} if it has no crossings. The set of noncrossing partitions is denoted $\mathrm{NC}(I)\subset\mathrm{P}(I)$.  Note that a crossing requires to blocks of size $\ge 2$; it follows that $\mathrm{NC}(I)=\mathrm{P}(I)$ whenever $|I|<4$.

In a $\ast$-probability space, the {\bf free cumulants} $\kappa_n$ are the multi-linear functionals defined (by their multiplicative extension $\kappa_\pi$) inductively by the free moment-cumulant formula
\begin{equation} \label{e.moment.cumulant.f} \varphi[a_1\cdots a_n] = \sum_{\pi\in\mathrm{NC}(n)} \kappa_\pi[a_1,\ldots,a_n]. \end{equation}
Since $\mathrm{NC}(n) = \mathrm{P}(n)$ for $n<4$, the free cumulants $\kappa_1,\kappa_2,\kappa_3$ agree with the classical cumulants $c_1,c_2,c_3$: they are mean, covariance, and skewness.  But since $\mathrm{P}(4)$ possesses the partition $\{\{1,3\},\{2,4\}\}$ which has a crossing, $\mathrm{NC}(4)\subsetneq\mathrm{P}(4)$, and $\kappa_4\ne c_4$.
\end{definition}
The free cumulants of a collection of random variables (and their adjoints) in a $\ast$-probability space contain the same information as the $\ast$-distribution, as \eqref{e.moment.cumulant.f} makes clear.  In the case of a single selfadjoint random variable $a$, where the $\ast$-distribution is encapsulated by the moments $\{\varphi(a^n)\}_{n\in\mathbb{N}}$, the (ordinary) generating function $M_a(z) = \sum_{n=0}^\infty \varphi(a^n)z^n$ and the free cumulant generating function $R_a(z)= \sum_{n=0}^\infty \kappa_n[a,\ldots,a] z^n$ are connected by a functional equation discussed below.  For the general non-selfadjoint case, the connection is harder to describe.  But the relationship to freeness is entirely analogous to the classical cumulants' relationship to independence.

\begin{proposition}[\cite{SpeicherNicaBook}, Theorem 11.20] \label{prop.mixed.fcumulants} A collection $(a_i)_{i\in I}$ of random variables in a $\ast$-probability space $(\mathscr{A},\varphi)$ is freely independent if and only if their mixed free cumulants vanish; i.e.\ for any $n\in\mathbb{N}$ and $i_1,\ldots,i_n\in I$ and $\varepsilon_1,\ldots,\varepsilon_n\in\{1,\ast\}$, $\kappa_n[a_{i_1}^{\varepsilon_1},\ldots,a_{i_n}^{\varepsilon_n}]=0$ if the indices $i_1,\ldots,i_n$ are not all equal. \end{proposition}

\begin{example} \label{ex.freeness} Following are some relevant examples of freeness and free cumulants of some random variables.
\begin{enumerate}
    \item The {\bf semicircle distribution} $\mu_{\mathrm{sc}}$ of variance $t$ is the probability distribution with density $\frac{1}{\pi t}\sqrt{(4t-x^2)_+}$ supported on the interval $|x|\le 2\sqrt{t}$.  A selfadjoint random variable $x$ in a $\ast$-probability space whose $\ast$-distribution is given by the law $\mu_{\mathrm{sc}}$ has very simple free cumulants: $\kappa_n[x,\ldots,x] = t\delta_{n2}$, i.e.\ all cumulants of order higher than $2$ are $0$, while the mean is $0$ and the variance is $t$.  Hence for any noncrossing partition $\pi\in\mathrm{NC}(n)$, $\kappa_\pi[x,\ldots,x]\ne 0$ only if all the blocks of $\pi$ have size $2$; we denote this set of {\bf noncrossing pairings} by $\mathrm{NC}_2(n)$ (which is nonempty only if $n$ is even).  From the moment-cumulant formula \eqref{e.moment.cumulant.f}, it follows that the even moment $\varphi(x^{2n})$ is equal to $t^n |\mathrm{NC}_2(2n)|$; this classic enumeration problem is one of the main examples that lead to the Catalan numbers $|\mathrm{NC}_2(2n)| = C_n = \frac{1}{n+1}\binom{2n}{n}$, which are indeed the even moments of the semicircular distribution.
    \item \label{ex.fields.ops.semicirc.free} In the free Gaussian $W^\ast$-probability space $(\Gamma_0(H),\tau_0)$ of Example \ref{ex.*-prob.spaces}(\ref{ex.*-prob.spaces.Fock}), the field operators $\omega(h) = \ell(h)+\ell^\ast(h)$ are semicricular with variance $\|h\|^2$.  Moreover: if $\{h_j\}_{j\in J}$ are orthogonal vectors in $H$, then $\{\omega(h_j)\}_{j\in J}$ are freely independent.
    \item \label{ex.circular.cumulants} Let $x,y$ be two freely independent semicircular random variables each of variance $t$; for example take $x=\omega(h)$ and $y=\omega(g)$ for $h\perp g$ with $\|h\|^2 = \|g\|^2=t$ as in (\ref{ex.fields.ops.semicirc.free}).  The random variable $c = \frac{1}{\sqrt{2}}(x+iy)$ is called Voiculescu's {\bf circular} element.  It is not a normal operator: $[c,c^\ast] = i[x,y]$, and since $x$ and $y$ are freely independent and nonconstant they cannot commute.  The free cumulants of $c,c^\ast$ are very easy to describe: $\kappa_n[c^{\varepsilon_1},\ldots,c^{\varepsilon_n}] = 0$ if $n\ne 2$, and the second-order cumulants are
    \[ \kappa_2[c,c] = \kappa_2[c^\ast,c^\ast]=0, \qquad \kappa_2[c,c^\ast] = \kappa_2[c^\ast,c] = t. \]
    Thus, from the moment-cumulant formula \eqref{e.moment.cumulant.f}, the $\ast$-distribution of $c$ is as follows: letting $\varepsilon = (\varepsilon_1,\ldots,\varepsilon_n) \in \{1,\ast\}^n$,
    \[ \varphi(c^{\varepsilon_1}\cdots c^{\varepsilon_n}) = t^n |\mathrm{NC}_2(\epsilon)| \]
    where $NC_2(\epsilon)$ is the set of noncrossing pairings where each block pairs a $1$ with a $\ast$ in the string $\varepsilon$.  See \cite{Kemp-JCTA} for detailed combinatorial analysis of this enumeration problem.
    \item \label{ex.Haar.cumulants} A random variable $u$ in a $\ast$-probability space $(\mathscr{A},\varphi)$ is {\bf Haar unitary} if it is unitary $uu^\ast = u^\ast u = 1$ and $\varphi(u^n)= \delta_{n0}$.  This $\ast$-distribution is encoded by the uniform probability distribution on the unit circle.  Haar unitaries can be constructed, for example, from semicircular random variables: if $x$ is semicircular of any variance then in the polar decomposition $x = ur$, $u$ is Haar unitary; i.e.\ $(xx^\ast)^{-1/2}x = |x|^{-1}x$ is Haar unitary when $x$ is semicircular.
    
    The free cumulants $\kappa_n[u,\ldots,u]$ and $\kappa_n[u^\ast,\ldots,u^\ast]$ all $=0$ for $n>0$, but this does not capture the free cumulants of mixtures of $u$ and $u^\ast$.  In \cite[Proposition 15.1]{SpeicherNicaBook}, it is shown that only the {\em alternating} free cumulants $\kappa_n[u,u^\ast,\ldots,u,u^\ast]$ and $\kappa_n[u^\ast,u,\ldots,u^\ast,u]$ are nonzero, taking value $(-1)^nC_{n-1}$ (signed Catalan numbers).
    \item \label{ex.p-Haar.unitary} Let $p\in\mathbb{N}$. A random variable $u$ in a $\ast$-probability space $(\mathscr{A},\varphi)$ is {\bf $p$-Haar unitary} if it is unitary $u^\ast u=uu^\ast=1$ and $\varphi(u^n)=1$ if $n$ is a multiple of $p$, and $=0$ otherwise.  The $\ast$-distribution corresponds to the discrete measure on the unit circle which has a mass of size $\frac{1}{p}$ at each of the $p$th roots of unity.
\end{enumerate}
\end{example}

Much like classical independence, free independence abounds in the sense that one can always extend the $\ast$-probability space to add free random variables with any individual distributions.  There is a free product construction, akin to Kolmogorov's extension theorem, that yields the following.

\begin{proposition}[Voiculescu \cite{Voiculescu1991}] \label{prop.Voiculescu.ext} Let $(\mathscr{A},\varphi)$ be a $W^\ast$-probability space.  Let $(\tau_i)_{i\in I}$ be any collection of $\ast$-distributions.  There is a $W^\ast$-probability space $(\mathscr{B},\eta)$ with $\mathscr{A}\subseteq\mathscr{B}$ and $\left.\eta\right|_{\mathscr{A}} = \varphi$ which contains $(b_i)_{i\in I}$ so that $\{\mathscr{A},(b_i)_{i\in I}\}$ are all freely independent and the $\ast$-distribution of $b_i$ is $\tau_i$ for each $i\in I$.
\end{proposition}

\subsubsection{Free Convolution and Analytic Tools of Free Probability} \hfill

\medskip

Free independence is a rule to determine joint $\ast$-distributions from individual $\ast$-distributions.  As a result, if $x,y$ are freely independent then the $\ast$-distributions of nice functions of the two, such as $x+y$ and $xy$, are directly computable in terms of the $\ast$-distributions of $x$ and $y$.

If $x$ and $y$ are selfadjoint random variables in a $\ast$-probability space $(\mathscr{A},\varphi)$, then $x+y$ is also selfadjoint, and the $\ast$-distributions of these three random variables are encapsulated in their spectral measures $\mu_x$, $\mu_y$, and $\mu_{x+y}$, all probability distributions on $\mathbb{R}$. Freeness of $x,y$ then implies that $\mu_{x+y}$ is totally determined by $\mu_x$ and $\mu_y$ separately.  If $x,y$ were classically independent random variables, we would have $\mu_{x+y} = \mu_x\ast\mu_y$ given by convolution.  In the free setting, we denote the relevant operation by $\mu_{x+y} = \mu_x\boxplus\mu_y$, {\bf free (additive) convolution}.

For bounded random variables in $\mathscr{A}$ whose distributions are determined by their moments, we can appeal to the method of moments to understand the operation $\boxplus$.  In a sense, Proposition \ref{prop.mixed.fcumulants} describes free convolution by allowing the computation of joint free cumulants from individual ones.  The same could be said of the vanishing of mixed classical cumulants describing classical convolution: it is technically true but not computationally meaningful.  In that setting, one effectively computes convolutions by using the relation between the cumulant generating function and the moment generating function.  Voiculescu introduced an analogous analytic theory of freeness \cite{Voiculescu1986}, where the key analytic tool is the Cauchy transform (also known as the Stieltjes transform).

\begin{definition} \label{def.Cauchy-transform}  For any Borel probability measure $\mu$ on $\mathbb{R}$, the {\bf Cauchy transform} of $\mu$ is the holomorphic function defined by
\[G_\mu(z) = \int_{\mathbb{R}}\frac{\mu(dx)}{z-x},\quad z\in\mathbb{C}\setminus\mathbb{R}.\]
Since $G_\mu(\bar z) = \overline{G_\mu(z)}$ it is customary to consider it as a function on the upper half-plane $\mathbb{C}^+ = \{z\in\mathbb{C}\colon \mathrm{Im}\,z >0\}$.  Its range is then contained in $\mathbb{C}^- = -\mathbb{C}^+$, the lower half-plane.  Note also that $\lim_{y\to\infty} iy G_\mu(iy) = 1$; indeed $zG_\mu(z)\to 1$ as $|z|\to\infty$ in any cone $\mathrm{Im}\,z \ge \epsilon |\mathrm{Re}\,z|$ for $\epsilon>0$.
%In the case that $\mu$ is compactly-supported on $\mathbb{R}$, then $G_\mu$ has a power series expansion at $\infty$ given by $G_\mu(z) = \frac{1}{z}\sum_{n\ge 0} m_n(\mu) \frac{1}{z^n}$ where $m_n(\mu) = \int t^n\,\mu(dt)$ are the moments of $\mu$.
\end{definition}

Much like the characteristic function in classical probability, the Cauchy transform $G_\mu$ encodes the probability distribution $\mu$ recoverably and robustly.

\begin{proposition} \label{prop.Cauchy.robust} Let $\mu$ be a probability measure on $\mathbb{R}$, with Cauchy transform $G_\mu$.  The {\bf Stieltjes inversion formula} recovers $\mu$ from $G_\mu$ as the weak limit $\mu_\epsilon \rightharpoonup \mu$ as $\epsilon\downarrow 0$, where $\mu_\epsilon$ has real analytic density $\varrho_\epsilon$
\[ \varrho_\epsilon(x) =  -\frac1\pi\mathrm{Im}\, G_{\mu}(x+i\epsilon). \]

Let $\{\mu_n\}_{n=1}^\infty$ be a sequence of probability measures on $\mathbb{R}$.
\begin{enumerate}
\item \label{robust.1} If $\mu_n\rightharpoonup \mu$ for some probability measure $\mu$, then $G_{\mu_n}$ converges to $G_\mu$ uniformly on compact subsets of $\mathbb{C}^+$.
\item \label{robust.2} Conversely, if $G_{\mu_n}$ converges pointwise to a function $G$ that is analytic on $\mathbb{C}^+$, then $G=G_\mu$ for some finite measure $\mu$ with $\mu(\mathbb{R})\le 1$, and $\{\mu_n\}_{n=1}^\infty$ converges vaguely to $\mu$.
\end{enumerate}
\end{proposition}
\noindent A detailed proof can be found in \cite[Section 8]{Kemp-RMT}.

For a selfadjoint random variable $x$ in a $\ast$-probability space, the free cumulants $\kappa_n^x:=\kappa_n[x,\ldots,x]$ grows at most exponentially in $\|x\|$.  As such, the formal power series $R_x(z) = \sum_{n\ge 1} \kappa_n^x z^n$ converges on a neighborhood of $0$ in $\mathbb{C}$.  For historical reasons, the shifted power series $\mathscr{R}_x(z) = \sum_{n\ge 0} \kappa_{n+1}^x z^n = R_x(z)/z$ was chosen instead; this is the {\bf $\mathscr{R}$-transform} of $x$.  Voiculescu showed that the moment-cumulant formula \eqref{e.moment.cumulant.f} in this context can be expressed as a functional relation between $G_{\mu_x}$ and $\mathscr{R}_x$, which we relable $\mathscr{R}_{\mu_x}$ for convenience.

\begin{proposition}[\cite{Voiculescu1986}] \label{prop.R-transform}
Let $\mu$ be a compactly-supported probability measure on $\mathbb{R}$.  The Cauchy transform $G_\mu$ is one-to-one on a neighborhood of $\infty$ in $\mathbb{C}$, and $\mathscr{R}_\mu(z):=G_\mu^{\langle-1\rangle}(z)-1/z$ is a well-defined analytic function on a neighborhood of $0$ in $\mathbb{C}$.  That is:
\begin{equation} \label{eq.G-R.relation}
G_{\mu}(\mathscr{R}_\mu(z)+1/z) = z.
\end{equation}
uniquely defines the analytic function $\mathscr{R}_\mu$.  If $x,y$ are selfadjoint random variables in a $W^\ast$-probability space, then they are freely independent iff
\begin{equation} \label{eq.R.free.additive}
\mathscr{R}_{x+y} = \mathscr{R}_x + \mathscr{R}_y.
\end{equation}
\end{proposition}
\noindent The connection between the $\mathscr{R}$-transform defined in Proposition \ref{prop.R-transform} and free cumulants (i.e.\ the free cumulants of $x$ are the power series coefficients of the analytic function $\mathscr{R}_{\mu_x}$ as described above) was discovered later by Speicher \cite{Speicher1994}.

Equations \eqref{eq.G-R.relation} and \eqref{eq.R.free.additive} together give a surprisingly computationally effective framework for computing the free additive convolution $\mu\boxplus\nu$ of two measures $\mu,\nu$.  Relabeling $\mathscr{R}_x = \mathscr{R}_{\mu_x}$ by the distribution of $x$, we have
\begin{align*} G_{\mu\boxplus\nu}(\mathscr{R}_\mu(z)+\mathscr{R}_\nu(z)+1/z)&=z \\
G_\mu(\mathscr{R}_\mu(z)+1/z)&=z \\
G_\nu(\mathscr{R}_\nu(z)+1/z)&=z.
\end{align*}
These functional equations allow, in principal, the computation of $G_{\mu\boxplus\nu}$ from $G_\mu$ and $G_\nu$; thence $\mu\boxplus\nu$ can be recovered via the Stieltjes inversion formula.

\begin{remark} Proving that the above equations uniquely determined $\mu\boxplus\nu$ in general requires some complex analysis, to prove in particular that Cauchy transforms are univalent on appropriate neighborhoods of $\infty$ with compatible domains and ranges.  The details can be found in \cite{BercoviciVoiculescu1993}.  Notably: since the Cauchy transform $G_\mu$ makes sense even of $\mu$ is not compactly-supported, the above can be taken as the definition of free additive convolution $\mu\boxplus\nu$ in the case that the measures $\mu,\nu$ are not compactly-supported.
\end{remark}

If $x,y$ are freely independent, then the $\ast$-distribution of the product $xy$ is also completely determined by the individual $\ast$-distributions.  In this case, however, even if both $x$ and $y$ are selfadjoint, their product is not (unless they commute, in which case freeness requires one of them to be constant).  That being said: if $x$ and $y$ are both positive semidefinite elements of a $W^\ast$-probability space, then $xy$, $\sqrt{x}y\sqrt{x}$, $\sqrt{y}x\sqrt{y}$, and $yx$ all have the same $\ast$-distribution.  The operators $\sqrt{x}y\sqrt{x}$ and $\sqrt{y}x\sqrt{y}$ are positive semidefinite, and so the common $\ast$-distribution can be described by the spectral measure $\mu_{\sqrt{x}y\sqrt{x}}$ which is a probability distribution on $[0,\infty)$.

If $x,y$ are freely independent and positive semidefinite then the the $\ast$-distirbution of $xy$, and hence the spectral measure $\mu_{\sqrt{x}y\sqrt{x}}$ is determined by the spectral measures $\mu_x$ and $\mu_y$; we denote it the {\bf free (multiplicative) convolution} $\mu_{\sqrt{x}y\sqrt{x}} = \mu_x\boxtimes\mu_y$.  Because $\sqrt{x}y\sqrt{x}$ and $\sqrt{y}x\sqrt{y}$ have the same $\ast$-distribution, $\mu_x\boxtimes\mu_y = \mu_y\boxtimes\mu_x$.  As with free additive convolution, there is an analytic theory of free multiplicative convolution, determined by another transform.

\begin{proposition}[\cite{Voiculescu1987}] \label{prop.S-transform} Let $\mu$ be a compactly-supported probability measure on $[0,\infty)$ that is not $\delta_0$. Define
\[ M_\mu(z) = \int_{\mathbb{R}} \frac{\mu(dx)}{1-zx} = \frac{1}{z}G_\mu(1/z). \]
Then $M_\mu$ is one-to-one and analytic on a neighborhood of $0$ in $\mathbb{C}$, with image containing a neighborhood of $1$ in $\mathbb{C}$.  The {\bf $S$-transform} of $\mu$, $S_\mu$, is defined on a neighborhood of $1 $ in $\mathbb{C}$ by
\begin{equation} \label{eq.S-transform} S_\mu(z) = \frac{1+z}{z}M_\mu^{\langle -1\rangle}(z). \end{equation}
If $x,y$ are positive semidefinite operators in a $W^\ast$-probability space, neither of which is $0$, then $S_{\mu_{\sqrt{x}y\sqrt{x}}} = S_{\mu_x}\cdot S_{\mu_y}$.
\end{proposition}

With free (additive) convolution $\boxplus$ in hand, we note the corresponding theory of {\bf subordination} which plays an important role in understanding conditional expectation in free probability.  For any probability distribution on $\R$, the Cauchy transform $G_\mu$ is univalent (invertible) on a neighborhood of $\infty$, with inverse $G_\mu^{\langle -1\rangle}$.  Generally the local inverse cannot be extended to all of $\mathbb{C}^+$; however, given two probability distributions $\mu_1,\mu_2$ on $\R$, the function
\[ \omega = G_{\mu_1\boxplus\mu_2}^{\langle-1\rangle}\circ G_{\mu_1} \]
does extend to an analytic map from $\mathbb{C}^+$ to $\mathbb{C}^+$.  These subordinator functions, introduced by Voiculescu \cite{Voiculescu1993} and Biane \cite{Biane1998}, will play an important role in our proofs of Theorems \ref{thm.Brown.Measure.1}, \ref{thm.Brown.Measure.2}, \ref{thm.Brown.Measure.3}, and \ref{thm.lima.bean}, so we collect important facts about them here.

For convenience, we also define
\[H_\mu(z) = \frac{1}{G_\mu(z)}-z,\quad z\in\mathbb{C}^+.\]
The function $H_\mu(z)$ is analytic on $\mathbb{C}^+$ and maps $\mathbb{C}^+$ to itself.

The following theorem, combining \cite{Voiculescu1993, Biane1998, Belinschi2008, Belinschi2014, BelinschiBercoviciHo2022}, describes the subordination property of free convolution of probability measures.
\begin{theorem}
    \label{thm.subordination.general}
    Let $\mu_1$ and $\mu_2$ be probability measures on $\mathbb{R}$ and write $\mu=\mu_1\boxplus\mu_2$ be the free convolution of $\mu_1$ and $\mu_2$. If neither of them is a point mass, then there exist unique continuous function $\omega_1, \omega_2:\overline{\mathbb{C}^+}\to\overline{\mathbb{C}^+}$ that are analytic on $\mathbb{C}^+$ such that
    \begin{enumerate}
        \item $G_\mu(z) = G_{\mu_1}(\omega_1(z))=G_{\mu_2}(\omega_2(z))$, for all $z\in\mathbb{C}^+$
        \item $\omega_1(z)+\omega_2(z) = z+\frac{ 1}{G_\mu(z)}$, for all $z\in\mathbb{C}^+$
        \item For each $z\in\overline{\mathbb{C}^+}$, $\omega_1(z)$ is the Denjoy--Wolff point of the analytic self-map on $\mathbb{C}^+$
        \[w\mapsto z+H_{\mu_2}(z+H_{\mu_1}(w)),\quad w\in\mathbb{C}^+.\]
        If we interchange $H_{\mu_1}$ and $H_{\mu_2}$ in the above function, the Denjoy--Wolff point of the resulting function is $\omega_2(z)$. In particular, $\omega_1(0)$ is the Denjoy--Wolff point of $H_{\mu_2}\circ H_{\mu_1}$ and $\omega_2(0)$ is the Denjoy--Wolff point of $H_{\mu_1}\circ H_{\mu_2}$.
    \end{enumerate}
\end{theorem}

\subsubsection{Asymptotic Freeness} \hfill

\medskip

While it is always possible to generate more freely independent random variables with prescribed $\ast$-distributions by potentially expanding the $\ast$-probability space (cf.\ Proposition \ref{prop.Voiculescu.ext}, in the confines of a fixed $\ast$-probability space without expansion freeness may be difficult to realize.  This is especially true in spaces of (random) matrices, where freeness generally only occurs for degenerate reasons (e.g.\ constant multiples of the identity are free from all matrices).  The connection between freeness and random matrices mostly comes asymptotically.

\begin{definition} Let $(\mathscr{A}^N,\varphi^N)_{N\in\mathbb{N}}$ be $\ast$-probability spaces, and let $(\mathbf{a}^N_i)_{i\in I}$ be collections of random variables with $\mathbf{a}^N_i\in\mathscr{A}^N$ for each $N\in\mathbb{N}$ and $i\in I$.  Say $(\mathbf{a}^N_i)_{i\in I}$ are {\bf asymptotically freely independent} if there is a $\ast$-probability space $(\mathscr{A},\varphi)$ containing freely independent collections $(\mathbf{a}_i)_{i\in I}$ such that $\mathbf{a}^N_i\rightharpoonup\mathbf{a}_i$ in $\ast$-distribution as $N\to\infty$ for each $i\in I$.
\end{definition}

The main connection between free probability theory and random matrices comes from many important examples of multi-matrix ensembles demonstrating asymptotic freeness.

\begin{definition} \label{def.classic.ensembles} A {\bf Gaussian Unitary Ensemble} ($\mathrm{GUE}$) is a random Hermitian matrix $X=X^N\in \MNC$ whose density with respect to the Lebesgue measure on the space of Hermitian matrices is $Z_N \exp(-c_N \mathrm{Tr}[X^2])$ for $c_N>0$ and appropriate normalizing constant $Z_N$.  Equivalently: the upper triangular entries are independent, with those strictly above the diagonal complex normal random variables and those on the diagonal real normal random variables.  The scaling most useful for us is the one that gives the entries of $X^N$ variance $\frac1N$, i.e.\ so that $\frac{1}{N}\mathbb{E}\mathrm{Tr}[X^2] = 1$; this corresponds to $c_N = N$.

A {\bf Ginibre Ensemble} is a random matrix $W=W^N\in \MNC$ whose density with respect to Lebesgue measure on $\MNC$ is $Z_N \exp(-c_N \mathrm{Tr}[X^\ast X])$ for some scale constant $c_N>0$ and normalizing constant $Z_N$.  Equivalently: all $N^2$ entries of $W$ are i.i.d.\ complex normal random variables.  Again, our choice for $c_N$ is $N$ to make the variance of entries $\frac{1}{N}$.

A {\bf Haar Unitary Ensemble} is a random matrix $U=U^N\in \mathrm{U}(N)$ whose distribution is the Haar probability measure on the group $\mathrm{U}(N)$.
\end{definition}

\begin{remark} The three ensembles in Definition \ref{def.classic.ensembles} are connected to each other in the following ways.
\begin{itemize}
    \item If $X$ and $Y$ are independent $\mathrm{GUE}$s then $W = \frac{1}{\sqrt{2}}(X+iY)$ is a Ginibre Ensemble.  (This follows from elementary computation.)
    \item If $X$ is a $\mathrm{GUE}$, then in its diagonalization $X = U\Lambda U^\ast$, the matrices $U$ and $\Lambda$ are independent, and $U$ is a Haar Unitary Ensemble.  (This follows from the change of variables theorem and rotational invariance properties of Gaussians, cf.\ \cite{Kemp-RMT}.)
    \item If $W$ is a Ginibre ensemble, then in its singular value decomposition $W = U T V^\ast$ the matrices $U$, $V$, and $T$ are all independent, and $U$, $V$ are Haar Unitary Ensembles.  (This follows more generally for bi-invariant ensembles as discussed in the next section.)
\end{itemize}
\end{remark}

The following proposition summarizes the main tools for constructing examples of asymptotic freeness and convergence of $\ast$-distributions.

\begin{proposition}[\cite{SpeicherNicaBook}] \label{prop.asymp.free.Haar.GUE} Let $\mathbf{X}^N = (X_1^N,\ldots,X_n^N)$, $\mathbf{U}^N = (U_1^N,\ldots,U_m^N)$, and $\mathbf{T}^N = (T_1^N,\ldots,T_\ell^N)$ all be independent random matrices with $X_j^N$ $\mathrm{GUE}$s, $U_j^N$ Haar Unitary Ensembles, and $T_j^N$ selfadjoint random matrices.  Suppose that there is a $\ast$-probability space containing a tuple $\boldsymbol{\vartheta} = (\vartheta_1,\ldots,\vartheta_\ell)$ with $\mathbf{T}^N\rightharpoonup\boldsymbol{\vartheta}$ converging in $\ast$-distribution a.s.  Then there is a $W^\ast$-probability space containing $\mathbf{a}$ along with $\mathbf{x} = (x_1,\ldots,x_n)$ and $\mathbf{u} = (u_1,\ldots,u_m)$ all freely independent and free from $\boldsymbol{\vartheta}$, with $x_j$ semicircular, $u_j$ Haar unitary, and $(\mathbf{X}^N,\mathbf{U}^N,\mathbf{T}_N)\rightharpoonup(\mathbf{x},\mathbf{u},\boldsymbol{\vartheta})$ converging in $\ast$-distribution a.s. 
\end{proposition}

As Ginibre Ensembles have the form $\frac{1}{\sqrt{2}}(X+iY)$ for independent $\mathrm{GUE}$s (that are asymptotically free from each other) while circular random variables have the form $\frac{1}{\sqrt{2}}(x+iy)$ for freely independent semicircular random variables, Proposition \ref{prop.asymp.free.Haar.GUE} immediately implies that independent Ginibre Ensembles are asymptotically free from independent $\mathrm{GUE}$, Haar Unitary, and other independent ensembles, and converge in $\ast$-distribution a.s.\ to freely independent circular random variables.

\begin{remark} More recent results show that the convergence in Proposition \ref{prop.asymp.free.Haar.GUE} is not only in $\ast$-distribution but also {\em strong convergence} in distribution -- meaning that not only traces of noncommutative polynomials but also operator norms of noncommutative polynomials converge -- provided the operator norms of $\mathbf{A}^N$ are a.s.\ uniformly bounded for large $N$. \end{remark}

\subsubsection{Haar Unitaries, Bi-Invariant Ensembles, and $\mathscr{R}$-Diagonal Elements\label{sec.R-diag}} \hfill

\medskip

It is always possible to ``liberate'' random matrices from each other as follows.

\begin{proposition}[\cite{SpeicherNicaBook}] \label{prop.conj.free} Let $(A_i = A^N_i)_{i\in I\sqcup 0}$ be random matrices in $\MNC$, and suppose that there is a $\ast$-probability space $(\mathscr{A},\varphi)$ containing elements $(a_i)_{i\in I\sqcup 0}$ so that, for each $i$, $A^N_i\rightharpoonup a_i$ in $\ast$-distribution a.s.  Let $(U_i = U^N_i)_{i\in I}$ be independent Haar Unitary Ensembles in $\mathrm{U}(N)$.  Let $(u_i)_{i\in I}$ be freely independent Haar unitaries in $\mathscr{A}$ (expanding if necessary), free from $(a_i)_{i\in I}$.  Then:
\begin{itemize}
    \item[(a)] $(A_0, U_iA_iU_i^\ast)_{i\in I}$ converges a.s.\ in $\ast$ distribution to $(a_0,u_ia_iu_i^\ast)_{i\in I}$.
    \item[(b)] $\{a_0,u_ia_iu_i^\ast\}_{i\in I}$ are all freely independent, and for each $i\in I$ $u_ia_iu_i^\ast$ has the same $\ast$-distribution as $a_i$.
\end{itemize}
In particular: $(A_0,U_iA_iU_i^\ast)_{i\in I}$ are asymptotically free a.s., and have the same component-wise $\ast$-distributions as $(A_i)_{i\in 0\sqcup I}$.
\end{proposition}

Proposition \ref{prop.conj.free} highlights the central role that Haar unitaries take in free probability: conjugating by free Haar unitaries is the canonical way to construct free random variables with given $\ast$-distributions.  What's more: free Haar unitaries ``absorb'' other unitaries and ``spread'' freeness.  We formalize and prove this folklore statement as follows.

\begin{lemma}
    \label{lem.free.Haar}
    Let $(u_i)_{i\in I}$ be freely independent Haar unitaries in $(\mathscr{A},\varphi)$. Let $\mathscr{B}$ be a subalgebra of $\mathscr{A}$ and $(v_i)_{i\in I}$ and $(w_i)_{i\in I}$ be unitaries in $\mathscr{B}$. Assume that $(u_i)_{i\in I}$ and $\mathscr{B}$ are freely independent. Then $\varpi_i := v_iu_iw_i$ are Haar unitaries for $i\in I$, and $(\varpi_i)_{i\in I},\mathscr{B}$ are freely independent.
\end{lemma}
\begin{proof}
    First, $(\varpi_i)_{i\in I}$ are all unitary (as they are products of unitaries). To show that they are indeed all Haar unitaries, we make the observation that, for each $j\in I$ and $\theta\in\mathbb{R}$, $u_j\equaldist e^{i\theta}u_j$.  Since $u_j$ is freely independent from $\mathscr{B}\ni v_j,w_j$, it follows that $(u_j,v_j,w_j) \equaldist (e^{i\theta} u_j,v_j,w_j)$.  Hence, for any $m\in\mathbb{N}$,
    \[\varphi((v_ju_jw_j)^m) = \varphi((v_j(e^{i\theta}u_j)w_j)^m) = e^{im\theta}\varphi((v_ju_jw_j)^m)\]
    for any $\theta\in\mathbb{R}$. This shows any non-zero moments of $\varpi_j$ must be $0$.

    To show freeness of $(\varpi_i)_{i\in I},\mathscr{B}$, we need to show that any $\ast$-monomial of the form
    \begin{equation}
    \label{eq.free.Haar.moments}
    (v_{i_1}u_{i_1}w_{i_1})^{m_1}a_1(v_{i_2}u_{i_2}w_{i_2})^{m_2}a_2\cdots a_{n-1}(v_{i_n}u_{i_n}w_{i_n})^{m_n}
    \end{equation}
    is centered, where $m_j$ are integers, $m_2,\ldots,m_{n-1}$ are nonzero, and $a_j\in\mathscr{B}$ is either centered or is a scalar but $i_j\neq i_{j+1}$. This is the general form of products of alternating centered words from $v_1u_1w_1,\ldots,v_mu_mw_m$, or $\mathscr{B}$. Since $v_i,w_i\in \mathscr{B}$ for each $i\in I$, we can rewrite \eqref{eq.free.Haar.moments} as
    \[b_1u_{j_1}^{\varepsilon_{1}}b_2u_{j_2}^{\varepsilon_{2}}\ldots b_lu_{j_l}^{\varepsilon_{l}}b_{l+1}\]
    where $j_1,\ldots,j_\ell$ are $i_1,\ldots,i_n$ with (possible) repeats, and $\varepsilon_l\in\{1,\ast\}$. By the structure in \eqref{eq.free.Haar.moments}, if $j_k=j_{k+1}$ then either $\varepsilon_{k} = \varepsilon_{k+1}$ or $b_{k+1}$ is centered (in which case $b_{k+1}$ is a unitary conjugation of one of the $a$'s in \eqref{eq.free.Haar.moments}). Now we calculate, using the moment-cumulant formula,
    \[\varphi(b_1u_{j_1}^{\varepsilon_{1}}b_2u_{j_2}^{\varepsilon_{2}}\ldots b_lu_{j_l}^{\varepsilon_{l}}b_{l+1}) = \sum_{\pi\in \mathrm{NC}(2l+1)}\kappa_\pi[b_1, u_{j_1}^{\varepsilon_1},b_2, u_{j_2}^{\varepsilon_2},\ldots, u_{j_l}^{\varepsilon_l},b_{l+1}].\]
    Since $(u_j)_{i\in I},\mathscr{B}$ are $\ast$-free, the only terms that may be nonzero are those $\pi\in \mathrm{NC}(2l+1)$ that decompose as $\pi = \pi_o\sqcup\pi_e$, where $\pi_o\in \mathrm{NC}(\{1,3,\ldots,2l+1\})$ and $\pi_e\in \mathrm{NC}(\{2,4,\ldots,2l\})$.

    Since $\pi_e$ is non-crossing, at least one block must be an interval (\cite{SpeicherNicaBook}). As $(u_j)_{i\in I}$ are free Haar unitaries, the only possibility that $\kappa_\pi$ is nonzero is when this interval is of even length and $\kappa_\pi$ factors into terms including
    $\kappa_{2p}[u_j, u_j^*,\ldots,u_j,u_j^*]$ or $\kappa_{2p}[u_j^*,u_j\ldots,u_j^*,u_j]$; see Example \ref{ex.freeness}(\ref{ex.Haar.cumulants}). But then since $\pi$ is non-crossing, $\pi_o$ contains a singleton $\{b_g\}$ where $b_g$ is in between $u_j$ and $u_j^*$. Thus, $\kappa_1(b_g)$ is a factor of $\kappa_\pi$. By the preceding paragraph, the structure of \eqref{eq.free.Haar.moments} forces $b_g$ to be centered; hence $\kappa_1(b_g) = 0$, showing $\kappa_\pi = 0$. We have completed the proof that any $\ast$-moment of the form \eqref{eq.free.Haar.moments} is centered. Therefore, our lemma is established.
    \end{proof}

The preceding proof depends fundamentally on the special structure of the free cumulants of a Haar unitary $u$: among all possible cumulants $\kappa_n[u^{\varepsilon_1},\ldots,u^{\varepsilon_n}]$, the only non-zero ones have $n$ even and have the sequence $\varepsilon_j\in\{1,\ast\}$ alternating (cf.\ Example \ref{ex.freeness}(\ref{ex.Haar.cumulants}).  Note from Example \ref{ex.freeness}(\ref{ex.circular.cumulants}) that circular elements have the same property: only alternating adjoints have (potentially) non-zero free cumulants.  Such random variables are called $\mathscr{R}$-diagonal.

\begin{definition} \label{def.R-diag} A random variable $a$ in a $\ast$-probability space is {\bf $\mathscr{R}$-diagonal} if, for all $n\in\mathbb{N}$ and $\varepsilon_1,\ldots,\varepsilon_n\in\{1,\ast\}$,
\[ \kappa_n[a^{\varepsilon_1},\ldots,a^{\varepsilon_n}] = 0 \quad \text{unless} \quad n\text{ is even and}\;\varepsilon_j\ne \varepsilon_{j+1}\text{ for }1\le j<n. \]
\end{definition}

The name $\mathscr{R}$-diagonal comes from a (noncommutative) multivariate generalization of the $\mathscr{R}$-transform: $\mathscr{R}_a$, which is the formal power series of the free cumulants of $a$:
\[ \mathscr{R}_a(z,z^\ast) = \sum_{n=0}^\infty \sum_{\varepsilon\in\{1,\ast\}^n} \kappa_n[a^{\varepsilon_1},\ldots,a^{\varepsilon_n}] z^{\varepsilon_1}\cdots z^{\varepsilon_n} \]
where $z,z^\ast$ are formal non-commuting indeterminates; $a$ is $\mathscr{R}$-diagonal iff $\mathscr{R}_a$ is supported on the ``diagonal'' terms $(z^\ast z)^n$ and $(zz^\ast)^n$.

Circular operators and Haar unitaries are prominent examples of $\mathscr{R}$-diagonal $\ast$-distributions, but the class is much larger.  Indeed, set $\alpha_n(a) = \kappa_{2n}[a,a^\ast,\ldots,a,a^\ast] = \kappa_{2n}[a^\ast,a,\ldots,a^\ast,a]$ (the equality follows from the traciality of the state).  The sequence $(\alpha_n(a))_{n\in\mathbb{N}}$ completely determines all free cumulants of $a$, and hence its $\ast$-distribution.  Letting $\alpha_\pi(a)$ denote the mutiplicative extension of the sequence $\alpha_n(a)$ for $\pi\in\mathrm{P}(n)$, in \cite[Proposition 15.6]{SpeicherNicaBook} it is proved that
\[ \kappa_n[a^\ast a,\ldots,a^\ast a] = \kappa_n[aa^\ast,\ldots,aa^\ast] = \sum_{\pi\in\mathrm{NC}(n)} \alpha_\pi(a). \]
As with the moment cumulant formula, the right-hand sums can be inductively inverted to compute $\alpha_n(a)$; it therefore follows that the $\ast$-distribution of $a$ is completely determined by the free cumulants on the left-hand-side, which encode the $\ast$-distribution of $a^\ast a$ (and $aa^\ast$ which has the same $\ast$-distribution).  We summarize this and related fact in the following.

\begin{proposition}[Lecture 15, \cite{SpeicherNicaBook}] \label{prop.R-diag} Let $(\mathscr{A},\varphi)$ be a $W^\ast$-probability space.
\begin{enumerate}
    \item If $a\in\mathscr{A}$ is $\mathscr{R}$-diagonal, the $\ast$-distribution of $a$ is completely determined by the (common) $\ast$-distribution of $a^\ast a$ or $aa^\ast$.
    \item \label{prop.R-diag.b} A random variable $a\in\mathscr{A}$ is $\mathscr{R}$-diagonal iff for any Haar unitary $u$ freely independent from $a$, $a$ and $ua$ have the same $\ast$-distribution.
    \item \label{prop.R-diag.c} If $a\in\mathscr{A}$ is $\mathscr{R}$-diagonal, then $a$ has the same $\ast$-distribution as $uq$ where $u$ is Haar unitary, $q = \sqrt{a^\ast a}$, and $u,q$ are freely independent.  Also, $a$ has the same $\ast$-distribution as $u\tilde{q}$, where $u$ is Haar unitary free from $\tilde{q}$, a selfadjoint operator whose distribution is the symmetrization of the distribution of $\sqrt{a^\ast a}$.
\end{enumerate}
\end{proposition}

It is clear from Definition \ref{def.R-diag} that $a$ is $\mathscr{R}$-diagonal iff $a^\ast$ is.  It therefore follows from Proposition \ref{prop.R-diag}\ref{prop.R-diag.b} that $a$ is $\mathscr{R}$-diagonal iff $a$ and $au$ have the same $\ast$-distribution for any Haar unitary free from $a$ (since $u^\ast$ is also Haar unitary).  Iterating this with \ref{prop.R-diag.b} and \ref{prop.R-diag.c} above yields the following.

\begin{corollary} \label{cor.R-diag.SVD} In a $W^\ast$-probability space, a random variable $a$ is $\mathscr{R}$-diagonal iff it has the same $\ast$-distribution as $uqv^\ast$ where $u,v$ are Haar unitaries, $q$ is a non-negative definite operator (with the same distribution as $\sqrt{a^\ast a}$), and $u,v,q$ are freely independent.
\end{corollary}

In particular: there is a one-to-one correspondence between (compactly-supported) probability distributions on $[0,\infty)$ (and symmetric compactly supported distributions on $\mathbb{R}$) and $\mathscr{R}$-diagonal $\ast$-distributions.  This also suggests what random matrix ensembles should approximate $\mathscr{R}$-diagonal elements.

\begin{proposition} \label{prop.bi-inv} For a random matrix $A \in \MNC$, the following are equivalent.
\begin{enumerate}
    \item \label{bi-inv} $A$ is unitarily bi-invariant: for any deterministic $Q_1,Q_2\in\mathrm{U}(N)$, the random matrices $A$ and $Q_1AQ_2^\ast$ have the same distribution.
    \item \label{bi-inv.SVD} There are independent random matrices $U,V,T$ with $U,V$ Haar Unitary Ensembles, $T$ diagonal with non-negative entries, and $A$ has the same distribution as $UTV^\ast$.
\end{enumerate}
Moreover, if $A=A^N$ has the same distribution as $UTV^\ast$ as above, and if the $\mathrm{ESD}$ of the diagonal matrix $T=T^N$ has an a.s.\ limit distribution $\nu$, then $A^N$ converges a.s.\ in $\ast$-distribution as $N\to\infty$ to an $\mathscr{R}$-diagonal random variable $a$ for which $\sqrt{a^\ast a}$ has distribution $\nu$.
\end{proposition}

\begin{proof} First assume \ref{bi-inv.SVD}.  Since $U$ is Haar distribution, if $Q$ is any deterministic unitary matrix, $QU$ is also Haar distributed.  Thus, since $U,V,T$ are independent, for $Q_1,Q_2\in\mathrm{U}(N)$, $(Q_1U,Q_2 V,T)$ has the same distribution as $(U,V,T)$, and hence
\[ A \equaldist UTV^\ast \equaldist (Q_1U)\cdot T\cdot (Q_2 V)^\ast = Q_1\cdot UTV^\ast\cdot Q_2^\ast \equaldist Q_1AQ_2. \]
This establishes the bi-invariance in \ref{bi-inv}.

Conversely, suppose $A$ is bi-invariant, i.e.\ \ref{bi-inv}.  Fix any $\sigma_1\ge\sigma_2\ge\cdots\ge\sigma_N\ge 0$, and consider the conditional distribution of $A$ given its singular values are $\sigma_1,\ldots,\sigma_j$.  This law is supported on the $\mathrm{U}(N)\times\mathrm{U}(N)$-orbit of the matrix $T=\mathrm{diag}(\sigma_1,\ldots,\sigma_N)$ under the map $(U,V)\cdot T = UTV^\ast$.  That orbit is the quotient of the group $\mathrm{U}(N)\times\mathrm{U}(N)$ by the stabilizer $H$ of $T$.  The bi-invariance of $H$ implies that this conditional law is bi-invariant; but there is a unique probability measure on $\mathrm{U}(N)\times\mathrm{U}(N)/H$ that is bi-invariant: the push-forward of the Haar measure on $\mathrm{U}(N)\times\mathrm{U}(N)$ under the quotient map $(U,V)\mapsto (U,V)\cdot T = UTV^\ast$.

Hence, the conditional distribution of of $A$ given $T=\mathrm{diag}(\sigma_1,\ldots,\sigma_N)$ is the push-forward of the Haar measure on $\mathrm{U}(N)\times\mathrm{U}(N)$ under the quotient map -- which does not depend on $\sigma_1,\ldots,\sigma_N$.  That means precisely that the marginal distribution of $(U,V)$ is independent from the distribution of $T$, and that this marginal distribution is the Haar measure on $\mathrm{U}(N)\times\mathrm{U}(N)$, which is the two-fold product of the Haar measure on $\mathrm{U}(N)$.  In other words, the marginal distributions of $U,V,T$ are all independent, and $U,V$ are Haar unitary, establishing \eqref{bi-inv.SVD}.

Finally, if $T=T^N$ is independent from the independent Haar Unitary Ensembles $U=U^N$ and $V=V^N$, and if $T^N\rightharpoonup q$ in probability in $\ast$-distribution, then by Proposition \ref{prop.asymp.free.Haar.GUE} it follows that $(U^N,V^N,T^N)\rightharpoonup(u,v,q)$ in probability in $\ast$-distribution, where $u,v$ are Haar unitaries and $u,v,q$ are freely independent.  Thus $UTV^\ast \rightharpoonup uqv^\ast$, which is $\mathscr{R}$-diagonal with $\sqrt{a^\ast a}$ sharing the same $\ast$-distribution as $q$ by Corollary \ref{cor.R-diag.SVD}.  This concludes the proof.
\end{proof}

\begin{remark} Given Proposition \ref{prop.R-diag}\ref{prop.R-diag.c}, one might expect a random matrix model for $\mathscr{R}$-diagonal random variables to be simply $A = UT$ where $U$ is a Haar Unitary Ensemble independent from the diagonal matrix $T$.  This also works, but is equivalent (in terms of eigenvalues and $\ast$-distribution) to the model in Proposition \ref{prop.bi-inv}.  Indeed, let $V$ be another Haar Unitary Ensemble independent from $U,T$.  Then $T$ has the same eigenvalues as $VTV^\ast$, and $U$ is independent from $VTV^\ast$, so the joint $\ast$-distribution of $(U,T)$ is the same as that of $(U,VTV^\ast)$.  Hence $UT$ has the same $\ast$-distribution as $U\cdot VTV^\ast$.  If $T=T^N$ converges in $\ast$-distribution to $q$, Proposition \ref{prop.conj.free} implies that $U\cdot VTV^\ast$ converges in $\ast$-distribution to $uvqv^\ast$; and Lemma \ref{lem.free.Haar} shows that $\tilde{u} = uv$ is Haar unitary and $\tilde{u},v,q$ are all freely independent; hence $UT$ converge in $\ast$-distribution to $uqv^\ast$.
\end{remark}

\begin{example} \begin{enumerate}
    \item From Proposition \ref{prop.bi-inv}, we see one constructive way to generate unitarily bi-invariant random matrix ensembles is by fixing their singular values: Let $T^N$ be a (deterministic) diagonal matrix with non-negative entries on the diagonal.  If $U^N,V^N$ are two independent Haar Unitary Ensembles, then $A^N:= U^N T^N V^{N\ast}$ is a bi-invariant ensemble.  If the values on the diagonal of $T^N$ are chosen so their empirical law $\nu^N$ converges weakly to some probability distribution $\nu$ on $[0,\infty)$, then $A^N$ converges in $\ast$-distribution to the $\mathscr{R}$-diagonal distribution determined by $\nu$.
    \item A different approach is to mimic the joint density of the Ginibre Ensemble.  Fix some convex function $V\colon[0,\infty)\to[0,\infty)$, and let $A^N$ be a random matrix whose joint density $\rho_{A^N}$ of entries is
    \[ \rho_{A^N}(X) = C_N \cdot \exp\left(-c_N \mathrm{Tr}[V(X^\ast X)]\right) \]
    where $c_N,C_N>0$ are chosen so that $\rho_{A^N}$ is a probability density.  (Here $V(X^\ast X)$ is defined through the spectral theorem; $V$ is applied to the eigenvalues of $X^\ast X$.)  For example, the Ginibre ensemble has this form with $V(x)=x$ and $c_N = N$.  Because the density explicitly depends on $X$ only through its singular values, it is straightforward to check that this matrix is unitarily bi-invariant.  If $C_N \sim N$, $A^N$ will converge to an $\mathscr{R}$-diagonal $a$, determined (in an indirect fashion) by $V$.
\end{enumerate}
\end{example}

\begin{myproof}{\it Lemma}{\ref{lem.Bk->bk}} Since the random matrices $\{A_1^N,\ldots,A_k^N\}$ are independent and unitarily bi-invariant, by Proposition \ref{prop.bi-inv} their singular value decompositions have the form $A_j^N = U_j^N T_j^N V_j^{N\ast}$ where $\{U_1^N,\ldots,U_k^N,V_1^N,\ldots,V_k^N\}$ are all independent Haar Unitary Ensembles and $\{T_1^N,\ldots,T_k^N\}$ are (diagonal) selfadjoint matrices independent from the $U_j^N$ and $V_j^N$.  The $\mathrm{ESD}$ of $T_j^N$ equals $\nu^N_j$ which convergess weakly to $\nu_j$.
By Proposition \ref{prop.asymp.free.Haar.GUE}, $\{U_1^N,\ldots,V_k^N,T^N_0,T_1^N,\ldots,T_k^N\}$ converge jointly in $\ast$-distribution to $\{u_1,\ldots,v_k,\vartheta_0,\vartheta_1,\ldots,\vartheta_k\}$, all freely independent with $u_j,v_j$ Haar unitary and $\vartheta_j$ having distribution $\nu_j$.

Let $a_j = u_j\vartheta_jv_j^\ast$.  By Corollary \ref{cor.R-diag.SVD}, $a_j$ are $\mathscr{R}$-diagonal; they are freely independent because the $u_j,v_j,\vartheta_j$ are all freely independent.  Note that $a_j^\ast a_j = v_j\vartheta_j^2 v_j^\ast$ has the same distribution as $\vartheta_j^2$, and therefore $|a_j|$ has the same distribution as $\vartheta_j$ which is $\nu_j$, as claimed.  Also, $u_0:=\exp(i\vartheta_0)$ is freely independent from $a_1,\ldots,a_k$.

Finally, note that there is a fixed noncommutative polynomial $P$ in $3k+1$ variables so that
\begin{align*} U_0^N B^N_k(t) &= P\left(\exp(iT^N_0),(U_j^N,V_j^N,T_j^N)_{1\le j\le k}\right), \qquad \text{and} \\
u_0 b_k(t) &= P\left(\exp(i\vartheta_0),(u_j,v_j,\vartheta_j)_{1\le j\le k}\right).
\end{align*}
The convergence in $\ast$-distribution of $U_0^N B_k^N(t)$ to $u_0 b_k(t)$ follows immediately (cf.\ Definition \ref{def.conv.*-dist} of convergence in $\ast$-distribution).
\end{myproof}

\subsubsection{Free Brownian motion\label{sect.fbm}} \hfill

\medskip

Let $(\mathscr{A},\varphi)$ be a $W^\ast$-probability space that contains countably many freely independent semicircular random variables.  Then $\mathscr{A}$ also contains {\bf free Brownian motion}: a ``process'' $x(t))_{t\ge0}$ with the properties
\begin{enumerate}
\item $t\mapsto x(t)$ is norm-continuous.
\item \label{fBm.2} For $t>s\ge 0$, the increments $x(t)-x(s)$ is semicircular with variance $t-s$.
\item If $\mathscr{A}_t = W^\ast\{x(s)\colon s\le t\}$, then $x(t)-x(s)$ is freely independent from $\mathscr{A}_s$ for $t>s\ge 0$.
\end{enumerate}

Similarly, one may construct a {\bf circular Brownian motion} $(w(t))_{t\ge 0}$ simply as $w(t) = \frac{1}{\sqrt{2}}(x(t)+iy(t))$ where $x,y$ are two freely independent free Brownian motions; it is equivalently characterized to the above, where \eqref{fBm.2} is replaced with the condition that the increments $w(t)-w(s)$ are circular random variables with variance $t-s$.  Both processes arise naturally as large-$N$ limits of matrix Brownian motions, where the notion of convergence is as follows.

\begin{definition} \label{def.conv.fd*d} Let $a_N = (a_N(t))_{t\ge 0}$ be a process in a $\ast$-probability space $(\mathscr{A}_N,\varphi_N)$ for each $N\in\mathbb{N}$, and let $a = (a(t))_{t\ge 0}$ be a process in another $\ast$-probability space $(\mathscr{A},\varphi)$.  Say that $a_N$ {\bf converges to $a$ in finite dimensional $\ast$-distributions} if, for any $t_1,\ldots,t_k\ge 0$, the tuple $(a_N(t_1),\ldots,a_N(t_k))$ converges in $\ast$-distribution to the tuple $(a(t_1),\ldots,a(t_k))$, cf.\ Definition \ref{def.conv.*-dist}; i.e.\ if, for any $n\in\mathbb{N}$, any $i_1,\ldots,i_n\in\{1,\ldots,k\}$, and any $\varepsilon_1,\ldots,\varepsilon_n\in\{1,\ast\}$,
\begin{equation} \label{eq.fd*d} \lim_{N\to\infty} \varphi_n[a_N(t_{i_1})^{\varepsilon_1}\cdots a_N(t_{i_n})^{\varepsilon_n}] 
= \varphi_n[a(t_{i_1})^{\varepsilon_1}\cdots a(t_{i_n})^{\varepsilon_n}]. \end{equation}
If the process $a_N$ is classically stochastic (i.e.\ each $a_N(t)$ is a $\mathscr{A}$-valued random variable, all on a common classical probability space), we say $a_N$ converges is finite-dimensional $\ast$-distributions a.s. / in probability if each $\ast$-moment in \eqref{eq.fd*d} converges a.s. / in probability accordingly.
\end{definition}

Let $W^N(t)$ denote {\bf matrix Brownian motion}:\ the process of the form \eqref{eq:Lie.BM} on the vector space $\MNC$ equipped with the inner product $\langle A,B\rangle = N\,\mathrm{Re}\,\mathrm{Tr}(B^\ast A)$.  Let $X^N(t)$ denote {\bf Hermitian Brownian motion}, on the subspace $\mathrm{M}_N^{\mathrm{sa}}(\C)$ of Hermitian matrices.  In \cite{Voiculescu1991}, Voiculescu proved that $(X^N(t))_{t\ge 0}$ converges in finite dimensional $\ast$-distributions to $(x(t))_{t\ge0}$ a.s.; it follows that $(W^N(t))_{t\ge0}$ converges in finite dimensional $\ast$-distributions to $(w(t))_{t\ge0}$ a.s.

The Brownian motion $(B^N(t))_{t\ge 0}$ on $\mathrm{GL}(N,\C)$   may be constructed as the strong solution to the matrix SDE
\begin{equation}
dB^N(t) = B^N(t)\,dW^N(t), \qquad B^N(0)=I
\end{equation}
cf.\ Definition \ref{def.BM.on.G} (together with \eqref{eq.Xi=0}).  Given the above statement about the large-$N$ limit of $W^N$, the following definition is sensible.

\begin{definition} \label{def.free.mult.BM} The {\bf free multiplicative Brownian motion} $b(t)$ is the unique solution to the free stochastic differential equation
\begin{equation} \label{eq.free.mult.BM} db(t) = b(t)\,dw(t), \qquad b(0)=1. \end{equation}
\end{definition}
\noindent (There is a fairly well-developed theory of free stochastic differential equations, utilizing the usual local Lipschitz and growth conditions to establish global solutions to free SDEs using Picard iteration similar to the classical case; see, for example, \cite{BianeSpeicher1998}.)

The processes $(B^N(t))_{t\ge 0}$ and $(b(t))_{t\ge 0}$ are (a.s.) non-normal for all $t>0$ and $N\ge 2$.  This makes their analysis, from the perspective of $\ast$-distribution, very tricky.  In \cite{Kemp2016}, the author showed that $(B^N(t))_{t\ge 0}$ converges in finite dimensional $\ast$-distributions to $(b(t))_{t\ge 0}$.

\subsection{Brown Measure and Non-Normal Random Matrices} \hfill

\medskip

For a normal (random) matrix $A\in \MNC$, the ESD \eqref{eq.def.ESD} coincides with the spectral measure: that is, there is a direct connection between the distribution of eigenvalues and the $\ast$-distribution.  For an ensemble of normal matrices $A^N\in \MNC$, if the ESD sequence $\mu_{A^N}$ is known to be tight (for example if the operator norm $\|A^N\|$ is known to be bounded in $N$) then convergence of the $\ast$-distribution $A^N\rightharpoonup a$ implies convergence of the ESD $\mu_{A^N}\rightharpoonup \mu_a$ (i.e. convergence of moments \eqref{eq.*-moments.measure} and tightness imply weak convergence).

For non-normal (random) matrix ensembles, the link between the $\ast$-distribution and the ESD is less direct.  As a result, convergence properties of the two may fail to coincide.

\begin{example} \label{ex.Jordan.block} In $\MNC$, consider the two matrices
\[ S = \begin{bmatrix}
0 & 1 & 0 & \cdots & 0 & 0 \\
0 & 0 & 1 & \cdots & 0 & 0 \\
\vdots & \vdots & \vdots & \ddots & \vdots & \vdots \\
0 & 0 & 0 & \cdots & 1 & 0 \\
0 & 0 & 0 & \cdots & 0 & 1 \\
0 & 0 & 0 & \cdots & 0 & 0
\end{bmatrix}
\qquad
U = \begin{bmatrix}
0 & 1 & 0 & \cdots & 0 & 0 \\
0 & 0 & 1 & \cdots & 0 & 0 \\
\vdots & \vdots & \vdots & \ddots & \vdots & \vdots \\
0 & 0 & 0 & \cdots & 1 & 0 \\
0 & 0 & 0 & \cdots & 0 & 1 \\
1 & 0 & 0 & \cdots & 0 & 0
\end{bmatrix}
\]
The matrix $S$ is upper triangular with all $0$s on the diagonal, so all its eigenvalues are $0$: $\mathrm{ESD}(S) = \delta_0$.  On the other hand $U$ is the permutation matrix for the single cycle shift $(N\,N-1\,\ldots\,2\,1)$.  Hence, the eigenvalues of $U$ are in the unit circle; indeed, $U$ is a $N$-Haar unitary (of Example \ref{ex.p-Haar.unitary}): $\mathrm{ESD}(U) = \frac1N\sum_{j=1}^N \delta_{\zeta_N^j}$ where $\zeta_N^j$ are the $N$th roots of unity, which converges to the uniform probability measure on the circle (the spectral measure of a Haar unitary).

As a normal matrix, the $\ast$-distribution of $U$ is determined by its ESD; ergo, as $N\to\infty$, $U=U^N$ converges in $\ast$-distribution to a Haar unitary.  But $U = S + E_{N1}$ is a rank-$1$ matrix with norm $1$.  Hence by Lemma \ref{lem.perturb.*-dist} it follows that $S=S^N$ also converges in $\ast$-distribution to a Haar unitary.  Thus, its $\ast$-distribution and its ESD have different limits.
\end{example}

\subsubsection{Pseudospectrum and Brown Measure\label{sect.Brown.measure}} \hfill

\medskip

The disconnect observed in the well-known Example \ref{ex.Jordan.block} has to do with the  pseudospectrum of non-normal matrices.

\begin{definition} Let $A\in \MNC$.  For $\epsilon>0$, the {\bf $\epsilon$-psuedospectrum} $\mathrm{spec}_\epsilon(A)$ is the set of all eigenvalues of perturbations of $A$ with norm $\le\epsilon$:
\begin{equation} \label{eq.def.pseudo.1} \mathrm{spec}_\epsilon(A):= \bigcup_{\|E\|\le\epsilon}\mathrm{spec}(A+E). \end{equation}
\end{definition}

An alternative and equivalent description of the pseudospectrum (cf.\ \cite[Chapter 2]{Trefethen}) is
\begin{equation} \label{eq.def.pseudo.2} \mathrm{spec}_\epsilon(A)=\left\{\lambda\in\mathbb{C}\colon \|(A-\lambda I)^{-1}\|\ge\frac{1}{\epsilon}\right\}. \end{equation}
Both \eqref{eq.def.pseudo.1} and \eqref{eq.def.pseudo.2} show that $\mathrm{spec}_0(A) = \mathrm{spec}(A)$.  The form \eqref{eq.def.pseudo.2} is revealing, as it connects to the {\bf singular values} of $A-\lambda I$: the square roots of the eigenvalues of $(A-\lambda I)^\ast (A-\lambda I)$.  Of particular note: for any matrix $A$, the largest and smallest singular values are
\[ \sigma_{\max}(A) = \sqrt{\max_j \lambda_j(A^\ast A)} =  \|A\|  \qquad \sigma_{\min}(A) = \sqrt{\min_j \lambda_j(A^\ast A)} = \frac{1}{\|A^{-1}\|} \]
where $\|A^{-1}\|:=\infty$ if $A$ is not invertible.  Thus, the pseudospectrum can be identified as
\begin{equation} \label{eq.def.pseudo.3} \mathrm{spec}_\epsilon(A)=\left\{\lambda\in\mathbb{C}\colon \sigma_{\min}(A-\lambda I)\le\epsilon \right\}. \end{equation}

The pseudospectrum of $A$ is a numerical analysis tool to quantify the stability of the spectrum of $A$ under small perturbations; from \eqref{eq.def.pseudo.3} we see that it is quantified in terms of the behavior of the smallest singular values of all the shifted matrices $A-\lambda I$.  The spectral theorem keeps the spectra of normal matrices relatively well-behaved, but it is well-known that there are many examples (such as in Example \ref{ex.Jordan.block}) of non-normal matrices whose pseudospectrum behaves more wildly.  One might expect, however, that a {\em random} non-normal matrix should have better behavior under perturbation.  A key {\em asymptotic} result in that direction was provided by Rudelson and Vershynin in \cite{RudelsonVershynin2014}.

\begin{theorem}[Theorem 1.1, \cite{RudelsonVershynin2014}] \label{thm.RV} Let $U\in\mathbb{U}(N)$ be a Haar Unitary Ensemble, and let $D$ be any (random) matrix independent from $U$.  There are absolute constants $\alpha_0,\beta_0>0$ that are independent of $N$ and $D$ so that, for all $\epsilon>0$,
\[ \mathbb{P}\{\sigma_{\min}(U+D)\le\epsilon\}\le \epsilon^{\alpha_0} N^{\beta_0}. \]
In particular: if $\epsilon = \epsilon(N) < N^{-\gamma}$ for any $\gamma < \beta/\alpha$, then
\[ \sup_{D\in \MNC} \mathbb{P}\left\{\sigma_{\min}(U+D)\le \epsilon(N)\right\} \to 0 \quad \text{as} \quad N\to\infty. \]
\end{theorem}
Theorem \ref{thm.RV} shows (by replacing $D$ with $D-\lambda I$ in the above supremum) that, for a Haar Unitary Ensemble $U=U^N$ and any sequence of matrices $D=D^N$, the $\epsilon(N)$-pseudospectrum $\mathrm{spec}_{\epsilon(N)}(U^N+D^N)$ vanishes in probability as $N$ grows.  Put another way: the smallest shifted singular value may decay as $N$ grows, but it decays at most polynomially with high probability.

\begin{example} \label{ex.Jordan.2} Consider again the Jordan block matrices $S=S^N$ from Example \ref{ex.Jordan.block}.  For $\lambda\in\mathbb{C}$, let $\mathbf{v}_\lambda\in\mathbb{C}^N$ be the vector $[\mathbf{v}_\lambda]_j = \lambda^{j-1}$.  Note that
\[  (S-\lambda I)\mathbf{v}_\lambda = 
\begin{bmatrix}
-\lambda & 1 & 0 & \cdots & 0 & 0 \\
0 & -\lambda & 1 & \cdots & 0 & 0 \\
\vdots & \vdots & \vdots & \ddots & \vdots & \vdots \\
0 & 0 & 0 & \cdots & 1 & 0 \\
0 & 0 & 0 & \cdots & -\lambda & 1 \\
0 & 0 & 0 & \cdots & 0 & -\lambda
\end{bmatrix}
\renewcommand\arraystretch{1.22}
\begin{bmatrix}
1 \\ \lambda \\ \lambda^2 \\ \vdots \\  \lambda^{N-1}
\end{bmatrix}
=
\begin{bmatrix}
0 \\ 0 \\ \vdots \\ 0 \\  \lambda^N
\end{bmatrix} = \lambda^N \mathbf{e}_N
\]
Hence, it follows that if $(S-\lambda I)$ is invertible then $(S-\lambda I)^{-1}\mathbf{e}_N = \lambda^{-N}\mathbf{v}_\lambda$, and so
\[ \frac{1}{\sigma_{\min}(S-\lambda I)} = \|(S-\lambda I)^{-1}\| \ge |\lambda^{-N}\mathbf{v}_\lambda| \ge |\lambda|^{-N} \]
since $|\mathbf{v}_\lambda|^2 = 1+|\lambda|^2+\cdots+|\lambda|^{2(N-1)} \ge 1$.  Thus: $\sigma_{\min}(S-\lambda I) \le |\lambda|^N$, which is {\em exponentially small} in $N$ for any $|\lambda|<1$.
\end{example}

Understanding why the rate of decay of the smallest shifted singular value of a matrix can impact the convergence properties of its eigenvalues requires an understanding of the connection between the {\em eigenvalues} of $A$ and the {\em singular values} of $\{A-\lambda I\colon \lambda\in\mathbb{C}\}$.  This brings us to Girko's Hermitization method.

The starting point is the observation that the Coulomb potential $L$ -- the fundamental solution to the Laplace equation $\nabla^2 L = \delta_0$ -- in the plane is $L(z) = \frac{1}{2\pi}\log|z|$.  Hence, for a matrix $A\in \MNC$ with eigenvalues $\lambda_j(A)$, its ESD \eqref{eq.def.ESD} can be consructed as
\begin{equation} \label{eq.Hermitize.1} \mu_A = \nabla^2\, \frac{1}{2\pi N}\sum_{j=1}^N \log |z-\lambda_j(A)| = \nabla^2_{\!z} \frac{1}{2\pi N}\log|\det(A-zI)| \end{equation}
where the final equality is the statement that the determinant of any matrix is the product of its eigenvalues.  What's more, $|\det(A-zI)|^2 = \det[(A-zI)^\ast(A-zI)]$ and so, taking logs,
\[ \log|\det(A-zI)| = \frac12\log\det[(A-zI)^\ast (A-zI)] = \frac12\mathrm{Tr}\log[(A-zI)^\ast (A-zI)] \]
the latter equality using the Spectral Theorem for Hermitian matrices: the eigenvalues of $\log[(A-zI)^\ast(A-zI)]$ are the logarithms of the eigenvalues of $(A-zI)^\ast(A-zI)$. Combining this with \eqref{eq.Hermitize.1} yields
\begin{equation} \label{eq.Hermitize.2} \mu_A = \frac{1}{4\pi}\nabla^2_{\!z} \frac{1}{N}\mathrm{Tr}[\log[(A-zI)^\ast(A-zI)]]. \end{equation}
Let $\nu_{A-zI}$ denote the spectral measure of $\sqrt{(A-zI)^\ast(A-zI)}$; then \eqref{eq.Hermitize.2} says
\begin{equation} \label{eq.Hermitize.3} \mu_A = \frac{1}{2\pi}\nabla^2_{\!z} \int_0^\infty \log x\,\nu_{A-zI}(dx). \end{equation}
The expression \ref{eq.Hermitize.3} for the ESD suggests how to find the limit as the size of the matrix $A$ grows.  If $A=A^N$ converges in $\ast$-distribution to some $a$, then $(A^N-zI)^\ast(A^N-zI)\rightharpoonup(a-z)^\ast(a-z)$ in $\ast$-distribution.  We then expect the spectral measure $\nu_{A^N-zI}$ should converge to $\nu_{a-z}:=\mu_{((a-z)^\ast(a-z))^{1/2}}$ weakly.  If $\log$ were continuous and bounded on $[0,\infty)$, it would then follow that the integral in \eqref{eq.Hermitize.3} converges to the same expression with $\nu_{a-zI}$ in place of $\nu_{A-zI}$.

The above scheme can fail (as it does in Example \ref{ex.Jordan.block}) because mass can escape to the singularities of $\log$ at $0$ and $\infty$.  In practice, the sequence $A^N$ is typically norm-bounded which avoids any problems at $\infty$; but $\nu_{A-zI}$ definitely has support near $0$ when $z$ is close to a singular value of $A$.  Hence, the failure of expected convergence in Example \ref{ex.Jordan.block} is caused by small singular values of $S^N-\lambda I$ decaying to $0$ quickly as $N$ grows, cf.\ Example \ref{ex.Jordan.2}.

It turns out that failure of the above convergence scheme is a rare pathology, as shown in a string of work starting with \cite{Sniady} and continuing with \cite{GWZ2014,FPZ2015,Wood2016,BPZ2019,BPZ202}.  Actually proving that the desired convergence does occur in any given example is usually very involved, but the above scheme at least points to what that limit {\em ought} to be, at least generically.

\begin{definition} \label{def.Brown.measure} Let $(\mathscr{A},\varphi)$ be a $W^\ast$-probability space, and let $a\in L^2(\mathscr{A},\varphi)$.  By H\"older's inequality $a^\ast a\in L^1(\mathscr{A},\varphi)$, and so $|a|:=\sqrt{a^\ast a}$ has a spectral measure (cf.\ Remark \ref{rk.unbounded.spec.meas}) which we denote $\nu_a$.  Since $|a|\in L^2(\mathscr{A},\varphi)$, $\int_0^\infty x^2\,\nu_a(dx) = \varphi(a^\ast a) = \|a\|_2^2<\infty$.  Hence $\int_{\epsilon}^\infty \log x\,\nu_a(dx)\in\mathbb{R}$ for all $\epsilon>0$.  It follows that the function $L\colon\mathbb{C}\to\mathbb{R}$
\[ L_a(z) := \frac{1}{2\pi} \int_0^\infty \log x\, \nu_{a-z}(dx) \]
takes finite or $-\infty$ values, and is (weakly) subharmonic on $\mathbb{C}$.  Its Laplacian is a probability measure on $\mathbb{C}$: the {\bf Brown measure} of $a$.  It is denoted $\mu_a$:
\[ \mu_a = \nabla^2_{\!z} L_a(z) \]
i.e.\ it is uniquely determined by the condition
\[ \int_{\mathbb{C}} f\,d\mu_a = \int_{\mathbb{C}} \nabla^2 f(z) L_a(z)\,d^2z \qquad \forall f\in C_c^\infty(\mathbb{C}). \]
\end{definition}

\begin{remark} \label{rk.Brown.measure.def} Although the $L^2$-norm $\|a\|_2$ will play a role in the study of Brown measures for the operators of interest in this paper, the only required condition for Definition \ref{def.Brown.measure} to make sense is that $\int_{\epsilon}^\infty \log x\,\nu_a(dx)$ should be finite for some $\epsilon$; canonically $\epsilon$ is taken to $=1$ here, with the condition given by the integrability of $\nu_a$ against $\log^+(x) = \max\{\log x, 0\}$; as $\nu_a$ is the law of the selfadjoint (possibly unbounded) operator $|a|=\sqrt{a^\ast a}$, the equivalent condition is $\log^+(|a|)\in L^1(\mathscr{A},\varphi)$.  See \cite[Definition 2.13]{HaagerupSchultz2007}.
\end{remark}

Brown measure, also known as Brown's spectral measure, was introduced in \cite{Brown1986} in the theory of von Neumann algebras, around the same time Girko \cite{Girko} independently introduced the closely related Hermitization method in random matrix theory.  The next proposition is a summary of important properties of Brown measure; proofs may be found in \cite{Brown1986} or \cite[Section 11]{MingoSpeicherBook}.

\begin{proposition} \label{prop.Brown} Let $a$ satisfy the conditions of Definition \ref{def.Brown.measure}, and let $\mu_a$ denote the Brown measure of $a$.
\begin{enumerate}
    \item If $a$ is normal, the Brown measure $\mu_a$ coincides with the spectral measure of $a$ (hence the notation collision is justified).
    \item \label{prop.Brown.spec} The support of $\mu_a$ is contained in $\mathrm{spec}(a)$.
    \item \label{prop.Brown.holo} If $f$ is holomorphic on a neighborhood of $\mathrm{spec}(a)$, then
    \[ \int f(\zeta)\,\mu_a(d\zeta) = \varphi[f(a)] \]
    where the right-hand-side is defined by holomorphic functional calculus.
    \item \label{prop.Brown.robust} If $A^N$ is a sequence of matrices which converges to $a$  in $\ast$-distribution and for which $(\nu_{A^N})_{N\ge 1}$ is tight (e.g.\ if $\|A^N\|$ is uniformly bounded for large $N$), then $\mathrm{ESD}(A^N)\rightharpoonup \mu_a$ iff $L_{A^N}(z) \to L_a(z)$ for (Lebesgue) almost all $z\in\mathbb{C}$.
\end{enumerate}
\end{proposition}

The Brown measure obeys the change of variables formula in Proposition \ref{prop.Brown}\ref{prop.Brown.holo} for holomorphic (or antiholomorphic) functions, but the mixed moment formula \eqref{eq.*-moments.measure} generally does not hold, as is well-known for non-normal matrices: $\mathrm{Tr}(A^n)$ is indeed the Newton sum of the $n$th powers of the eigenvlues of $A$, but $\mathrm{Tr}(A^n A^{\ast m})$ is not expressible simply in terms of the eigenvalues of $A$.  There is one exception, which relates back to the reason the log potential comes into play in defining Brown measure: the $\log$ is the Coublomb potential on the plane.  This yields the following.

\begin{lemma} In a $W^\ast$-probability space $(\mathscr{A},\varphi)$, let $a\in L^2(\mathscr{A},\varphi)$, and let $\mu_a$ be the Brown measure of $a$.  The log potential $L_a$ satisfies
\begin{equation} \label{eq.log.potential.integral.Brown}
L_a(z) = \frac{1}{2\pi}\int_{\mathbb{C}} \log|\zeta-z|\,\mu_a(d\zeta).
\end{equation}
\end{lemma}
See \cite[Eq.\ (4), p.\ 18]{Brown1986}.  For a more accessible source, see \cite[Eq.\ (2.17), p.\ 217]{HaagerupSchultz2007} where this is proved even for unbounded operators $a$.  One useful corollary is that $L_a$ is $L^p_{\mathrm{loc}}$ for any $p\ge 1$, with bounds depending on $a$ only through $\|a\|$.

\begin{corollary} \label{cor.La.unif.L1.loc} For $p\ge 1$ there is an increasing function $\Lambda_p\colon[1,\infty)\to(0,\infty)$ so that, for any $R\ge 1$ and any bounded $a\in(\mathscr{A},\varphi)$, the log potential $L_a$ satisfies
\begin{equation} \label{eq.L1.loc.bound} \int_{|z|\le R} |L_a(z)|^p\,d^2z \le \Lambda_p(R+\|a\|). \end{equation}
\end{corollary}

\begin{proof} Applying Tonnelli's theorem and Jensen's inequality, followed by the change of variables $\tilde{z}=z-\zeta$, we have
\begin{align*} \int_{|z|\le R} |L_a(z)|^p\,d^2z &= \frac{1}{2\pi}\int_{|z|\le R} \left|\int_{\mathbb{C}} \log|\zeta-z|\,\mu_a(d\zeta)\right|^p\,d^2z \\ 
&\le \frac{1}{2\pi}\int_{\mathbb{C}} \mu_a(d\zeta) \int_{|z|\le R} |\log|\zeta-z||^p\,d^2z \\
&= \frac{1}{2\pi}\int_{\mathbb{C}} \mu_a(d\zeta) \int_{|\tilde{z}+\zeta|\le R}|\log |\tilde{z}||^p\,d^2\tilde{z}.
\end{align*}
The inside integral of the non-negative function $\tilde{z}\mapsto \vert \log|\tilde{z}|\vert^p$ is over the disk or radius $R$ centered at $-\zeta$.  This disk is contained in the disk of radius $R+|\zeta|$ centered at $0$, and so
\begin{equation} \label{eq.Lp.bound.1}  \int_{|z|\le R} |L_a(z)|^p\,d^2z \le \frac{1}{2\pi}\int_{\mathbb{C}} \mu_a(d\zeta) \int_{|\tilde{z}|\le R+|\zeta|} |\log|\tilde{z}||^p\,d^2\tilde{z}.
\end{equation}
The inside integral can be computed in polar coordinates as follows: for $\tilde{R}>1$,
\begin{align*} \frac{1}{2\pi}\int_{|\tilde{z}|\le \tilde{R}} |\log|\tilde{z}||^p\,d^2 z &= \int_0^{\tilde{R}} r|\log r|^p\,dr \\
&= \int_0^1 r|\log r|^p\,dr + \int_1^{\tilde{R}}r(\log r)^p\,dr \\
&\le  2^{-p-1}\Gamma(p+1) + \frac12\tilde{R}^2(\log\tilde{R})^p \end{align*}
where the first term is the exact value of the integral over $[0,1]$, and the second is estimated using $(\log r)^p \le (\log\tilde{R})^p$ for $r\in[1,\tilde{R}]$ (and further estimating $\tilde{R}^2-1<\tilde{R}$ for good measure).  Combining with \eqref{eq.Lp.bound.1}, we have
\[ \int_{|z|\le R} |L_a(z)|^p\,d^2z \le \int_{\mathbb{C}}  \left(2^{-p-1}\Gamma(p+1) + \frac12(R+|\zeta|)^2(\log(R+|\zeta|))^p\right)\,\mu_a(d\zeta). \]
The Brown measure $\mu_a$ is supported on the spectrum of $a$, cf.\ Proposition \ref{prop.Brown}\ref{prop.Brown.spec}.  Since the spectrum of $a$ is contained in $\{\zeta\in\mathbb{C}\colon |\zeta|\le \|a\|\}$, and since the function $\tilde{R}\mapsto\tilde{R}^2(\log\tilde{R})^p$ is increasing for $\tilde{R}>1$ (easily verified by calculus), we see that \eqref{eq.L1.loc.bound} holds true with 
\begin{equation} \label{eq.Lambda.p}
\Lambda_p(\tilde{R}) = 2^{-p-1}\Gamma(p+1) + \frac{1}{2}\tilde{R}^2(\log\tilde{R})^p.
\end{equation}
\end{proof}

Using either \eqref{eq.log.potential.integral.Brown} or Definition \ref{def.Brown.measure}, there are a few classes of operators whose log potentials and thence whose Brown measures can be calculated exactly.  For example, in 2000 Haagerup and Larsen gave a fairly explicit description of the Brown measure of any $\mathscr{R}$-diagonal operator.

\begin{theorem}[\cite{HaagerupLarsen2000}] \label{thm.Haag.Lars} Let $(\mathscr{A},\varphi)$ be a $W^\ast$-probability space, and let $a\in\mathscr{A}$ be $\mathscr{R}$-diagonal.  Assume that $a$ is injective (i.e.\ that $a^{-1}$ exists as a possibly unbounded operator), and that $a^\ast$ is not a constant multiple of the identity.  Then the Brown measure $\mu_a$ is rotationally invariant on $\mathbb{C}$, and its support is the annulus
\[ \mathrm{supp}\,\mu_a = \{z\in\mathbb{C}\colon \|a^{-1}\|_2^{-1}\le|z|\le\|a\|_2\} \]
where $\|a^{-1}\|_2^{-1}:=0$ if $a^{-1}\notin L^2(\mathscr{A},\varphi)$.  Moreover, $\mu_a$ is determined by
\[ S_{a^\ast a}(\mu_a(r\mathbb{D})-1) = \frac{1}{r^2}, \qquad r>0 \]
where $S_{a^\ast a}$ is the $S$-transform of the positive operator $a^\ast a$ (cf.\ Proposition \ref{prop.S-transform}) and $r\mathbb{D}$ is the disc of radius $r$ centered at $0$ in $\mathbb{D}$.  In particular: $\mu_a$ has a real analytic, rotationally invariant density on the interior of its support annulus or disk.
\end{theorem}

Using very different methods (based on stochastic and partial differential equations), in 2019 three of the authors of the present paper computed the Brown measure of the free multiplicative Brownian motion (Definition \ref{def.free.mult.BM}); the following year, another present author collaborated on an extension to allow arbitrary unitary initial conditions.

\begin{theorem}[\cite{DHKBrown,HoZhong2020Brown}] \label{thm.fmbm} In  a $W^\ast$-probability space $(\mathscr{A},\varphi)$, let $(b(t))_{t\ge0}$ be the free multiplicative Brownian motion \eqref{eq.free.mult.BM}, and let $u_0\in\mathscr{A}$ be a unitary operator freely independent from $(b(t))_{t\ge0}$.  Denote by $\mu_{u_0}$ the spectral measure of $u_0$, which is a probability distribution supported in the unit circle $\mathbb{T}$.

For any $t>0$, the Brown measure $\mu_{u_0b(t)}$ of the free multiplicative Brownian motion $u_0b(t)$ with starting distribution $u_0$ is supported  on the closure of the region
\[ \Sigma_\infty(u_0,t) = \{z\in\mathbb{C}\colon T_\infty(u_0,z)<t\} \]
where
\begin{equation} \label{eq.lifetime.k=infty} T_\infty(u_0,z) = \frac{\log(|z|^2)}{|z|^2-1}\left(\int_{\mathbb{T}} \frac{\mu_{u_0}(d\xi)}{|\xi-z|^2}\right)^{-1}. \end{equation}

There is a function $r_{u_0}\colon\mathbb{R}_+\times(-\pi,\pi]\to [1,\infty)$ such that
\begin{equation} \label{eq.Sigma.infty.geom} \Sigma_\infty(u_0,t) = \{re^{i\theta}\colon r_{u_0}(t,\theta)^{-1}<r<r_{u_0}(t,\theta)\}. \end{equation}

The Brown measure $\mu_{u_0b(t)}$ is absolutely continuous with respect to Lebesgue measure on $\Sigma_\infty(u_0,t)$, with density
\begin{equation} \label{eq.bt.density}
\rho_\infty(t,re^{i\theta}) = \frac{1}{4\pi r^2}\left(\frac{2}{t}+\frac{\partial}{\partial\theta}\int_{-\pi}^\pi\frac{2r_{u_0}(t,\theta)\sin(\theta-\phi)}{r_{u_0}(t,\theta)^2+1-2r_{u_0}(t,\theta)\cos(\theta-\phi)}\nu_0(d\phi)\right)
\end{equation}
where $\nu_0$ is the law of any $\vartheta_0$ for which $e^{i\vartheta_0}=u_0$.  This density of the form $\rho_\infty(t,re^{i\theta}) = w_t(\theta)/r^2$, with the radial function $w_t$ is bounded $w_t(\theta)<\frac{1}{\pi t}$.
\end{theorem}

\subsubsection{The Circular Law and the Single Ring Theorem\label{sect.SRT}} \hfill

\medskip

If $W=W^N$ is a Ginibre Ensemble (cf.\ Example \ref{def.classic.ensembles}), Proposition \ref{prop.asymp.free.Haar.GUE} shows that $W^N$ converges in $\ast$-distribution to Voiculescu's circular element $c=\frac{1}{2}(x+iy)$ where $x,y$ are freely independent.  The Brown measure of $c$ can be computed directly: it is the uniform probability measure on the unit disk $\mathbb{D}$ in $\mathbb{C}$.  In \cite{Ginibre1965} Ginibre proved by direct computation with Gaussian processes that the $\mathrm{ESD}$ of $W^N$ converges to this law; this is why the ensemble carries his name.  Now, if $A^N\in \MNC$ be a random matrix with all independent $L^{\infty-}$ entries, centered and with variance $1/N$, then direct comparison of moments shows that $A^N$ also converges in $\ast$-distribution to a circular element.  Girko \cite{Girko} conjectured that in general (even if the entries of $A^N$ are $L^2$ and no more) the $\mathrm{ESD}$ of $A^N$ should converge to the uniform probability measure on $\mathbb{D}$.  He invented the Hermitization method described above in order to prove this, although he still left this so-called {\bf circular law} as a conjecture as he did not establish the requisite uniform integrability of the logarithm.  Over the subsequence decades, incremental progress on this conjecture was made, first by Bai under higher moment assumptions.  The final theorem is as follows.

\begin{theorem}[Girko, Bai, Tao-Vu] \label{thm.circular.law} Let $A^N$ be a random $N\times N$ matrix whose entries are all independent $L^2$ random variables, centered and with variance $1/N$.  Then the $\mathrm{ESD}$ converges weakly a.s.\ to the uniform probability measure on $\mathbb{D}$.
\end{theorem}

The Circular Law was the first prominent example of a non-normal random matrix ensemble whose eigenvalue distribution was shown to converge they way it should: to the Brown measure of the $\ast$-distributional limit.  In this case, computing the Brown measure of the limit is fairly straightforward; the methodology to prove uniform integrability for the Hermitization method to work took decades to develop, and still stands as the basic roadmap for proving $\mathrm{ESD}$ convergence in any non-normal ensemble.

By Proposition \ref{prop.bi-inv}, if $A^N = U^N T^N V^{N\ast}$ is a bi-invariant ensemble, and if the singular value measure $\nu_{T^N}$ converges a.s.\ to $\nu$, then $A^N$ converges in $\ast$-distribution to an $\mathscr{R}$-diagonal $a$ for which $\sqrt{a^\ast a}$ has distribution $\nu$.  The question remains whether the eigenvalue distribution of $A^N$ also converges to the Brown measure $\mu_a$ which is identifies in Theorem \ref{thm.Haag.Lars}.  Without further assumptions on $T^N$, this can certainly fail to be true; but in \cite{GuionnetKZ-single-ring} (and follow-up improvements in \cite{RudelsonVershynin2014}), the answer was shown to be affirmative under mild conditions on how $\mu_{T^N}$ converges to $\nu$.

\begin{theorem}[Single Ring Theorem, \cite{GuionnetKZ-single-ring,RudelsonVershynin2014}] \label{thm.SRT} Let $A^N = U^N T^N V^{N\ast}$ be a bi-invariant ensemble, and suppose $\mu_{T^N}$ converges weakly in probability to a distribution $\nu$ that is compactly-supported in $[0,\infty)$.  Assume further:
\begin{enumerate}
    \item[(SRT1)] \label{SRT.1} There exists $M>0$ so that $\mathbb{P}\{\|T^N\|>M\}\to 0$ as $N\to\infty$.
    \item[(SRT2)] \label{SRT.2} There are $\gamma,\delta>0$ so that for all large $N$
    \[ |\mathrm{Im}\, G_{\mu_{T^N}}(z)|\le \gamma \quad  \forall\,z\in\mathbb{C}^+\;\text{ with }\;\mathrm{Im}(z)>N^{-\delta}. \]
\end{enumerate}
Then the $\mathrm{ESD}$ of $A^N$ converges weakly in probability to the Brown measure $\mu_a$.
\end{theorem}
\noindent Condition \ref{SRT.2} is an anti-concentration restriction on the singular values of $A^N$: it says, in strong quantitative terms, that the singular values to not concentrate on short intervals.  Together with \ref{SRT.1} which prevents singular values from escaping to $\infty$, provides enough regularity to prove that the {\em shifted} singular value measures $\nu_{A^N-zI}$ are uniformly integrable against the logarithm, completing the Hermitization method for these ensembles.

There have been a number of variations on the Single Ring Theorem and the Circular Law where convergence to the Brown measure has been established using methods similar to those in the original theorems; for example \cite{BGR2016,BenaychGeorges2017,BGR2017,BaoES2019singlering,HoZhong2025}.  And there have been a number of cases of explicit computation of Brown measures of operators that are not $\mathscr{R}$-diagonal; for example \cite{HaagerupLarsen2000,BianeLehner2001,HaagerupSchultz2007,DHKBrown,HoZhong2020Brown}.  Both approaches are still considered very challenging.

To summarize: in studying the limits of eigenvalues of non-normal random matrices, the techniques fall into two separate yet equally important groups: the computation of the Brown measure of the $\ast$-distribution limit, and the establishment of uniform integrability of the log potential, via decay rates of small shifted singular values.  These are their stories.

\section{Random Walk Approximations to Brownian Motion}

\subsection{Proof of Theorem \ref{main.thm.1.MoBettaWZ}}\label{sec.main.thm.1.MoBettaWZ}

Key to the proof of Theorem \ref{main.thm.1.MoBettaWZ} are bounds on the $L^p$ norms of matrix-valued stochastic integrals coming from the following Hilbert space--valued Burkholder--Davis--Gundy (BDG) inequalities.

\begin{theorem}[BDG inequalities, {\cite[Thm.\ 1.1]{MR2016}}]\label{thm.BDGinH}
There exist increasing families $(\alpha_p)_{1 \leq p < \infty}$ and $(\gamma_p)_{1 \leq p < \infty}$ of strictly positive constants such that the following holds:
If $H$ is a separable Hilbert space, $M$ is an $H$-valued RCLL local martingale such that $M_0 = 0$, $1 < p < \infty$, and $t \geq 0$, then
\[
\alpha_p^{-1}\norm{\int_0^t \ip{\ds M(s), \ds M(s)}}_{L^{\frac{p}{2}}}^{\frac12} \leq \norm{\norm{M}_t^*}_{L^p} \leq \gamma_p\norm{\int_0^t \ip{\ds M(s), \ds M(s)}}_{L^{\frac{p}{2}}}^{\frac12},
\]
where
\[
\norm{M}_t^* \coloneqq \sup_{0 \leq s \leq t}\norm{M(s)}.
\]
To be clear, $\norm{\cdot}$ is the norm induced by $\ip{\cdot,\cdot}$ of $H$, $\norm{\cdot}_{L^q} = \E[|\cdot|^q]^{1/q}$, and
\[
\int_0^t \ip{\ds M(s),\ds M(s)} = L^0\text{-}\lim_{|\Pi| \to 0} \sum_{s \in \Pi} \ip{\Delta_sM,\Delta_sM}
\]
is the scalar quadratic variation of $M$.
(Above, $L^0$\text{-}$\lim$ denotes a limit in probability, and the limit is over partitions $\Pi$ of $[0,t]$.)
\end{theorem}

To harvest the fruits of these inequalities in our case, and to do some other important calculations, we compute some matrix-valued quadratic covariations.
For $i=1,2$, let $A_i = (A_i(t))_{t \geq 0}$ and $C_i = (C_i(t))_{t \geq 0}$ be adapted continuous $M_{N \times N}(\C)$-valued processes, $\e_i \in \{1,\ast\}$, and $M_i \coloneqq \into A_i(t)\,\ds W^{\e_i}(t)\,C_i(t)$.
Then
\begin{align*}
    \ds M_1(t)\,\ds M_2(t) & = A_1(t)\left(\sum_{\xi \in \beta} \xi^{\e_1}\,\ds W_\xi(t)\right)C_1(t)A_2(t)\left(\sum_{\eta \in \beta} \eta^{\e_2}\,\ds W_\eta(t)\right)C_2(t) \\
    & = \sum_{\xi,\eta \in \beta}A_1(t)\xi^{\e_1} C_1(t)A_2(t)\eta^{\e_2} C_2(t) \,\underbrace{\ds W_\xi(t)\,\ds W_\eta(t)}_{\delta_{\xi\eta}\,\ds t} \\
    & = \sum_{\xi \in \beta}A_1(t)\xi^{\e_1} C_1(t)A_2(t)\xi^{\e_2} C_2(t) \,\ds t \numberthis\label{eq.mat.QC}
\end{align*}
since $\{W_\xi\}_{\xi \in \beta}$ is a collection of independent standard Brownian motions on $\R$. 
Combining this calculation with Theorem \ref{thm.BDGinH}, we obtain the following bounds.

\begin{corollary}\label{cor.matstochintbound}
Suppose $A$ and $C$ are adapted continuous $\MNC$-valued processes.
If $1 < p < \infty$ and $t \geq 0$, then
\begin{align*}
    & \alpha_p^{-1}\norm{\int_0^t \tr\left(\sum_{\xi \in \beta} C^*(s)\xi^* A^*(s)A(s)\xi C(s)\right) \ds s}_{L^{\frac{p}{2}}}^{\frac12} \\
    & \leq \norm{\left|\into A(s)\,\ds W(s)\,C(s)\right|_t^*}_{L^p} \\
    & \leq \gamma_p\norm{\int_0^t\tr\left(\sum_{\xi \in \beta} C^*(s)\xi^* A^*(s)A(s)\xi C(s)\right)\ds s}_{L^{\frac{p}{2}}}^{\frac12},
\end{align*}
where $|\cdot| = \norm{\cdot}_{L^2(\tr)}$.
\end{corollary}

\begin{proof}
Note that if $M \coloneqq \into A(t)\,\ds W(t)\,C(t)$, then $M$ is an $M_{N \times N}(\C)$-valued local martingale starting at zero.
Furthermore, $M^* = \into C^*(t)\,\ds W(t)\,A^*(t)$.
Consequently, by \eqref{eq.mat.QC} with $M_1 = M^*$ and $M_2 = M$,
\begin{align*}
    \int_0^t \ip{\ds M(s), \ds M(s)} & = \int_0^t \tr\left(\ds M^*(s)\, \ds M(s)\right) \\
    & = \int_0^t \tr\left(\sum_{\xi \in \beta}C^*(s)\xi^* A^*(s)A(s)\xi C(s)\right) \ds s.
\end{align*}
The desired result then follows immediately from Theorem \ref{thm.BDGinH}.
\end{proof}

\begin{proposition}\label{prop.Binvertible+estimates}
The following hold.
\begin{enumerate}[font=\normalfont,label=(\roman*)]
    \item $B$ takes values in $\GL(N,\C)$ (almost surely).
    Specifically, if $H$ is the unique solution to the $M_{N \times N}(\C)$-valued SDE
    \[
    \ds H(t) = -\ds W(t)\,H(t) + \frac12 \Xi\,H(t)\,\ds t, \qquad H(0) = I,
    \]
    then $H = B^{-1}$ almost surely.\label{item.inv(B)SDE}
    \item If $\norm{\cdot}_{\infty} = \norm{\cdot}_{\mathrm{op}}$ is the operator norm on $M_{N \times N}(\C)$ and
    \[
    c(\beta) \coloneqq \max\left\{\norm{\sum_{\xi \in \beta} \xi^*\xi}_{\infty},\norm{\sum_{\xi \in \beta} \xi\xi^*}_{\infty}\right\},
    \]
    then
    \[
    \norm{\left|B^{\e}\right|_t^*}_{L^p} \leq \sqrt3e^{\frac34\left(2\gamma_p^2c(\beta) + \norm{\Xi}_{\infty}^2t\right) t}, \qquad t \geq 0, \; \e =\pm 1,
    \]
    where $\gamma_p$ and $|\cdot|_t^*$ are as in Theorem \ref{thm.BDGinH}.\label{item.B,inv(B)bounds}
    \item In the notation of the previous item,\label{item.BPibound}
    \[
    \norm{\left|B_{\Pi}\right|_t^*}_{L^p} \leq \sqrt3e^{\frac34\left(2\gamma_p^2c(\beta) + \norm{\Xi}_{\infty}^2t\right) t}, \qquad 0 \leq t \leq T.
    \]
\end{enumerate}
\end{proposition}

\begin{proof}
We take each item in turn.

\ref{item.inv(B)SDE} By the (matrix-valued) It\^o product rule and \eqref{eq.mat.QC},
\begin{align*}
    \ds (BH)(t) & = \ds B(t)\,H(t) + B(t)\,\ds H(t) + \ds B(t) \, \ds H(t) \\
    & = B(t)\,\ds W(t)\,H(t) + \frac12 B(t)\Xi H(t)\,\ds t \\
    &\hspace{15mm} - B(t) \,\ds W(t)\,H(t) + \frac12 B(t)\Xi H(t)\,\ds t - B(t)\left(\ds W(t)\right)^2H(t) \\
    & = B(t)\Xi H(t)\,\ds t - B(t)\Xi H(t)\,\ds t = 0.
\end{align*}
Consequently, $BH \equiv B(0)H(0) = I$ almost surely.
Since $B$ and $H$ are $M_{N \times N}(\C)$-valued processes, this establishes that $B$ is always invertible and $B^{-1} = H$ almost surely.

\ref{item.B,inv(B)bounds} For a continuous $M_{N \times N}(\C)$-valued process $A$, write
\[
\vertiii{A}_t \coloneqq \norm{\left|A\right|_t^*}_{L^p} = \norm{\sup_{0 \leq s \leq t} |A(s)|}_{L^p}.
\]
We now estimate $\norm{B}_t$ using the SDE it satisfies.
In the course of doing so, we shall use shorthand for various inequalities:
BDG for the BDG inequalities in Theorem \ref{thm.BDGinH}, MI for Minkowski's integral inequality, NCH for noncommutative H\"older's inequality, and CS for the Cauchy--Schwarz inequality.
With this in mind, if $t \geq 0$, then
\begin{align*}
    \vertiii{B}_t^2 & = \vertiii{I + \into B(s)\,\ds W(s)+\frac12\into B(s)\Xi\,\ds s}_t^2 \\
    & \leq 3\vertiii{I}_t^2+3\vertiii{\into B(s)\,\ds W(s)}_t^2 + \frac32\vertiii{\into B(s)\Xi\,\ds s}_t^2 \\
    & \leq 3+3\gamma_p^2\norm{\int_0^t \tr\left( \sum_{\xi \in \beta} \xi^*B^*(s)B(s)\xi\right)\ds s}_{L^\frac{p}{2}} + \frac32\vertiii{\into B(s)\Xi\,\ds s}_t^2 \tag{BDG}
\end{align*}
We estimate this further by
\begin{align*}
    & \leq 3+3\gamma_p^2\norm{\int_0^t \tr\left( \sum_{\xi \in \beta} \xi\xi^*B^*(s)B(s)\right)\ds s}_{L^\frac{p}{2}} + \frac32\left(\int_0^t\vertiii{B(\cdot)\Xi}_s\,\ds s\right)^2 \tag{MI}\\
    & \leq 3 + 3\gamma_p^2\norm{\sum_{\xi \in \beta}\xi\xi^*}_{\infty}\norm{\int_0^t \left|B(s)\right|^2\,\ds s}_{L^\frac{p}{2}} + \frac32\norm{\Xi}_{\infty}^2\left(\int_0^t\vertiii{B}_s\,\ds s\right)^2 \tag{NCH}\\
    & \leq 3 + 3\gamma_p^2c(\beta)\int_0^t \underbrace{\norm{\left|B(s)\right|^2}_{L^\frac{p}{2}}}_{\norm{\left|B(s)\right|}_{L^p}^2} \,\ds s + \frac32\norm{\Xi}_{\infty}^2t\int_0^t\vertiii{B}_s^2\,\ds s\tag{MI, CS} \\
    & \leq 3 + \frac32\left(2\gamma_p^2c(\beta) + \norm{\Xi}_{\infty}^2t\right)\int_0^t\vertiii{B}_s^2\,\ds s
\end{align*}
It then follows from Gr\"onwall's inequality that
\[
\vertiii{B}_t^2 \leq 3e^{\frac32\left(2\gamma_p^2c(\beta) + \norm{\Xi}_{\infty}^2t\right) t}, \qquad t \geq 0.
\]
Taking square roots on both sides yields the desired bound.
By a nearly identical argument, the bound on $\vertiii{B^{-1}}_t$ follows from the SDE for $B^{-1}$ in the previous part.

\ref{item.BPibound} We proceed as in the previous item's proof, from which we retain the notation.
If $X \coloneqq B_{\Pi}$ and $0 \leq t \leq T$, then
\begin{align*}
    & \vertiii{X}_t^2 = \vertiii{I + \into X(\ell_{\Pi}(s))\,\ds W(s)+\frac12\into X(\ell_{\Pi}(s))\Xi\,\ds s}_t^2 \\
    & \leq 3\vertiii{I}_t^2+3\vertiii{\into X(\ell_{\Pi}(s))\,\ds W(s)}_t^2 + \frac32\vertiii{\into X(\ell_{\Pi}(s))\Xi\,\ds s}_t^2 \\
    & \leq 3+3\gamma_p^2\norm{\int_0^t \tr\left( \sum_{\xi \in \beta} \xi^*X^*(\ell_{\Pi}(s))X(\ell_{\Pi}(s))\xi\right)\ds s}_{L^\frac{p}{2}} + \frac32\vertiii{\into X(\ell_{\Pi}(s))\Xi\,\ds s}_t^2\\
    & \leq 3+3\gamma_p^2\norm{\int_0^t \tr\left( \sum_{\xi \in \beta} \xi\xi^*X^*(\ell_{\Pi}(s))X(\ell_{\Pi}(s))\right)\ds s}_{L^\frac{p}{2}} + \frac32\left(\int_0^t\vertiii{X(\ell_{\Pi}(\cdot))\Xi}_s\,\ds s\right)^2 \\
    & \leq 3 + 3\gamma_p^2\norm{\sum_{\xi \in \beta}\xi\xi^*}_{\infty}\norm{\int_0^t \left|X(\ell_{\Pi}(s))\right|^2\,\ds s}_{L^\frac{p}{2}} + \frac32\norm{\Xi}_{\infty}^2\left(\int_0^t\vertiii{X(\ell_{\Pi}(\cdot))}_s\,\ds s\right)^2 \\
    & \leq 3 + 3\gamma_p^2c(\beta)\int_0^t \underbrace{\norm{\left|X(\ell_{\Pi}(s))\right|^2}_{L^\frac{p}{2}}}_{\norm{\left|X(\ell_{\Pi}(s))\right|}_{L^p}^2} \,\ds s + \frac32\norm{\Xi}_{\infty}^2t\int_0^t\vertiii{X(\ell_{\Pi}(\cdot))}_s^2\,\ds s \\
    & \leq 3 + \frac32\left(2\gamma_p^2c(\beta) + \norm{\Xi}_{\infty}^2t\right)\int_0^t\vertiii{X(\ell_{\Pi}(\cdot))}_s^2\,\ds s \\
    & \leq 3 + \frac32\left(2\gamma_p^2c(\beta) + \norm{\Xi}_{\infty}^2t\right)\int_0^t\vertiii{X}_s^2\,\ds s.
\end{align*}
The desired result follows again from Gr\"onwall's inequality.
\end{proof}

Next, we show $B_{\Pi}B^{-1}$ is close to $I$ (when $|\Pi|$ is small), the most important step.

\begin{proposition}\label{prop.Bpartinv(B)bound}
Suppose $2 \leq p < \infty$.
There exists a constant $C(N,\beta,T,p) < \infty$ such that
\[
\norm{\left|B_{\Pi}\,B^{-1} - I\right|_T^*}_{L^p} \leq C(N,\beta,T,p)|\Pi|^{\frac12}
\]
for all partitions $\Pi$ of $[0,T]$.
\end{proposition}

\begin{proof}
First, note that if $s \in \Pi$ and $t \in [s_-,s]$, then
\begin{align*}
    B_{\Pi}(t) - B_{\Pi}(s_-) & = I+\int_0^t B_{\Pi}(\ell_{\Pi}(r))\,\ds W(r) + \frac12\int_0^t B_{\Pi}(\ell_{\Pi}(r))\Xi \,\ds r \\
    & \hspace{7.5mm} - \left(I + \int_0^{s_-} B_{\Pi}(\ell_{\Pi}(r))\,\ds W(r) + \frac12\int_0^{s_-} B_{\Pi}(\ell_{\Pi}(r))\Xi \,\ds r\right)\\
    & = \int_{s_-}^t B_{\Pi}(\ell_{\Pi}(r))\,\ds W(r) + \frac12\int_{s_-}^t B_{\Pi}(\ell_{\Pi}(r))\Xi \,\ds r\\
    & = B_{\Pi}(s_-)\left(W(t) - W(s_-) + \frac12 \Xi(t-s_-) \right),
\end{align*}
i.e., $B_{\Pi}(t) = B_{\Pi}(s_-)\,(I+W(t) - W(s_-) + \Xi(t-s_-)/2)$.
Thus,
\begin{align*}
    B_{\Pi} - B_{\Pi}\circ \ell_{\Pi} & = \sum_{s \in \Pi} \left(B_{\Pi} - B_{\Pi}(s_-)\right)1_{(s_-,s]} \\
    & = \sum_{s \in \Pi} B_{\Pi}(s_-) \left(W(\cdot) - W(s_-) + \frac12\Xi (\cdot-s_-)\right)1_{(s_-,s]}(\cdot).\numberthis\label{eq.mat.otherrecursion}
\end{align*}
Let $X \coloneqq B_{\Pi}$ and $H \coloneqq B^{-1}$.
By the It\^o product rule, \eqref{eq.mat.QC}, Proposition \ref{prop.Binvertible+estimates}\ref{item.inv(B)SDE}, and \eqref{eq.mat.otherrecursion},
\begin{align*}
    \ds (XH)(t) & = \ds X(t)\,H(t) + X(t)\,\ds H(t) + \ds X(t)\,\ds H(t) \\
    & = X(\ell_{\Pi}(t))\,\ds W(t)\,H(t) + \frac12 X(\ell_{\Pi}(t))\Xi H(t)\,\ds t\\
    & \hspace{7.5mm} -X(t)\,\ds W(t)\,H(t) + \frac12 X(t)\Xi H(t)\,\ds t - X(\ell_{\Pi}(t))\left(\ds W(t)\right)^2H(t)\\
    & = (X(\ell_{\Pi}(t)) - X(t))\,\ds W(t)\,H(t) - \frac12(X(\ell_{\Pi}(t)) - X(t))\Xi H(t)\,\ds t \\
    & = -\sum_{s \in \Pi} X(s_-)\left(W(t)-W(s_-) + \frac12 \Xi(t-s_-)\right)1_{(s_-,s]}(t)\,\ds W(t)\,H(t) \\
    & \hspace{7.5mm} + \frac12\sum_{s \in \Pi} X(s_-)\left(W(t)-W(s_-) + \frac12 \Xi(t-s_-)\right)1_{(s_-,s]}(t)\,\Xi H(t)\,\ds t
\end{align*}
Now, define
\begin{align*}
    A & \coloneqq \sum_{r \in \Pi} X(r_-)\left(W(\cdot)-W(r_-) - \frac12 \Xi (\cdot - r_-)\right)1_{(r_-,r]}(\cdot), \\
    I_1 & \coloneqq \into A(s)\,\ds W(s)\,H(s), \; \text{ and} \\
    I_2 & \coloneqq \into A(s)\Xi H(s)\,\ds s.
\end{align*}
Also, recall that
\[
W(t) - W(s) \equaldist \sqrt{t-s}W(1), \qquad t \geq s \geq 0.
\]
Retaining the notation and shorthand from the proof of Proposition \ref{prop.Binvertible+estimates}\ref{item.B,inv(B)bounds}, we have that if $0 \leq t \leq T$, then
\begin{align*}
    \vertiii{I_1}_t^2 & \leq \gamma_p^2\norm{\int_0^t\tr\left(\sum_{\xi \in \beta} H^*(s)\xi A^*(s)A(s)\xi H(s) \right)\ds s}_{L^\frac{p}{2}} \tag{BDG} \\
    & \leq \gamma_p^2\sum_{\xi \in \beta}\norm{\xi}_{\infty}^2\norm{\int_0^t \norm{A(s)}_{\infty}^2 \left|H(s)\right|^2\ds s}_{L^\frac{p}{2}} \tag{NCH} \\
    & \leq \gamma_p^2\sum_{\xi \in \beta}\norm{\xi}_{\infty}^2\int_0^t \norm{\norm{A(s)}_{\infty}^2 \left|H(s)\right|^2}_{L^\frac{p}{2}} \,\ds s \tag{MI} \\
    & \leq \gamma_p^2\sum_{\xi \in \beta}\norm{\xi}_{\infty}^2\int_0^t \norm{\norm{A(s)}_{\infty}^2}_{L^p} \norm{\left|H(s)\right|^2}_{L^p} \,\ds s \tag{H\"older's ineq.} \\
    & = \gamma_p^2\sum_{\xi \in \beta}\norm{\xi}_{\infty}^2\int_0^t \norm{\norm{A(s)}_{\infty}}_{L^{2p}}^2 \norm{\left|H(s)\right|}_{L^{2p}}^2 \,\ds s \\
    & \leq 3\gamma_p^2e^{\frac32\left(2\gamma_{2p}^2c(\beta) + \norm{\Xi}_{\infty}^2t\right) t}\sum_{\xi \in \beta}\norm{\xi}_{\infty}^2\int_0^t \norm{\norm{A(s)}_{\infty}}_{L^{2p}}^2 \,\ds s. \tag{Prop.\ \ref{prop.Binvertible+estimates}\ref{item.B,inv(B)bounds}}
\end{align*}
Next, if $0 \leq s \leq t$, then
\begin{align*}
    & \norm{\norm{A(s)}_{\infty}}_{L^{2p}} = \sum_{r \in \Pi}1_{(r-_-,r]}(s)\norm{\norm{X(r_-)\left(W(s)-W(r_-) - \frac12 \Xi (s - r_-)\right)}_{\infty}}_{L^{2p}} \\
    & \leq \sum_{r \in \Pi}1_{(r-_-,r]}(s)\norm{\norm{X(r_-)\left(W(s)-W(r_-) - \frac12 \Xi (s - r_-)\right)}_{\infty}}_{L^{2p}} \\
    & \leq \sum_{r \in \Pi}1_{(r-_-,r]}(s)\norm{\norm{X(r_-)}_{\infty}\left(\norm{W(s)-W(r_-)}_{\infty} +\frac12 \norm{\Xi}_{\infty} (s - r_-)\right)}_{L^{2p}} \\
    & \leq \sum_{r \in \Pi}1_{(r-_-,r]}(s)\norm{\norm{X(r_-)}_{\infty}}_{L^{4p}}\left(\norm{\norm{W(s)-W(r_-)}_{\infty}}_{L^{4p}} +\frac12 \norm{\Xi}_{\infty} (s - r_-)\right) \\
    & \leq \norm{\norm{X}_{\infty,t}^*}_{L^{4p}}\sum_{r \in \Pi}1_{(r-_-,r]}(s)\left(\sqrt{s-r_-}\norm{\norm{W(1)}_{\infty}}_{L^{4p}} +\frac12 \norm{\Xi}_{\infty} (s - r_-)\right) \\
    & \leq \sqrt{N}\norm{\left|X\right|_t^*}_{L^{4p}}|\Pi|^{\frac12}\sum_{r \in \Pi}1_{(r-_-,r]}(s)\left(\norm{\norm{W(1)}_{\infty}}_{L^{4p}} +\frac12 \norm{\Xi}_{\infty} \sqrt{s - r_-}\right) \\
    & \leq  \sqrt{3N}e^{\frac34\left(2\gamma_{4p}^2c(\beta) + \norm{\Xi}_{\infty}^2t\right) t}|\Pi|^{\frac12}\left(\norm{\norm{W(1)}_{\infty}}_{L^{4p}} +\frac12 \norm{\Xi}_{\infty} \sqrt{t}\right), \numberthis\label{eq.Abound}
\end{align*}
where we used $\norm{\cdot}_{\infty} \leq \sqrt{N}|\cdot|$ in the second-to-last inequality and Proposition \ref{prop.Binvertible+estimates}\ref{item.BPibound} in the last inequality.
Putting it together yields
\[
\vertiii{I_1}_t^2 \leq 9N\gamma_p^2e^{3\left(\gamma_{2p}^2c(\beta) + \gamma_{4p}^2c(\beta) + \norm{\Xi}_{\infty}^2t\right) t}\sum_{\xi \in \beta}\norm{\xi}_{\infty}^2\left(\norm{\norm{W(1)}_{\infty}}_{L^{4p}} +\frac12 \norm{\Xi}_{\infty} \sqrt{t}\right)^2t|\Pi|.
\]
Finally, for $I_2$, we have
\begin{align*}
    \vertiii{I_2}_t & = \norm{\into A(s)\Xi H(s)\,\ds s}_t \leq \int_0^t \norm{\left|A(s)\Xi H(s)\right|}_{L^p}\,\ds s \\
    & \leq \norm{\Xi}_{\infty}\int_0^t \norm{\norm{A(s)}_{\infty}\left|H(s)\right|}_{L^p}\,\ds s \\
    & \leq \norm{\Xi}_{\infty}\int_0^t \norm{\norm{A(s)}_{\infty}}_{L^{2p}}\norm{\left|H(s)\right|}_{L^{2p}}\,\ds s \\
    & \leq \norm{\Xi}_{\infty}\sqrt3e^{\frac34\left(2\gamma_{2p}^2c(\beta) + \norm{\Xi}_{\infty}^2t\right) t}\int_0^t \norm{\norm{A(s)}_{\infty}}_{L^{2p}}\,\ds s \tag{Prop.\ \ref{prop.Binvertible+estimates}\ref{item.B,inv(B)bounds}} \\
    & \leq 3\sqrt{N}\norm{\Xi}_{\infty}e^{\frac32\left(\gamma_{2p}^2c(\beta) + \gamma_{4p}^2c(\beta) + \norm{\Xi}_{\infty}^2t\right) t}\left(\norm{\norm{W(1)}_{\infty}}_{L^{4p}} +\frac12 \norm{\Xi}_{\infty} \sqrt{t}\right)t|\Pi|^{\frac12}
\end{align*}
by \eqref{eq.Abound}.
Consequently,
\[
\vertiii{XH - I}_T = \vertiii{-I_1+\frac12 I_2}_T \leq \vertiii{I_1}_T + \frac12\vertiii{I_2}_T \leq C|\Pi|^{\frac12},
\]
where
\begin{align*}
    C & = 3\gamma_pe^{\frac32\left(\gamma_{2p}^2c(\beta) + \gamma_{4p}^2c(\beta) + \norm{\Xi}_{\infty}^2t\right) t}\left(\norm{\norm{W(1)}_{\infty}}_{L^{4p}} +\frac12 \norm{\Xi}_{\infty} \sqrt{t}\right)\sqrt{Nt\sum_{\xi \in \beta}\norm{\xi}_{\infty}^2} \\
    & \hspace{7.5mm} + \frac32\norm{\Xi}_{\infty}e^{\frac32\left(\gamma_{2p}^2c(\beta) + \gamma_{4p}^2c(\beta) + \norm{\Xi}_{\infty}^2t\right) t}\left(\norm{\norm{W(1)}_{\infty}}_{L^{4p}} +\frac12 \norm{\Xi}_{\infty} \sqrt{t}\right)t\sqrt{N}.
\end{align*}
This completes the proof.
\end{proof}

We can now cross the finish line.
If $\Pi$ is a partition of $[0,T]$ and $2 \leq p < \infty$, then
\begin{align*}
    \norm{\left|B_{\Pi} - B\right|_T^*}_{L^p} & = \norm{\left|(B_{\Pi}B^{-1} - I)B\right|_T^*}_{L^p} \leq \norm{\left|B_{\Pi}B^{-1} - I\right|_T^*\norm{B}_{\infty,T}^*}_{L^p} \\
    & \leq \norm{\left|B_{\Pi}B^{-1} - I\right|_T^*}_{L^{2p}}\norm{\norm{B}_{\infty,T}^*}_{L^{2p}} \\
    & \leq \sqrt{N}\norm{\left|B\right|_T^*}_{L^{2p}}\norm{\left|B_{\Pi}B^{-1} - I\right|_T^*}_{L^{2p}} \\
    & \leq \sqrt{3N}e^{\frac34\left(2\gamma_p^2c(\beta) + \norm{\Xi}_{\infty}^2t\right) t}\norm{\left|B_{\Pi}B^{-1} - I\right|_T^*}_{L^{2p}} \tag{Prop.\ \ref{prop.Binvertible+estimates}\ref{item.B,inv(B)bounds}}\\
    & \leq \sqrt{3N}e^{\frac34\left(2\gamma_p^2c(\beta) + \norm{\Xi}_{\infty}^2t\right) t}C(N,\beta,T,2p)|\Pi|^{\frac12}. \tag{Prop.\ \ref{prop.Bpartinv(B)bound}}
\end{align*}
Thus, we can take
\[
K = \left(3N\right)^p e^{\frac{3}{4}\left(2\gamma_p^2c(\beta) + \norm{\Xi}_{\infty}^2tp\right) t}C(N,\beta,T,2p)^p
\]
in Theorem \ref{main.thm.1.MoBettaWZ}.\qed

\subsection{Proof of Theorem \ref{thm.free.MoBettaWZ}}\label{sec.thm.free.MoBettaWZ}

Write $\GL(\cA) \coloneqq \{$invertible elements of $\cA\}$.

We follow the same basic outline as the previous section.
However, the proof is substantially simpler because of Biane and Speicher's operator-norm bound for free stochastic integrals (\cite[Thm.\ 3.2.1]{BS1998}), which we use in place of the BDG inequalities (Theorem \ref{thm.BDGinH}).
It is also superficially simpler because of the lack of a ``$\Xi$ term.''

We begin with some preliminary calculations and estimates.

\begin{proposition}\label{prop.binvertible+estimates}
The following hold.
\begin{enumerate}[font=\normalfont,label=(\roman*)]
    \item $b(t) \in \GL(\cA)$ for all $t \geq 0$.
    Specifically, if $h \colon \R_+ \to \cA$ is the unique solution to the free stochastic differential equation (FSDE)
    \[
    \ds h(t) = -\ds w(t)\,h(t), \qquad h(0) = 1,
    \]
    then $h = b^{-1}$.\label{item.inv(b)SDE}
    \item $\displaystyle\sup_{0 \leq s \leq t} \norm{b^{\e}(s)}_{\infty} \leq \sqrt2 e^{16t}$ for all $t \geq 0$ and $\e =\pm 1$.\label{item.b,inv(b)bounds}
    \item $\displaystyle\sup_{0 \leq s \leq t} \norm{b_{\Pi}(s)}_{\infty} \leq \sqrt2e^{16t}$ for all $t \in [0,T]$.\label{item.bPibound}
\end{enumerate}
\end{proposition}

\begin{proof}
We take each item in turn.

\ref{item.inv(b)SDE} By the free It\^o product rule (\cite[Thm.\ 3.2.5]{NikitopoulosIto}),
\begin{align*}
    \ds (bh)(t) & = \ds b(t)\,h(t) + b(t)\,\ds h(t) + \ds b(t) \, \ds h(t) \\
    & = b(t)\,\ds w(t)\,h(t) - b(t)\,\ds w(t) h(t) - b(t)\,(\ds w(t))^2\,h(t) = 0.
\end{align*}
Consequently, $bh \equiv b(0)h(0) = 1$.
Also,
\begin{align*}
    \ds (hb)(t) & = \ds h(t)\,b(t) + h(t)\,\ds b(t) + \ds h(t) \, \ds b(t) \\
    & = -\ds w(t)\,h(t)b(t) + h(t)b(t)\,\ds w(t) - \ds w(t) \,h(t)b(t)\,\ds w(t) \\
    & = -\ds w(t)\,(hb)(t) + (hb)(t)\,\ds w(t).
\end{align*}
In other words, $u \coloneqq hg$ solves the FSDE
\[
\ds u(t) = -\ds w(t)\,u(t) + u(t)\,\ds w(t), \qquad u(0) = h(0)b(0) = 1.
\]
Since the constant $1$ process also solves the above FSDE, we conclude from the uniqueness of solutions to the above FSDE that $hb = u \equiv 1$.
This establishes that $b$ is always invertible and $b^{-1} = h$.

\ref{item.b,inv(b)bounds}
If $t \geq 0$, then
\begin{align*}
    \norm{b(t)}_{\infty}^2 & = \norm{1 + \int_0^t b(s)\,\ds w(s)}_{\infty}^2 \\
    & \leq 2 + 2\norm{\int_0^t b(s)\,\ds w(s)}_{\infty}^2 \\
    & \leq 2 + 32\int_0^t \norm{b(s)}_{\infty}^2\,\ds s
\end{align*}
by Biane and Speicher's operator-norm bound on free stochastic integrals from \cite{BS1998}.
Therefore, by Gr\"onwall's inequality, $\norm{b(t)}_{\infty}^2 \leq 2 e^{32t}$, i.e.,
\[
\sup_{0 \leq s \leq t} \norm{b(s)}_{\infty} \leq \sqrt2 e^{16t} \qquad (t \geq 0).
\]
By a nearly identical argument, the bound on $\sup\left\{ \norm{b^{-1}(s)}_{\infty} : 0 \leq s \leq t\right\}$ follows from the FSDE for $b^{-1}$ in the previous part.

\ref{item.bPibound} If $0 \leq t \leq T$, then
\begin{align*}
    \sup_{0 \leq s \leq t} \norm{b_{\Pi}(s)}_{\infty}^2 & = \sup_{0 \leq s \leq t}\norm{1 + \int_0^s b_{\Pi}(\ell_{\Pi}(r))\,\ds z(r)}_{\infty}^2 \\
    & \leq 2 + 32 \int_0^t \norm{b_{\Pi}(\ell_{\Pi}(s))}_{\infty}^2\,\ds s \\
    & \leq 2 + 32 \int_0^t \sup_{0 \leq r \leq s}\norm{b_{\Pi}(r)}_{\infty}^2\,\ds s.
\end{align*}
Consequently, by Gr\"onwall's inequality, $\sup \left\{ \norm{b_{\Pi}(s)}_{\infty}^2 : 0 \leq s \leq t\right\} \leq 2 e^{32t}$, i.e.,
\[
\sup_{0 \leq s \leq t} \norm{b_{\Pi}(s)}_{\infty} \leq \sqrt2 e^{16t} \qquad (t \geq 0),
\]
as desired.
\end{proof}

\begin{proposition}\label{prop.bpartinv(b)bound}
There exists a constant $C(T) < \infty$ such that
\[
\sup_{0 \leq t \leq T}\norm{b_{\Pi}(t)\,b^{-1}(t) - 1}_{\infty} \leq C(T)|\Pi|^{\frac12}
\]
for all partitions $\Pi$ of $[0,T]$.
\end{proposition}

\begin{proof}
First, note that if $s \in \Pi$ and $t \in [s_-,s]$, then
\begin{align*}
    b_{\Pi}(t) -b_{\Pi}(s_-) & = 1+\int_0^t b_{\Pi}(\ell_{\Pi}(r))\,\ds w(r) - \left(1 - \int_0^{s_-} b_{\Pi}(\ell_{\Pi}(r))\,\ds w(r)\right)\\
    & = \int_{s_-}^t b_{\Pi}(\ell_{\Pi}(r))\,\ds w(r) = b_{\Pi}(s_-)\,(w(t) - w(s_-)),
\end{align*}
i.e., $b_{\Pi}(t) = b_{\Pi}(s_-)\,(1+w(t) - w(s_-))$.
Thus,
\begin{equation}
    b_{\Pi} - b_{\Pi}\circ \ell_{\Pi} = \sum_{s \in \Pi} \left(b_{\Pi} - b_{\Pi}(s_-)\right)1_{(s_-,s]} = \sum_{s \in \Pi} b_{\Pi}(s_-) \left(w - w(s_-)\right)1_{(s_-,s]}.\label{eq.free.otherrecursion}
\end{equation}
Now, let $x \coloneqq b_{\Pi}$ and $h \coloneqq b^{-1}$.
By the It\^o product rule, Proposition \ref{prop.binvertible+estimates}\ref{item.inv(b)SDE}, and \eqref{eq.free.otherrecursion},
\begin{align*}
    \ds (xh)(t) & = \ds x(t)\,h(t) + x(t)\,\ds h(t) + \ds x(t)\,\ds h(t) \\
    & = x(\ell_{\Pi}(t))\,\ds w(t)\,h(t) - x(t)\,\ds w(t)\,h(t) - x(\ell_{\Pi}(t))(\ds w(t))^2h(t) \\
    & = (x(\ell_{\Pi}(t)) - x(t))\,\ds w(t)\,h(t) \\
    & = -\sum_{s \in \Pi} x(s_-)\left(w(t)-w(s_-)\right)1_{(s_-,s]}(t)\,\ds w(t)\,h(t).
\end{align*}
Consequently, after recalling that
\[
\norm{w(s) - w(t)}_{\infty}^2 \leq 8|t-s| \qquad (s,t \geq 0),
\]
we obtain, again from the Biane--Speicher operator-norm bound on free stochastic integrals, that if $0 \leq t \leq T$, then
\begin{align*}
    \|x(t)h(t) & - 1\|_{\infty}^2 = \norm{{-}\int_0^t\sum_{r \in \Pi} x(r_-)\left(w(s)-w(r_-)\right)1_{(r_-,r]}(s)\,\ds w(s)\,h(s)}_{\infty}^2 \\
    & \leq 16\int_0^t\norm{\sum_{r \in \Pi} x(r_-)\left(w(s)-w(r_-)\right)1_{(r_-,r]}(s)}_{\infty}^2\norm{h(s)}_{\infty}^2\,\ds s \\
    & = 16\int_0^t\sum_{r \in \Pi}\norm{ x(r_-)\left(w(s)-w(r_-)\right)1_{(r_-,r]}(s)}_{\infty}^2\norm{h(s)}_{\infty}^2\,\ds s \\
    & \leq 16\sup_{\underset{|s-r| \leq |\Pi|}{0 \leq r,s \leq t}}\norm{w(s)-w(r)}_{\infty}^2\int_0^t \left(\sum_{r \in \Pi} \norm{x(r_-)}_{\infty}^21_{(r_-,r]}(s)\right)\norm{h(s)}_{\infty}^2\,\ds s \\
    & = 16\sup_{\underset{|s-r| \leq |\Pi|}{0 \leq r,s \leq t}}\norm{w(s)-w(r)}_{\infty}^2\int_0^t \norm{x(\ell_{\Pi}(s))}_{\infty}^2\norm{h(s)}_{\infty}^2\,\ds s \\
    & \leq 128|\Pi|\int_0^t \norm{h(s)}_{\infty}^2\sup_{0 \leq r \leq s}\norm{x(r)}_{\infty}^2\,\ds s
\end{align*}
for any $t \in [0,T]$.
Applying Proposition \ref{prop.binvertible+estimates}\ref{item.b,inv(b)bounds}--\ref{item.bPibound} then yields
\[
\norm{x(t)h(t) - 1}_{\infty}^2 \leq 512|\Pi|\int_0^t e^{64s}\,\ds s \leq 512e^{64t}t|\Pi|.
\]
Hence, the desired result holds with $C(T) = 16\sqrt{2T}e^{32T}$.
\end{proof}

Let $\Pi$ be a partition of $[0,T]$.
By Propositions \ref{prop.bpartinv(b)bound} and \ref{prop.binvertible+estimates}\ref{item.b,inv(b)bounds},
\begin{align*}
    \sup_{0 \leq t \leq T}\norm{b_{\Pi}(t) - b(t)}_{\infty} & = \sup_{0 \leq t \leq T}\norm{\left(b_{\Pi}(t)\,b^{-1}(t) - 1\right)b(t)}_{\infty} \\
    & \leq \sup_{0 \leq t \leq T}\norm{b_{\Pi}(t)\,b^{-1}(t) - 1}_{\infty}\sup_{0 \leq t \leq T}\norm{b(t)}_{\infty} \\
    & \leq C(T)|\Pi|^{\frac12}\sqrt2 e^{16T}.
\end{align*}
Thus, we can take $\overline{K} = C(T)e^{16T}\sqrt2$ in Theorem \ref{thm.free.MoBettaWZ}.\qed

\section{Convergence of Eigenvalues\label{sect.proof.conv.eig}}

Throughout this section, $B_k^N(t)$ refers to a matrix random walk \eqref{eq.Bk} in $\MNC$ for some $N,k\in\mathbb{N}$, $t>0$, and steps $A_1,\ldots,A_k$ independent with common fixed unitarily bi-invariant distribution.  Constants described as {\em universal} do not explicitly depend on $N$, $t$, or $A_1,\ldots,A_k$.

\subsection{Polynomial Decay of the Pseudospectrum} \hfill

\medskip

We begin with the following lemma on the smallest singular value of a product of square matrices, which will be useful throughout this section.

\begin{lemma} \label{lem.sigma(XY)} Let $X_1,\ldots,X_n\in\MNC$, and let $\epsilon_1,\ldots,\epsilon_n>0$.  Then
\begin{equation} \label{ineq.sigma.min.n} \mathbb{P}\{\sigma_{\min}(X_1\cdots X_n)\le\epsilon_1\cdots\epsilon_n\} \le \sum_{j=1}^n \mathbb{P}\{\sigma_{\min}(X_j)\le\epsilon_j\}.
\end{equation}
\end{lemma}

\begin{proof} Through a straightforward induction on $n$, it suffices to prove the lemma in the case $n=2$; i.e.
\begin{equation} \label{ineq.sigma.min} \mathbb{P}\{\sigma_{\min}(XY)\le\epsilon\} \le \mathbb{P}\{\sigma_{\min}(X)\le\epsilon_1\} + \mathbb{P}\{\sigma_{\min}(Y)\le\epsilon_2\}
\end{equation}
where $\epsilon = \epsilon_1\epsilon_2$.

To begin assume $X$ and $Y$ are both invertible.  Then so is $XY$, and
\[ \frac{1}{\sigma_{\min}(XY)} = \|(XY)^{-1}\| = \|Y^{-1}X^{-1}\| \le \|Y^{-1}\|\|X^{-1}\| = \frac{1}{\sigma_{\min}(X)\sigma_{\min}(Y)} \]
and hence $\sigma_{\min}(XY)\ge\sigma_{\min}(X)\sigma_{\min}(Y)$.  This also holds as an equality ($0=0$) if either $X$ or $Y$ fails to be invertible.

Hence, if $\sigma_{\min}(XY)\le\epsilon$ then $\sigma_{\min}(X)\sigma_{\min}(Y)\le \sigma_{\min}(XY)\le\epsilon$; i.e.\
\begin{equation} \label{eq.sigma.containment} \{\sigma_{\min}(XY)\le\epsilon\} \subseteq \{\sigma_{\min}(X)\sigma_{\min}(Y)\le\epsilon\}. \end{equation}
Now, since $\sigma_{\min}(X)$ and $\sigma_{\min}(Y)$ are $\ge 0$, if $\sigma_{\min}(X)>\epsilon_1$ and $\sigma_{\min}(Y)>\epsilon_2$ then $\sigma_{\min}(X)\sigma_{\min}(Y) > \epsilon_1\epsilon_2=\epsilon$.  The contrapositive statement is therefore that $\{\sigma_{\min}(X)\sigma_{\min}(Y)\le\epsilon\}\subseteq \{\sigma_{\min}(X)\le \epsilon_1\}\cup\{\sigma_{\min}(Y)\le\epsilon_2\}$.  Combining this with \eqref{eq.sigma.containment} and using subadditivity of $\mathbb{P}$ proves \eqref{ineq.sigma.min}.
\end{proof}

One corollary of Lemma \ref{lem.sigma(XY)} is that the non-zero eigenvalues of any bi-invariant ensemble are continuous random variables.  The proof also offers a preview of the core idea of the proof of this section's main Theorem \ref{thm.pseudospectrum}.

\begin{lemma} \label{lem.I+A.inv} Fix $N\in\mathbb{N}$.  If $A$ as a.s.\ invertible and bi-invariant, then for any fixed $\lambda\in\mathbb{C}$, $\mathbb{P}\{\lambda\in\mathrm{spec}(A)\}=0$.
\end{lemma}

\begin{proof} Since $A$ is bi-invairant, its singular value decomposition has the form $A = UTV^\ast$ where $U,V,T$ are independent, $U,V$ are Haar Unitary Ensembles, and $T$ is diagonal with the singular values of $A$ on the diagonal; cf.\ Proposition \ref{prop.bi-inv}. Now, on the $\mathbb{P}=1$ event that $A$ (and hence $T$) is invertible, we have
\[ A-\lambda I = UTV^\ast-\lambda I = (U-\lambda VT^{-1})TV^\ast. \]
Applying \eqref{ineq.sigma.min} with $X=U-\lambda VT^{-1}$, $Y=TV^\ast$, and $\epsilon_1=\epsilon_2=\sqrt{\epsilon}$ yields
\begin{align} \nonumber \mathbb{P}\{\sigma_{\min}(A-\lambda I)\le \epsilon\} 
&\le \mathbb{P}\{\sigma_{\min}(U-\lambda VT^{-1})\le \sqrt{\epsilon}\} + \mathbb{P}\{\sigma_{\min}(TV^\ast)\le\sqrt{\epsilon}\} \\
& \le \epsilon^{\frac12\alpha}N^\beta + \mathbb{P}\{\sigma_{\min}(TV^\ast)\le\sqrt{\epsilon}\} \label{eq.A.spec.2}
\end{align}
where $\alpha$ and $\beta$ are universal constants, coming from Theorem \ref{thm.RV} (which applies since $-\lambda VT^{-1}$ is independent from the Haar Unitary matrix $U$).

Letting $\epsilon\downarrow 0$, the left-hand-side of \eqref{eq.A.spec.2} converges to $\mathbb{P}\{\sigma_{\min}(A-\lambda I)= 0\} = \mathbb{P}\{\lambda\in\mathrm{spec}(A)\}$.  On the right-hand-side, the first term converges to $0$, while $\mathbb{P}\{\sigma_{\min}(TV^\ast)\le\sqrt{\epsilon}\}\to \mathbb{P}\{\sigma_{\min}(TV^\ast)=0\} =\mathbb{P}\{TV^\ast\text{ is not invertible}\}$.  Since $A = UTV^\ast$ where $U$ is unitary, $TV^\ast$ is invertible iff $A$ is invertible, which has probability $1$ by assumption. Thus $\mathbb{P}\{\sigma_{\min}(TV^\ast)=0\} = 0$, concluding the proof.
\end{proof}

\begin{corollary} \label{cor.Bk.inv} Under the assumption that the step distributions $A_j$ are a.s.\ invertible, the matrix random walk $B_k^N(t)$ of \eqref{eq.Bk} is a.s.\ invertible for any $k,N\in\mathbb{N}$ and $t>0$.
\end{corollary}

\begin{proof}
Let $\tilde{A}_j = (t/k)^{1/2}A_j$, which is still bi-invariant, and is invertible in the same event that $A_j$ is invertible.  Then $B_k^N(t) = (I+\tilde{A}_1)\cdots(I+\tilde{A}_k)$.
By Lemma \ref{lem.I+A.inv} applied with $\lambda = -1$, each of the terms $I+\tilde{A}_j$ is a.s.\ invertible; hence so is the product $B_k^N(t)$.
\end{proof}

\begin{theorem} \label{thm.pseudospectrum} For each $N,k\in\mathbb{N}$ and $t>0$, let $U^N_0$, $A^N_j$, and $B_k^N(t)$ be as in Notation \ref{notation.Bk->bk}.  Suppose that $A_j=A_j^N$ are a.s.\ invertible and have singular values satisfy Assumption \ref{assump.SV2}.  There are universal exponents $\overline{\alpha},\overline{\beta}>0$ and a constant $C(t,k)<\infty$ so that, for any $z\in\mathbb{C}$ and $\epsilon>0$,
\begin{equation} \label{eq.pseudospectrum} \mathbb{P}\{\sigma_{\min}(U_0^N B^N_k(t)-zI) \le \epsilon\} \le C(t,k) \epsilon^{\overline{\alpha}/k} N^{\overline{\beta}}. \end{equation}
\end{theorem}

\begin{remark} We assume \ref{assump.SV2} holds for each $A_j^N$; i.e.
\[ \mathbb{P}\{\sigma_{\min}(A_j^N)\le\epsilon\}\le \gamma_j \epsilon^{\alpha_j} N^{\beta_j}. \]
Wlog assuming $\epsilon<1$, by taking constants
\[ \alpha := \min\{\alpha_1,\ldots,\alpha_k\}, \quad \beta := \max\{\beta_1,\ldots,\beta_k\}, \quad \gamma := \max\{\gamma,\ldots,\gamma_k\} \]
we may assume the smallest singular values of the $k$ ensembles each satisfy \eqref{assump.SV2} with the same constants $\alpha,\beta,\gamma$.
\end{remark}

\begin{proof} As in the proof of Corollary \ref{cor.Bk.inv}, we expand the matrix random walk product as $B_k^N(t) = (I+\tilde{A}_1)\cdots(I+\tilde{A}_{k-1})(I+\tilde{A}_k) = B_{k-1}(t)(I+\tilde{A}_k)$, where in the case $k=1$ we use the convention $B_0^N(t) = I$. In the singular value decomposition, $\tilde{A}_1 = U_1\tilde{T}_1V_1^\ast$ where $U_1,V_1,\tilde{T}_1$ are independent, $U_1,V_1$ are Haar Unitary Ensembles, and $\tilde{T}_1 = (\frac{t}{k})^{1/2}T_1$ where $T_1$ is diagonal with the singular values of $A_1$ on the diagonal; cf.\ Proposition \ref{prop.bi-inv}.

By assumption $A_k$ is a.s.\ invertible, hence so are $T_k$ and $\tilde{T}_k$.  Corollary \ref{cor.Bk.inv} also shows that $B_k^N(t)$ is a.s.\ invertible for each $k$.  Therefore,
\begin{align*} U_0 B_k^N(t)-zI &= U_0 B_{k-1}^N(t)(I+U_k \tilde{T}_k V_k^\ast)-zI \\
&=U_0B_{k-1}^N(t)U_k\tilde{T}_k(\tilde{T}_k^{-1}U_k^\ast + V_k^\ast)-zI \\
&= M_k(\tilde{T}_k^{-1}U_k^\ast + V_k^\ast)-zI
\end{align*}
where $M_k := U_0B_{k-1}^N(t)U_k\tilde{T}_k$, which is invertible.  Thus
\begin{equation*} U_0B_k^N(t)-zI = M_k(\tilde{T}_k^{-1}U_k^\ast + V_k^\ast)-zI = M_k(\tilde{T}_k^{-1}U_k^\ast-zM_k^{-1} + V_k^\ast). \end{equation*}
Now, $V_k$ is independent from $U_0$, $U_k$, $\tilde{T}_k$, and $B_{k-1}^N(t)$; hence $V_k$ is independent from $D_k:=\tilde{T}_k^{-1}U_k^\ast-zM_k^{-1}$, and we have the decomposition
\begin{equation} \label{eq.UBk-z} U_0B_k^N(t)-zI = (U_0B_{k-1}^N(t))(U_k\tilde{T}_k)(D_k+V_k^\ast). \end{equation}

Applying \eqref{ineq.sigma.min.n} to \eqref{eq.UBk-z} with $X_1 = U_0B_{k-1}^N(t)$, $X_2=U_k\tilde{T}_k$, and $X_3=X_k+V_k^\ast$ yields
\begin{align*} \mathbb{P}\{\sigma_{\min}(U_0B_k^N(t)-zI)\le\epsilon_1\epsilon_2\epsilon_3\}
&\le \mathbb{P}\{\sigma_{\min}(U_0B_{k-1}^N(t))\le\epsilon_1\} \\
&+ \mathbb{P}\{\sigma_{\min}(\tilde{A}_k)\le\epsilon_2\} + \mathbb{P}\{\sigma_{\min}(D_k+V_k^\ast)\le\epsilon_3\}
\end{align*}
where we have used the fact that $\sigma_{\min}(U_k\tilde{T}_k) = \sigma_{\min}(\tilde{A}_k)$ (because multiplying by the unitary $V$ does not change singular values).  Define
\begin{equation} \label{eq.pkz} p_t(k,z,\eta):=\mathbb{P}\{\sigma_{\min}(U_0B_k^N(t)-zI)\le\eta\}.
\end{equation}
So we have established that
\begin{align} \nonumber
p_t(k,z,\epsilon_1\epsilon_2\epsilon_3) &\le p_t(k-1,0,\epsilon_1) \\
&+ \mathbb{P}\{\sigma_{\min}(\tilde{A}_k)\le\epsilon_2\} + \mathbb{P}\{\sigma_{\min}(D_k+V_k^\ast)\le\epsilon_3\}, \qquad k\ge 1.  \label{eq.def.pkz}
\end{align}

Since $\tilde{A}_k = (t/k)^{1/2}A_k$, $\sigma_{\min}(\tilde{A}_k) = (t/k)^{1/2}\sigma_{\min}(A_k)$.  Thus, by Assumption \ref{assump.SV2},
\begin{align*}
\mathbb{P}\{\sigma_{\min}(\tilde{A}_k)\le\epsilon_2\} = \mathbb{P}\{(t/k)^{1/2}\sigma_{\min}(A_k)\le\epsilon_2\} \le \gamma((k/t)^{1/2}\epsilon_2)^{\alpha}N^{\beta}.
\end{align*}
Meanwhile, Theorem \ref{thm.RV} asserts that $\mathbb{P}\{\sigma_{\min}(D_k+V_k^\ast)\le\epsilon_3\} \le (\epsilon_3)^{\alpha_0}N^{\beta_0}$.  Let $\overline{\beta}=\max\{\beta_0,\beta\}$; thus
\[ \mathbb{P}\{\sigma_{\min}(\tilde{A}_k)\le\epsilon_2\} + \mathbb{P}\{\sigma_{\min}(D_k+V_k^\ast)\le\epsilon_3\} \le \left(\gamma((k/t)^{1/2}\epsilon_2)^{\alpha}+(\epsilon_3)^{\alpha_0}\right)N^{\overline{\beta}}. \]
Let $\epsilon_2' = \epsilon_2\epsilon_3$, and choose $\epsilon_2,\epsilon_3$ so that $(\epsilon_2)^{\alpha} = (\epsilon_3)^{\alpha_0}$.  Define $\alpha>0$ by $\frac{1}{\overline{\alpha}}=\frac{1}{\alpha_0}+\frac{1}{\alpha}$.  A calculation then shows that
\[ \gamma((k/t)^{1/2}\epsilon_2)^{\alpha}+(\epsilon_3)^{\alpha_0} = \gamma_t(k)(\epsilon_2')^{\overline{\alpha}} N^{\overline{\beta}} \]
where $\gamma_t(k) = \gamma\cdot(k/t)^{\alpha/2}$.  Note that this estimate is independent of $z$.  Thus, \eqref{eq.def.pkz} yields the estimates
\begin{equation} \label{eq.pkz2}
p_t(k,z,\epsilon_1\epsilon_2')\le p_t(k-1,0,\epsilon_1)+ \gamma_t(k)(\epsilon_2')^{\overline{\alpha}} N^{\overline{\beta}}.
\end{equation}

This allows us to proceed by induction.  In the base case $k=1$, note that
\[ p_t(0,0,\epsilon_1) = \mathbb{P}\{\sigma_{\min}(U_0)\le \epsilon_1\} = \mathbb{P}\{1\le\epsilon_1\} \]
since $U_0$ is unitary and so its singular values are all $1$.  Thus, provided we ensure $\epsilon_1<1$, $p_t(0,0,\epsilon_1)=0$, and \eqref{eq.pkz2} yields
\begin{equation} \label{eq.induction.base} p_t(1,z,\epsilon_1\epsilon_2') \le \gamma_t(1)(\epsilon_2')^{\overline{\alpha}} N^{\overline{\beta}}, \qquad \epsilon_1\in(0,1). \end{equation}
By taking $\epsilon_1\uparrow 1$, we can then combine \eqref{eq.induction.base} with \eqref{eq.pkz2} to give the induction
\begin{equation} \label{eq.full.induction}
\begin{aligned}
p_t(1,z,\epsilon) &\le \gamma_t(1)\epsilon^{\overline{\alpha}} N^{\overline{\beta}} \\
p_t(k,z,\epsilon_1\epsilon_2) &\le p_t(k-1,0,\epsilon_1)+\gamma_t(k)(\epsilon_2)^{\overline{\alpha}}N^{\overline{\beta}} \qquad k\ge 1.
\end{aligned}
\end{equation}
(Note: The base of the induction can also be derived directly from \eqref{eq.UBk-z} in the case $k=1$, by combining $U_0$ with $U_k$ and using Lemma \ref{lem.sigma(XY)} on the two terms, then optimizing as above.)

The conclusion is that
\begin{equation} \label{eq.induction.concl} p_t(k,z,\epsilon) \le [\gamma_t(1)+\cdots+\gamma_t(k)]\epsilon^{\overline{\alpha}/k}N^{\overline{\beta}}. \end{equation}
Indeed, the base case in \eqref{eq.full.induction} is precisely the $k=1$ case of \eqref{eq.induction.concl}.  Assuming then that \eqref{eq.induction.concl} has been established up to $k-1$, then \eqref{eq.full.induction} yield
\[ p_t(k,z,\epsilon) \le [\gamma_t(1)+\cdots+\gamma_t(k-1)](\epsilon_1)^{\overline{\alpha}/(k-1)}N^{\overline{\beta}} +\gamma_t(k)(\epsilon_2)^{\overline{\alpha}} N^{\overline{\beta}} \]
for any $\epsilon_1,\epsilon_2>0$ with $\epsilon_1\epsilon_2=\epsilon$.  Taking $\epsilon_1 = \epsilon^{1-1/k}$ and $\epsilon_2 = \epsilon^{1/k}$ yields \eqref{eq.induction.concl} at level $k$.

This concludes the proof of \eqref{eq.pseudospectrum}, with
\[ C(t,k) = \gamma_t(1)+\cdots+\gamma_t(k) = \frac{\gamma}{t^{\alpha/2}}\sum_{j=1}^k j^{\alpha/2} \le \frac{\gamma k^{\alpha/2+1}}{t^{\alpha/2}}. \]
\end{proof}

\subsection{Wegner Estimates\label{section.Wegner}} \hfill

\medskip

In the previous section, we proved quantitative control on the decay of the pseudospectrum of our matrix random walks $U_0^N B_k^N(t)$ as $N\to\infty$.  The second ingredient needed to prove convergence of the $\mathrm{ESD}$ of these ensembles is reasonable control on ``moderately small'' singular values: i.e.\ knowing from Theorem \ref{thm.pseudospectrum} that the smallest singular value is unlikely to be smaller than $N^{-c}$ for some $c>0$, we need only verify that those singular values between $N^{-c}$ and some small $\epsilon>0$ do not bunch up quickly near $0$ as $N\to\infty$.  Bounds to prove this kind of anti-concentration property generally go under the name {\em Wegner estimates}, cf.\ \cite{Wegner}.

In practice, Wegner estimates can be stated efficiently as regularity properties of the Cauchy transform of the empirical law.  For example: Assumption \ref{SRT.2} in the Single Ring Theorem is in fact a (strong) Wegner estimate.  That assumption was motivated by a similar property which was originally proved in the context of $\mathrm{GUE}$ or Ginibre ensembles \cite{HaagerupT2005Annals} (see \cite[Lemma 5.5.4]{GuionnetKZ-single-ring} for a detailed statement and proof).  As with that case, below we prove the requisite estimate in  greater generality than needed. %not simply for the matrix random walk $U_0 B_k(t)$ but for any noncommutative polynomial in independent Haar Unitary Ensembles $U_i$ and selfadjoint ensembles $T_j$.

\begin{theorem} \label{thm.Wegner} Let $N\in\mathbb{N}$, and let $\mathbf{A}^N=(A_1^N,\ldots,A_k^N)$ be independent bi-invariant ensembles in $\MNC$.  For $1\le j\le k$, let $\nu^N_j$ denote the empirical law of singular values of $A_j$, and suppose $\nu_j^N\rightharpoonup\nu_j$ weakly a.s.  Assume these measures satisfy \ref{assump.SV1} and \ref{assump.SV3} of Assumption \ref{assump.sing.conv}.  We also make {\em one} of the following assumptions:
\begin{enumerate}
    \item[(R)] \label{assump.R} Assumption \ref{assump.SV4}  holds true for $\{\nu^N_j\}$ and $\{\nu_j\}$; or
    \item[(G)] \label{assump.G} Each $A^N_j$ is a polynomial in independent $\mathrm{GUE}$s; or
    \item[(U)] \label{assump.U} Each $A^N_j$ is a polynomial in Haar Unitary Ensembles.
\end{enumerate}

Let $P$ be a positive selfadjoint noncommutative polynomial in $k$ variables.  Denote by $P^N = P(\mathbf{A}^N)$.  Then there exist constants $\overline{C},c_1>0$ and $c_2\in(0,1)$ not depending on $N$ so that, for $N^{-c_1}\le\eta\le 1$,
\begin{equation} \label{eq.Wegner} \left|\mathrm{Im}\,\mathbb{E}G_{P^N}(i\eta)\right| \le \overline{C}\eta^{-c_2}. \end{equation}
\end{theorem}

We will apply Theorem \ref{thm.Wegner} to the polynomial $P^N = (B^N_k(t)-zI)^\ast(B^N_k(t)-zI)$ in the next section, to complete the proof of Theorem \ref{thm.conv.ESD}.  Condition \ref{assump.R} is, by far, the most general, and the theorem with this condition will require several tools and estimates (from high dimensional probability and free entropy theory) to prove below.

We begin with a result that gets most of the way to Theorem \ref{thm.Wegner} without any of the three competing conditions \ref{assump.R}, \ref{assump.G}, or \ref{assump.U}.  Since each $A_j^N$ decomposes as $U_j^N T_j^N V_j^{N\ast}$ where $U_j^N,V_j^N$ are Haar Unitary Ensembles and all three matrices are independent, any polynomial in $\mathbf{A}^N$ is a polynomial in independent Haar Unitary Ensembles and selfadjoint ensembles.  We state and prove the following allowing also an arbitrary additional unitary ``initial condition'' for future use.

\begin{proposition} \label{prop.Wegner.1} Let $N\in\mathbb{N}$, and let $\mathbf{U}^N=(U_1^N,\ldots,U_d^N)$ be independent Haar Unitary Ensembles in $\mathrm{U}(N)$ and let $\mathbf{T}^N=(T_0^N,T_1^N,\ldots,T_r^N)$ be selfadjoint ensembles in $\MNC$ independent from each other and from $\mathbf{U}^N$.  Let $T_j^N$ have $\mathrm{ESD}$ $\nu_j^N$, with weak limit $\nu_j$.  Assume that, for $1\le j\le k$, $(\nu_j^N)_{N\in\mathbb{N}}$ and $\nu_j$ satisfy Assumption \ref{assump.sing.conv}.  Let $U_0^N = \exp(iT_0^N)$, and set $\mathbf{S}^N = (U_0^N,\mathbf{T}^N)$.

Let $(\mathscr{A},\varphi)$ be a $W^\ast$-probability space containing freely independent operators $\mathbf{u}=(u_1,\ldots,u_d)$ and $\boldsymbol{\vartheta}=(\vartheta_0,\vartheta_1,\ldots,\vartheta_r)$ where $u_j$ are Haar unitaries and $\vartheta_j$ has distribution $\nu_j$. let $u_0=\exp(i\vartheta_0)$, and set $\mathbf{s}=(u_0,\boldsymbol{\vartheta})$.

Let $P$ be a positive selfadjoint noncommutative polynomial in $d+r+1$ variables, and set $P^N = P(\mathbf{S}^N,\mathbf{U}^N)$.  There exist constants $\overline{C},c_1>0$ not depending on $N$ so that, for $N^{-c_1}\le\eta\le 1$,
\begin{equation} \label{eq.Wegner.partial}  \left|\mathrm{Im}\,\mathbb{E}G_{P^N}(i\eta)\right| \le \overline{C}_0\,\eta^{-\frac12} + \left|\varphi[(P(\mathbf{s},\mathbf{u})-i\eta)^{-1}]\right|. \end{equation}
\end{proposition}

To prove Proposition \ref{prop.Wegner.1}, we begin by noting from the definition of the Cauchy transform that
\begin{equation} \label{eq.GPieta} \mathbb{E}G_{P^N}(i\eta) = -\mathbb{E}\,\mathrm{tr}_N\left[(P(\mathbf{S}^N,\mathbf{U}^N)-i\eta)^{-1}\right] \end{equation}
where $\mathrm{tr}_N = \frac{1}{N}\mathrm{Tr}$ is the normalized trace on $\MNC$.  (We prefer to write the denominators with the $i\eta$ second, hence the minus sign in front.)  To estimate this quantity, we introduce an expanded $W^\ast$-probability space $(\mathscr{A}_N,\varphi_N)$ containing both $(\MNC,\mathrm{tr})$ and $(\mathscr{A},\varphi)$ {\em as freely independent subalgebras}.
\ignore{
and also containing random variables $\mathbf{u} = (u_1,\ldots,u_d)$ and selfadjoint random variables $\boldsymbol{\vartheta} = (\vartheta_0,\vartheta_1,\ldots,\vartheta_r)$ such that
\begin{enumerate}
    \item \label{FP.1} $\mathbf{u},\boldsymbol{\vartheta}$ are freely independent from $\MNC$.
    \item \label{FP.2} $u_1,\ldots,u_d,\vartheta_0,\vartheta_1,\ldots,\vartheta_r$ are all freely independent.
    \item \label{FP.3} For $1\le j\le d$, $u_j$ is Haar unitary.
    \item \label{FP.4} For $0\le j\le r$, $\vartheta_j$ has distribution $\nu_j$, the large-$N$ limit of the $\mathrm{ESD}$ $\nu_j^N$ of $T_j^N$.
\end{enumerate} 
}
(Such a $W^\ast$-probability space exists by Proposition \ref{prop.Voiculescu.ext}.) In particular: this means that the random matrices $\mathbf{S}^N,\mathbf{U}^N$ are a.s.\ freely independent from the operators $\mathbf{s},\mathbf{u}$.

It is important to note that the matrices in $\mathbf{S}^N,\mathbf{U}^N$ are still (classically) independent from each other, not freely independent.  It will be convenient to ``liberate'' them and make them actually free, without changing their individual $\ast$-distributions (with respect to the new combined stated $\varphi_N$).  Utilizing Proposition \ref{prop.conj.free}, we can achieve this by conjugating by free Haar unitaries in $(\mathscr{A}_N,\varphi_N)$.  So, let $\mathbf{v} = (v_0,v_1,\ldots,v_r)$ be freely independent Haar unitaries, all freely independent from $\MNC\subset\mathscr{A}_N$.  Define
\[ \mathbf{v}\mathbf{S}^N\mathbf{v}^\ast := (v_0U^N_0v_0^\ast,v_1T^N_1v_1^\ast,\ldots,v_rT^N_rv_r^\ast). \]

We estimate \eqref{eq.GPieta} as follows.  Let Let $\mathscr{S}_N$ be the $\sigma$-field generated by the random matrices $\mathbf{S}^N$.  Thus conditioning $\mathbb{E}_{\mathscr{S}_N}$ treats $\mathbf{S}^N$ as deterministic matrices.  Double conditioning $\mathbb{E}\circ\mathbb{E}_{\mathscr{S}_N}$ and using the triangle inequality yields that
\begin{align} \nonumber
&\left|\mathrm{Im}\left(\mathbb{E}\,\mathrm{tr}_N\left[(P(\mathbf{S}^N,\mathbf{U}^N)-i\eta)^{-1}\right]\right)\right| \\
\label{Felix.1}\le &\;\mathbb{E}\left|\mathbb{E}_{\mathscr{S}_N}\,\mathrm{tr}_N\left[(P(\mathbf{S}^N,\mathbf{U}^N)-i\eta)^{-1}\right]-\varphi_N\left[(P(\mathbf{v}\mathbf{S}^N\mathbf{v}^\ast,\mathbf{u})-i\eta)^{-1}\right]\right| \\
\label{Felix.2}+ &\;\mathbb{E}\left|\varphi_N\left[(P(\mathbf{v}\mathbf{S}^N\mathbf{v}^\ast,\mathbf{u})-i\eta)^{-1}\right]-\varphi_N\left[(P(\mathbf{s},\mathbf{u})-i\eta)^{-1}\right]\right| \\
\label{Felix.3}+ &\; \left|\varphi\left[(P(\mathbf{s},\mathbf{u})-i\eta)^{-1}\right]\right|.
\end{align}
(The last term could be written with a $\varphi_N$; since $\mathbf{s},\mathbf{u}$ are contained in the subalgebra $\mathscr{A}\subset\mathscr{A}_N$ where $\varphi_N$ reduces to $\varphi$, it is more convenient to write it this way.)
Note that \eqref{Felix.3} is precisely the second term on the right in \eqref{eq.Wegner.partial}; hence, proving Proposition \ref{prop.Wegner.1} is achieved by bounding each of the terms \eqref{Felix.1} and \eqref{Felix.2} $\lesssim \eta^{-1/2}$ for $N$ polynomially small.  We do so in the following two lemmas, in which we make all the assumptions from Proposition \ref{prop.Wegner.1}.

\begin{lemma} \label{lem.Felix.1} There is a constant $\overline{C}_1>0$ so that the quantity \eqref{Felix.1} is $\le \overline{C}_1\eta^{-1/2}$ whenever $N^{-2/13}\le\eta\le 1$.
\end{lemma}

\begin{proof} By Assumption \ref{assump.SV1}, $\sup_{N,j}\|T^N_j\|<\infty$ a.s.  Therefore, by \cite[Theorem 1.1]{Parraud2022}, there is a constant $C_1>0$ independent of $N$ so that, for every $f\in C^6_b(\R\to\C)$ (functions with at least $6$ bounded derivatives),
\[ \left|\mathbb{E}_{\mathscr{S}_N}\mathrm{tr}_N[f(P(\mathbf{S}^N,\mathbf{U}^N))] - \varphi_N[f(P(\mathbf{v}\mathbf{S}^N\mathbf{v}^\ast,\mathbf{u}))] \right|\ \le C_1\frac{(\log N)^2}{N^2}\|f\|_{C^6_b} \quad \text{a.s.} \]
where $\|f\|_{C^6_b} = \sup_{n\le 6,x\in\mathbb{R}}|f^{(n)}(x)|$.  Applying this with $f(x) = (x-i\eta)^{-1}$, we compute that $\|f\|_{C^6_b} = 6!\eta^{-7}$ for $0<\eta\le 1$.  Hence, since $\mathbf{S}^N$ are a.s.\ bounded, by the bounded convergence theorem, \eqref{Felix.1} is
\[ \le 6!C_1 \frac{(\log N)^2}{N^2\eta^7} = 6!C_1\frac{(\log N)^2}{N^2\eta^{6.5}}\cdot \eta^{-1/2}. \]
Hence, if $N^{-2/13}\le \eta\le 1$, we have
\[ \frac{(\log N)^2}{N^2 \eta^{6.5}} \le \frac{(\log N)^2}{N^2}\cdot N^{\frac{2}{13}\cdot 6.5} = \frac{(\log N)^2}{N} \le 1 \quad \text{for all }N\ge 1. \]
Thus the lemma holds true with $\overline{C}_1 := 6!C_1$.
\end{proof}

Estimating \eqref{Felix.2} requires Assumption \ref{assump.SV3}, measuring the convergence rate of $\nu_j^N$ to $\nu_j$ in terms of the Wasserstein distance $\mathcal{W}_1$, which we presently discuss.  Given two probability measures $\nu_1,\nu_2$ on a Polish space $(S,\mathcal{B})$, let $\Gamma(\mu,\nu)$ denote the set of all couplings of $\mu$ and $\nu$: i.e.\ probability measures $\gamma$ on
$(S\times S,\mathcal{B}\otimes\mathcal{B})$ whose marginals are $\nu_1$ and $\nu_2$: $\gamma(B\times S) = \nu_1(B)$ and $\gamma(S\times B) = \nu_2(B)$ for all $B\in\mathcal{B}$. For $1\le p<\infty$, the {\bf Wasserstein-$p$ distance} $\mathcal{W}_1$ between $\nu_1$ and $\nu_2$ is defined to be
\begin{equation} \label{eq.Wass.p} \mathcal{W}_p(\nu_1,\nu_2) = \inf_{\gamma\in\Gamma(\nu_1,\nu_2)} \left(\int_{S\times S}d_S(x,y)^p\,\gamma(dx\,dy)\right)^{1/p} \end{equation}
which is well-defined on probability measures $\nu_j$ with $x\mapsto d_S(\star,x)\in L^p(\nu_j)$ for some (and therefore any) base point $\star\in S$.  $\mathcal{W}_p$ metrizes weak convergence for such probability measures.

The Wasserstein distance (particularly with $p=1$) is fundamental to the theory of optimal transport, and has also become ubiquitous in Stein's method for normal and Poisson approximation.  In the latter case, where the state space $S=\mathbb{R}$, there are simpler, computable expressions for $\mathcal{W}_p$ in terms of cumulative distribution functions and quantile functions.

\begin{definition}  Let $\nu$ be a Borel probability distribution on $\mathbb{R}$, and denote $F_\nu(x) = \nu((-\infty,x])$ the cumulative distribution function.  The {\bf quantile function} $Q_\nu$, also denoted $F_\nu^{-1}$, is the non-decreasing function $[0,1]\to\mathbb{R}$ given by
\[ Q_\nu(t) = \inf\{x\in\mathbb{R}\colon F_\nu(x)\ge t\}. \]
If $F_\nu$ is continuous and strictly increasing then $Q_\nu$ is the function inverse of $F_\nu$.  In general, $(Q_\nu)_\ast\mathrm{Unif}[0,1] = \nu$; i.e.\ if $x\equaldist\mathrm{Unif}[0,1]$ then $Q_\nu(x) \equaldist \nu$.
\end{definition}

For probability measures $\nu_1,\nu_2$ on $\mathbb{R}$ possessing a $p$th absolute moment, the Wasserstein distance \eqref{eq.Wass.p} can be computed as
\[ \mathcal{W}_p(\nu_1,\nu_2) = \left(\int_0^1 |Q_{\nu_1}(t)-Q_{\nu_2}(t)|^p\,dt\right)^{1/p} \]
i.e.\ the $L^p[0,1]$ distance between the quantile functions.  In the special case $p=1$ of greatest interest to us, a change of variables shows that in fact
\begin{equation} \label{eq.Wass.1} \mathcal{W}_1(\nu_1,\nu) = \int_0^1 |Q_{\nu_1}(t)-Q_{\nu_2}(t)|\,dt = \int_{\R} |F_{\nu_1}(x)-F_{\nu_2}(x)|\,dx. \end{equation}

\begin{lemma} \label{lem.Felix.2} There are constants $\overline{C}_2,c_{1,2}>0$ so that the quantity \eqref{Felix.2} is $\le \overline{C}_2\eta^{-1/2}$ whenever $N^{-c_{1,2}}\le\eta\le 1$.
\end{lemma}

\begin{proof} Fix $r+1$ freely independent random variables $x_0,x_1,\ldots,x_r\in\mathscr{A}_N$ each distributed uniformly on $[0,1]$.  Using quantile functions: define
\begin{align} \label{eq.tilde.1} \tilde{T}_j^N := Q_{\nu_j^N}(x_j) \quad &\text{and} \quad \tilde\vartheta_j := Q_{\nu_j}(x_j) \\ \label{eq.tilde.2}
\tilde{U}_0^N := \exp(i\tilde{T}_0^N) \quad &\text{and} \quad  \tilde{u}_0:=\exp(i\tilde\vartheta_0).
\end{align}
Then each $\tilde{T}_j^N$ has the same distribution as $T^N_j$, and each $\tilde{\vartheta}_j$ has the same distribution as $\vartheta_j$; it follows that $\tilde{U}^N_0$ has the same distribution as $\tilde{u}_0$.  Since the joint $\ast$-distribution of free random variables is totally determined by their individual distributions, it follows that 
\begin{align*} (\mathbf{v}\tilde{\mathbf{S}}^N\mathbf{v}^\ast,\mathbf{u}) \quad \text{has the same }&\varphi_N\; \ast\text{-distribution as} \quad (\mathbf{v}\mathbf{S}^N\mathbf{v}^\ast,\mathbf{u})\quad\text{and} \\
(\mathbf{v}\tilde{\mathbf{s}}\mathbf{v}^\ast,\mathbf{u}) \quad \text{has the same }&\varphi_N\; \ast\text{-distribution as} \quad (\mathbf{s},\mathbf{u})
\end{align*}
where $\tilde{\mathbf{S}}^N := (\tilde{U}^N_0,\tilde{T}_1^N,\ldots,\tilde{T}_r^N)$ and $\tilde{\mathbf{s}}:= (\tilde{u}_0,\tilde{\vartheta}_1,\ldots,\tilde{\vartheta}_r)$. (In the second line, we use the fact that $(\mathbf{v}\tilde{\mathbf{s}}\mathbf{v}^\ast,\mathbf{u})$ has the same $\ast$-distribution as $(\tilde{\mathbf{s}},\mathbf{u})$ since the $s_j$ are already freely independent, and conjugating each $\tilde{s}_j$ does not change its individual distribution.) Hence
\[ \eqref{Felix.2} = 
\mathbb{E}\left|\varphi_N\left[(P(\mathbf{v}\tilde{\mathbf{S}}^N\mathbf{v}^\ast,\mathbf{u})-i\eta)^{-1}\right]-\varphi_N\left[(P(\mathbf{v}\tilde{\mathbf{s}}\mathbf{v}^\ast,\mathbf{u})-i\eta)^{-1}\right]\right|. \]

Inside the expectation, we use the identity $X^{-1}-Y^{-1} = X^{-1}(Y-X)Y^{-1}$ and H\"older's ineuquality in the form $|\varphi_N(X^{-1}(Y-X)Y^{-1}|\le \|X\|\|X-Y\|_1\|Y\|$ applied with $X = (P(\mathbf{v}\tilde{\mathbf{S}}^N\mathbf{v}^\ast,\mathbf{u})-i\eta)^{-1}$ and $Y=(P(\mathbf{v}\tilde{\mathbf{s}}\mathbf{v}^\ast,\mathbf{u})-i\eta)^{-1}$.  Since $P$ is a selfadjoint polynomial, $\tilde{P}^N = P(\mathbf{v}\tilde{\mathbf{S}}^N\mathbf{v}^\ast,\mathbf{u})$ is selfadjoint, and so $\|X\| = \lim_{p\to\infty} \|X\|_p$, where
\[\|X\|_p = \varphi_N(|(\tilde{P}^N-i\eta)^{-1}|^p)^{1/p} = \left(\int_{\R}\frac{\mu_{\tilde{P}^N}(dx)}{|x-i\eta|^p}\right)^{1/p} \]
and since $|x-i\eta|^p = (x^2+\eta^2)^{p/2} \ge \eta^p$ it follows that $\|X\|_p^p \le \int \eta^{-p}\,d\mu_{\tilde{P}^N} = \eta^{-p}$ for all $p$; it follows that $\|X\|\le 1/\eta$.  An analogous argument shows that $\|Y\|\le 1/\eta$.  Hence
\begin{equation} \label{eq.Felix.L1.1}  \eqref{Felix.2} \le \frac{1}{\eta^2} \mathbb{E}\|P(\mathbf{v}\tilde{\mathbf{S}}^N\mathbf{v}^\ast,\mathbf{u})-P(\mathbf{v}\tilde{\mathbf{s}}\mathbf{v}^\ast,\mathbf{u})\|_{L^1(\varphi_N)}. \end{equation}
Now, the polynomial $P$ is a finite sum $P=\sum_n p_n M_n$ where each $M_n$ is a noncommutative monomial.  By the triangle inequality,
\[ \eqref{eq.Felix.L1.1} \le \frac{1}{\eta^2}\sum_n |p_n|\cdot \mathbb{E} \| M_n(\mathbf{v}\tilde{\mathbf{S}}^N\mathbf{v}^\ast,\mathbf{u})-M_n(\mathbf{v}\tilde{\mathbf{s}}\mathbf{v}^\ast,\mathbf{u})\|_{L^1(\varphi_N)}. \]
For convenience, let us label the full lists $(\mathbf{v}\tilde{\mathbf{S}}^N\mathbf{v}^\ast,\mathbf{u}) = \mathbf{X} = (X_1,X_2,\ldots,X_{d+r+1})$ and $(\mathbf{v}\tilde{\mathbf{s}}\mathbf{v}^\ast,\mathbf{u}) = \mathbf{Y}=(Y_1,Y_2,\ldots,Y_{d+r+1})$.  ($X_j=Y_j$ for $j>r+1$, but it is not notationally useful to take account of that fact.)  Now, $M_n(\mathbf{X}) = X_{i_1}X_{i_2}\cdots X_{i_m}$ for some ($n$-dependent) indices $i_1,\ldots,i_m\in\{1,\ldots,d+r+1\}$.  We expand the difference as a telescoping sum:
\[ M_n(\mathbf{X})-M_n(\mathbf{Y}) = \sum_{\ell=1}^m X_{i_1}\cdots X_{i_{\ell-1}}(X_{i_\ell}-Y_{i_\ell})Y_{i_{\ell+1}}\cdots Y_{i_m} \]
and therefore, applying H\"older's inequality again,
\begin{align*}  &\| M_n(\mathbf{v}\tilde{\mathbf{S}}^N\mathbf{v}^\ast,\mathbf{u})-M_n(\mathbf{v}\tilde{\mathbf{s}}\mathbf{v}^\ast,\mathbf{u})\|_{L^1(\varphi_N)} \\
\le &\sum_{\ell=1}^m \|X_{i_1}\|\cdots\|X_{i_{\ell-1}}\| \|X_{i_\ell}-Y_{i_\ell}\|_{L^1(\varphi_N)} \|Y_{i_{\ell+1}}\|\cdots\|Y_{i_m}\|.
\end{align*}

By assumption, there is some $K<\infty$ so that $\|T^N_j\|\le K$ a.s.\ and $\|\vartheta_j\|\le K$ for $1\le j\le k$; wlog take $K\ge 1 = \|U_0^N\|=\|U_j^N\|=\|u_0\|=\|u_j\|$.  Thus
\[ \mathbb{E}\| M_n(\mathbf{v}\tilde{\mathbf{S}}^N\mathbf{v}^\ast,\mathbf{u})-M_n(\mathbf{v}\tilde{\mathbf{s}}\mathbf{v}^\ast,\mathbf{u})\|_{L^1(\varphi_N)} \le mK^{m-1}\sup_\ell \|X_{i_\ell}-Y_{i_\ell}\|_{L^1(\varphi_N)}. \]
Note that $m = \mathrm{deg}(M_n)$ is the degree of the monomial $M_n$.

Now, for each $i$, $X_i$ is one of $v_0\tilde{U}_0^Nv_0^\ast$, $v_j\tilde{T}_j^Nv_j^\ast$ for some $j$, or $u_j$, while $Y_i$ is one of $v_0\tilde{u}_0v_0^\ast$, $v_j\tilde{\vartheta}_jv_j^\ast$, or $u_j$.  Note that, for any operators $\tilde{S}$ and $\tilde{s}$ and any unitary $v$, $\|v\tilde{S}v^\ast-v\tilde{s}v^\ast\|_{L^1} = \|v(\tilde{S}-\tilde{s})v^\ast\|_{L^1} = \|\tilde{S}-\tilde{s}\|_{L^1}$.  Hence, we have established that
\begin{equation} \label{eq.Felix.L1.2} \eqref{Felix.2} \le
\frac{1}{\eta^2}\cdot C(P) K^{\mathrm{deg}(P)-1}\max_{1\le j\le k}\{\|\tilde{U}_0^N-\tilde{u}_0\|_{L^1(\varphi_N)},\|\tilde{T}^N_j-\tilde{\vartheta}_j\|_{L^1(\varphi_N)}\}\end{equation}
where $C(P) = \|P\|_1\cdot\mathrm{deg}(P)$, where $\|P\|_1 = \sum_n |p_n|$ and $\mathrm{deg}$ the degree of $P$, i.e.\ the maximum degree of any monomial term in $P$.

Using \eqref{eq.tilde.2} and Duhamel's formula, we have
\[ \tilde{U}_0^N-\tilde{u}_0 = \exp(i\tilde{T}^N_0)-\exp(i\tilde\vartheta_0) = \int_0^1 e^{(1-s)i\tilde{T}^N_0}(\tilde{T}^N_0-\tilde{\vartheta}_0)e^{si\tilde{\vartheta}_0}\,ds \]
and thus (using a combination of H\"older's inequality and the triangle inequality for integrals)
\begin{align*} \|\tilde{U}_0^N-\tilde{u}_0\|_{L^1(\varphi_N)} &\le \int_0^1 \|e^{(1-s)i\tilde{T}^N_0}\|\cdot \|\tilde{T}^N_0-\tilde{\vartheta}_0\|_{L^1(\varphi_N)}\cdot \|e^{si\tilde{\vartheta}_0}\|\,ds \\
&=  \|\tilde{T}^N_0-\tilde{\vartheta}_0\|_{L^1(\varphi_N)}
\end{align*}
because the two operator norms in the integral are identically $=1$ since those operators are unitary.  No, by \eqref{eq.tilde.1}, for $0\le j\le r$,
\begin{align*} \|\tilde{T}_j^N-\tilde{\vartheta}_j^N\|_{L^1(\varphi)} &= \|Q_{\nu_j^N}(x_j)-Q_{\nu_j}(x_j)\|_{L^1(\varphi_N)} \\
&= \int_{\R} |Q_{\nu_j^N}(t)-Q_{\nu_j}(t)| \mu_{x_j}(dt) \\
&= \int_0^1 |Q_{\nu_j^N}(t)-Q_{\nu_j}(t)|\,dt
= \mathcal{W}_1(\nu_j^N,\nu_j) \le \Upsilon_j N^{-\delta_j}
\end{align*}
where the last inequality is by Assumption \ref{assump.SV3}.  Thus, setting $\Upsilon = \max\{\Upsilon_0,\ldots,\Upsilon_r\}$ and $\delta = \min\{\delta_0,\ldots,\delta_r,1\}$, \eqref{eq.Felix.L1.2} shows that
\[ \eqref{Felix.2} \le C(P)K^{\mathrm{deg}(P)-1}\Upsilon\cdot \frac{1}{\eta^2}\cdot N^{-\delta} = \overline{C}_2 \frac{1}{N^\delta \eta^{3/2}}\cdot \eta^{-1/2} \]
where we've define $\overline{C}_2:= C(P)K^{\mathrm{deg}(P)-1}\Upsilon$.  Thus, if $N^{-2\delta/3}\le \eta\le 1$, then $N^\delta \eta^{3/2} \ge N^\delta N^{-\delta}=1$; thus, the lemma holds true with $c_{1,2}=\frac23\delta$.
\end{proof}

Lemmas \ref{lem.Felix.1} and \ref{lem.Felix.2}, together with estimates \ref{Felix.1}, \ref{Felix.2}, and \ref{Felix.3} prove Proposition \ref{prop.Wegner.1}, taking $\overline{C}' = \max\{\overline{C}_1,\overline{C}_2\}$ and $c_1 = \min\{\frac{2}{13},c_{1,2}\}$.

Now, let $\mathbf{a}=(a_1,\cdots,a_k)$ be freely independent $\mathscr{R}$-diagonal elements in a $W^\ast$-probability space where, for $1\le j\le k$, $\sqrt{a_j^\ast a_j}$ has distribution $\nu_j$.  The special case of Proposition \ref{prop.Wegner.1}, letting $\mathbf{u} = (u_1,\ldots,u_k,v_1,\ldots,v_k)$ and realizing $a_j = u_j t_j v_j^\ast$, shows that for $N^{-c_1}\le \eta\le 1$,
\begin{equation}
\label{eq.Wegner.partial.2}  \left|\mathrm{Im}\,\mathbb{E}G_{P^N}(i\eta)\right| \le \overline{C}\eta^{-\frac12} + \left|\varphi[(P(\mathbf{a})-i\eta)^{-1}]\right|. 
\end{equation}
Thus, to complete the proof of Theorem \ref{thm.Wegner}, it suffices to show that the second term in \eqref{eq.Wegner.partial.2} is $\lesssim \eta^{-c_2}$ for some $c_2\in(0,1)$, under any one of the assumptions \ref{assump.R}, \ref{assump.G}, or \ref{assump.U}.  We begin with the last two.

\begin{proposition} \label{prop.ShlyakhSkouf} Under either assumption \ref{assump.G} or \ref{assump.U}, there are constants $\overline{C}_3>0$ and $c_3\in(0,1)$ such that, for all $\eta\in(0,1]$,
\begin{equation} \label{eq.prop.GU} \left|\varphi\left[(P(\mathbf{a})-i\eta)^{-1}\right]\right| \le \overline{C}_3\, \eta^{-c_3}. \end{equation}
\end{proposition}

\begin{proof} A collection of independent $\mathrm{GUE}$s converges in $\ast$-distribution to freely independent semicircular variables; a collection of independent Haar Unitary ensembles converges in $\ast$-distribution to freely independent Haar unitaries.  Therefore, Under Assumptions \ref{assump.G} or \ref{assump.U}, each $a_j$ is a polynomial in free semicirculars or in free Haar unitaries, which means that $P(\mathbf{a})$ is a selfadjoint noncommutative polynomial in either free semicirculars or free Haar unitaries.

The result now follows from \cite[Theorem 1.1]{ShlySkou}: the Cauchy transform $G_{P(\mathbf{a})}(z) = \varphi[(P(\mathbf{a})-z)^{-1}]$ is an algebraic function of $z\in\mathbb{C}$.  Hence, it is meromorphic, with at most a finite set of points $\zeta$ (necessarily in $\mathbb{R}$) where it is not analytic.  At those points, $G_{P(\mathbf{a})}(z-\zeta)$ has power series expansion in $\zeta^{1/q_\zeta}$ (defined by the usual branch cut along the positive real axis) for some $q_\zeta\in\mathbb{N}$.

Thus, if $0$ is one of the non-analytic points, then \eqref{eq.prop.GU} holds with $c_3 = 1/q_0$; otherwise, if $0$ is a point of analyticity, then $G_{P(\mathbf{a})}$ is analytic and hence bounded in a neighborhood of $0$, so \eqref{eq.prop.GU} holds for any $c_3>0$.
\end{proof}

\begin{proposition} \label{prop.BannaMai} Under assumption \ref{assump.R}, there are constants $\overline{C}_4>0$ and $c_4\in(0,1)$ such that, for all $\eta\in(0,1]$,
\begin{equation} \label{eq.prop.R} \left|\varphi\left[(P(\mathbf{a})-i\eta)^{-1}\right]\right| \le \overline{C}_4 \eta^{-c_4}. \end{equation}
\end{proposition}

\begin{proof} Here we use free Fisher information $\Phi^\ast$ introduced by Voiculescu in \cite{Voiculescu1993}. If $\nu_j$ has a density $\rho_j$, Nica, Shlyakhtenko, and Speicher proved in \cite[Theorem 1.1]{NShS} that
\[ \Phi^\ast(a_j,a_j^\ast) = \frac23\int_0^\infty t\rho(t)^3\,dt. \]
Hence, under assumption \ref{assump.R}, i.e.\ under the assumption that each $\nu_j$ satisfies \eqref{assump.SV4}, we have $\Phi^\ast(a_j,a_j^\ast)<\infty$ for $1\le j\le k$.

It follows from \cite[Corollary 6.8]{Voiculescu1993} that $\Phi^\ast(a_j,a_j^\ast) = \Phi^\ast(\mathrm{Re} (a_j),\mathrm{Im}(a_j))$.  (The corollary states that $\Phi^\ast$ is invariant under orthogonal transformations of selfadjoint variables, but the proof applies to unitary transformations between not-necessarily-selfadjoint variables without change.) It is also shown in \cite{Voiculescu1993} that $\Phi^\ast$ is additive over free variables.  Thus, since $a_1,\ldots,a_k$ are freely independent, we have
\[ \Phi^\ast(\mathrm{Re}(a_1),\mathrm{Im}(a_1),\ldots,\mathrm{Re}(a_k),\mathrm{Im}(a_k)) < \infty \]
under assumption \ref{assump.SV4}.

$P(\mathbf{a})$ is a selfadjoint noncommutative polynomial in the selfadjoint operators $\mathrm{Re}(a_1),\mathrm{Im}(a_1),\ldots,\mathrm{Re}(a_k),\mathrm{Im}(a_k)$, it follows directly from \cite[Theorem 1.1]{BannaMai} that the cumulative distribution function $F_{P(\mathbf{a})}$ is H\"older continuous $C^\alpha$ for some $\alpha\in(0,1]$; Wlog we may assume $\alpha<1$ (since $C^\beta\subset C^\alpha$ if $\alpha<\beta$).  For readability, let $X$ be a positive random variable with the same distirbution as $P(\mathbf{a})$.  Thus $F_X(t) = \mathbf{P}\{X\le t\}$ is H\"older $C^\alpha$ continous, and since $X\ge 0$, it follows that, for some constant $\overline{C}_5>0$,
\begin{equation} \label{eq.Holder.cont} F_X(x) \le \overline{C}_5\, x^\alpha, \qquad x\ge 0. \end{equation}
We use \eqref{eq.Holder} to control the behavior of the Cauchy transform $G_X(i\eta)$ for small $\eta>0$.  Note that
\[ \mathrm{Im}\, G_X(i\eta) = -\mathrm{Im} \int_0^\infty \frac{\mu_X(dx)}{x-i\eta} = -\eta \int_0^\infty \frac{\mu_X(dx)}{x^2+\eta^2}. \]
We now use the layercake representation $\mathbb{E}[f(X)] =\int_0^\infty \mathbb{P}\{f(X)\ge x\}\,dx$ to express this in terms of $F_X$.  Taking $f(t) = (t^2+\eta^2)^{-1}$ gives
\begin{align*} \int_0^\infty \frac{\mu_X(dx)}{x^2+\eta^2} = \int_0^\infty \mathbb{P}\{(X^2+\eta^2)^{-1}\ge x\}\,dx &= \int_0^\infty \mathbb{P}\{X^2+\eta^2 \le 1/x\}\,dx \\
&= \int_0^\infty \mathbb{P}\{X^2+\eta^2\le y\}\,\frac{dy}{y^2}.
\end{align*}
Since $X^2\ge 0$, $\mathbb{P}\{X^2+\eta^2\le y\}=0$ if $y<\eta^2$; hence the last integral is equal to
\begin{align*} \int_{\eta^2}^\infty \mathbb{P}\{X^2+\eta^2\le y\}\,\frac{dy}{y^2} &= \int_0^\infty \mathbb{P}\{X^2\le y\}\,\frac{dy'}{(y'+\eta^2)^2} \\
&= \int_0^\infty \mathbb{P}\{X^2\le x^2\} \frac{2x\,dx}{(x^2+\eta^2)^2} \\
&= \int_0^\infty F_X(x)\frac{2x\,dx}{(x^2+\eta^2)^2},
\end{align*}
the last equality following from the positivity of $X$, so $\{X^2\le x^2\} = \{X\le x\}$ for $x\ge 0$.  Combining, this means that
\begin{equation} \label{eq.Holder.calc.1} -\mathrm{Im}\,G_X(i\eta) = \int_0^\infty F_X(x) \frac{2\eta x}{(x^2+\eta^2)^2}\,dx \le \int_0^\infty F_X(x)\frac{dx}{x^2+\eta^2}
\end{equation}
using the fact that $2\eta x \le x^2+\eta^2$.  We break up this integral into two pieces, wlog at $x=1$, and then use \eqref{eq.Holder}:
\begin{align*} |\mathrm{Im}\, G_X(i\eta)| \le \int_0^1 F_X(x)\frac{dx}{x^2+\eta^2} + \int_1^\infty F_X(x)\frac{dx}{x^2+\eta^2}
&\le \int_0^1 \frac{\overline{C}_5\,x^\alpha}{x^2+\eta^2}\,dx + 1
\end{align*}
where the last inequality follows from $F_X(x)\le 1$ and $\frac{1}{x^2+\eta^2}\le \frac{1}{x^2}$.  For the remaining integral, if we may the change of variables $x=\eta y$ we have
\[ \int \frac{x^\alpha}{x^2+\eta^2}\,dx = \int_0^{1/\eta}\frac{\eta^\alpha y^\alpha}{\eta^2 y^2+\eta^2}\,\eta\, dy = \eta^{\alpha-1}\int_0^{1/\eta} \frac{y^\alpha}{y^2+1}\,dy \]
Since $\alpha\in(0,1)$,
\begin{equation} \label{eq.Holder.calc.2} \int_0^{1/\eta} \frac{y^\alpha}{y^2+1}\,dy \le \int_0^{\infty} \frac{y^\alpha}{y^2+1}\,dy =I_\alpha < \infty. \end{equation}
Hence, we have shown that
\[ |\mathrm{Im}\, G_X(i\eta)| \le \eta^{\alpha-1} \cdot I_\alpha\cdot \overline{C}_5+1 \]
Thus, the proposition holds true, taking for example $c_4 = 1-\alpha\in(0,1)$ and $\overline{C}_4 = 2\max\{1,I_\alpha\cdot \overline{C}_5\}$.
\end{proof}

\begin{remark} If the cumulative distribution function is actually Lipschitz (a case we excluded by wlog downgrading in the proof), we could explicitly evaluate \eqref{eq.Holder.calc.2} without bluntly estimating $1/\eta<\infty$ in the limits of integration; we would then bound this integral by $\log\eta$, and the result would be that we could take $c_4$ to be any constant in $(0,1)$.
\end{remark}

\begin{myproof}{\it Theorem}{\ref{thm.Wegner}}  Combining Propositions \ref{prop.Wegner.1}, \ref{prop.ShlyakhSkouf}, and \ref{prop.BannaMai} shows that, for $N^{-c_1}\le\eta\le 1$,
\begin{align*} \left|\mathrm{Im}\,\mathbb{E}G_{P^N}(i\eta)\right| &\le \overline{C}_0\,\eta^{-\frac12} + \left|\varphi[(P(\mathbf{s},\mathbf{u})-i\eta)^{-1}]\right| \\
&\le \overline{C}_0\,\eta^{-\frac12} + \overline{C}_6\eta^{-c_6}
\end{align*}
where $\overline{C}_6 = \max\{\overline{C}_3,\overline{C}_4\}$ and $c_6 = \min\{c_3,c_4\}$.  Thus, the theorem holds true with $\overline{C} = \max\{\overline{C}_0,\overline{C}_6\}$ and $c_2=\min\{\frac12,c_6\}$.
\end{myproof}

\subsection{Proof of Theorem \ref{thm.conv.ESD}} \hfill

\medskip

To prove that the $\mathrm{ESD}$ of $U_0^NB_k^N(t)$ (i.e.\ its Brown measure) converges to the Brown measure of its $\ast$-distribution limit $u_0b_k(t)$, we prove that the log potential $L_{U_0^NB_k^N(t)}(z)$ converges pointwise to the log potential $L_{u_0b_k(t)}(z)$ for $z\in\mathbb{C}$; cf.\ Definition \ref{def.Brown.measure}.  In fact we need that the Laplacian $\nabla^2_z L_{U_0^NB^N_k(t)}(z) \to \nabla^2_z L_{u_0b_k(t)}$, which requires slightly stronger than a.e.\ pointwise convergence; but the necessary upgrade is relatively easy and left for the end.  Our approach closely follows \cite[Section 8]{Cook}, which summarizes and rigorizes the original Hermitization method of Girko \cite{Girko}.

For any $a$ in a $W^\ast$-probability space, we denote by $\nu_a$ the distribution of $|a|=\sqrt{a^\ast a}$.  It is convenient to work with $\upsilon_a^z$, the push-forward of $\nu_{a-z}$ under the square map $x\mapsto x^2$; i.e. $\upsilon_a^z$ is the distribution of $|a-z|^2 = (a-z)^\ast(a-z)$.  Then
\[ L_a(z) = \frac{1}{4\pi} \int_0^\infty \log x\,\upsilon_a^z(dx) \]
cf.\ Definition \ref{def.Brown.measure}.  Our goal is to show that
\begin{equation} \label{eq.log.pot.conv}\int_0^\infty \log x\,\upsilon^z_N(dx) \to_{\mathbb{P}} \int_0^\infty \log x\,\upsilon^z(dx) \quad \text{as }N\to\infty \end{equation}
for (Lebesgue) almost every $z\in\mathbb{C}$, where
\begin{equation} \label{eq.def.upsilon} \upsilon^z_N = \upsilon^z_{U_0^NB^N_k(t)} \quad \upsilon^z = \upsilon^z_{u_0b_k(t)}. \end{equation}

We break the proof into several lemmas.

\begin{lemma} \label{lem.Bkt.bounded} For $t>0$, $k\in\mathbb{N}$, $R>0$, and $z\in\mathbb{C}$ with $|z|\le R$, there is a constant $M_t(k,R)<\infty$ so that
\[ \limsup_{N\to\infty} \max\{\|U_0^NB_k^N(t)-zI\|,\|u_0b_k(t)-z\|\} \le M_t(k,R) \quad \text{a.s.} \]
In particular, $\upsilon^z_N$ converges weakly a.s.\ to $\upsilon^z$ for each $z\in\mathbb{C}$.
\end{lemma}

\begin{proof} Since $A_j^N = U_j^N T_j^N V_j^{N\ast}$, $\|A_j^N\| = \|T_j^N\| = \sigma_{\max}(T_j^N)$.  Therefore, by Assumption \ref{assump.SV1}, there are constants $M_1,\ldots,M_k<\infty$ so that $\limsup_{N\to\infty}\|A_j^N\|\le M_j$ a.s.
Thus
\begin{align*} &\limsup_{N\to\infty}\|U_0^N B_k^N(t)-zI\| \\
=& \limsup_{N\to\infty}\|U_0^N(I+(t/k)^{1/2}A^N_1)\cdots(I+(t/k)^{1/2}A^N_k)\|+|z| \\
\le& \prod_{j=1}^k \left(1+(t/k)^{1/2}M_j\right) +R =: M_t(k,R) \quad \text{a.s.}
\end{align*}
Now, $B^N := U_0^N B_k^N(t)-zI$ converges in $\ast$-distribution to $b:= u_0b_k(t)-z$ (cf.\ Lemma \ref{lem.Bk->bk}), and so in particular $B^{N\ast}B^N$ converges in $\ast$-distribution to $b^\ast b$.  Fix $\epsilon>0$.  The above estimate shows that the distribution $\upsilon^z_N$ of $B^{N\ast}B^N$ is supported in $[0,M_t(k,z)^2+\epsilon)$ for all large $N$.  In particular, this uniform (in $N$) compact support means that the sequence $\{\upsilon^N_z\}_{N\in\mathbb{N}}$ is tight; hence, the convergence of moments of $\upsilon_N^z$ to $\upsilon^z$ (the distribution of $b^\ast b$) implies weak convergence $\upsilon_N^z\rightharpoonup\upsilon^z$.

Thus, if $f\in C_c(\mathbb{R})$ has support in $[M_t(k,|z|)^2+\epsilon,\infty)$, then
\[ \int_{\mathbb{R}} f\,d\upsilon^z_N = 0 \quad \text{for all large }N \]
and therefore
\[\int_{\mathbb{R}} f\,d\upsilon^z  = \lim_{N\to\infty}\int_{\mathbb{R}} f\,d\upsilon^z_N(dx) = 0. \]
This shows that the support of $\upsilon^z$ is contained in $[0,M_t(k,|z|)^2+\epsilon]$.  As this holds for all $\epsilon>0$, we see that $\|u_0 b_k(t)-z\|^2 = \max\mathrm{supp}\,\upsilon^z \le M_t(k,|z|)^2$, which concludes the proof.
\end{proof}

\begin{remark} \label{rk.bulk.vs.outliers} The above proof does not imply, and it is generally not true, that the norm $\|X^N\|$ of a selfadjoint matrix must converge to the norm $\|x\|$ of its $\ast$-distribution limit --- even if $\limsup_{N\to\infty}\|X^N\|<\infty$.  For example, suppose the largest eigenvalue $\lambda_{\max}(X^N) \ge 2$ for all $N$ but all other eigenvalues are $\le 1$ for all $N$; then mass that the $\mathrm{ESD}$ $\mu_{X^N}$ assigns to the interval $[1,\infty)$ is $\le 1/N$, and hence its weak limit is supported in $(-\infty,1]$.  The $\ast$-distribution only detects the {\em bulk} distribution; it cannot detect {\em outliers}.
\end{remark}

\begin{corollary} \label{cor.eig.conv.1} For any $\epsilon>0$,
\begin{equation} \label{eq.log.pot.conv.1}\int_\epsilon^\infty \log x\,\upsilon^z_N(dx) \to \int_\epsilon^\infty \log x\,\upsilon^z(dx) \quad \text{a.s.} \end{equation}
\end{corollary}

\begin{proof} Lemma \ref{lem.Bkt.bounded} shows that $\upsilon_N^z$ converges weakly a.s. to $\upsilon^z$ for each $z\in\mathbb{C}$, and that the support of all the measures $\upsilon^z_N,\upsilon^z$ is contained in a uniform (in $N$) compact interval $[0,M_t(k,|z|)^2]$.  For any fixed $\epsilon>0$, the $\log$ function is continuous and bounded on $[\epsilon,M_t(k,|z|)^2]$.  Hence
\begin{align*} \int_\epsilon^\infty \log x\,\upsilon^z_{N}(dx) &=  \int_\epsilon^{M_t(k,|z|)} \log x\,\upsilon^z_{N}(dx) \\
&\to \int_\epsilon^{M_t(k,|z|)} \log x\,\upsilon^z(dx) = \int_\epsilon^\infty \log x\,\upsilon^z(dx) \quad \text{a.s.} \end{align*}
\end{proof}

From \eqref{eq.log.pot.conv.1}, we see that to prove the desired \eqref{eq.log.pot.conv} it suffices to show that the integrals of the $\log$ function against all the measures $\upsilon_N^z$ and $\upsilon^z$ over $[0,\epsilon]$ are uniformly (in $N$) negligible as $\epsilon\downarrow 0$, at least for a.e.\ $z\in\mathbb{C}$.  For any one log potential, this is generally true thanks to them being $L^1_{\mathrm{loc}}$.

\begin{lemma} \label{lem.log.pot.finite} For any operator $b$ in a $W^\ast$-probability space,
\[ \limsup_{\epsilon\downarrow 0}\int_0^\epsilon |\log x|\,\upsilon_b^z(dx) = 0 \quad \text{for Lebesgue a.e. }z\in\mathbb{C}. \]
\end{lemma}

\begin{proof} From Corollary \ref{cor.La.unif.L1.loc}, we have for any $R\ge 1$
\[\int_{|z|\le R}|L_b(z)|\,d^2z \le \Lambda_1(R+\|b\|)<\infty.\]
Now,
\[ |L_b(z)| = \frac{1}{4\pi} \int_0^1 |\log x|\,\upsilon_b^z(dx) + \frac{1}{4\pi}\int_1^\infty\log x\,\upsilon_b^z(dx). \]
Both terms are $\ge 0$, so we conclude that
\[ \int_{|z|\le R} \int_0^1 |\log x|\,\upsilon^z_b(dx)\,d^2 z \le \Lambda_1(R+\|b\|)<\infty. \]
It follows that for Lebesgue a.e.\ $z$ with $|z|\le R$, $\upsilon_b^z(\{0\}) = 0$ and
\[\int_0^1 |\log x| \upsilon_b^z(dx)<\infty.\]
Fix a full measure set $B_R\subset \{z\in\mathbb{C}\colon |z|\le R\}$ (i.e.\ $\mathrm{Leb}(B_R) = \pi R^2$) so that, for all $z\in B_R$, $\log \in L^1(\upsilon^z_b)$.  For $z\in B_R$ it follows that
\[\lim_{\varepsilon\downarrow 0}\int_0^{\varepsilon}|\log x|\upsilon_b^z(dx) = \lim_{\epsilon\downarrow 0} \int \mathbbm{1}_{[0,\epsilon]}|\log|\,d\upsilon^z_b =0\]
by the dominated convergence theorem.  The lemma is then fully established by constructing the full measure set of $z\in\mathbb{C}$ as $B_1\sqcup (B_2\setminus B_1)\sqcup\cdots\sqcup (B_n\setminus B_{n-1})\sqcup\cdots$.
\end{proof}

This leaves us to prove that $\int_0^\epsilon |\log| \,d\upsilon^z_N$ is negligible (in probability) uniformly in $N$ for small $\epsilon>0$.  To that end, we prove the following, which uses Theorems \ref{thm.pseudospectrum} and \ref{thm.Wegner}.

\begin{lemma} \label{lem.CGH} For each $N\in\mathbb{N}$, there is an event $\mathscr{G}_N$ satisfying $\lim_{N\to\infty} \mathbb{P}\{\mathscr{G}_N\}= 1$ such that, for all $z\in\mathbb{C}$,
\begin{equation} \label{eq.CGH}
\lim_{\epsilon\downarrow 0}\limsup_{N\to\infty} \mathbb{E}\left[\mathbbm{1}_{\mathscr{G}_N}\int_0^\epsilon |\log x|\,\upsilon_N^z(dx)\right] = 0.
\end{equation}
\end{lemma}

\begin{proof} Let $\overline{\alpha},\overline{\beta},C(t,k)$ denote the constants from Theorem \ref{thm.pseudospectrum}.  Fix some $\rho>k\overline{\beta}/\overline{\alpha}$.   Define
\begin{equation} \label{eq.GN} \mathscr{G}_N:=\{\sigma_{\min}(U_0^N B_k^N(t)-zI)> N^{-\rho}\}. \end{equation}
By Theorem \ref{thm.pseudospectrum}, $1-\mathbb{P}\{\mathscr{G}_N\} \le C(t,k) N^{-\rho\overline{\alpha}/k}N^{\overline{\beta}} \to 0$ as $N\to\infty$.  Now, in the event $\mathscr{G}_N$, the support of $\upsilon_N^z$ is contained in $[N^{-\rho},\infty)$ for all $N\in\mathbb{N}$ and all $z\in\mathbb{C}$.  Hence
\begin{equation} \label{eq.CGH.1} \mathbbm{1}_{\mathscr{G}_N}\int_0^\epsilon |\log x|\,\upsilon^z_N(dx) = \mathbbm{1}_{\mathscr{G}_N}\int_{N^{-\rho}}^\epsilon |\log x|\,\upsilon^z_N(dx).
\end{equation}
Hence, to prove \eqref{eq.CGH}, it suffices to prove that
\begin{equation} \label{eq.CGH.rho}
\lim_{\epsilon\downarrow 0}\limsup_{N\to\infty} \mathbb{E}\left[\mathbbm{1}_{\mathscr{G}_N}\int_{N^{-\rho}}^\epsilon |\log x|\,\upsilon_N^z(dx)\right] = 0.
\end{equation}

To prove \eqref{eq.CGH.rho}, we introduce the {\em averaged} empirical measure $\overline{\upsilon}^z_N = \mathbb{E}[\upsilon^z_N]$; i.e.\ it is the unique probability measure on $[0,\infty)$ with the property that
\[ \int f\,d\overline{\upsilon}^z_N = \mathbb{E}\left[\int f\,d\upsilon^z_N\right] \quad \forall  f\in C_b([0,\infty)). \]
In particular, given that $y\mapsto (\zeta-y)^{-1}$ is a continuous bounded function of $y$ for any $\zeta\in\mathbb{C}^+$, it follows that the Cauchy transform satisfies $G_{\overline{\upsilon}^z_N}(\zeta) = \mathbb{E}[G_{\upsilon^z_N}(\zeta)]$.

For given $\eta\ge 0$, $f_x(y)= \frac{\eta^2}{y^2+\eta^2}$ is in $C_b([0,\infty))$; moreover, $f_\eta(y)\ge\frac12$ for $y\in[0,\eta]$.  Therefore
\begin{equation} \label{eq.CDF.1} \int_0^\eta \frac{\eta^2}{y^2+\eta^2}\,\overline{\upsilon}^z_N(dy) \ge\int_0^\eta \frac12 d\overline{\upsilon}^z_N = \frac12\overline{\upsilon}^z_N([0,\eta]). \end{equation}
On the other hand, note that
\[ \mathrm{Im}\, G_{\overline{\upsilon}^z_N}(i\eta) = \mathrm{Im} \int_{\mathbb{R}} \frac{1}{i\eta-y}\,\overline{\upsilon}^z_N(dy) = \int_{\mathbb{R}} \frac{-\eta}{y^2+\eta^2}\,\overline{\upsilon}^z_N(dy) \]
and combining this with \eqref{eq.CDF.1} yields
\begin{equation} \label{eq.CDF.2} \overline{\upsilon}^z_N([0,\eta]) \le \left|-2\eta \mathrm{Im}\,G_{\overline{\upsilon}^z_N}(i\eta)\right| = 2\eta\,\mathbb{E}\!\left[\left|\mathrm{Im}\,G_{\upsilon^z_N}(i\eta)\right|\right]. \end{equation}

Now, $\upsilon^z_N$ is the spectral measure of $P^N = (U_0^N B_k^N(t)-zI)^\ast(U_0^N B_k^N(t)-zI)$, which is a selfadjoint polynomial in $U_0^N$ and $A_j^N = U_j^N T_j^N V_j^{N\ast}$.  Thus, by Theorem \ref{thm.Wegner}, there are constants $\overline{C},c_1>0$ and $c_2\in(0,1)$ so that, for $\eta\ge N^{-c_1}$
\begin{equation} \label{eq.CDF.3} \overline{\upsilon}^z_N([0,\eta]) \le 2\eta\,\mathbb{E}\!\left[\left|\mathrm{Im}\,G_{P^N}(i\eta)\right|\right] 
 \le 2\eta\cdot \overline{C} \eta^{-c_2} = 2\overline{C} \eta^{1-c_2}. \end{equation}
This inequality for the cumulative distribution function of $\overline{\upsilon}^z_N$ can now be used to estimate the integral in \eqref{eq.CGH.rho} as follows.  Fix $\varsigma>0$ with $\varsigma<\min\{\rho,c_1\}$, so that $N^{-\varsigma} > N^{-c_1}$ and $N^{-\varsigma}>N^{-\rho}$.  Then
\[ \int_{N^{-\rho}}^\epsilon |\log x|\,\upsilon^z_N(dx) = \int_{N^{-\rho}}^{N^{-\varsigma}} |\log x|\,\upsilon^z_N(dx) + \int_{N^{-\varsigma}}^\epsilon |\log x|\,\upsilon^z_N(dx) \]
and hence
\begin{align} \nonumber \mathbb{E}\left[\mathbbm{1}_{\mathscr{G}_N}\int_{N^{-\rho}}^\epsilon |\log x|\,\upsilon^z_N(dx)\right] &\le \mathbb{E}\left[\int_{N^{-\rho}}^\epsilon |\log x|\,\upsilon^z_N(dx)\right] \\
&= \int_{N^{-\rho}}^{N^{-\varsigma}} |\log x|\,\overline{\upsilon}^z_N(dx) + \int_{N^{-\varsigma}}^\epsilon |\log x|\,\overline{\upsilon}^z_N(dx). \label{eq.CDF.2-term}
\end{align}
For the first term in \eqref{eq.CDF.2-term}, we make the blunt estimate
\begin{align} \nonumber 
\int_{N^{-\rho}}^{N^{-\varsigma}} |\log x|\,\overline{\upsilon}^z_N(dx) &\le \overline{\upsilon}^z_N([N^{-\rho},N^{-\varsigma}])\cdot (|\log N^{-\varsigma}|-|\log N^{-\rho}|) \\ \label{eq.CDF.2-term.1}
&\le \overline{\upsilon}^z_N([0,N^{-\varsigma}])|\log N^{-\varsigma}| \le 2\overline{C} N^{-\varsigma(1-c_2)}\cdot \varsigma \log N
\end{align}
where the final inequality comes from \eqref{eq.CDF.3}, which holds at $\eta=N^{-\varsigma}$ because $\eta=N^{-\varsigma}>N^{-c_1}$.

For the second term in \eqref{eq.CDF.2-term}, we use integration by parts.  Letting $F_N^z(x)=\overline{\upsilon}_N^x([0,x])$ denote the cumulative distribution function, from Riemann--Stieltjes integration theory we have
\begin{align*}
\int_{N^{-\varsigma}}^\epsilon |\log x|\,\overline{\upsilon}^z_N(dx)
&=-\int_{N^{-\varsigma}}^{\epsilon} \log x\,dF^z_N(x) \\
&=-\log\epsilon\, F^z_N(\epsilon) + \log (N^{-\varsigma}) F^z_N(N^{-\varsigma}) + \int_{N^{-\varsigma}}^\epsilon \frac{1}{x} F^z_N(x)\,dx \\
&\le -\log\epsilon\, F^z_N(\epsilon) + \int_{N^{-\varsigma}}^\epsilon \frac{1}{x} F^z_N(x)\,dx
\end{align*}
where the inequality is from dropping the middle term which is negative.  We may assume that $\epsilon<1$, hence $-\log\epsilon>0$, and so in both terms above we may use \eqref{eq.CDF.3} to upper bound $F_N^z(x) = \overline{\upsilon}_N^z([0,x])$:
\begin{align} \nonumber
\int_{N^{-\varsigma}}^\epsilon |\log x|\,\overline{\upsilon}^z_N(dx) &\le
-\log\epsilon\, F^z_N(\epsilon) + \int_{N^{-\varsigma}}^\epsilon \frac{1}{x} F^z_N(x)\,dx \\ \nonumber
&\le -\log\epsilon\,  2\overline{C}\epsilon^{1-c_2} + \int_{N^{-\varsigma}}^\epsilon \frac{1}{x}\cdot 2\overline{C} x^{1-c_2}\,dx \\ \label{eq.CDF.2-term.2}
&\le -2\overline{C}\epsilon^{1-c_2}\log\epsilon + \frac{2\overline{C}}{1-c_2}\epsilon^{1-c_2}
\end{align}
where the final inequality comes from bounding the integral over $[N^{-\varsigma},\epsilon]$ by the integral over $[0,\epsilon]$.

Combining \eqref{eq.CDF.2-term} with \eqref{eq.CDF.2-term.1} and \eqref{eq.CDF.2-term.2} we have
\begin{align*}  &\mathbb{E}\left[\mathbbm{1}_{\mathscr{G}_N}\int_{N^{-\rho}}^\epsilon |\log x|\,\upsilon^z_N(dx)\right] \\
\le &2\varsigma\overline{C} N^{-\varsigma(1-c_2)}\log N -2\overline{C}\epsilon^{1-c_2}\log\epsilon + \frac{2\overline{C}}{1-c_2}\epsilon^{1-c_2}.
\end{align*}
Because $\varsigma>0$ and $c_2<1$, the first term tends to $0$ as $N\to\infty$.  The last two terms are independent of $N$, and (again owing to $c_2<1$) tend to $0$ as $\epsilon\downarrow 0$.
\end{proof}

\begin{corollary} \label{cor.conv.prob} For any $\delta,\delta'>0$, there is an $\epsilon\in(0,1)$ so that, for all $z\in\mathbb{C}$,
\begin{equation} \limsup_{N\to\infty}\mathbb{P}\left\{\left|\int_0^\epsilon \log x\,\upsilon^z_N(dx)\right|\ge\delta\right\}\le \delta'.
\end{equation}
\end{corollary}

\begin{proof} For each $N\in\mathbb{N}$, let $\mathscr{G}_N$ be the event defined in \eqref{eq.GN}. My Markov's inequality, for any $\epsilon\in(0,1)$ and $\delta>0$, the conditional probability $\mathbb{P}_{\mathscr{G}_N}$ satisfies
\begin{align*} \mathbb{P}_{\mathscr{G}_N}\left\{\left|\int_0^\epsilon \log x\,\upsilon^z_N(dx)\right|\ge\delta\right\} &\le \frac{1}{\delta}\mathbb{E}_{\mathscr{G}_N}\left[\left|
\int_0^\epsilon \log x\,\upsilon^z_N(dx)\right|\right] \\
&= \frac{1}{\delta\cdot \mathbb{P}(\mathscr{G}_N)}\mathbb{E}\left[\mathbbm{1}_{\mathscr{G}_N}\int_0^\epsilon |\log x|\,\upsilon_N^z(dx)\right].
\end{align*}
Theorem \ref{thm.pseudospectrum} asserts that $\lim_{N\to\infty}\mathbb{P}(\mathscr{G}_N)=1$, so let $N_0$ be large enough that, for $N\ge N_0$, $\mathbb{P}(\mathscr{G}_N)>\frac12$.  By Lemma \ref{lem.CGH}, there is $\epsilon>0$ so that, for all sufficiently large $N\ge N_0$, the above expectation is $< \delta\delta'/4$.  It thence follows that, for that $\epsilon$ and sufficiently large $N$,
\begin{equation} \label{eq.conv.P.delta'} \mathbb{P}_{\mathscr{G}_N}\left\{\left|\int_0^\epsilon \log x\,\upsilon^z_N(dx)\right|\ge\delta\right\}\le \frac{\delta'}{2}. \end{equation}
Now, since $\mathbb{P}(\mathscr{G}_N)\to 1$, it follows that $\mathbb{P}_{\mathscr{G}_N}(\mathscr{E})\to \mathbb{P}(\mathscr{E})$ for any event $\mathscr{E}$, and in particular for $\mathscr{E}$ equal to the above event.  Indeed: by the principal of inclusion-exclusion, $\mathbb{P}(\mathscr{G}_N\cap\mathscr{E}) = \mathbb{P}(\mathscr{E})+\mathbb{P}(\mathscr{G}_N)-\mathbb{P}(\mathscr{G}_N\cup\mathscr{E})$, and $1\ge \mathbb{P}(\mathscr{G}_N\cup\mathscr{E})\ge \mathbb{P}(\mathscr{G}_N)\to 1$, so
\[ \mathbb{P}_{\mathscr{G}_N}(\mathscr{E}) = \frac{\mathbb{P}(\mathscr{G}_N\cap E)}{\mathbb{P}(\mathscr{G}_N)} \to \frac{\mathbb{P}(\mathscr{E})+1-1}{1} = \mathbb{P}(\mathscr{E}). \]
Therefore, we can choose $N_1\in\mathbb{N}$ so that, for all $N\ge N_1$, $|\mathbb{P}_{\mathscr{G}_N}(\mathscr{E})-\mathbb{P}(\mathscr{E})|\le \delta'/2$.  Finally, then, for all sufficiently large $N\ge\max\{N_0,N_1\}$, we have
\begin{align*} \mathbb{P}(\mathscr{E}) &= \mathbb{P}(\mathscr{E})-\mathbb{P}_{\mathscr{G}_N}(\mathscr{E}) + \mathbb{P}_{\mathscr{G}_N}(\mathscr{G}) \\
& \le |\mathbb{P}(\mathscr{E})-\mathbb{P}_{\mathscr{G}_N}(\mathscr{E})| + \mathbb{P}_{\mathscr{G}_N}(\mathscr{G}) \le \frac{\delta'}{2}+\frac{\delta'}{2} = \delta'
\end{align*}
This concludes the proof.
\end{proof}

Combining all of the above yield the a.e.\ pointwise convergence of log potentials.  Going forward, denote
\begin{align*} L_N(z) &= \frac{1}{4\pi}\int_0^\infty\,\log x\,\upsilon^z_N(dx) = L_{U_0^N B^N_k(t)}(z) \\
L_\infty(z) &= \frac{1}{4\pi}\int_0^\infty\,\log x\,\upsilon^z(dx) = L_{u_0 b_k(t)}(z).
\end{align*}

\begin{proposition} \label{prop.a.e.convergence} Under Assumption \ref{assump.sing.conv}, for Lebesgue a.e.\ $z\in\mathbb{C}$, \eqref{eq.log.pot.conv} holds: i.e.\ $L_N(z)\to_{\mathbb{P}} L_\infty(z)$ for a.e.\ $z\in\mathbb{C}$.
\end{proposition}

\begin{proof} Using the notation from \eqref{eq.def.upsilon}, we begin by noting that, for any $\epsilon>0$,
\begin{align*} &\left|\int_0^\infty \log x\,\upsilon^z_N(dx) - \int_0^\infty \log x\,\upsilon^z(dx)\right| \\
&\le \left|\int_\epsilon^\infty \log x\,\upsilon^z_N(dx) - \int_\epsilon^\infty \log x\,\upsilon^z(dx)\right| + \int_0^\epsilon |\log x|\,\upsilon^z_N(dx) + \int_0^\epsilon |\log x|\,\upsilon^z(dx).
\end{align*}
Let $\delta>0$.  If each of these three terms is $\le \delta/3$, then the difference of log potentials is $\le \delta$; hence
\begin{align*}
&\mathbb{P}\left\{\left|\int_0^\infty \log x\,\upsilon^z_N(dx) - \int_0^\infty \log x\,\upsilon^z(dx)\right|\ge\delta\right\} \\
\le\;\;&
\mathbb{P}\left\{\left|\int_\epsilon^\infty \log x\,\upsilon^z_N(dx) - \int_\epsilon^\infty \log x\,\upsilon^z(dx)\right|\ge \delta/3\right\} \\
+\;\;&\mathbb{P}\left\{\int_0^\epsilon |\log x|\,\upsilon^z_N(dx)\ge\delta/3\right\}  + \mathbb{P}\left\{\int_0^\epsilon |\log x|\,\upsilon^z(dx)\ge\delta/3\right\}.
\end{align*}
For any $\delta'>0$, Lemma \ref{lem.log.pot.finite} shows that, for some $\epsilon_0>0$, the third term is $\le\delta'/2$ for all $z$ in some Lebesgue full measure subset of $\mathbb{C}$; this also holds (for the same full measure set of $z$) for any $\epsilon<\epsilon_0$ since the integral is increasing in $\epsilon$.  By Corollary \ref{cor.conv.prob}, there is some $\epsilon_1\in(0,1)$ for which the $\limsup_{N\to\infty}$ of the second term is $\le \delta'/2$ for any $z\in\mathbb{C}$; again, this also holds true for any $\epsilon<\epsilon_1$ as the integral is increasing in $\epsilon$.  So taking $\epsilon=\min\{\epsilon_0,\epsilon_1\}$, the first term converges to $0$ as $N\to\infty$ for any $z\in\mathbb{C}$ by Corollary \ref{cor.eig.conv.1}.

Hence, we've shown that for any $\delta,\delta'>0$,
\[ \limsup_{N\to\infty} \mathbb{P}\left\{\left|\int_0^\infty \log x\,\upsilon^z_N(dx) - \int_0^\infty \log x\,\upsilon^z(dx)\right|\ge\delta\right\} \le \delta'. \]
It follows that this $\limsup_{N\to\infty}$ is actually equal to $0$, concluding the proof.
\end{proof}

To complete the proof of Theorem \ref{thm.conv.ESD}, that the Brown measures converge in probability, i.e.\ that the (distributional) Laplacians of the log potentials converge, we need to upgrade the log potential convergence from a.e.\ to $L^1_{\mathrm{loc}}$.

\begin{lemma} \label{prop.L1loc} Let $K\subset\mathbb{C}$ be compact.  For $M>0$, define the event
\[ \mathscr{H}_N(K,M):= \left\{\|U_0^N B_k^N(t)-zI\|\le M\text{ for a.e. }z\in K\right\}. \]
Then
\begin{equation} \label{eq.ELN-L} \lim_{N\to\infty}\mathbb{E}\left[\mathbbm{1}_{\mathscr{H}_N(K,M)} \int_{K} |L_N(z)-L_{\infty}(z)|\,d^2z\right] = 0.
\end{equation}
\end{lemma}

\begin{proof} Let $(\Omega,\mathscr{F},\mathbb{P})$ be the underlying probability space.  On the joint finite measure space $(\Omega\times K, \mathscr{F}\otimes\mathscr{B}(K),\mathbb{P}\otimes\mathrm{Leb})$ (where $\mathscr{B}(K)$ is the Borel $\sigma$-field over $K$ and $\mathrm{Leb}$ is Lebesgue measure), let $F_N\colon\Omega\times K\to\mathbb{R}_+$ be the measurable function $F_N(\omega,z) = \mathbbm{1}_{\mathscr{H}_N(K,M)}(\omega)|L_N(z)-L_\infty(z)|$.  By Proposition \ref{prop.a.e.convergence}, for a.e.\ $z\in K$ $F_N(\cdot,z)\to_{\mathbb{P}} 0$; i.e.\ for any $\epsilon>0$, $f_N(z)=\mathbb{P}\{F_N(\cdot,z)\ge\epsilon\}\to 0$ for $\mathrm{Leb}$-a.e.\ $z\in K$ as $N\to\infty$.  The function $f_N$ is bounded, and hence by the bounded convergence theorem for $L^1(K)$, it follows that
\begin{align*} 0 = \lim_{N\to\infty} \int_K f_N(z)\,d^2z &= \int_K \mathbb{P}\{F_N(\cdot,z)\ge\epsilon\}\,d^2z  \\
&= \int_K \int_\Omega \mathbbm{1}_{F_N(\omega,z)\ge\epsilon}\,\mathbb{P}(d\omega)\,d^2 z \\
&= \mathbb{P}\otimes\mathrm{Leb}\{F_N\ge\epsilon\}.
\end{align*}
Hence, $F_N\to_{\mathbb{P}\otimes\mathrm{Leb}} 0$.

Now, note that $F_N \le \mathbbm{1}_{\mathscr{H}_N(K,M)}(|L_N| + |L_\infty|)$, and hence for any $p\ge 1$,
\[ \|F_N\|_{L^p} \le \|\mathbbm{1}_{\mathscr{H}_N(K,M)}L_N\|_{L^p} + \|\mathbbm{1}_{\mathscr{H}_N(K,M)}L_\infty\|_{L^p}. \]
Well,
\begin{align*} \|\mathbbm{1}_{\mathscr{H}_N(K,M)}L_N\|_{L^p}^p &= \mathbb{E}\left[\mathbbm{1}_{\mathscr{H}_N(K,M)}\int_K |L_N|^p\,d^2z\right] \\
\|\mathbbm{1}_{\mathscr{H}_N(K,M)}L_\infty\|_{L^p}^p &= \mathbb{E}\left[\mathbbm{1}_{\mathscr{H}_N(K,M)}\int_K |L_\infty|^p\,d^2z\right].
\end{align*}

In the event $\mathscr{H}_N(K,M)$, $\|U_0^N B_k^N(t)-zI\|\le M$ a.e.  Because $U_0^N B_k^N(t)-zI$ converges in $\ast$-distribution to $u_0 b_k(t)-z$, it then follows that $\|u_0 b_k(t)-z\|\le M$ a.e.\ as well (cf.\ Lemma \ref{lem.Bkt.bounded} and Remark \ref{rk.bulk.vs.outliers}).  Therefore, Corollary \ref{cor.La.unif.L1.loc} implies that
\begin{align*}
\mathbbm{1}_{\mathscr{H}_N(K,M)}\int_K |L_N|^p\,d^2z &\le \mathbbm{1}_{\mathscr{H}_N(K,M)}\Lambda_p(R_K+M) \le \Lambda_p(R_K+M) \\
\mathbbm{1}_{\mathscr{H}_N(K,M)}\int_K |L_\infty|^p\,d^2z &\le \mathbbm{1}_{\mathscr{H}_N(K,M)}\Lambda_p(R_K+M) \le \Lambda_p(R_K+M)
\end{align*}
for any $R_K\ge 1$ for which $K\subseteq\{z\in\mathbb{C}\colon |z|\le R_K\}$.  
Hence, we conclude that $\|F_N\|_{L^p} \le 2\Lambda_p(R_K+M)^{1/p}$, a bound which is uniform in $N$.  Thus, $\{F_N\}_{N\in\mathbb{N}}$ is $L^p$ bounded for every $p\ge 1$.  Applying this with any $p>1$ implies that $\{F_N\}_{N\in\mathbb{N}}$ is uniformly integrable.

Hence: $\{F_N\}_{N\in\mathbb{N}}$ is uniformly integrable and $F_N\to 0$ in measure.  It follows from the Vitali Convergence Theorem \cite[Theorem 17.56]{DriverProbNotes} that $\|F_N\|_{L^1}\to 0$, which is precisely the desired statement \eqref{eq.ELN-L}.
\end{proof}

\begin{corollary} \label{cor.L1loc} For any compact $K\subseteq\mathbb{C}$,
\[ \int_K |L_N(z)-L_\infty(z)|\,d^2z \to_{\mathbb{P}} 0 \quad \text{as}\;N\to\infty. \]\end{corollary}

\begin{proof} By Lemma \ref{lem.Bkt.bounded}, there is a constant $M_K>0$ so that
\[ \mathbb{P}(\mathscr{H}_N(K,M_K)) = \mathbb{P}\left\{\|U_0^N B_k^N(t)-zI\|\le M\text{ for a.e. }z\in K\right\}=1 \]
for all large $N$.  (For example, let $R_K\ge 1$ be large enough that $K\subseteq\{z\in\mathbb{C}\colon |z|\le R_K\}$, and take $M_K = M_t(k,R_K)+1$.)  Let $\mathbb{P}_N = \mathbb{P}(\,\cdot\,|\mathscr{H}_N(K,M_K))$ denote the conditional probability and let $\mathbb{E}_N$ be the expectation with respect to $\mathbb{P}_N$. Markov's inequality yields for any $\epsilon>0$
\begin{align*} &\mathbb{P}_N\left(\int_K |L_N(z)-L_\infty(z)|\,d^2z\ge\epsilon\right) \\
\le &\frac{1}{\epsilon}\mathbb{E}_N\left[\int_K |L_N(z)-L_\infty(z)|\,d^2z\right] \\
=&\frac{1}{\epsilon\mathbb{P}(\mathscr{H}_N(K,M_K))}\mathbb{E}\left[\mathbbm{1}_{\mathscr{H}_N(K,M_K)}\int_K |L_N(z)-L_\infty(z)|\,d^2z\right].
\end{align*}
Since $\mathbb{P}(\mathscr{H}_N(M,K_M))=1$ for large $N$, it follows that $\mathbb{P}_N = \mathbb{P}$ for large $N$.  Hence, we've shown that, for all large $N$,
\[ \mathbb{P}\left(\int_K |L_N(z)-L_\infty(z)|\,d^2z\ge\epsilon\right) \le \frac{1}{\epsilon}\mathbb{E}\left[\mathbbm{1}_{\mathscr{H}_N(K,M_K)}\int_K |L_N(z)-L_\infty(z)|\,d^2z\right] \]
and this tends to $0$ as $N\to\infty$ by Lemma \ref{prop.L1loc}.  \end{proof}

Finally, this gives us all the necessary pieces to prove the convergence of Brown measures.

\begin{myproof}{\it Theorem}{\ref{thm.conv.ESD}}
Let $L_N(z)$ denote the log potential of $U^N_0 B_k^N(t)$ and let $L_{\infty}(z)$ denote the log potential of $u_0 b_k(t)$.  Similarly, denote the $\mathrm{ESD}$ (i.e.\ Brown measure) of $ U^N_0 B_k^N(t)$ by $\mu_N$ and the Brown measure of $u_0 b_k(t)$ by $\mu_\infty$.  Then $\mu_N = \frac{1}{2\pi}\nabla^2 L_N$ and $\mu_{\infty} = \frac{1}{2\pi}\nabla^2 L_\infty$, both in distributional sense.  Our goal is to prove that
\begin{equation} \label{eq.final.Brown.conv.} \int_{\mathbb{C}} f\,d\mu_N \to_{\mathbb{P}} \int_{\mathbb{C}} f\, d\mu_\infty \qquad \forall\,f\in C_c^\infty(\mathbb{C}).
\end{equation}
Note that
\[ \int_{\mathbb{C}} f\,d\mu_N = \int_{\mathbb{C}} f(z)\nabla^2 L_N(z)\,d^2z = \int_{\mathbb{C}} \nabla^2 f(z)\,L_N(z)\,d^2z \]
by definition of the distributional Laplacian.  Hence, \eqref{eq.final.Brown.conv.} is equivalent to showing
\begin{equation} \label{eq.final.Brown.conv.2} \int_{\mathbb{C}} \nabla^2 f(z)\cdot[L_N(z)-L_\infty(z)]\,d^2z \to_{\mathbb{P}} 0 \qquad \forall\,f\in C_c^\infty(\mathbb{C}).
\end{equation}
Since $f\in C_c^\infty(\mathbb{C})$, so is $\nabla^2 f$, supported in the same compact set $\mathrm{supp}\, f$.  Hence
\[ \left|\int_{\mathbb{C}} \nabla^2 f(z)\cdot[L_N(z)-L_\infty(z)]\,d^2z\right| \le \|\nabla^2 f\|_{\infty} \int_{\mathrm{supp}\,f}|L_N(z)-L_\infty(z)|\,d^2 z. \]
Wlog, assume that $\|\nabla^2 f\|_\infty\ne 0$.  Then for any $\epsilon>0$
\begin{align} \label{eq.main.thm.final.1} &\mathbb{P}\left\{\left|\int_{\mathbb{C}} \nabla^2 f(z)\cdot[L_N(z)-L_\infty(z)]\,d^2z\right|\ge\epsilon\right\} \\ \label{eq.main.thm.final.2}
\le &\mathbb{P}\left\{\int_{\mathrm{supp}\,f}|L_N(z)-L_\infty(z)|\,d^2 z \ge \epsilon/\|\nabla^2 f\|_\infty\right\}. \end{align}
Since $f$ is compactly-supported, Corollary \ref{cor.L1loc} shows that \eqref{eq.main.thm.final.2} $\to 0$ as $N\to\infty$; hence, it follows that \eqref{eq.main.thm.final.1} $\to 0$ as $N\to\infty$.  Thus, \eqref{eq.final.Brown.conv.2} is established, concluding the proof.
\end{myproof}

\section{Linearized Matrix Models and Freeness}
Denote by $\MkC$ the $k\times k$ matrices with complex entries. For a $W^\ast$-probability space $(\mathscr{A},\varphi)$, $\varphi\otimes \mathrm{tr}$ is a faithful normal tracial state on $\mathscr{A}\otimes \MkC$. Thus, $(\mathscr{A}\otimes \MkC, \varphi\otimes \mathrm{tr})$ is also a noncommutative probability space.

\begin{remark} The random matrix ensembles in $\MNC$ that are involved in the construction of matrix random walks throughout previous sections are all a.s.\ bounded for large $N$, and so their large-$N$ limits can all be constructed genuinely within a $W^\ast$-algebra $\mathscr{A}$.  Nevertheless, {\em all} of the techniques we use henceforth to analyze the Brown measures of such models work perfectly well for unbounded operators, and so we develop the tools in that level of generality. \end{remark}

Denote by $\widetilde{\mathscr{A}}$ the set of all closed, densely defined (possibly unbounded) operators affiliated with $\mathscr{A}$. We will consider (possibly) unbounded $\ast$-free $\mathscr{R}$-diagonal element $a\in\widetilde{\mathscr{A}}$ with unit variance (i.e.\ the $L^2$-norm $\|a\|_2$ is $1$). Let $a_j\in\widetilde{\mathscr{A}}$ be freely independent elements with the same $\ast$-distribution as $a$. Also let $u_0\in \mathscr{A}$ be unitary $\ast$-freely independent from all $a_j$. We consider the $k$-step random walk with unitary initial condition
\begin{equation}
\label{eq.RW.unitary}
u_0b_k(t) = u_0\left(1+\sqrt{\frac{t}{k}}a_1\right)\cdots \left(1+\sqrt{\frac{t}{k}}a_k\right)
\end{equation}
for $t$ ranging through $[0,\infty)$.  The random walk $u_0b_k(t)$ is an element in $\widetilde{\mathscr{A}}$. When $u_0=1$ and $a_j\in\mathscr{A}$ are bounded,  $b_k(t)$ reduces to the $k$-step random walk defined in \eqref{def.bk}. 

We study the Brown measure of $u_0b_k(t)$ in this section. In Section \ref{sec.linearization}, we first apply a linearization trick to $u_0b_k(t)$. The random walk $u_0b_k(t)$ is a noncommutative polynomial in $u_0, a_1, \ldots,a_k$. The linearization allows us to study $Z+\sqrt{t/k}A$, where
 \begin{equation}
     \label{eq.Z.A.def}
     Z = \sum_{j=1}^k u_0^{1/k}\otimes E_{j,j+1} \quad\textrm{and}\quad A = \sum_{j=1}^k a_j\otimes E_{j,j+1}
 \end{equation}
 are elements in $\widetilde{\mathscr{A}}\otimes \MkC$. The nonzero entries of $Z$ and $A$ are in cyclic order. The operators $Z$ and $A$ are closed, densely defined (possibly) unbounded operators affiliated with $(\mathscr{A}\otimes \MkC,\varphi\otimes\mathrm{tr})$; so is $Z+\sqrt{t/k}A$. Corollary \ref{cor.unitary.linearized} shows that the push-forward of the Brown measure of $Z+\sqrt{t/k}A$ by $\lambda\mapsto \lambda^k$ is the Brown measure of $u_0b_k(t)$.

 The elements $Z$ and $A$ are free with amalgamation over $\MkC$ (\cite[Chapter 9.2]{MingoSpeicherBook}). We want to compute the Brown measure of $Z+\sqrt{t/k}A$ as an operator affiliated with $(\mathscr{A}\otimes \MkC,\varphi\otimes\mathrm{tr})$. For this purpose, we want to show that $Z$ and $A$ are $\ast$-free with respect to the trace $\varphi\otimes\mathrm{tr}$; this $\ast$-freeness result will be proved in Section \ref{sec.linearized.freeness}.

\subsection{A Linearization Trick\label{sec.linearization}} \hfill

%\medskip

\begin{proposition}
    \label{prop.linearization}
    Let $b_1,\ldots,b_k\in\widetilde{\mathscr{A}}$ be $\ast$-freely independent, with $\log^+|b_j|\in L^1(\mathscr{A},\varphi)$ for $1\le j\le k$ (where $\log^+ x = \max\{\log x,0\}$), so that the Brown measure of $b_1\cdots b_k$ is defined.  (See Definition \ref{def.Brown.measure} and Remark \ref{rk.Brown.measure.def}.)
%    \begin{equation}
%        \label{eq.log.Brown.cond}
%        \varphi(\log^+\vert b_j\vert)<\infty,
%    \end{equation}
%    where $\log^+$ is defined by $\log^+(x) = \max\{\log(x),0\}$ for $x\geq 0$. Thus, the Brown measure of any polynomial in $b_1,\ldots,b_k$ is defined (\cite[Definition 2.13]{HaagerupSchultz2007}).
    
    We define the matrix $M\in\widetilde{\mathscr{A}}\otimes \MkC$ by
    \[M = \sum_{j=1}^k b_j\otimes E_{j,j+1}.\]
    Then $M$ is affiliated with $\mathscr{A}\otimes \MkC$ and $\log^+|M|\in L^1(\mathscr{A}\otimes \MkC,\varphi\otimes\mathrm{tr})$,
%    \[(\varphi\otimes\mathrm{tr})(\log^+\vert M\vert)<\infty\]
    so that the Brown measure of $M$ is defined.% in \cite{HaagerupSchultz2007}.
    
    Moreover, the Brown measure of $b_1\cdots b_k$ is the push-forward of the Brown measure of $M$ by the map $\lambda\mapsto \lambda^k$.
\end{proposition}
%\begin{remark}
%    \label{rem.Brown.welldefined}
%    By Remark 2.2 of \cite{HaagerupSchultz2007}, any unbounded operator with finite variance satisfies \eqref{eq.log.Brown.cond}. In particular, the Brown measure of $u_0b_k(t)$ and the Brown measure of $Z+\sqrt{t/k}A$ are defined.
%\end{remark}
\begin{proof}
    The matrix $M^k$ is diagonal, and the $\ast$-distribution of the entries of $M^k$ along the diagonal is the same as $b_1\cdots b_k$. The Brown measure is determined by $\ast$-distribution; hence, the Brown measure of $M^k$ is the same as the Brown measure of $b_1\cdots b_k$.  By \cite[Proposition 2.15]{HaagerupSchultz2007}, the Brown measure of $M^k$ is the push-forward of the Brown measure of $M$ by the map $z\mapsto z^k$.
\end{proof}

\begin{theorem}
    \label{thm.unitary.distribute}
    Let $u_0, a_1,\ldots,a_k\in \widetilde{\mathscr{A}}$ be as in \eqref{eq.RW.unitary}. If $u_0^{1/k}$ is any $k$-th root of $u_0$ (i.e. $(u_0^{1/k})^k = u_0$), then $u_0b_k(t)$ and 
    \[\left(u_0^{1/k}+\sqrt{\frac{t}{k}}a_1\right)\cdots \left(u_0^{1/k}+\sqrt{\frac{t}{k}}a_k\right)\]
    have the same $\ast$-distribution.
\end{theorem}
For convenience, we will later specify $u_0^{1/k}$ to be a $k$-th root of $u$ such that its spectrum is contained in an arc with argument in $(\pi,\pi]$.
\begin{proof}
    We calculate
    \begin{align*}
        &u_0\left(1+\sqrt{\frac{t}{k}}a_1\right)\cdots \left(1+\sqrt{\frac{t}{k}}a_k\right) \\
        &= \left(u_0^{1/k}+\sqrt{\frac{t}{k}}u a_1(u_0^{1/k})^{-(k-1)}\right)(u_0^{1/k})^{k-1}\left(1+\sqrt{\frac{t}{k}}a_2\right)\cdots \left(1+\sqrt{\frac{t}{k}}a_k\right)\\
        &=\prod_{j=1}^{k}\left(u_0^{1/k}+\sqrt{\frac{t}{k}}(u_0^{1/k})^{k-j+1}a_{j}(u_0^{1/k})^{-(k-j)}\right)
    \end{align*}
    Recall that each $\mathscr{R}$-diagonal element $a_j$ has the same $\ast$-distribution as $v_j x_j w_j$ for some $\ast$-free $v_j, x_j, w_j$, where $v_j$, $w_j$ are Haar unitaries and $x_j$ is self-adjoint. Lemma \ref{lem.free.Haar} then shows that $(u_0^{1/k})^{k-j}a_j(u_0^{1/k})^{-(k-j-1)}$ are all $\ast$-free $\mathscr{R}$-diagonal elements, with the same $\ast$-distribution as $a$, all $\ast$-free from $u_0^{1/k}$. The proof is completed.
\end{proof}

\begin{corollary}
    \label{cor.unitary.linearized}
    Denote by $u_0^{1/k}$ a fixed $k$-th root of $u$; that is, $(u_0^{1/k})^k = u$. Let $Z$ and $A$ be affiliated with $\mathscr{A}\otimes \MkC$ as defined in \eqref{eq.Z.A.def}. Then the push-forward of the Brown measure of $Z+\sqrt{t/k}A$ by $\lambda\mapsto \lambda^k$ is the Brown measure of the $k$-step random walk $u_0b_k(t)$, defined in \eqref{eq.RW.unitary}.
\end{corollary}
\begin{proof}
    By Theorem \ref{thm.unitary.distribute}, $u_0b_k(t)$ has the same $\ast$-distribution as 
    \[\left(u_0^{1/k}+\sqrt{\frac{t}{k}}a_1\right)\cdots \left(u_0^{1/k}+\sqrt{\frac{t}{k}}a_k\right).\] 
    The corollary now follows from applying Proposition \ref{prop.linearization} with $b_j = u_0^{1/k}+\sqrt{\frac{t}{k}}a_j$.
\end{proof}

\subsection{Freeness of Circulant Invariant Block Matrices\label{sec.linearized.freeness}} \hfill

\medskip

Let $S_k$ denote the symmetric group of permutations on the set $[k]=\{1,2,\ldots,k\}$.  For each $\sigma\in S_k$, let $U_\sigma$ denote the associated permutation matrix
\[ U_\sigma = \sum_{i=1}^k E_{i,\sigma(i)}. \]

\begin{definition} Let $(\mathscr{A},\varphi)$ be a tracial noncommutative probability space.  For some index set $L$, let $\mathbf{A} = \{a_{i,j}^\ell\}_{i,j\in[k],\ell\in L}$ be random variables in $\mathscr{A}$.  We say $\mathbf{A}$ is {\bf $\sigma$-circulant invariant} if the joint $\ast$-distribution of $\mathbf{A}$ is equal to the joint $\ast$-distribution of $\sigma\cdot\mathbf{A}$, where
\[ \sigma\cdot\mathbf{A} = \{a_{\sigma(i),\sigma(j)}^\ell\}_{i,j\in[k],\ell\in L}. \]
\end{definition}
Equivalently: if we think of $\mathbf{A}$ as a collection $\{A_\ell\}_{\ell\in L}$ of matrix-valued operators in $\mathscr{A}\otimes \MkC$ via $A_\ell = \sum_{i,j=1}^k a^{\ell}_{i,j} \otimes E_{i,j}$, $\sigma$-circulant invariance is equivalent to
\[ \{A_\ell\}_{\ell\in L} \equaldist \{U_\sigma A_\ell U_{\sigma}^\ast\}_{\ell\in L} \]
where we identify $U_\sigma = 1\otimes U_\sigma\in \mathscr{A}\otimes \MkC$.  This is because 
\[ U_\sigma A_\ell U_\sigma^\ast = \sum_{i=1}^k E_{i,\sigma(i)} \sum_{n,m=1}^k a_{n,m}\otimes E_{n,m} \sum_{j=1}^k E_{\sigma(j),j}
= \sum_{i,j=1}^k a_{\sigma(i),\sigma(j)}\otimes E_{i,j}\]
and
\[ U_\sigma A_\ell^\ast U_\sigma^\ast = (U_\sigma A_\ell U_\sigma^\ast)^\ast = \sum_{i,j=1}^k a_{\sigma(i),\sigma(j)}^\ast\otimes E_{i,j} \]

Note: we are talking about the $\ast$-distribution {\bf with respect to $\varphi$} of the entries of the matrices; not the $\ast$-distribution of the operator-valued matrices $A_\ell$ in the noncommutative probability space $(\mathscr{A}\otimes \MkC,\varphi\otimes\mathrm{tr})$.  (Our goal is to understand relations between these two.)

\begin{definition} For $\sigma\in S_k$, $\sigma\times\sigma$ is a permutation on $[k]\times[k]$.  The cycles of $\sigma\times\sigma$ partition the entries of each matrix $A_\ell$ into {\bf $\sigma$-orbits}.
\end{definition}

For example let $1+_k\in S_k$ denote the shift-up-one full cycle $(1+_k)(i) = 1+i\;\mathrm{mod}\;k$.  The $1+_k$-orbits in $A_\ell$ are the $k$ ``circulant diagonals'' $\{a^\ell_{i,i+n}\}_{i\in[k]}$ for each $n\in 0,1,\ldots,k-1$; here addition of indices is taken $\mathrm{mod}\;k$.

\begin{example} \label{ex.circ.inv} Following are two common setups exhibiting $\sigma$-circulant invariance.
\begin{enumerate}
    \item \label{ex.circ.inv.a} Suppose all the random variables $\{a^\ell_{i,j}\}_{i,j\in[k],\ell\in L}$ are $\ast$-freely independent.  Moreover, suppose that within each matrix $A_\ell$, all the entries within each $\sigma$-orbit have the same $\ast$-distribution.  (Entries in $A_{\ell_1}$ and $A_{\ell_2}$ for $\ell_1\ne\ell_2$ need not have related entry distributions.)  Then $\mathbf{A}$ is $\sigma$-circulant invariant.
    
    \item \label{ex.circ.inv.b} Suppose that the matrices $A_\ell$ all have $\ast$-freely independent entries for different $\ell$.  Moreover, within each matrix $A_\ell$, suppose that the $\sigma$-cycles are $\ast$-freely independent, and within each $\sigma$-cycle all entries are equal (but not necessarily equal for different $\sigma$-orbits, or in different matrices).  Then $\mathbf{A}$ is $\sigma$-circulant invariant.
    \end{enumerate} 
\end{example}

Both of the above examples explain the notation ``$\sigma$-circulant'' invariance.  A (scalared valued deterministic) {\em circulant matrix} is defined to have equal entries in each circulant diagonal (so all the entries are equal to one in the first row, tracking back up along a shifted diagonal).

The primary efficacy of $\sigma$-circular invariance comes from the fact that it is expressed in terms of a unitary conjugation, which interacts well with matrix product.  In particular, we have the following

\begin{lemma} \label{lem.circulant.*alg} Let $\sigma\in S_k$ be a permutation, and let $\mathbf{A} = \{A_\ell\}_{\ell\in L}$ be a family of $\sigma$-circulant invariant matrices in $\mathscr{A}\otimes \MkC$.  If $B\in \mathbb{C}^\ast \langle \mathbf{A} \rangle$ (the $\ast$-subalgebra of $\mathscr{A}\otimes \MkC$ generated by the matrices in $\mathbf{A}$), then $B$ is also $\sigma$-circulant invariant.
\end{lemma}

\begin{proof} Any such $B$ is a linear combination of monomials $A_{\ell_1}^{\epsilon_1}\cdots A_{\ell_p}^{\epsilon_p}$ for some $\ell_1,\ldots,\ell_p\in L$ and $\epsilon_1,\ldots,\epsilon_p\in\{1,\ast\}$.  Note that
\[ U_\sigma A_{\ell_1}^{\epsilon_1}\cdots A_{\ell_p}^{\epsilon_p} U_\sigma^\ast = (U_\sigma A_{\ell_1}^{\epsilon_1} U_\sigma^\ast)\cdots (U_\sigma A_{\ell_p}^{\epsilon_p} U_\sigma^\ast). \]
By definition of $\sigma$-circulant invariance, this product (a fixed function) of the conjugated matrices has the same joint distribution of entries as the product $A_{\ell_1}^{\epsilon_1}\cdots A_{\ell_p}^{\epsilon_p}$.  Extending this argument to the linear combination (a fixed function) of these terms shows that $U_\sigma B U_\sigma^\ast$ has the same joint distribution of entries as $B$, verifying the claim.
\end{proof}

\begin{corollary} \label{cor.diagonal.centered} Let $\sigma\in S_k$ be a permutation with one cycle of length $k$.  Let $\mathbf{A} = \{A_\ell\}_{\ell\in L}$ be a family of $\sigma$-circulant invariant matrices in $\mathscr{A}\otimes \MkC$.  If $B = \sum_{i,j=1}^k b_{i,j}\otimes E_{i,j} \in \mathbb{C}^\ast \langle \mathbf{A} \rangle$, then $(\varphi\otimes\mathrm{id})(B) = \varphi(b_{1,1}) I_k$ i.e. the diagonal entries all have the same trace, $\varphi(b_{i,i}) = \varphi(b_{1,1})$ for all $i\in[k]$.
\end{corollary}

\begin{proof} From Lemma \ref{lem.circulant.*alg}, $B$ is $\sigma$-circulant invariant: $B$ has the same joint law of entries as $U_\sigma B U_\sigma^\ast$.  In particular,for each $i,j\in[k]$, the entry $b_{i,j}$ of $B$ has the same $\ast$-distribution as $b_{\sigma(i),\sigma(j)}$ in $(\mathscr{A},\varphi)$.  Thus $b_{1,1} \equaldist b_{\sigma(1),\sigma(1)}$; also $b_{\sigma(1),b_{\sigma(1)}} \equaldist b_{\sigma(\sigma(1)),\sigma(\sigma(1))}$, and so forth.  It follows that all $i$ in the same $\sigma$-cycle as $1$ have $\varphi(b_{i,i}) = \varphi(b_{1,1})$.  Since $\sigma$ has a single cycle of length $k$, it contains all indices in $[k]$ including $1$.  Thus $\varphi\otimes\mathrm{id}(B) = \sum_{i=1}^k \varphi(b_{i,i})\otimes E_{i,i} = \varphi(b_{1,1})\sum_{i=1}^k E_{i,i} = \varphi(b_{1,1}) I_k$ as claimed.
\end{proof}

Now we turn for a moment to diagonal operator-valued matrices with free identically distributed entries, a kind of ``dual'' to the ensembles considered above.

\begin{lemma} \label{lem.diag.Haar}
Let $\{v^m_i\}_{i\in[k],m\in M}$ be a family of random variables in $(\mathscr{A},\varphi)$, and assume that $\{v^m_i\}_{m\in M}$ have the same $\ast$-distribution for all $i\in[k]$.  Let $V_m = \sum_{i=1}^k v^m_i\otimes E_{i,i} \in \mathscr{A}\otimes \MkC$.  Then the joint $\ast$-distribution of $\mathbf{V} = \{V_m\}_{m\in M}$ in $(\mathscr{A}\otimes \MkC,\varphi\otimes\mathrm{tr})$ is the same as the joint $\ast$-distribution of $\{v^m_1\}_{m\in M}$ in $(\mathscr{A},\varphi)$.  In particular, if $\{v^m_1\}_{m\in M}$ are $\ast$-free in $(\mathscr{A},\varphi)$, then $\{V_m\}_{m\in M}$ are $\ast$-free in $(\mathscr{A}\otimes \MkC,\varphi\otimes\mathrm{tr})$.
\end{lemma}

\begin{proof} The statement follows easily from the fact that, for any indices $m_1,\ldots,m_p\in M$ and any $\varepsilon_1,\ldots,\varepsilon_p\in\{1,\ast\}$,
\[ (V_{m_1})^{\varepsilon_1}\cdots (V_{m_p})^{\varepsilon_p} = \sum_{i=1}^k (v^{m_1}_i)^{\varepsilon_1}\cdots (v^{m_p}_i)^{\varepsilon_p}\otimes E_{i,i}. \]
The assumptions that the $\ast$-distribution of $\{v^m_i\}_{m\in M}$ is equal to the $\ast$-distribution of $\{v^m_i\}_{m\in M}$ is precisely to say that
\[ \varphi[(v^{m_1}_i)^{\varepsilon_1}\cdots (v^{m_p}_i)^{\varepsilon_p}] = \varphi[(v^{m_1}_1)^{\varepsilon_1}\cdots (v^{m_p}_1)^{\varepsilon_p}] \]
for all $i\in[k]$; thus
\begin{align*}  &(\varphi\otimes\mathrm{tr})[(V_{m_1})^{\varepsilon_1}\cdots (V_{m_p})^{\varepsilon_p}] \\ =&\frac{1}{k} \sum_{i=1}^k \varphi[(v^{m_1}_i)^{\varepsilon_1}\cdots (v^{m_p}_i)^{\varepsilon_p}] = \varphi[(v^{m_1}_1)^{\varepsilon_1}\cdots (v^{m_p}_1)^{\varepsilon_p}].
\end{align*}
Hence, the joint $\varphi\otimes\mathrm{tr}$ $\ast$-moments of $\{V_m\}_{m\in M}$ are the same as the $\varphi$ $\ast$-moments of $\{v^m_1\}_{m\in M}$, as claimed.  \end{proof}

This bring us to the main result of this section: freeness of cycle-circulant invariant matrices from diagonal matrices.

\begin{theorem} \label{thm.general.freeness} Let $\mathbf{A} = \{A_{\ell}\}_{\ell\in L}$ be a family of $\sigma$-circulant invariant matrices in $\mathscr{A}\otimes \MkC$, for some permutation $\sigma\in S_k$ with a single cycle of length $k$.  Let $\{\mathbf{V}_m\}_{m\in M}$ be diagonal matrices as in Lemma \ref{lem.diag.Haar}, and further assume that the families $\{v^m_i\}_{m\in M}$ are $\ast$-free for $i\in[k]$.  Assume that the entries of $\mathbf{A}$ are all $\ast$-free from all the entries of $\mathbf{V}$ with respect to $\varphi$.  Then $\mathbf{A}$ and $\mathbf{V}$ are $\ast$-freely independent with respect to $\varphi\otimes\mathrm{tr}$.
\end{theorem}

\begin{proof} We work directly from the definition, and show that if $B_1,\ldots,B_n\in\mathbb{C}^\ast\langle\mathbf{A}\rangle$ are centered and $W_1,\ldots,W_n\in\mathbb{C}^\ast\langle \mathbf{V}\rangle$ are centered, then $B_1W_1\cdots B_nW_n$ is centered. (For completeness one must also consider the product starting and ending with elements of $\mathbb{C}^\ast\langle\mathbf{A}\rangle$ or starting and ending with elements of $\mathbb{C}^\ast\langle\mathbf{V}\rangle$; the reader can readily check that these cases are treated nearly identically as the one we detail below.)

Denoting the entries as $B_m = \sum_{i,j=1}^k b^m_{i,j}\otimes E_{i,j}$ and $W_m = \sum_{h=1}^k w^m_h\otimes E_{h,h}$, we may expand
\begin{align*} B_1W_1\cdots B_nW_n &= \sum_{i_1,j_1,h_1,\ldots,i_n,j_n,h_n=1}^k b^1_{i_1,j_1}w^1_{h_1} b^2_{i_2,j_2}w^2_{h_2}\cdots b^n_{i_n,j_n}w^n_{h_n} \\  & \hspace{1.4in}\otimes E_{i_1,j_1}E_{h_1,h_1}E_{i_2,j_2}E_{h_2,h_2}\cdots E_{i_n,j_n}E_{h_n,h_n}.
\end{align*}
The product of matrix units yields
\[ \delta_{j_1,h_1}\delta_{h_1,i_2}\delta_{j_2,h_2}\delta_{h_2,i_3}\cdots \delta_{j_{n-1},h_{n-1}}\delta_{h_{n-1},i_n}\delta_{j_n,h_n} E_{i_1,h_n} \]
and so the sum reduces to
\begin{align*} B_1W_1\cdots B_nW_n &= \sum_{i_1,i_2,\ldots,i_n,h_n=1}^k b^1_{i_1,i_2}w^1_{i_2} b^2_{i_2,i_3}w^2_{i_3}\cdots b^n_{i_n,h_n}w^n_{h_n} \otimes E_{i_1,h_n}.
\end{align*}
Taking $\varphi\otimes\mathrm{tr}$ means only the $i_1=h_n$ terms survive, and
\begin{equation} \label{eq.alt.trace.b.w}
(\varphi\otimes\mathrm{tr})(B_1W_1\cdots B_nW_n) = \frac{1}{k}\sum_{i_1,i_2,\ldots,i_n=1}^k \varphi(b^1_{i_1,i_2}w^1_{i_2} b^2_{i_2,i_3}w^2_{i_3}\cdots b^n_{i_n,i_1}w^n_{i_1}).
\end{equation}

Now, for $1\le g\le n$, $W_g = \sum_{i=1}^k w^g_i\otimes E_{i,i}$, and also $W_g\in\mathbb{C}^\ast\langle\mathbf{V}\rangle$; so there is some noncommutative $\ast$-polynomial $P_g$ with $W_g = P_g(\{V_m\}_{m\in M})$; hence, since $V_m$ are diagonal, we have $w^g_i = P_g(\{v^m_i\}_{m\in M})$.  Note that this is the same polynomial $P_g$ used for all $i$.  Hence, our assumption that $\{v_i^m\}_{m\in M}$ has the same $\ast$-distribution as $\{v_i^m\}_{m\in M}$ implies immediately that $\varphi(w_i^g) = \varphi(w_1^g)$ for $i\in[k]$.  By assumption, the $W_g$ are centered, meaning that $0 = (\varphi\otimes\mathrm{tr})(W_g) = \frac{1}{k}\sum_{i=1}^k \varphi(w^g_i) = \varphi(w^g_1)$.  Thus, the centering assumption on $W^g$ implies that all the terms $w^g_i$ in \eqref{eq.alt.trace.b.w} are centered.

The same does not quite hold true for the $b^g_{i,j}$; however, the assumption that $\mathbf{A}$ is a $\sigma$-circulant family for a full cycle $\sigma$ implies, by Corollary \ref{cor.diagonal.centered}, that $\varphi(b^g_{i,i})=0$ for all $g,i$: i.e. the diagonal terms are centered.

As noted above, $w^g_i\in\mathbb{C}^\ast\langle \{v_i^m\}_{m\in M}\rangle$; similarly the $b^g_{i,j}$ are entries of $B_g$ which is a noncommutative polynomial in $\mathbf{A}$; hence, the $b^g_{i,j}$ are in the $\ast$-algebra generated by the entries of $\mathbf{A}$.  By the assumption of the theorem, it follows that $\{b^g_{i,j}\}_{g\in[n],i,j\in[k]}$ and $\{w^g_i\}_{g\in[n],i\in[k]}$ are $\ast$-free.  Note also that $w^g_1,\ldots,w^g_k$ are $\ast$-free from each other, due to the assumption of the theorem that the families $\{v^m_i\}_{m\in M}$ are $\ast$-free for different $i$.

If only diagonal terms $b^g_{i,i}$ appeared in \eqref{eq.alt.trace.b.w}, it would follow immediately from the definition of freeness that all the terms in \eqref{eq.alt.trace.b.w} are $0$.  This conclusion is still true, but it takes a little more work to show it.

We may compute each term in the sum using the moment cumulant formula:
\begin{align} \nonumber &\varphi(b^1_{i_1,i_2}w^1_{i_2} b^2_{i_2,i_3}w^2_{i_3}\cdots b^n_{i_n,i_1}w^n_{i_1}) \\
& \hspace{1.2in} = \sum_{\pi\in NC(2n)} \kappa_\pi[b^1_{i_1,i_2},w^1_{i_2},b^2_{i_2,i_3},w^2_{i_3},\ldots,b^n_{i_n,i_1},w^n_{i_1}].
\label{eq.mom.cumulant}
\end{align}

Since the entries in odd positions (all $b^g_{i,j}$) are free from the entries in even positions (all $w^g_i$), any $\pi\in NC(2n)$ with a block that connects even and odd positions yields a mixed cumulant $\kappa_\pi = 0$.  Therefore, in \eqref{eq.mom.cumulant}, the only terms that may be non-zero are those $\pi\in NC(2n)$ that decompose as $\pi = \pi_o\sqcup\pi_e$, where $\pi_o\in NC(\{1,3,\ldots,2n-1\})$ and $\pi_e\in NC(\{2,4,\ldots,2n\})$.

Now, since $\pi_e$ is non-crossing, at least one block must be an interval,cf. \cite{SpeicherNicaBook}.  If $\pi_e$ contains a singleton, say $\{2g\}\in\pi_e\subset\pi$, then the above cumulant factors into terms including $\kappa_1[w^g_{i_g}] = \varphi(w^g_{i_g})=0$.  Hence, $\kappa_\pi$ can only contribute if $\pi_e$ contains an interval of length $\ge 2$: so some block of $\pi_e$ connects $2g,2g+2$ for some $g\in[n]$.  But blocks of $\pi_e$ are blocks of $\pi$, and so $\kappa_\pi$ factors including a cumulant $\kappa_p[w^g_{i_g},w^{g+1}_{i_{g+1}},\ldots]$.  Since $w^g_i$ are freely independent for different $i$, this mixed cumulant must be $0$ (and hence $\kappa_\pi=0$) unless $i_g=i_{g+1}$.  In that case, note that since $\pi$ is non-crossing, it follows that the singleton position $2g+1$ nested between the connected $2g$ and $2g+2$ must be a singleton block in $\pi$, meaning that $\kappa_\pi$ has a factor of $\kappa_1[b^g_{i_g,i_{g+1}}] = \varphi(b^g_{i_g,i_g})=0$; once again, $\kappa_\pi$ has a factor of $0$.

Thus, we have shown that there are no $\pi\in NC(2n)$ that contribute non-zero terms to the sum \eqref{eq.mom.cumulant}, and therefore all the terms in \eqref{eq.alt.trace.b.w} are $0$.  This finally proves that $(\varphi\otimes\mathrm{tr})(B_1W_1\cdots B_nW_n)=0$, which concludes this proof of freeness.
\end{proof}

Next we will apply Theorem \ref{thm.general.freeness} repeatedly to derive a distribution and general freeness result for linearized $\mathscr{R}$-diagonal operators.  But first, it is worth noting the following natural but perhaps not fully clear result on transferrence of joint distributions.

\begin{lemma} \label{lem.joint.dist}
Let $\{a^m_{i,j}\}_{m\in M, i,j\in[k]}$ and $\{\tilde{a}^m_{i,j}\}_{m\in M, i,j\in[k]}$ be families in $(\mathscr{A},\varphi)$, and suppose they have the same $\ast$-distribution.  Let
\[ A_m = \sum_{i,j=1}^k a^m_{i,j}\otimes E_{i,j}  \qquad \tilde{A}_m = \sum_{i,j=1}^k \tilde{a}^m_{i,j}\otimes E_{i,j}. \]
Then $\{A_m\}_{m\in M}$ and $\{\tilde{A}_m\}_{m\in M}$ have the same $\ast$-distribution in $(\mathscr{A}\otimes \MkC,\varphi\otimes\mathrm{tr})$.
\end{lemma}

\begin{proof} For any choice of $m_1,\ldots,m_p\in M$ and $\varepsilon_1,\ldots,\varepsilon_p\in\{1,\ast\}$, we can expand
\begin{align*} & (\varphi\otimes\mathrm{tr})(A_{m_1}^{\varepsilon_1}\cdots A_{m_p}^{\varepsilon_p}) \\
& \hspace{0.4in} = \sum_{i_1,j_1,\ldots,i_p,j_p=1}^k (\varphi\otimes\mathrm{tr})\left[(a^{m_1}_{i_1,j_1})^{\varepsilon_1}\cdots (a^{m_p}_{i_p,j_p})^{\varepsilon_p}\otimes E_{i_1j_1}\cdots E_{i_pj_p}\right].
\end{align*}
The product of matrix units is $\delta_{j_1,i_2}\delta_{j_2,i_3}\cdots \delta_{j_{p-1},i_p}\,E_{i_1,j_p}$, and taking $\mathrm{tr}$ adds the further constraint that $j_p=i_1$; thus
\[ (\varphi\otimes\mathrm{tr})(A_{m_1}^{\varepsilon_1}\cdots A_{m_p}^{\varepsilon_p})
= \frac{1}{k}\sum_{i_1,\ldots,i_p=1}^k \varphi\left[(a^{m_1}_{i_1,i_2})^{\varepsilon_1}\cdots (a^{m_p}_{i_p,i_1})^{\varepsilon_p}\right] = \varphi(P(\{a^m_{i,j}\}_{m\in M,i,j\in[k]}) \]
where $P$ is the noncommutative $\ast$-polynomial
\[ P(\{a^m_{i,j}\}_{m\in M,i,j\in[k]}) = \frac{1}{k} \sum_{i_1,\ldots,i_p=1}^k (a^{m_1}_{i_1,i_2})^{\varepsilon_1}\cdots (a^{m_p}_{i_p,i_1})^{\varepsilon_p}. \]
By the equal $\ast$-distribution assumption,
\[ \varphi(P\{a^m_{i,j}\}_{m\in M,i,j\in[k]}) = \varphi(P\{\tilde{a}^m_{i,j}\}_{m\in M,i,j\in[k]}). \]
Now tracking back up each equality, this shows that
\[ (\varphi\otimes\mathrm{tr})(A_{m_1}^{\varepsilon_1}\cdots A_{m_p}^{\varepsilon_p}) = (\varphi\otimes\mathrm{tr})(\tilde{A}_{m_1}^{\varepsilon_1}\cdots \tilde{A}_{m_p}^{\varepsilon_p}) \]
as claimed.
\end{proof}

Before coming to our main application, we need one somewhat surprising feature of Haar unitaries.

\begin{lemma} \label{lem.free.Haar.products} 
Let $u,v$ be two $\ast$-free Haar unitaries, and let $\tilde{u}=vu$.  Then $\tilde{u}$ is Haar unitary, and $\tilde{u}$ and $v$ are $\ast$-free.
\end{lemma}

\begin{proof} 
    Take $(u_i)_{i\in I}$ be $\{u\}$ and $(v_i)_{i\in I}$ be $\{v\}$, and $(w_i)_{i\in I}$ be an identity in Lemma \ref{lem.free.Haar}. The conclusion of Lemma \ref{lem.free.Haar} proves $v$ and $vu$ are $\ast$-free.
\end{proof}
%%%%%%%%%%%%%%%%%%%%%%%%%%%%%%%%% OLD PROOF %%%%%%%%%%%%%%%%%%%%%%%%%%%%%
% As the statement is about the tracial von Neumann algebra generated by $\{u,v\}$, it suffices to prove the statement about the canonical representation of that algebra, $L(\mathbb{F}_2)$ wherein the generators of $\mathbb{F}_2$ induce via the left-regular representation two Haar unitaries $u,v$ with respect to the canonical trace $\tau$ which is defined by $\tau(w)=\delta_{w=1}$ for $\ast$-monomials $w$ in $\{u,v\}$.

% Thus, consider the subgroup generated by $\{\tilde{u}=vu,v\}$ in $L(\mathbb{F}_2)$.  In fact, this subgroup is isomorphic to the full free group generated by $\{u,v\}$.  To check this, we simply define a homomorphism $\psi\colon\langle u,v\rangle\to \langle \tilde{u},v\rangle$ by $\psi(u) = \tilde{u} = vu$ and $\psi(v) = v$.  (As $u,v$ have no algebraic relations, this extends to a unique well-defined group homomorphism.)  One can readily check that this is a bijection, whose inverse is uniquely determined by $\psi^{-1}(u) = v^{-1}u$ and $\psi^{-1}(v)=v$.

% Now, because $\psi$ is a group isomorphism, $\mathrm{ker}\,\psi = 1$, and so if any word in $\tilde{u},v$ and their inverses reduces to the identity, it follows that the same word in $u,v$ reduces to the identity.  This means exactly that $\psi$ preserves the trace $\tau$ restricted to the group.

% As such, $\psi$ induces a unique $\tau$-preserving $\ast$-automorphism of $L(\mathbb{F}_2)$.  Hence, $\ast$-freeness of $u,v$ implies $\ast$-freeness of $\tilde{u},v$, completing the proof.

Now, we come to the main application of the preceding results: freeness of entry-wise circulant $\mathscr{R}$-diagonal matrices from general circulant matrices.

\begin{theorem} \label{thm.freeness.R.diagonal} Let $\{a^\ell_{i,j}\}^{\ell\in L}_{i,j\in[k]}$ and $\{z^m_{i,j}\}^{m\in M}_{i,j\in[k]}$ be two $\ast$-free families in $(\mathscr{A},\varphi)$, and suppose that $a^{\ell}_{i,j}$ are $\ast$-free $\mathscr{R}$-diagonal elements.  Further, let $\mathbf{A}=\{A_\ell\}_{\ell\in L}$ and $\mathbf{Z}=\{Z_m\}_{m\in M}$ where
\[ A_\ell = \sum_{i,j=1}^k a^\ell_{i,j}\otimes E_{i,j} \qquad\text{and}\qquad Z_m = \sum_{i,j=1}^k z^m_{i,j}\otimes E_{i,j}. \]
Let $\sigma\in S_k$ be a permutation with a single cycle of length $k$ and suppose that the family $\{\mathbf{A},\mathbf{Z}\}$ is $\sigma$-circulant invariant.

Then the families $\mathbf{A}$ and $\mathbf{Z}$ are $\ast$-freely independent with respect to $\varphi\otimes\mathrm{tr}$.  Moreover, each $A_\ell$ decomposes as a sum $A_\ell = A_{\ell,1}+\cdots+A_{\ell,k}$ of $\sigma$-circulant invariant matrices, such that $\{A_{\ell,n}\}_{\ell\in L,n\in[k]}$ are all $\ast$-freely independent with respect to $\varphi\otimes\mathrm{tr}$, and
\[ A_{\ell,n} \equaldist a^{\ell}_{1,\sigma^n(1)} \qquad \ell\in L,\; 1\le n\le k \]
i.e.\ the $\ast$-distribution of $A_{\ell,n}$ with respect to $\varphi\otimes\mathrm{tr}$ is equal to the $\ast$-distribution of the first row entry $a^\ell_{1,\sigma^n(1)}$ with respect to $\varphi$.  In particular, each $A_{\ell,n}$ is $\mathscr{R}$-diagonal, and therefore their ($\ast$-free) sum $A_\ell$ is $\mathscr{R}$-diagonal, with $\ast$-distribution equal to that of $a^{\ell}_{1,\sigma(1)} + \cdots+ a^{\ell}_{1,\sigma^k(1)} = a^\ell_{1,1}+\cdots +a^\ell_{1,k}$.
\end{theorem}

\begin{proof} Decompose each matrix $A_\ell=A_{\ell,1}+\cdots+A_{\ell,k}$ where
\begin{equation} \label{eq.Aln.def} A_{\ell,n} = \sum_{i=1}^{k} a^{\ell}_{\sigma^i(1),\sigma^{i+n}(1)} \otimes E_{\sigma^i(1),\sigma^{i+n}(1)} \qquad 1\le n\le k. \end{equation}
So: $A_{\ell,k}$ is the diagonal part of $A_\ell$, and in general each $A_{\ell,n}$ has $k$ non-zero entries, index by the full orbit of $(1,\sigma^n(1))$ under $\sigma\times\sigma$.  By the assumption of $\sigma$-circulant invariance of $A_\ell$, it follows that all the non-zero entries of $A_{\ell,n}$ are identically distributed.  Note also that each $A_{\ell,n}$ is $\sigma$-circulant invariant on its own.

Let $\{u^{\ell,n}_i,v^{\ell,n}_i\}_{i,n\in[k]}^{\ell\in L}$ be a family of $\ast$-free Haar unitaries in $(\mathscr{A},\varphi)$, all $\ast$-free from all the entries of $\{\mathbf{A},\mathbf{Z}\}$.  Define
\[ U_{\ell,n} = \sum_{i=1}^k u_i^{\ell,n}\otimes E_{i,i}\qquad\text{and}\qquad V_{\ell,n} = \sum_{i=1}^k v^{\ell,n}_{i}\otimes E_{i,i}. \]
By Lemma \ref{lem.diag.Haar}, $\{U_{\ell,n},V_{\ell,n}\}_{\ell\in L,n\in[k]}$ are $\ast$-freely independent Haar unitaries with respect to $\varphi\otimes\mathrm{tr}$.  Notice that
\begin{align*} U_{\ell,n}A_{\ell,n}V_{\ell,n} &= \sum_{j,i,j'=1}^k u^{\ell,n}_{j}a^\ell_{\sigma^{i}(1),\sigma^{i+n}(1)}v^{\ell,n}_{j'} \otimes E_{j,j}E_{\sigma^i(1),\sigma^{i+n}(1)}E_{j',j'} \\
&= \sum_{i=1}^k u_{\sigma^i(1)}^{\ell,n} a^{\ell}_{\sigma^i(1),\sigma^{i+n}(1)}v^{\ell,n}_{\sigma^{i+n}(1)}\otimes E_{\sigma^i(1),\sigma^{i+n}(1)}.
\end{align*}
For each $\ell$ and each fixed $i,n\in[k]$, the Haar unitaries $\{u^{\ell,n}_{\sigma^i(1)},v^{\ell,n}_{\sigma^{i+n}(1)}\}$ are freely independent from the $\mathscr{R}$-diagonal element $a^\ell_{\sigma^i(1),\sigma^{i+n}(1)}$.  Hence, by \cite{SpeicherNicaBook},
\[ u_{\sigma^i(1)}^{\ell,n} a^{\ell}_{\sigma^i(1),\sigma^{i+n}(1)}v^{\ell,n}_{\sigma^{i+n}(1)} \equaldist a^{\ell}_{\sigma^i(1),\sigma^{i+n}(1)}. \]
This shows (by the freeness of the entries of $\mathbf{Z}$ from all the $a^{\ell}_{i,j}$ and $u^{\ell,n}_i$ and $v^{\ell,n}_i$) that the joint distributions of all the entries of $\{\mathbf{Z},A_{\ell,n}\}_{\ell\in L}^{n,\in[k]}$ is the same as the joint distribution of all the entires of $\{\mathbf{Z},U_{\ell,n}A_{\ell,n}V_{\ell,n}\}_{\ell\in L}^{n,\in[k]}$.  It therefore follows from Lemma \ref{lem.joint.dist} that
\begin{equation} \label{eq.=dist.1} \{\mathbf{Z},A_{\ell,n}\}_{\ell\in L}^{n,\in[k]} \equaldist \{\mathbf{Z},U_{\ell,n}A_{\ell,n}V_{\ell,n}\}_{\ell\in L}^{n,\in[k]} \qquad\text{with respect to}\;\; \varphi\otimes\mathrm{tr}. \end{equation}

Now, applying Theorem \ref{thm.general.freeness}, we find that the families $\{\mathbf{Z},A_{\ell,n}\}_{\ell\in L}^{n\in[k]}$ (that are all $\sigma$-circulant invariant) and $\{U_{\ell,n},V_{\ell,n}\}_{\ell\in L}^{n\in[k]}$ are $\ast$-freely independent with respect to $\varphi\otimes\mathrm{tr}$.  Well,
\begin{equation} \label{eq.Haar.conj.2} U_{\ell,n}A_{\ell,n}V_{\ell,n} = V^\ast_{\ell,n}(V_{\ell,n}U_{\ell,n}A_{\ell,n})V_{\ell,n} = V^\ast_{\ell,n}(\tilde{U}_{\ell,n}A_{\ell,n})V_{\ell,n} \end{equation}
where $\tilde{U}_{\ell,n} = V_{\ell,n}U_{\ell,n}$.  By construction, for any particular $(\ell_0,n_0)$, $\tilde{U}_{\ell_0,n_0}$ is $\ast$-freely independent from $\mathbf{Z}$, $\{A_{\ell,n}\}_{\ell\in L}^{n\in[k]}$, and $\{V_{\ell,n}\}_{(\ell,n)\ne(\ell_0,n_0)}$.  Moreover, $\tilde{U}_{\ell_0,n_0}$ is freely independent from $V_{\ell_0,n_0}$ by Lemma \ref{lem.free.Haar.products}.  Hence, combining \eqref{eq.=dist.1} and \eqref{eq.Haar.conj.2}, it follows from \cite{SpeicherNicaBook} that $\mathbf{Z}$ and $\{A_{\ell,n}\}_{\ell\in L}^{n\in[k]}$ are $\ast$-freely independent, as claimed.

Finally, since $A_{\ell,n} \equaldist U_{\ell,n} A_{\ell,n} V_{\ell,n}$ where $\{U_{\ell,n},V_{\ell,n}\}$ and $A_{\ell,n}$ are $\ast$-freely independent, it follows from \cite{SpeicherNicaBook} that $A_{\ell,n}$ is $\mathscr{R}$-diagonal.  The assumption that $A_\ell$ is $\sigma$-circulant invariant implies (since the non-zero entries of $A_{\ell,n}$ fill a full orbit of $\sigma\times\sigma$) that all the non-zero entries in $A_{\ell,n}$ are identically distributed.  Hence, they all have the same distribution as the $i=k$ term in \eqref{eq.Aln.def}, which is
\[ a^\ell_{\sigma^k(1),\sigma^{k+n}(1)} = a^{\ell}_{1,\sigma^n(1)} \]
since $\sigma$ is a cycle of length $k$.  For convenience, relabel these entries as $b_i = a^{\ell}_{\sigma^i(1),\sigma^{i+n}(1)}$ (suppressing the $\ell$ and $n$ which will be fixed for now); thus
\[ A_{\ell,n} = \sum_{i=1}^k b_i\otimes E_{\sigma^i(1),\sigma^{i+n}(1)} \]
where the non-zero entries $b_i$ are $\ast$-freely independent identically distributed $\equaldist a^{\ell}_{1,\sigma^n(1)}$.  Then we compute
\[ A_{\ell,n}^\ast A_{\ell,n} = \sum_{i,j=1}^k b_i^\ast b_j \otimes E_{\sigma^{i+n}(1),\sigma^i(1)} E_{\sigma^j(1),\sigma^{j+n}(1)}. \]
The product of matrix units is $\delta_{\sigma^i(1),\sigma^j(1)}E_{\sigma^{i+n}(1),\sigma^{j+n}(1)}$; this reduces the sum to a single index $i=j$ (really $i=j\;\mathrm{mod}\;k$, but since $i,j$ only range through $[k]$ we must have $i=j$), and thus
\[ A_{\ell,n}^\ast A_{\ell,n} = \sum_{i=1}^k b_i^\ast b_i\otimes E_{\sigma^{i+n}(1),\sigma^{i+n}(1)}. \]
As the set $\{\sigma^{i+n}(1)\}_{i\in[k]}$ is equal to $[k]$ (because $\sigma$ is a cycle of length $k$), and since the $b_i^\ast b_i$ are all identically distributed, it follows that the $\ast$-distribution of the diagonal matrix $A_{\ell,n}^\ast A_{\ell,n}$ with respect to $\varphi\otimes\mathrm{tr}$ is equal to the $\ast$-distribution of $b_1^\ast b_1$ with respect to $\varphi$.  As both $A_{\ell,n}$ and $b_1$ are $\mathscr{R}$-diagonal, and the $\ast$-distribution of any $\mathscr{R}$-diagonal element $b$ is determined by the $\ast$-distribution of $b^\ast b$, it now follows that $A_{\ell,n}\equaldist b_1 \equaldist a^{\ell}_{1,\sigma^n(1)}$ as claimed.  The remaining statement of the theorem now follows immediately from $\ast$-freeness of the $A_{\ell,n}$ and $a^\ell_{1,\sigma^{n}(1)}$ for $n\in[k]$.  \end{proof}

\begin{remark} Using \cite[Proposition 3.5]{HaagerupLarsen2000} (See also \cite[Proposition 5.2]{Nica-Speicher-1998duke}), the distribution of the $\mathscr{R}$-diagonal matrices $A_\ell$ can be computed exactly in principle, as follows.  Let $\mu^\ell_k$ denote the symmetrization of the distribution of $|a^\ell_{1,k}|$, and let
\[ \mu^\ell = \mu^\ell_1\boxplus\cdots\boxplus\mu^\ell_k. \]
then $A_\ell^\ast A_\ell$ has $\mu^\ell$ as its symmetrized distribution.  In particular: the Brown measure of $A_\ell^\ast A_\ell$ is the unique rotationally-invariant measure on $\mathbb{C}$ whose trace along $\mathbb{R}$ is $\mu^\ell$.
\end{remark}

\begin{remark} It is important to remember that $\sigma$-circulant invariance is a propoerty of {\em joint distributions of entries}.  Namely: the assumption that $\{\mathbf{A},\mathbf{Z}\}$ is $\sigma$-circulant invariant is fairly strong and restrictive on the form that $\mathbf{Z}$ can take.  It requires that the $z^m_{i,j}$ are identically distributed for all $(i,j)$ in a single orbit of $\sigma\times\sigma$; and moreover, to preserve the joint distribution of entries of $\mathbf{Z}$ with $\mathbf{A}$, one cannot have any generic joint distribution between the entries of $\mathbf{Z}$.  They could be $\ast$-free, or $\ast$-free between orbits and equal within, or perhaps have a common pairwise covariance structure coming from a highly symmetric (in each orbit) joint law.  In any case, while $\mathscr{R}$-diagonality of $z^m_{i,j}$ is not needed here, there are still many constraints on the joint distribution.
\end{remark}

% \begin{remark} It is possible to generalize Theorem \ref{thm.freeness.R.diagonal} slightly: the diagonal entries of each $A_\ell$ (i.e. the matrix $A_{\ell,k}$) do not need to be $\mathscr{R}$-diagonal; so long as they are $\ast$-freely independent and identically distributed, they can have any distribution at all, and it will remain true that $\mathbf{Z}$ and $\{A_{\ell,n}\}_{\ell\in L,n\in[k]}$ are $\ast$-free from each other.  It will no longer be true that $A_\ell$ is $\mathscr{R}$-diagonal, but it will still have the same disribution as $a^\ell_{1,1}+a^\ell_{1,2} +\cdots+a^{\ell}_{1,k}$; here the sum of the last $k-1$ terms is $\mathscr{R}$-diagonal, then added to another generic $\ast$-free random variable.

% Proving this requires another layer of proof, including another family of free Haar unitary diagonal matrices $W^\ell$ to conjugate the diagonal entries; and this requires a generalization of Lemma \ref{lem.free.Haar.products}, showing that if $\{u,v,w\}$ are $\ast$-freely independent Haar unitaries then so are $\{u,uv,uw\}$.  This is true and the proof isn't very different.

% {\bf However}, unless someone convinces me that this very mild generalization of the diagonal entries is useful for something we might care about, I am disinclined to modify the proof and make it any longer.
% \end{remark}

\begin{corollary}  
Let $\{a^\ell_{i,j}\}^{\ell\in L}_{i,j\in[k]}$ and $\{z^m_{i,j}\}^{m\in M}_{i,j\in[k]}$ be two $\ast$-free families in $(\mathscr{A},\varphi)$, and suppose that $a^{\ell}_{i,j}$ are all $\ast$-free, and $a^{\ell}_{i,j}$ are $\mathscr{R}$-diagonal elements for $i\neq j$.  Further, let $\mathbf{A}=\{A_\ell\}_{\ell\in L}$ and $\mathbf{Z}=\{Z_m\}_{m\in M}$ where
\[ A_\ell = \sum_{i,j=1}^k a^\ell_{i,j}\otimes E_{i,j} \qquad\text{and}\qquad Z_m = \sum_{i,j=1}^k z^m_{i,j}\otimes E_{i,j}. \]
Let $\sigma\in S_k$ be a permutation with a single cycle of length $k$ and suppose that the family $\{\mathbf{A},\mathbf{Z}\}$ is $\sigma$-circulant invariant.

Then $\mathbf{Z}$ and all the $A_\ell$ are also $\ast$-freely independent with respect to $\varphi\otimes\mathrm{tr}$.
\end{corollary}
\begin{proof}
For each $\ell\in L$, define 
\[D_{\ell} = \sum_{i=1}^k a_{i,i}^{\ell}\otimes E_{i,i} \qquad\text{and}\qquad \tilde A_\ell = \sum_{i\neq j}a_{i,j}^\ell\otimes E_{i,j}.\]
Then $\{Z_m, D_\ell: m\in M, \ell\in L\}$ and $\{\tilde A_\ell: \ell\in L\}$ satisfy the hypothesis of Theorem~\ref{thm.freeness.R.diagonal}. Hence, $\{Z_m, D_\ell: m\in M, \ell\in L\}$ and all $\tilde A_\ell$ are $\ast$-free with respect to $\varphi\otimes \mathrm{tr}$.

Since $a_{i,i}^\ell$ are all $\ast$-free and $\{z^m_{i,j}\}^{m\in M}_{i,j\in[k]}$ is $\ast$-free from all $a_{i,i}^\ell$, by Theorem \ref{thm.general.freeness}, $\{Z_m\}_{m\in M}$ and $ \{D_\ell\}_{\ell\in L}$ are $\ast$-free. By Lemma \ref{lem.diag.Haar}, the joint $\ast$-distribution of $\{D_\ell\}_{\ell\in L}$ is the same as the joint $\ast$-distribution of $\{a_{1,1}\}_{\ell\in L}$; in particular, all the $D_\ell$ are $\ast$-freely independent. Combining the $\ast$-freeness from the preceding paragraph, $\{Z_m\}_{m\in M}$, all $D_\ell$, and all $\tilde A_\ell$ are $\ast$-freely independent. By construction, $A_\ell = D_\ell + \tilde A_\ell$. It then follows that $\{Z_m\}_{m\in M}$ and all $A_\ell$ are $\ast$-freely independent.
\end{proof}

\subsection{Freeness in Our Application\label{sect.freeness.ZA}}
We apply the results in the preceding section to a linearized matrix $Z+\sqrt{t/k}A$. Let 
\[Z = \sum_{j=1}^k u_0^{1/k}\otimes E_{j,j+1},\; U = \sum_{j=1}^k w_j\otimes E_{j,j}, \; V=\sum_{j=1}^k v_j\otimes E_{j,j}, \; X = \sum_{j=1}^k x\otimes E_{j,j},\]
where $w_1,\ldots,w_k, v_1,\ldots,v_k$ are free Haar unitaries, $x\in\widetilde{\mathcal{A}}$ is a (possibly unbounded) self-adjoint random variable such that $u_0^{1/k}, x, w_1,\ldots,w_k, v_1,\ldots,v_k$ are $\ast$-free. The distribution of $x$ is chosen such that $a = v_1 x w_1$ in $\ast$-distribution, so that $A = (U\Sigma)XV$ in $\ast$-distribution where $\Sigma = \sum_{j=1}^k E_{j,j+1}$ is a cyclic permutation matrix.

\begin{proposition}
    \label{prop.freeness.unbounded}
    The sets of random variables $\{U\Sigma\}$, $\{V\}$, and $\{Z, X\}$ are $\ast$-free. 
\end{proposition}
\begin{proof}
    The elements $U\Sigma$ and $V$ are $\ast$-free by Theorem \ref{thm.general.freeness}. By Proposition 4.1(iii) and Definition 4.2 of \cite{BercoviciVoiculescu1993}, $\{Z, X\}$ is $\ast$-free from $\{U\Sigma, V\}$ means the subset
    \[S_1=\{Z, f(X): f \textrm{ is a bounded continuous function on $\mathbb{R}$}\}\]
    of $\mathcal{A}$ is $\ast$-free from $S_2=\{U\Sigma, V\}$.

    To show that $S_1$ is free from $S_2$, we consider an arbitrary order-$n$ mixed free cumulant $\kappa_n$ of elements in $S_1$ and $S_2$ (i.e. involves $S_1$ and $S_2$), and want to show that $\kappa_n=0$. There are nonnegative intergers $m_1, m_2, m_3, m_4$ such that $n=m_1+m_2+m_3+m_4$ and $m_1$ of these $n$ elements are $U\Sigma$, $m_2$ are $V$, $m_3$ are $Z$, and $f_1(X),\ldots,f_{m_4}(X)$ for some bounded continuous functions $f_1,\ldots,f_{m_4}$ on $\mathbb{R}$. In other words, the free cumulant $\kappa_n$ involves elements in 
    \[S_1' = \{Z,f_1(X),\ldots,f_m(X)\}\]
    and $S_2$. Now, each $f_i(X) = \sum_{j=1}^k f_i(x)\otimes E_{j,j}$. Thus, we can apply Theorem \ref{thm.freeness.R.diagonal} with $\mathbf{Z} = S_1'$ and $\mathbf{A} = S_2$ to conclude $\kappa_n = 0$ because $\kappa_n$ is a mixed cumulant of $S_1'$ and $S_2$.
    %It remains to show that $Z$ and $X$ are $\ast$-free; that is, we need to show that $\{Z\}$ and 
    %\[S = \{f(X): f \textrm{ is a bounded continuous function on $\mathbb{R}$}\}\]
    %are $\ast$-free. The sets $\{Z\}$ and $S$ are $\ast$-free by Theorem \ref{thm.general.freeness} because $f(X) = \sum_{j=1}^k f(x_j)\otimes E_{j,j+1}$ and $f(x_1),\ldots,f(x_k)$ are free. 
\end{proof}

\begin{corollary}
\label{cor.Z.A.freeness}
The $Z$ and $A$ defined in \eqref{eq.Z.A.def} are $\ast$-free, and $A$ has the same $\ast$-distribution as $a$.
\end{corollary}
\begin{proof}
By Lemma \ref{lem.joint.dist}, the distribution of $x$ is chosen such that $A = (U\Sigma)XV$ in $\ast$-distribution, and $\{Z, A\}$ and $\{Z, (U\Sigma)XV\}$ have the same joint $\ast$-distribution.

By Proposition \ref{prop.freeness.unbounded}, $\{U\Sigma\}$, $\{V\}$, $\{Z,X\}$ are $\ast$-free. We write 
\[(U\Sigma)XV = [(U\Sigma) X (U\Sigma)^*] [(U\Sigma) V].\] 
Then $(U\Sigma) X (U\Sigma)^*$ is $\ast$-free from $Z$ since the spectral projections of $(U\Sigma) X (U\Sigma)^*$ are all $\ast$-free from $Z$, by \cite[Exercise 5.24]{SpeicherNicaBook}. Thus, the von Neumann algebra $\mathcal{W}$ generated by the spectral projections of $(U\Sigma) X (U\Sigma)^*$ and $(U\Sigma)V$ is $\ast$-free from $Z$. Now, $(U\Sigma)XV$ is affiliated with $\mathcal{W}$, which proves $Z$ and $(U\Sigma)XV$, hence $Z$ and $A$, are $\ast$-free.

Finally, we show that $A$ has the same $\ast$-distribution as $a$. Since $X$ is a diagonal matrix, $X$ has the same distribution as $x$. Since $U\Sigma$ and $V$ are $\ast$-free Haar unitaries also $\ast$-free from $X$ by Proposition \ref{prop.freeness.unbounded}, $(U\Sigma)XV$ has the same $\ast$-distribution as $v_1 x w_1$, which in turn has the same $\ast$-distribution as $a$, where $v_1$, $w_1$ are $\ast$-free Haar unitaries $\ast$-free from $x$.
\end{proof}

\section{Brown Measures of Invariant Random Walks\label{sect.BM.computation.1}}

\subsection{Brown measure of random walk} \hfill

\medskip

In this subsection, we compute the Brown measure of the $k$-step random walk $u_0b_k(t)$ (see \eqref{eq.RW.unitary}). We first compute the Brown measure of $Z+\sqrt{t/k} A$, where $Z$ and $A$ are defined in \eqref{eq.Z.A.def}. The $\ast$-free operators $Z$ and $A$ are affiliated with $\mathscr{A}\otimes \MkC$, and $A$ is $\mathscr{R}$-diagonal with the same $\ast$-distribution as $a$, by Corollary \ref{cor.Z.A.freeness}. The Brown measure of $Z+\sqrt{t/k}A$ can, in principle, be computed using the results in \cite{BercoviciZhong2022Rdiag}. Once we have the Brown measure of $Z+\sqrt{t/k}A$, the Brown measure of $u_0b_k(t)$ is the push-forward of the Brown measure of $Z+\sqrt{t/k}A$ by $\lambda\mapsto \lambda^k$ by Corollary \ref{cor.unitary.linearized}.

The unit circle is the image of $\theta\mapsto e^{i\theta}$ for $\theta\in (-\pi ,\pi]$. From now on, we abuse our notation and realize the law $\mu_{u_0}$ of $u_0$ as a probability measure on the interval $(-\pi, \pi]$. The nonzero entries in $Z$ are $u_0^{1/k}$, where $u_0^{1/k}$ can be chosen to be any $k$-root of $u_0$; that is, $u_0^{1/k}$ is any unitary operator satisfying $(u_0^{1/k})^k=u_0$. For convenience, we will fix $u_0^{1/k}$ to be the unitary operator with law $\mu_{u_0^{1/k}}$, realized as a probability measure on $(-\pi, \pi]$ determined by
 \begin{equation}
     \label{eq.law.u.root}
     \int_{(-\pi/k,\pi/k]}f(e^{i\alpha})\mu_{u_0^{1/k}}(d\alpha)=\int_{(\pi,\pi]}f(e^{\frac{i\alpha}{k}})\mu_{u_0}(d\alpha)
 \end{equation}
 for any bounded continuous function $f$.

To describe the relevant sets for the Brown measure of $Z+\sqrt{t/k}A$ in Proposition \ref{prop.Brown.add.support}, we need the following definition. For any $X\in \mathscr{A}$, denote by $\mathrm{ker}(X)$ the orthogonal projection, which is also in $\mathscr{A}$, into the kernel of $X$.
\begin{definition}
    \label{def.sets.Brown.add}
    We first emphasize that the $\ast$-distribution of $Z$ depends on $k$, and $A$ has the same $\ast$-distribution as $a$ for all $k$. We define the following sets. 
    \begin{enumerate}
        \item $S_k = \{\lambda: \varphi(\mathrm{ker}(Z-\lambda))+\varphi(\mathrm{ker}(a))\geq 1\}$
        \item $E_{1,k}(t) =\displaystyle \left\{\lambda: \frac{t}{k} \leq \frac{1}{\|(Z-\lambda)^{-1}\|_2^2}\right\}$
        \item $E_{2,k}(t) =\displaystyle \left\{\lambda: \vert\lambda\vert^2\leq\frac{t}{k\|a^{-1}\|_2^2}-1\right\}$
        \item $E_k(t) = E_{1,k}\cap E_{2,k}$
        \item $\Omega_k(t) =\mathbb{C}\setminus (E_{1,k}(t)\cup E_{2,k}(t))$.
    \end{enumerate}
\end{definition}
\begin{remark}
Since $a$ is not identically zero, $\varphi(\mathrm{ker}(a))<1$. We must have $\varphi(\mathrm{ker}(Z-\lambda))>1$ for any $\lambda\in S_k$, and thus $S_k\subset \mathrm{spec}(Z)$.
\end{remark}

\begin{definition}
    \label{notation.psi.epsilon}
    \begin{enumerate}
        \item For any Borel measure $\mu$ on $\mathbb{R}$, denote by $\tilde\mu$ the symmetrization of $\mu$; that is, for any Borel set $B$,
        \[\tilde\mu(B) = \frac{\mu(B)+\mu(-B)}{2}.\]
        \item Let $\mu_{\vert Z-\lambda\vert}$ and $\mu_{\vert a\vert}$ be the laws of $\vert Z-\lambda\vert$ and $\vert a\vert$ respectively.
        \item For any $z\in\mathbb{C}$, let $\lambda$ be any $k$-th root of $z$. By applying Theorem \ref{thm.subordination.general} with $\mu_1 = \tilde\mu_{\vert Z-\lambda\vert}$ and $\mu_2 = \tilde\mu_{\sqrt{t/k}\vert a\vert}$, we get two subordination functions $\omega_1$ and $\omega_2$ which depend on $k$ and $z\in\mathbb{C}$. We denote
        \[\psi_k(t,z) = \omega_1^{(\lambda)}(0).\]
        This definition is well-defined; that is, if $\lambda$ is another $k$-th root of $z$, we still have the same value of $\omega_1^{(\lambda)}(0)$, since the $\ast$-distributions of $Z$ and $a$ are invariant under rotation by any $k$-th root of unity.
        \item By Lemma 2.5 of \cite{BercoviciZhong2022Rdiag}, $\psi_k(t,z)$ is a nonnegative multiple of $i$. For convenience, we define
        \[\eta_k(t,z) = -\psi_k(t,z)^2,\]
        which is a nonnegative number.
    \end{enumerate}
\end{definition}

\begin{proposition}
    \label{prop.Brown.add.support}
    The sets defined in Definition \ref{def.sets.Brown.add} are related to the Brown measure of $Z+\sqrt{t/k}A$ as in the following. Recall that $\mu_{u_0}$ is the law of the initial condition $u_0$ of $u_0b_k(t)$.
    \begin{enumerate}
        \item The atoms of the Brown measure of $Z+\sqrt{t/k}A$ are contained in $S_k$. The set $S_k$ is a finite set disjoint from $E_{1,k}(t)\cup E_{2,k}(t)$. If $S_k$ is nonempty, then $E_{2,k}(t)$ is empty.
        \item The set $S_k$ is empty for $k$ large enough.
        \item $E_k(t) = \{\lambda: \vert Z-\lambda\vert = \sqrt{t/k}\vert A\vert = \sqrt{t/k}\|A\|\}$. Thus $E_k(t)$ is empty except if $a$ is Haar unitary and
        \begin{enumerate}[label=(\roman*)]
            \item $k=1$ and $\mu_{u_0} = p\delta_{e^{i\alpha}}+(1-p)\delta_{e^{i\beta}}$ for some $\alpha, \beta$ such that $\sin^2\frac{\alpha-\beta}{2}\leq t$. In this case, $E_k(t) = \left\{\left(\cos\frac{\alpha-\beta}{2}\pm\sqrt{t-\sin^2\frac{\alpha-\beta}{2}}\right)e^{i\frac{\alpha+\beta}{2}}\right\}$; or
            \item $k=2$, $t\geq 2$, and $\mu_{u_0} = \delta_{e^{i\alpha}}$. $E_k=\left\{  \pm ie^{i\frac{\alpha}{2}}\sqrt{\frac{t}{2}-1}\right\}.$
        \end{enumerate} 
        \item Both $E_{1,k}(t)$ and $E_{2,k}(t)$ are closed sets. The set $E_{2,k}(t)$ is either empty, singleton, or a closed disk centered at $0$. Moreover, $\psi_k(t,z) = 0$ if $z^{1/k}\in E_{1,k}(t)$ and $\psi_k(t,z) = \infty$ if $z^{1/k}\in E_{2,k}(t)$.
        \item Both $\Omega_k(t)\setminus S_k$ and $\Omega_k(t)$ are open. Moreover, $\psi_k(t,z)$ is a (finite) positive multiple of $i$ if and only if $z^{1/k}\in \Omega_k(t)\setminus S_k$.
        \item The Brown measure of $Z+\sqrt{t/k}A$ is supported in the closure $\overline{\Omega_k(t)}$ of $\Omega_k(t)$. In $\Omega_k(t)\setminus S_k$, the Brown measure is absolutely continuous relative to the Lebesgue measure on $\mathbb{C}$ with a real-analytic density.
    \end{enumerate}
\end{proposition}
\begin{proof}
    Apply $X_1=Z$, $X_2= \sqrt{t/k}A$ in \cite{BercoviciZhong2022Rdiag}. Recalling that $A$ is chosen to have unit variance, the sets $S_k, E_{1,k}(t), E_{2,k}(t), E_k(t), \Omega_k(t)\setminus S_k$ are respectively the sets $S$, $F_1$, $F_2$, $F$, $\Omega$ in Definition 3.1 of \cite{BercoviciZhong2022Rdiag} (Identifying $E_{2,k}(t)$ with $F_2$ requires an algebraic calculation). 
    
    Part (1) follows from Lemma 3.2(1) and Remark 3.5 of \cite{BercoviciZhong2022Rdiag}. To prove Part (2), we first note that the $\mathscr{R}$-diagonal element $a$ is nonzero because it has unit variance; thus $\varphi(\mathrm{ker}(a))<1$. Meanwhile, if $\varphi(\mathrm{ker}(Z-\lambda))>0$ for some $\lambda\in\mathbb{C}$, then by that the $\ast$-distribution of $Z$ is invariant under multiplication by $e^{2\pi i/k}$, $\varphi(\mathrm{ker}(Z-e^{\frac{2\pi i j}{k}}\lambda))>0$ for all integers $j=1,\ldots,k-1$. Since 
    \[\mathrm{ker}(Z-e^{\frac{2\pi i j_1}{k}}\lambda)\cap \mathrm{ker}(Z-e^{\frac{2\pi i j_2}{k}}\lambda) = \{0\}\]
    for any $j_1\neq j_2$, $\varphi(\mathrm{ker}(Z-e^{\frac{2\pi i j}{k}}\lambda))\leq \frac{1}{k}$ for each $j$; in particular, $\varphi(\mathrm{ker}(Z-\lambda))\leq \frac{1}{k}$. Now, if $k$ is large enough, $\varphi(\mathrm{ker}(Z-\lambda))+\varphi(\mathrm{ker}(a))<1$ for all $\lambda\in\mathbb{C}$ such that $\varphi(\mathrm{ker}(Z-\lambda))>0$. This shows $S_k$ must be empty for all $k$ large enough.
    
    The identification of $E_k(t)$ in Part (3) is Lemma 3.2(2) of \cite{BercoviciZhong2022Rdiag}. To investigate when $E_k(t)$ is nonempty, we first note that $\vert A\vert = \|A\|$ if and only if $a$ (which has the same $\ast$-distribution as $A$) is Haar unitary because $a$ is $\mathscr{R}$-diagonal; in this case, $\vert A\vert = 1$. If $A$ is Haar unitary, then the condition $\vert Z-\lambda \vert = \sqrt{t/k}\vert A\vert$ holds if and only if $1+\vert\lambda\vert^2-t/k=2\mathrm{Re}(\bar\lambda Z)$, which precisely happens when the law of $Z$ is supported only on a line (i.e. at most two points on the circle because $Z$ is unitary). We list all the possible cases below. Recall that $Z = u_0^{1/k}\otimes\left(\sum_{j=1}^k E_{j,j+1}\right)$. The cases when $E_k(t)$ is nonempty are:
    \begin{itemize}
        \item $k=1$ and $\mu_{u_0} = p\delta_{e^{i\alpha}}+(1-p)\delta_{e^{i\beta}}$ for some $\alpha, \beta$. In this case, that $\mathrm{Re}(\bar\lambda Z)$ is a scalar tells us the argument $\theta\mod 2\pi$ of $\lambda$ is either $(\alpha+\beta)/2$ or $\pi+(\alpha+\beta)/2$. Once we have $\theta$, we solve a quadratic equation for $\vert\lambda\vert$, where we look for nonnegative solutions. Therefore, the discriminant of the quadratic equation $t-\sin^2\frac{\alpha-\beta}{2}$ must be nonnegative. If $\cos\frac{\alpha-\beta}{2}>0$, then we only get solutions when $\theta=\frac{\alpha+\beta}{2}$, which gives 
        \[E_k(t) = \left\{\left(\cos\frac{\alpha-\beta}{2}\pm\sqrt{t-\sin^2\frac{\alpha-\beta}{2}}\right)e^{i\frac{\alpha+\beta}{2}}\right\};\] if $\cos\frac{\alpha-\beta}{2}<0$, then we only get solutions when $\theta=\pi+\frac{\alpha+\beta}{2}$, and 
        \[E_k(t) = \left\{\left(-\cos\frac{\alpha-\beta}{2}\pm\sqrt{t-\sin^2\frac{\alpha-\beta}{2}}\right)e^{i\pi+i\frac{\alpha+\beta}{2}}\right\}.\]
        \item $k=2$ an $\mu_{u_0} = \delta_{e^{i\alpha}}$. In this case, that $\mathrm{Re}(\bar\lambda Z)$ is a scalar tells us the argument $\theta\mod 2\pi$ of $\lambda$ is either $(\alpha+\pi)/2$ or $\pi+(\alpha-\pi)/2$. Either possible $\theta$ gives $\mathrm{Re}(\bar\lambda Z) = 0$. Thus, combining the two cases gives
        \[\lambda = \pm ie^{i\frac{\alpha}{2}}\sqrt{\frac{t}{2}-1}.\]
    \end{itemize} 

    Part (4) follows from Lemma 3.2(3) and Lemma 3.2(4) of \cite{BercoviciZhong2022Rdiag}. We can also identify $E_{2,k}(t)$ from Definition \ref{def.sets.Brown.add}. Part (5) is Lemma 3.2(5) of \cite{BercoviciZhong2022Rdiag}. Finally, Part (6) is Theorem 3.6 of \cite{BercoviciZhong2022Rdiag}.
\end{proof}
\begin{proposition}
    \label{prop.add.Cauchy}
    The Cauchy transform of $Z+\sqrt{t/k}A$ at $\lambda\in \Omega_k(t)\setminus S_k$ is given by
    \[\frac{1}{k}\sum_{-\frac{k}{2}<j\leq \frac{k}{2}}\int_{(\pi,\pi]}\frac{\bar\lambda-e^{-(2\pi i j+\alpha)/k}}{\vert\lambda-e^{(2\pi i j+\alpha)/k}\vert^2+\eta_k(t,\lambda^k)}\mu_{u_0}(d\alpha)\]
\end{proposition}
\begin{proof}
    The Cauchy transform is given by
    \[\frac{\partial}{\partial \lambda}\varphi(\log\vert(Z+\sqrt{t/k}A)-\lambda\vert^2).\]
    The theorem follows from Lemma 3.7 of \cite{BercoviciZhong2022Rdiag}, noting that $\psi_k(z)$ in Definition \ref{notation.psi.epsilon} is the $\omega_1^{(\lambda)}(0)$ in Lemma 3.7 of \cite{BercoviciZhong2022Rdiag}.
\end{proof}

Before we state our main theorem, we need two more formulas.
\begin{lemma}
\label{lem.domain.by.Tk}
\begin{enumerate}
\item Suppose that $z_j$ are the $k$-th roots of unity and $\lambda$ is not a $k$-th root of unity. Then,
\[\frac{1}{k^2}\sum_{j=1}^k\frac{1}{\vert z_j-\lambda\vert^2} = \frac{1}{k}\frac{\vert\lambda\vert^{2k}-1}{\vert\lambda\vert^2-1}\frac{1}{\vert\lambda^k-1\vert^2}=\frac{1}{k}\frac{\vert\lambda\vert^{2k}-1}{\vert\lambda\vert^2-1}\frac{1}{\vert\lambda\vert^{2k}-2\operatorname{Re}(\lambda^k)+1}.\]
\item Recall that $T_k$ is defined as in \eqref{eq.lifetime.k}. Then
\[T_k(u_0;r,\theta) = \frac{k}{\|Z-z^{1/k}\|^2}.\]
\end{enumerate}
\end{lemma}
\begin{proof}
Let $r>0$ and $\vert w\vert = 1$ so that $\lambda = rw$. We then do a partial fraction
\begin{align*}
    \frac{1}{\vert z_j-rw\vert^2} &= \frac{wz_j}{(rw-z_j)(rz_j-w)}\\
    &=\frac{1}{1-r^2}\left[\frac{z_j}{z_j-rw}-\frac{z_j}{z_j-w/r}\right]\\
    &=\frac{1}{1-\vert \lambda\vert^2}\left[\frac{\lambda}{z_j-\lambda}-\frac{\bar\lambda^{-1}}{z_j-\bar\lambda^{-1}}\right]
\end{align*}
The $z_j$ are the roots of the polynomial $p(u) = u^k-1$; $1/(z_j-x) = p'(x)/p(x)$ for any $x$. This shows
\begin{align*}
\sum_{j=1}^k\frac{1}{\vert z_j-rw\vert^2} &=\frac{1}{1-\vert\lambda\vert^2}\left[\lambda\sum_{j=1}^k\frac{1}{z_j-\lambda}-\bar\lambda^{-1}\sum_{j=1}^k\frac{1}{z_j-\bar\lambda^{-1}}\right]\\
 &=\frac{k}{\vert\lambda\vert^2-1}\left[\frac{\lambda^k}{\lambda^k-1}-\frac{(\bar\lambda^{-1})^k}{(\bar\lambda^{-1})^k-1}\right]\\
 &=\frac{k}{\vert\lambda\vert^2-1}\frac{\vert\lambda\vert^{2k}-1}{\vert\lambda^k-1\vert^2}.
\end{align*}
This proves (1).

To prove (2), recall that the law of $Z$ is the product measure of the laws of $u_0^{1/k}$ (see \eqref{eq.law.u.root}) and $\left(\sum_{j=1}^k E_{j,j+1}\right)$. We compute
\begin{align*}
    \frac{1}{k}\|(Z-z^{1/k})^{-1}\|_2^2 &= \frac{1}{k^2}\int_{-\pi/k}^{\pi/k}\sum_{-\frac{k}{2}<j\leq\frac{k}{2}}\frac{1}{\vert e^{i\alpha+2\pi i j/k} - r^{1/k}e^{i\theta/k}\vert^2}\mu_{u_0^{1/k}}(d\alpha)\\
    &= \frac{1}{k^2}\int_{-\pi}^{\pi}\sum_{-\frac{k}{2}<j\leq\frac{k}{2}}\frac{1}{\vert e^{2\pi i j/k} - r^{1/k}e^{i(\theta-\alpha)/k}\vert^2}\mu_{u_0}(d\alpha)\\
    &= \int_{-\pi}^{\pi}\frac{1}{k}\frac{r^2-1}{r^{2/k}-1}\frac{1}{\vert re^{i(\theta-\alpha)}-1\vert^2}\mu_{u_0}(d\alpha)
\end{align*}
where the last equality follows from Part (1). The last line of the above equation is the reciprocal of $T_k(u_0;r,\theta)$.
\end{proof}

\begin{theorem}
    \label{thm.RW.Brown.density}
    The Brown measure of $u_0b_k(t)$ is supported in the closure of the domain $\Sigma_k(u_0,t)\setminus \overline{D}_k(a,t)$, where $\Sigma_k(u_0,t)$ and $D_k(a,t)$ are defined in \eqref{eq.Sigma.intro} and \eqref{eq.Dkat} respectively.
    
    Denote by $S_k^k$ the image of $S_k$ under the map $\lambda\mapsto \lambda^k$. In the domain $\Sigma_k(u_0,t)\setminus (S_k^k\cup \overline{D}_k(a,t))$, the density $\rho_k(t,z)$ of the Brown measure of $u_0b_k(t)$ is given by
    \[\rho_k(t,z) = \frac{1}{\vert z\vert^{2-2/k}}\frac{\bar z^{1-1/k}}{k\pi}\frac{\partial}{\partial\bar z}\sum_{-\frac{k}{2}<j\leq \frac{k}{2}}\int_{(-\pi,\pi]}\frac{\bar z^{\frac{1}{k}}-e^{-i\frac{2\pi  j+\alpha}{k}}}{\vert z^{\frac{1}{k}}-e^{i\frac{2\pi j+\alpha}{k}}\vert^2+\eta_k(t,z)}\mu_{u_0}(d\alpha),\]
    where $z^{1/k}$ denotes the $k$-th root of $z$ with argument in $(-\pi/k,\pi/k]$.
\end{theorem}

The density $\rho_k(t,\cdot)$ is a subprobability density. The Brown measure of $u_0b_k(t)$ may have atoms at $S_k^k$. If $k$ is large, $S_k$ is empty by Proposition \ref{prop.Brown.add.support}(2). We also don't know whether the Brown measure has mass on the boundary. 

The formula for $\rho_k(t,z)$ is independent of the choice of the $k$-th root of $z$. However, it is more convenient to choose the $k$-th root of $z$ with argument in $(-\pi/k,\pi/k]$ when we investigate the limit of the Brown measure as $k\to\infty$. 

\begin{proof}
    We first compute 
    \[\mathbb{C}\setminus  (E_{1,k}(t)\cup E_{2,k}(t)) = \left\{\lambda: \frac{t}{k}>\frac{1}{\|(Z-\lambda)^{-1}\|_2^2}\right\}\setminus\left\{\lambda:\vert\lambda\vert^2\leq \frac{t}{k\|a^{-1}\|_2^2}-1\right\}.\]
    Lemma \ref{lem.domain.by.Tk}(2) then shows that the image of $\mathbb{C}\setminus  (E_{1,k}(t)\cup E_{2,k}(t))$ under the map $\lambda\mapsto \lambda^k$ has a formula given by $\Sigma_k(u_0,t)\setminus \overline{D}_k(a,t)$.

    The domain $\Sigma_k(u_0,t)\setminus (S_k^k\cup \overline{D}_k(a,t))$ is the image of $\Omega_k(t)\setminus S_k$ under the map $\lambda\mapsto \lambda^k$; hence, the Brown measure of $u_0b_k(t)$ has a density on $\Sigma_k(u_0,t)\setminus (S_k^k\cup \overline{D}_k(a,t))$ by Proposition \ref{prop.Brown.add.support}(6). By Proposition \ref{prop.add.Cauchy}, the density of the Brown measure of $Z+\sqrt{t/k}A$ on $\Omega_k(t)\setminus S_k$ is
    \begin{equation}
        \label{eq.CauchyTrans.ZA}
        \frac{1}{\pi}\frac{\partial}{\partial\bar\lambda}\frac{1}{k}\sum_{-\frac{k}{2}<j\leq \frac{k}{2}}\int_{(\pi,\pi]}\frac{\bar\lambda-e^{-(2\pi i j+\alpha)/k}}{\vert\lambda-e^{(2\pi i j+\alpha)/k}\vert^2+\eta_k(t,\lambda^k)}\mu_{u_0}(d\alpha).
    \end{equation}
    
    We then push-forward the Brown measure of $Z+\sqrt{t/k}A$ by the map $\lambda\to\lambda^k$, which is a $k$-to-$1$ map. By writing $z=\lambda^k$ and $z^{1/k}$ be the $k$-th root of $z$ with argument $(-\pi/k,\pi/k]$, the Jacobian of the map $\lambda\mapsto \lambda^k$ is $k^2\vert z\vert^{2-2/k}$. the density $\rho_k(t,\cdot)$ of $u_0b_k(t)$ in $\Sigma_k(u_0,t)\setminus (S_k^k\cup \overline{D}_k(a,t))$ is given by
    \[k\cdot\frac{1}{k^2\vert z\vert^{2-2/k}}\frac{1}{\pi}\frac{\partial}{\partial\bar\lambda}\frac{1}{k}\sum_{-\frac{k}{2}<j\leq \frac{k}{2}}\int_{(\pi,\pi]}\frac{\bar z^{\frac{1}{k}}-e^{-(2\pi i j+\alpha)/k}}{\vert z^{\frac{1}{k}}-e^{(2\pi i j+\alpha)/k}\vert^2+\eta_k(t,z)}\mu_{u_0}(d\alpha).\]
    The conclusion of the theorem follows from the change-of-variable formula $\partial f/\partial\bar\lambda = k\bar z^{1-1/k}(\partial f/\partial\bar z)$ for any differentiable function $f$.
\end{proof}

\subsection{Domain with trivial initial condition} \hfill

\medskip

In this section, we describe the domain when the initial condition $u_0$ is the identity using graphs of function on $\mathbb{C}$. We define a continuous function on $\mathbb{C}$ in polar coordinates by
\begin{equation}
    \label{eq.Tk.def}
    T_k(r,\theta) = \frac{k(r^{2/k}-1)}{r^2-1}(r^2-2r\cos\theta+1)
\end{equation}
where we identify $(0,\theta)$ as the same point for any $\theta$ and $T_k(0,\theta)=k$. The function $T_k(r,\theta) = T_k(1; r,\theta)$. We also understand that $T_k(1,\theta) = 2(1-\cos\theta)$. The pointwise limit of the above $T_k$ is the function $T$ in \cite[Eq.(3.1)]{DHKBrown}, which describes the domain $\Sigma_\infty(t)$ of the free multiplicative Brownian motion $b(t)$.

For the random walk $b_k(t)$ with trivial initial condition, we denote the domain $\Sigma_k(1,t)$ identified in Theorem \ref{thm.RW.Brown.density} by $\Sigma_k(t)$. The main theorem in this section is Theorem \ref{thm.Gammakt.rmin}, which describes the domain $\Sigma_k(t)$ identifies the support $\overline{\Sigma_k(t)\setminus\overline{D}_k(a,t)}$ of the Brown measure of $u_0b_k(t)$ using functions on $\mathbb{C}$ in the polar coordinates. Proposition \ref{prop.compare.to.DHK} compares $\Sigma_k(t)$ to the domain $\Sigma_\infty(t)$.

% Applying Theorem \ref{thm.RW.Brown.density} with the trivial initial condition, $\Sigma(k,a,t)$ is given by
% \begin{equation}
%     \label{eq.trivial.domain}
%     \Sigma(k,a,t) = \{re^{i\theta}: T_k(r,\theta)<t\}\setminus \{z: \vert z\vert^{2/k}\leq t\cdot (k\|a^{-1}\|_2^2)^{-1}-1\}.
% \end{equation}
We start by investigating the behavior of $T_k(r,\theta)$ with a fixed $\theta$. We first start with $r\in(0,1)$ and will study $T_k(r,\theta)$ for $r\geq 1$ later.

\begin{proposition}
    \label{prop.convex.Tk}
For all $k\geq 4$, $T_k$ is strictly convex in $r\in(0,1)$ for any fixed $\theta$. If $k=2$, $T_k(r,\theta)$ is strictly convex in $r\in(0,1)$ for any fixed $\theta$ such that $\cos(\theta)>-1$; if $\cos(\theta)=-1$, $T_k(r,\theta) = 2(1+r)$. If $k=3$, for every $\theta$, there is an $\tilde r\in(0,1]$ such that $T_k(r,\theta)$ is strictly convex for $r\in(0,\tilde r)$ and is increasing for $r\in[\tilde r,1)$; $\tilde r< 1$ unless $\cos\theta = 1$.

Furthermore, unless $k=2$ and $\cos\theta = -1$,
\[\lim_{r\to 0^+}\frac{\partial}{\partial r}T_k(r,\theta) = -\infty,\quad \lim_{r\to 1}\frac{\partial}{\partial r}T_k(r,\theta) = \frac{2(1-\cos\theta)}{k}.\]
Therefore, for $r\in[0,1]$, $T_k(r,\theta)$ has a unique minimum in $(0,1)$ except $\cos\theta = 1$, in which case $T_k(r,\theta)$ has a unique minimum at $r=0$.
\end{proposition}
\begin{proof}
The $k=3$ case follows from Lemma \ref{lem.k=3.lifetime}. We restrict to $k\neq 3$ for the rest of the proof.

Observe that $T_k(r,\theta)$ is linear in $\cos\theta$; indeed, $T_k(r,\theta)$ is a convex combination of $T_k(r,0)$ and $T_k(r,\pi)$. For any $r\in(0,1)$,
\begin{equation}
\label{eq.Tkr.bound}
\frac{\partial^2}{\partial r^2}T_k(r,\theta)=\frac{1+\cos\theta}{2}\frac{\partial^2}{\partial r^2}T_k(r,0)+\frac{1-\cos\theta}{2}\frac{\partial^2}{\partial r^2}T_k(r,\pi).
\end{equation}

 For $k\geq 4$, Lemmas \ref{lem.Tkr.convex} and \ref{lem.Tk-r.convex} show that $r\mapsto T_k(r,0)$ and $r\mapsto T_k(r,\pi)$ are strictly convex in $r\in(0,1)$. The right-hand side of \eqref{eq.Tkr.bound} is positive. If $k=2$, the right-hand side of \eqref{eq.Tkr.bound} shows that $T_k(r,\theta)$ is convex in $r\in(0,1)$ unless $\cos\theta = -1$. In the case $k=2$ and $\cos\theta = -1$, $T_k(r,\theta) = 2(1+r)$ by Lemma \ref{lem.Tk-r.convex}.

The limits as $r\to 0^+$ and $r\to 1$ follow from direct computation. The uniqueness of the minimum follows from convexity.
\end{proof}

\begin{lemma}
    \label{lem.Tkr.convex}
For any $k\geq 2$, the function $r\mapsto T_k(r,0)$ is convex for $r\in(0,1)$.
\end{lemma}
\begin{proof}
    We calculate, with the help of Mathematica,
    \[\frac{d^2T_k(r,0)}{dr^2}= \frac{2k}{(1+r)^3}f(r)\]
    for some function $f$ such that $f(1) = 8/k$ and 
    \[f'(r) = -\frac{(k-2)r^{-3+\frac{2}{k}}(1+r)^2[(k+2)r+2(k-1)]}{k^3}\]
    Note that $f'(r)\leq 0$ for all $k\geq 2$ and $r>0$. For all $r\in(0,1)$,
    $(d^2/dr^2)T_k(r)>0$, completing the proof.
\end{proof}
\begin{lemma}
    \label{lem.Tk-r.convex}
    When $k=2$, $T_k(r,\pi) = 2(1+r)$. For any $k\geq 4$, the function $r\mapsto T_k(r,\pi)$ is convex for $r\in(0,1)$.
\end{lemma}
\begin{proof}
 When $k=2$, $T_k(r,\pi) = 2(1+r)$ by direct calculation. 
 
 For the rest of the proof, we assume $k\geq 4$. We calculate, with the help of Mathematica,

 \[\frac{\partial^2T_k(r,\pi)}{\partial r^2}= \frac{2k}{(1-r)^3}f(r).\]
    where $f$ is a function satisfying $f(1) = 0$ and
\[
f'(r) = \frac{r^{-3+\frac{2}{k}}(k-2)(r-1)^2[(k+2)r+2(1-k)]}{k^3}
\]
Note that $f'(r)<0$ for all $k\geq 4$ and $r\in(0,1)$. Hence, for all $r\in(0,1)$,
$(d^2/dr^2)T_k(-r)>0$ for all $r\in(0,1)$, completing the proof.
\end{proof}

\begin{lemma}
    \label{lem.k=3.lifetime}
For $k=3$, for each fixed $\theta$,
\[T_k(r,\theta) = \frac{3(r^{2/3}-1)(r^2-2r\cos\theta+1)}{r^2-1}, \quad r\in[0,1]\]
has an $\tilde r\in(0,1]$ such that $T_k(r,\theta)$ is strictly convex for $r\in(0,\tilde r)$ and is increasing for $r\in[\tilde r,1)$. In particular, for each $\theta$, $T_k(\cdot,\theta)$ has a unique minimum in $(0,1)$ except when $\cos\theta = 1$, in which case the unique minimum is at $1$.

Furthermore,
\begin{equation}
    \label{eq.k3.partial.limit}
    \lim_{r\to 0^+}\frac{\partial}{\partial r}T_k(r,\theta) = -\infty,\quad \lim_{r\to 1}\frac{\partial}{\partial r}T_k(r,\theta) = \frac{2(1-\cos\theta)}{k}.
\end{equation}
\end{lemma}
From the proof, we know that if $\cos\theta\geq 0$, $\tilde r$ can be chosen to be $1$; if $\cos\theta< 0$, $\tilde r$ can be chosen to be $0.9^3$. The specific number $0.9^3$ has no significance. We only want to show that $T_k(\cdot,\theta)$ has a unique minimum in $(0,1)$ except when $\cos\theta = 1$, at which the unique minimum occurs at $1$.
\begin{proof}
For convenience for the rest of the proof, let
\[P(a,r) = \frac{(r^{2/3}-1)(r^2+2ar+1)}{r^2-1},\quad a\in[-1,1],\, r\geq 0\]
so that $T_k(r,\theta) = 3 P(-\cos\theta, r)$. We compute
\begin{equation}
    \label{eq.dPdr}
    \frac{\partial}{\partial r}P(a,r) =\frac{2(-1-2r^{2/3}+3r^{4/3}+2r^2+r^{8/3}+a(3r^{1/3}+r-r^{5/3}))}{3r^{1/3}(1+r^{2/3}+r^{4/3})^2}.
\end{equation}
The last assertion about limits follows from direct calculations.

Note that $P(a,r) = \frac{1+a}{2}P(1,r)+\frac{1-a}{2}P(-1,r)$ is a convex combination of $P(1,r)$ and $P(-1,r)$. We separate our argument for $a\leq 0$ and $a\geq 0$. We first assume $a\leq 0$. When $a=-1$, $P(-1, r) = (1/3)T_k(r,0)$ is convex for $r\in(0,1)$ by Lemma \ref{lem.Tkr.convex}. Meanwhile, $(\partial^2/\partial r^2)P(0,r) >0$ for all $0<r<1$. Thus, when $a\leq 0$, $P(a,r)$, as a convex combination of $P(-1,r)$ and $P(0,r)$, is a convex function for $r\in(0,1)$  Combining with the limits in \eqref{eq.k3.partial.limit}, $P(a,\cdot)$ has a unique minimum in $(0,1)$ for each $-1<a\leq 0$. When $a=-1$, $P(-1,\cdot)$ attains its minimum at $r=1$.

We now turn to the case $a\geq 0$. Our strategy is to show that $r\mapsto P(a,r)$ is increasing for $r\in[0.9^3, 1)$ and is convex for $r\in(0,0.9^3)$. Recall from the preceding paragraph that $(\partial^2/\partial r^2)P(0,r) >0$ for all $0<r<1$. Since. by \eqref{eq.dPdr}, $(\partial P/\partial r)(0,0.9^3)>0$ and $P(0,\cdot)$ is convex, $(\partial/\partial r)P(0,r)>0$ for all $r\geq 0.9^3$. We then combine the fact that the second-order derivative
\[\frac{\partial^2}{\partial a\partial r}P(a,r) = \frac{3r^{1/3}+(r-r^{5/3})}{3r^{1/3}(1+r^{2/3}+r^{4/3})^3}>0\]
for all $0<r<1$ to conclude $(\partial P/\partial r)(a,r)>0$
for all $a\geq 0$, $r\geq 0.9^3$. For every $a\geq 0$, $P(a,\cdot)$ is increasing on $[0.9^3,1]$.

We claim that $P(1,r)$ is convex for $r\in(0,0.9^3)$. We compute
\[\frac{\partial^2 P}{\partial r^2}(1,r) =\frac{2}{9}\frac{q(r)}{(1+r^{1/3}+r^{2/3})r^{4/3}}\]
where $q(r)=1+3r^{1/3}+6r^{2/3}-3r-6r^{4/3}-3r^{5/3}-r^2$. We can then verify that $q''(r)<0$; hence $q$ is concave in $(0,0.9^3)$ and $q(r)>\min(q(0), q(0.9^3))>0$. This proves the claim that $P(1,\cdot)$ is convex on $(0,0.9^3)$. 

In summary, $P(a,r)$, as a convex combination of $P(1,r)$ and $P(0,r)$, is convex for all $r\in(0,0.9^3)$ and $a\in[0,1]$, and $P(a,r)$ is increasing for all $r\in[0.9^3,1]$. For each $a\in[0,1]$, $P(a,\cdot)$ has a unique minimum in $(0,0.9^3)$. 
\end{proof}

After understanding the behavior of $T_k$ for $r\in(0,1]$, we prove that, for each $\theta$, $T_k(r,\theta)$ is increasing in $r$ for $r>1$.
\begin{proposition}
    \label{prop.Tk.increasing}
 For every $\theta$, $T_k(r,\theta)$ is increasing for $r>1$.
\end{proposition}
\begin{proof}
Let $\lambda = r^{1/k}e^{i\theta/k}$. Recall that
\[\frac{1}{T_k(r,\theta)} = \frac{1}{k^2}\sum_{j=1}^k\frac{1}{\vert e^{2\pi i j/k}-\lambda\vert^2} = \frac{1}{k^2}\sum_{j=1}^k\frac{1}{r^{2/k}-2r^{2/k}\cos\left(\frac{2\pi j-\theta}{k}\right)+1}.\]
For each $\theta$, $T_k(r,\theta)$ is increasing for $r>1$, because each term in the sum in the right-hand side has negative partial derivative with respect to $r$ for all $r>1$. 
\end{proof}

A main step to prove Theorem \ref{thm.Gammakt.rmin} is the following theorem, combining Propositions \ref{prop.convex.Tk} and \ref{prop.Tk.increasing}.
\begin{theorem}
    \label{thm.unimodality}
For each $\theta$, there exists a unique $r_0\leq 1$ such that $T_k(r_0,\theta)\leq T_k(r,\theta)$ for all $r>0$. Moreover, $T_k(r, \theta)$ is decreasing for $0<r<r_0$ and increasing for $r>r_0$. This $r_0$ is positive except $\theta = \pi$ and $k=2$, in which case $r_0 = 0$. This $r_0$ is $1$ when $\theta = 0$.
\end{theorem}
\begin{proof}
In the case $k=2$ and $\theta=\pi$, $T_k(r,\theta) = 2(1+r)$. We take $r_0 = 0$. If $\theta = 0$, then $r_0 = 1$ by Proposition \ref{prop.convex.Tk}.

For all the other cases, by Proposition \ref{prop.convex.Tk}, $r\mapsto T_k(r,\theta)$ is convex in $r$ for $r\in(0,\tilde r)$ for some $\tilde r\in (0,1]$, decreasing for small enough $r$, and increasing for $r$ close to $1$. Thus, there exists a unique $r_0\in(0,\tilde r)$ such that $T_k(r_0,\theta)\leq T_k(r,\theta)$ for all $0<r<1$. Now, by Proposition \ref{prop.Tk.increasing}, $T_k(r,\theta)$ is increasing for $r>1$ for each $\theta$. The theorem is established since $r\mapsto T_k(r,\theta)$ is continuous at $r=1$.
\end{proof}

\begin{definition}
\label{def.r.min.+.-}
For each $\theta$, define $r_k^{\min}(\theta)$ to be the $r_0$ given in Theorem \ref{thm.unimodality}.

Define $r_k^+(t,\theta)$ on $\{\theta: T_k(r_k^{\min}(\theta),\theta)<t\}$ to be the unique number such that 
\[r_k^{\min}(\theta)<r_k^+(t,\theta)\]
and $T_k(r_k^+(t,\theta),\theta)=t$. This function $r_k^+$ is well-defined by Theorem \ref{thm.unimodality} and the fact that $\lim_{r\to\infty}T_k(r,\theta) = \infty$.

For any $\theta\in \{\theta: T_k(r_k^{\min}(\theta),\theta)<t\}$, if there exists $r_0<r_k^{\min}(\theta)$ such that $T_k(r_k^-(t,\theta),\theta)=T_k(r_0,\theta)=t$. This $r_0$ must be unique by Theorem \ref{thm.unimodality} and we define $r_k^-(t,\theta) = r_0$. If such an $r_0$ does not exist, we define $r_k^-(t,\theta) = 0$.

Finally, we extend the definitions of $r_k^+(t,\theta)$ and $r_k^-(t,\theta)$ to $\{\theta: T_k(r_k^{\min}(\theta),\theta)\geq t\}$ by setting
\[r_k^+(t,\theta)=r_k^-(t,\theta) = r_k^{\min}(\theta).\]
\end{definition}

\begin{theorem}
    \label{thm.Gammakt.rmin}
    If $t\leq k$, then  $D_k(a,t)$ is empty and
    \[\Sigma_k(t)= \{re^{i\theta}:T_k(r_k^{\min}(\theta),\theta)<t, r_k^-(t,\theta)<r<r_k^+(t,\theta)\}.\]
    In this case, when $t<T_k(r_k^{\min}(\pi),\pi)$, $\Sigma_k(t)$ is a topological disk; at $t=T_k(r_k^{\min}(\pi),\pi)$, $\Sigma_k(t)$ is a topological disk but the closure $\overline{\Sigma_k(t)}$ is a topological annulus; when $T_k(r_k^{\min}(\pi),\pi)<t<k$, $\Sigma_k(t)$ is a topological annulus; when $t=k$, $\Sigma_k(t)$ is a topological annulus but its closure is a topological disk.
    
    If $k< t<k\|a^{-1}\|_2^2$, then $D_k(a,t)$ is empty and
    \[\Sigma_k(t) = \{re^{i\theta}: r<r_k^+(t,\theta)\},\]
    which is a topological disk.
    
    If $t\geq  k\|a^{-1}\|_2^2$, then $D_k(a,t)$ is nonempty, so that
    \[\Sigma_k(t)\setminus \overline{D}_k(a,t) = \{re^{i\theta}:r<r_k^+(t,\theta)\}\setminus \{z: \vert z\vert^{2/k}\leq t\cdot (k\|a^{-1}\|_2^2)^{-1}-1\}\]
    is a topological annulus.
\end{theorem}
The $R$-diagional element $a$ is chosen such that $\|a\|_2 = 1$. By Cauchy--Schwarz inequality, $1\leq \|a^{-1}\|_2$. We must have $T_k(r_k^{\min}(\pi),\pi)\leq k$; in fact, $T_k(r_k^{\min}(\theta),\theta)<k$ except $\theta = \pi$ and $k=2$ by Theorem \ref{thm.Gammakt.rmin}.

Before we prove Theorem \ref{thm.Gammakt.rmin}, we need the following lemma.

\begin{lemma}
\label{lem.domain.theta.interval}
\begin{enumerate}
    \item $T_k(r_k^{\min}(\theta),\theta)$ increases in $\theta$ for $\theta\in(0,\pi)$ and decreases in $\theta$ for $\theta\in(-\pi,0)$. 
    \item For any $t>0$, 
\[\{\theta\in[-\pi,\pi]: T_k(r_k^{\min}(\theta),\theta)<t\}\]
\end{enumerate}

 is an interval symmetric about $0$.
\end{lemma}
\begin{proof}
    We first calculate
    \[\frac{\partial T_k}{\partial \theta}(r,\theta) = 2kr\sin\theta\cdot \frac{1-r^{2/k}}{1-r^2}.\]
    By definition of $r_k^{\min}(\theta)$, $T_k(\cdot, \theta)$ has a local minimum at $r_k^{\min}(\theta)$. We differentiate
    \begin{align*}
        \frac{d}{d\theta}T_k(r_k^{\min}(\theta),\theta)&=\frac{\partial T_k}{\partial r}(r_k^{\min}(\theta),\theta)\frac{dr_k^{\min}(\theta)}{d\theta}+\frac{\partial T_k}{\partial \theta}(r_k^{\min}(\theta),\theta)\\
        &= \frac{\partial T_k}{\partial \theta}(r_k^{\min}(\theta),\theta)\\
        &=2kr_k^{\min}(\theta)\sin\theta\cdot\frac{1-r_k^{\min}(\theta)^{2/k}}{1-r_k^{\min}(\theta)^2}.
    \end{align*}
    Thus, $T_k(r_k^{\min}(\theta),\theta)$ increases in $\theta$ for $\theta\in(0,\pi)$ and $T_k(r_k^{\min}(\theta),\theta)$ decreases in $\theta$ for $\theta\in(-\pi,0)$. 
    
    Meanwhile, $r_k^{\min}(\theta) = r_k^{\min}(-\theta)$ because $T_k(r,\theta) = T_k(r,-\theta)$. The set $\{\theta\in[-\pi,\pi]: T_k(r_k^{\min}(\theta),\theta)<t\}$ is symmetric about $0$. Since $T_k(1,0) = 0$, $0$ is always in the set $\{\theta: T_k(r_k^{\min}(\theta),\theta)<t\}$. Hence $\{\theta\in[-\pi,\pi]: T_k(r_k^{\min}(\theta),\theta)<t\}$ must be an interval symmetric about $0$.
\end{proof}

\begin{proof}[Proof of Theorem \ref{thm.Gammakt.rmin}]
    Since $T_k(0,\theta) = k$, by Theorem \ref{thm.unimodality} and Definition \ref{def.r.min.+.-}, $r_k^-(t,\theta)=0$ for all $\theta$ if $t\geq k$, and $r_k^-(t,\theta)>0$ for all $\theta$ if $t<k$. The classification of $\Sigma_k(t)\setminus\overline{D}_k(a,t)$ then follows from \eqref{eq.Sigma.intro} and \eqref{eq.Dkat}.
    
    It remains to prove the topologies of $\Sigma_k(t)\setminus \overline{D}_k(a,t)$; it is clear that $\overline{D}_k(a,t)$ is empty for $t\leq k\|a^{-1}\|_2^2$. If $t<T_k(r_k^{\min}(\theta),\theta)$, then the set $\{\theta\in[-\pi,\pi]: T_k(r_k^{\min}(\theta),\theta)<t\}$ is an interval symmetric about $0$ that does not contain $\pi$ by Lemma \ref{lem.domain.theta.interval}. In this case, $\Sigma_k(t)$ is a topological disk. 
    
    If $t = T_k(r_k^{\min}(\theta),\theta)$, then the set $\{\theta\in[-\pi,\pi]: T_k(r_k^{\min}(\theta),\theta)<t\}$ is the interval $(-\pi,\pi)$, again by Lemma \ref{lem.domain.theta.interval}. In this case, the open set $\Sigma_k(t)$ is a topological disk, but the complement of its closure $\overline{\Sigma_k(t)}$ has two connected components, one bounded and one unbounded. The closure $\overline{\Sigma_k(t)}$ is thus a topological annulus.
    
    If $T_k(r_k^{\min}(\theta),\theta)<t<k$, then the set $\{\theta\in[-\pi,\pi]: T_k(r_k^{\min}(\theta),\theta)<t\}$ is the interval $[-\pi,\pi]$. The domain $\Sigma_k(t)$ is a topological annulus because its complement has two connected components, one bounded and one unbounded. When $t=k$, $r_k^-(t,\theta)=0$; therefore, $\overline{\Sigma_k(t)}=\{re^{i\theta}: r\leq r_k^+(t,\theta)\}$ is a topological disk.

    In the last two cases $k<t<k\|a^{-1}\|_2^2$ and $t\geq k\|a^{-1}\|_2^2$, $\overline{D}_k(a,t)$ is nonempty. It is then straightforward to see that the topologies of $\Sigma_k(t)\setminus \overline{D}_k(a,t)$ are topological disk and topological annulus respectively.
\end{proof}

Before we end this section, we prove the following result to compare the domain for the random walk $b_k(t)$ and the domain for the free multiplicative Brownian motion $b(t)$, both with trivial initial condition.

Recall that the law of $u$ is denoted by $\mu_{u_0}$. By \cite[Theorem 4.10]{HoZhong2020Brown}, the Brown measure of $u_0b(t)$ has full measure and positive density on the domain $\Sigma_\infty(u_0,t)$. The domain $\Sigma_\infty(u_0,t)$ is given by
\begin{equation}
    \label{eq.T.def}
    \Sigma_\infty(u_0,t)=\{re^{i\theta}: T(u_0;r,\theta)<t\}
\end{equation}
where
\[T(u_0;r,\theta) = \frac{\log r^2}{r^2-1}\left(\int_{(\pi,\pi]}\frac{1}{\vert re^{i\theta}-e^{i\alpha}\vert^2}\mu_{u_0}(d\alpha)\right)^{-1}.\]
If $r=1$, $\log(r^2)/(r^2-1)$ is understood as $1$. By \cite[Theorem 1.2]{HoZhong2020Brown} that there exists a function $r_{u_0}(t,\cdot):(-\pi,\pi]\to[1,\infty)$ such that
    \begin{equation}
        \label{eq.Sigma.u.bdry.funct}
        \Sigma_\infty(u_0,t) = \{re^{i\theta}: r_{u_0}(t,\theta)^{-1}<r<r_{u_0}(t,\theta)\}.
    \end{equation}
When $u_0 = 1$, we write $\Sigma_\infty(t)$ instead of $\Sigma_\infty(u_0,t)$, $r(t,\theta)$ instead of $r_{u_0}(t,\theta)$ and write $T(r,\theta)$ instead of $T(u_0;r,\theta)$. Corollary 3.5 of \cite{DHKBrown} computes the $\theta$ such that the ray $e^{i\theta}(0,\infty)$ with argument $\theta$ intersects $\Sigma_\infty(t)$; indeed,
    \begin{equation}
        \label{eq.Sigma.infty.def}
        \Sigma_\infty(t) = \{re^{i\theta}: \theta\in I_\infty(t), r_\infty(t,\theta)^{-1}<r<r_\infty(t,\theta)\}
    \end{equation}
    where $I_\infty(t) =\left(-\arccos\left(1-\frac{t}{2}\right),\arccos\left(1-\frac{t}{2}\right)\right)$ if $t\leq 4$ and $I_\infty = [-\pi,\pi]$ if $t>4$. 

We first compare the function $T_k$ defined in \eqref{eq.lifetime.k} with the function $T$.
\begin{lemma}
    \label{lem.Tk.T.compare}
    For any unitary initial condition $u_0$, $T_k(u_0;r,\theta)<T(u_0;r,\theta)$ for $r<1$, $T_k(u_0;r,\theta)>T(u_0;r,\theta)$ for $r>1$, and $T_k(u_0;r,\theta)=T(u_0;r,\theta)$ for $r=1$.

    As a consequence, $\Sigma_k(u_0,t)\cap (\mathbb{C}\setminus \mathbb{D})\subset \Sigma_\infty(u_0,t)\cap (\mathbb{C}\setminus \mathbb{D})$ and $\Sigma_\infty(u_0,t)\cap \mathbb{D}\subset \Sigma_k(u_0,t)\cap \mathbb{D}$.
\end{lemma}
\begin{proof}
    By the Fundamental Theorem of Calculus,
    \begin{align*}
        &\log r^2 = \int_1^r \frac{2}{s}\,ds\\
        &k(r^{2/k}-1) = \int_1^r \frac{2}{s}s^{2/k}\,ds.
    \end{align*}
    By comparing the two integrands above, we can immediately see that $T_k(u_0;r,\theta)>T(u_0;r,\theta)$ for $r>1$. Now, for $r<1$, we have $\log r^2 < k(r^{2/k}-1)$. This shows $T_k(u_0;r,\theta)<T(u_0;r,\theta)$ for all $r<1$ due to the $r^2-1$ in the denominators of $T_k$ and $T$. Finally, the values of $T_k$ and $T$ at $r=1$ are defined by limit as $r\to 1$; the limits of $(\log r^2)/(r^2-1)$ and $k(r^{2/k}-1)/(r^2-1)$ as $r\to 1$ are both $1$.

    The last assertion follows from that $\Sigma_k(u_0,t)$ and $\Sigma_\infty(u_0,t)$ consist of points $re^{i\theta}$ where $T_k(u_0;r,\theta)<t$ and $T(u_0;r,\theta)<t$ respectively.
\end{proof}
    
\begin{proposition}
    \label{prop.compare.to.DHK}
    For all $\theta\in I_\infty(t)$, $r_k^-(t,\theta)<r_\infty(t,\theta)^{-1}< r_k^+(t,\theta)<r_\infty(t,\theta)$. Moreover, if $t\leq 4$ and $\cos\theta = 1-t/2$, then $r_k^+(t,\theta)=1$, and $r_k^-(t,\theta)<1$.
\end{proposition}
\begin{proof}  
    For any $\theta\neq 0$,
    \[T_k(r_k^{\min}(\theta),\theta)<T_k(1,\theta) = 2(1-\cos\theta).\]
    Thus, $T_k(r_k^{\min}(\theta),\theta)<t$ for all $\theta\in I_\infty(t)$,
    so that $r_k^-(\theta)<r_k^+(\theta)$ for all $\theta\in I_\infty(t)$. The above display equation also shows that $r_k^-(t,\theta)<1$ when $t\leq 4$ and $\cos\theta = 1-t/2$ because $r_k^-(\theta)<r_k^{\min}(\theta)<1$. Since $T_k(r,\theta)$ is increasing for $r>1$ by Proposition \ref{prop.Tk.increasing}, $r_k^+(t,\theta)=1$ if $t\leq 4$ and $\cos\theta = 1-t/2$.
    
    The domain $\Sigma_\infty(t)$ for $b(t)$ can be identified using the function $T$ in \eqref{eq.T.def} with $u_0=1$. For $r<1$, we have $T_k(r,\theta)<T(r,\theta)$ by Lemma \ref{lem.Tk.T.compare}. Hence, $T_k(r_\infty(t,\theta)^{-1},\theta)<t$, which shows $r_k^-(t,\theta)<r_\infty(t,\theta)^{-1}$. Similarly, we have $T_k(r,\theta) >T(r,\theta)$ for all $r>1$. Thus, $T_k(r_\infty(t,\theta),\theta)>t$, and we must have $r_k^+(t,\theta)<r_\infty(t,\theta)$.
\end{proof}

\subsection{Explicit formulas when $k=2$\label{section.explicit.formulas}} \hfill

\medskip

In this section, we compute the Brown measure of $b_2(t)$ and $k=2$ when the $\mathscr{R}$-diagonal distribution $a$ is Haar unitary or circular.

We first compute the Cauchy transform of $\tilde\mu_{\vert\lambda-Z\vert}$ in the following lemma. In this section, we apply the lemma with $\mu_{u_0} = \delta_1$ to study the random walk $b_k(t)$ with trivial initial condition. 
\begin{lemma}
    \label{lem.Z.Cauchy}
     The Cauchy transform of $\tilde\mu_{\vert\lambda-Z\vert}$ is given by
    \[G_{\tilde\mu_{\vert\lambda-Z\vert}}(z) = \frac{1}{k}\sum_{j=1}^k\int_{(-\pi,\pi]}\frac{z}{z^2-\vert\lambda-e^{\frac{i(2\pi j+\alpha)}{k}}\vert^2}\mu_{u_0}(d\alpha).\]
\end{lemma}
\begin{proof}
    Recall that $Z = u_0^{1/k}\otimes\left(\sum_{j=1}^k E_{j,j+1}\right)$. Note that $u_0^{1/k}$ and $\left(\sum_{j=1}^k E_{j,j+1}\right)$ are unitary. The law of $u_0^{1/k}$ is given by \eqref{eq.law.u.root} and the law of $\left(\sum_{j=1}^k E_{j,j+1}\right)$ is uniform on the $k$-th roots of unity. 
   
    Thus, $Z$ is unitary, and the law of $Z$ is the product measure of the laws of $u_0^{1/k}$ and $\left(\sum_{j=1}^k E_{j,j+1}\right)$. Therefore,
    \[G_{\tilde\mu_{\vert\lambda-Z\vert}}(z) = \frac{1}{2k}\sum_{j=1}^k\int_{(-\pi/k,\pi/k]}\left(\frac{1}{z-\vert\lambda-e^{\frac{2\pi i j}{k}}e^{i\alpha}\vert}+\frac{1}{z+\vert \lambda-e^{\frac{2\pi i j}{k}}e^{i\alpha}\vert}\right)\mu_{u_0^{1/k}}(d\alpha).\]
    The conclusion follows from \eqref{eq.law.u.root} and an algebraic manipulation.
\end{proof}

\begin{proposition}
    Assume the $\ast$-distribution of $a$ is Haar unitary and $u_0=1$ in \eqref{eq.RW.unitary}. The density of the Brown measure of the two-step random walk $b_2(t)$ is given by
    \[\frac{1}{4\pi\vert z\vert}\left[ \frac{4t\bigg[\big(2(\vert z\vert+1)-t\big)^2+8\big(\mathrm{Re}(z)+\vert z\vert\big)\bigg]}{\big[4\vert z-1\vert^2+t^2-4t(\vert z\vert+1)\big]^2}-\frac{1}{\mathrm{Re}(z)+\vert z\vert}\right]\]
    for all $z\in\Sigma_2(t)\setminus \overline{D}_2(a,t)$.
\end{proposition}
\begin{proof}
    Let $\lambda\in\Omega_2(t)$ and $z = \lambda^2 $. We first compute $\eta = \eta_k(t,z)$. By Definition \ref{notation.psi.epsilon}, we need to first calculate $\psi_k(t,z)$, which is the Denjoy--Wolff point of $H_{\tilde\mu_{\vert a\vert}}\circ H_{\tilde\mu_{\vert Z-z^{1/2}}\vert}$. Since $\vert a\vert = 1$, the function $H_{\tilde\mu_{\vert a\vert}}(\zeta) = -t/(2\zeta)$. Thus, $\eta =\eta_k(t,z) = -\psi_k(t,z)^2$ satisfies the equation
    \[\frac{1}{\vert \lambda - 1\vert^2+\eta}+\frac{1}{\vert \lambda +1\vert^2+\eta} = \frac{2}{\eta+t/2}.\]

    To compute the density of $b_2(t)$, we first calculate the density of the Brown measure of $Z+\sqrt{t/2}A$ using \eqref{eq.CauchyTrans.ZA} and then push forward by the square function. With the help of Mathematica, the density of the Brown measure of $Z+\sqrt{t/2}A$ for $\lambda\in\Omega_k(t)$ using \eqref{eq.CauchyTrans.ZA} is
    \begin{align*}
        &\frac{1}{\pi}\left[\frac{t}{(2\vert\lambda-1\vert^2-t)^2}+\frac{t}{(2\vert\lambda+1\vert^2-t)^2}-\frac{1}{(\lambda+\bar\lambda)^2}\right]\\
        &=\frac{1}{\pi}\left[ \frac{2t\bigg[\big(2(\vert\lambda\vert^2+1)-t\big)^2+4\big(\lambda^2+\bar\lambda^2+2\vert\lambda\vert^2\big)\bigg]}{\big[4\vert\lambda^2-1\vert^2+t^2-4t(\vert\lambda\vert^2+1)\big]^2}-\frac{1}{\lambda^2+\bar\lambda^2+2\vert\lambda\vert^2}\right]
    \end{align*}
    The map $\lambda\mapsto\lambda^2$ is a $2$-to-$1$ map with Jacobian $4\vert z\vert$. Thus, the pushed-forward measure has a density as in the conclusion of this proposition.
\end{proof}

%%%%% Below is just to keep some calculations for the above proof %%%%%%%%%%
% \[(2\vert\lambda-1\vert^2-t)(2\vert\lambda+1\vert^2-t) =4\vert\lambda^2-1\vert^2+t^2-4t(\vert\lambda\vert^2+1)\]
% \[(2\vert\lambda-1\vert^2-t)^2+(2\vert\lambda+1\vert^2-t)^2 = 2[(2(\vert\lambda\vert^2+1)-t)^2+4(\lambda+\bar\lambda)^2]\]

\begin{proposition}
    Assume the $\ast$-distribution of $a$ is circular and $u_0=1$ in \eqref{eq.RW.unitary}. In this case, $\overline{D}_2(a,t)$ is empty. The density of the Brown measure of the two-step random walk $b_k(t)$ is given by
   \[\frac{1}{2\vert z\vert\pi}\left[\frac{1}{t}-\frac{1}{4(\vert z\vert+\mathrm{Re}(z))}+\frac{t}{4(\vert z\vert+\mathrm{Re}(z))\sqrt{t^2+4(\vert z\vert+\mathrm{Re}(z))}}\right]\]
    for $z\in\Sigma_2(t)$.
\end{proposition}
\begin{proof}
    The disk $\overline{D}_2(a,t)$ is empty because $\|a\|_2 = \infty$.

    Let $\lambda\in\Omega_2(t)$ and $z = \lambda^2 $. We first compute $\eta = \eta_k(t,z)$. By Definition \ref{notation.psi.epsilon}, we need to first calculate $\psi_k(t,z)$, which is the Denjoy--Wolff point of $H_{\tilde\mu_{\vert a\vert}}\circ H_{\tilde\mu_{\vert Z-z^{1/2}}\vert}$. Since $\sqrt{t/2}a$ is the circular variable with variance $t/2$, $\tilde\mu_{\vert a\vert}$ is a semicircular variable with variance $t/2$, and
    \[H_{\tilde\mu_{\vert a\vert}}(\zeta) = \frac{-\zeta+\sqrt{\zeta^2-2t}}{2}.\]
    Thus the Denjoy--Wolff equation $\psi_k(t,z) = H_{\tilde\mu_{\vert a\vert}}(H_{\tilde\mu_{\vert \lambda -Z}}(\psi_k(t,z)))$ for $\psi_k(t,z)$ can be rewritten as $-t G_{\tilde\mu_{\vert \lambda -Z}} (\psi_k(t,z))= 2\psi_k(t,z)$. Using the formula of the Cauchy transform of $\tilde\mu_{\vert \lambda -Z\vert}$ in Lemma \ref{lem.Z.Cauchy}, $\eta = \eta_k(t,z) = \psi_k(t,a)^2$ satisfies the equation
    \[\frac{1}{\vert \lambda-1\vert^2+\eta}+\frac{1}{\vert\lambda+1\vert^2+\eta} = \frac{1}{t}.\]

    To compute the density of $b_2(t)$, we first calculate the density of the Brown measure of $Z+\sqrt{t/2}A$ using \eqref{eq.CauchyTrans.ZA} and then push forward by the square function. With the help of Mathematica, the density of the Brown measure of $Z+\sqrt{t/2}A$ for $\lambda\in\Omega_k(t)$ using \eqref{eq.CauchyTrans.ZA} is
    \begin{align*}
        &\frac{1}{\pi}\left[\frac{1}{t}-\frac{1}{2(\lambda+\bar\lambda)^2}+\frac{t}{2(\lambda+\bar\lambda)^2\sqrt{t^2+2(\lambda+\bar\lambda)^2}}\right]\\
        &=\frac{1}{\pi}\left[\frac{1}{t}-\frac{1}{4(\vert\lambda\vert^2+\mathrm{Re}(\lambda^2))}+\frac{t}{4(\vert\lambda\vert^2+\mathrm{Re}(\lambda^2))\sqrt{t^2+4(\vert\lambda\vert^2+\mathrm{Re}(\lambda^2))}}\right]
    \end{align*}
    The map $\lambda\mapsto\lambda^2$ is a $2$-to-$1$ map with Jacobian $4\vert z\vert$. Thus, the pushed-forward measure has a density as in the conclusion of this proposition.
\end{proof}

\section{Convergence to the Lima Bean Law}
In this section, our goal is to show that the Brown measure of the random walk $u_0b_k(t)$ converges to the Brown measure of the free multiplicative Brownian motion $u_0b(t)$. 

The Brown measure of $u_0b_k(t)$ is supported on $\Sigma_k(u_0,t)\setminus  \overline{D}_k(a,t)$ by Theorem \ref{thm.RW.Brown.density}. Proposition \ref{prop.add.Cauchy} and Theorem \ref{thm.RW.Brown.density} compute the Cauchy transform and the density of the Brown measure $u_0b_k(t)$ in the domain $\Sigma_k(u_0,t)\setminus (S_k^k\cup \overline{D}_k(a,t))$. Both the Cauchy transform and the density involve the function $\eta_k$. Whether $\eta_k(t,z)$ is positive is related to whether $z$ is in $\Sigma_k(u_0,t)\setminus (S_k^k\cup \overline{D}_k(a,t))$. In Section \ref{sec.Brown.domain.conv}, we first show that the limit of $\Sigma_k(u_0,t)\setminus \overline{D}_k(a,t)$ is $\Sigma_\infty(u_0,t)$. Then, in Section \ref{sec.varepsilonk.conv}, we prove that $k^2\eta_k$ has a finite limit as $k\to\infty$. We then proceed to proving the convergence of the Brown measure in Section \ref{sec.Brown.density.conv}.

\subsection{Convergence of the domain\label{sec.Brown.domain.conv}} \hfill

\medskip

In this section, we prove that the domain $\Sigma_k(u_0,t)\setminus \overline{D}_k(a,t)$ converges to the domain $\Sigma_\infty(u_0,t)$ of the free multiplicative Brownian motion with unitary initial condition $u_0$. 

\begin{lemma}
    \label{lem.spec.u.Sigma}
    For any $k>0$ and $t>0$, the support of $\mu_{u_0}$ is contained in $\overline{\Sigma_\infty(u_0,t)}$ and $\overline{\Sigma_k(u_0,t)}$.
\end{lemma}
\begin{proof}
    We use the same reasoning in Lemma 5.2 of \cite{HallHo2025Spectrum}, which proves that the support of $\mu_{u_0}$ is contained in $\overline{\Sigma_\infty(u_0,t)}$. For self-consistency, we provide a complete proof to show that the support of $\mu_{u_0}$ is contained in $\overline{\Sigma_\infty(u_0,t)}$ and $\overline{\Sigma_k(u_0,t)}$. 
    
    By Lemma 4.5 in \cite{Zhong2021Brown}, the integral $\int_{(\pi,\pi]}\frac{1}{\vert \lambda-e^{i\alpha}\vert^2}\mu_{u_0}(d\alpha)$ is infinite for $\mu_{u_0}$-almost every $\lambda$. Thus $T_k(u_0;r,\theta)=0$ and $T(u_0;r,\theta)=0$ for $\mu_{u_0}$-almost every $re^{i\theta}$. This shows that either of $\Sigma_\infty(u_0,t)$ and $\Sigma_k(u_0,t)$ is a set of full measure of $\mu_{u_0}$. The support of $\mu_{u_0}$ is then contained in both of of $\overline{\Sigma_\infty(u_0,t)}$ and $\overline{\Sigma_k(u_0,t)}$.
\end{proof}

\begin{theorem}
    \label{thm.domain.conv}
    Let $K\subset \mathbb{C}$ be a compact set. If $K\subset \Sigma_\infty(u_0,t)$, then $K\subset \Sigma_k(u_0,t)$ for all large $k$. On the other hand, if $K\subset (\overline{\Sigma_\infty(u_0,t)})^c$, then $K\subset (\overline{\Sigma_k(u_0,t)})^c$ for all large $k$.
\end{theorem}
\begin{proof}
    For all $k$ large enough, the disk $\overline{D}_k(a,t)$ is empty. Thus, the support $\overline{\Sigma_k(u_0,t)\setminus \overline{D}_k(a,t)}$ of the Brown measure in Theorem \ref{thm.RW.Brown.density} becomes $\overline{\Sigma_k(u_0,t)}$ for large $k$.
    
    Let $K\subset \Sigma_\infty(u_0,t)$ be compact. We prove by contradiction that $K\subset\Sigma_k(u_0,t)$ for all large $k$. Suppose that, for all $k$, there is $r_ke^{i\theta_k}\in K$ such that $T_k(u_0;r_k,\theta_k)\geq t$. Since $K$ is compact, there exists a subsequence $r_{k_j}e^{i\theta_{k_j}}$ and $re^{i\theta}\in K$ such that  $r_{k_j}e^{i\theta_{k_j}}\to re^{i\theta}$. Notice that $T(u_0;r,\theta)<t$. We then apply Fatou's lemma to the integral in the definition of $T_k$ and get
    \[t\leq \liminf_{j\to\infty}T_k(u_0;r_{k_j},\theta_{k,j})\leq T(u_0;r,\theta)<t,\]
    which is a contradiction.

    Now let $K\subset (\overline{\Sigma_\infty(u_0,t)})^c$. 
    By Theorem 4.10 of \cite{HoZhong2020Brown} (see also Lemma 5.3 of \cite{HallHo2025Spectrum}), $T(u_0;r,\theta)> t$ for all $e^{i\theta}\in K$. By Lemma \ref{lem.spec.u.Sigma}, the support of $\mu_{u_0}$ is contained in $\overline{\Sigma_\infty(u_0,t)}$ and $\overline{\Sigma_k(u_0,t)}$; therefore, the support of $\mu_{u_0}$ has a positive distance from $K$. We claim that when $k$ large enough, $T_k(u_0;r,\theta)>t$ for all $re^{i\theta}\in K$. Suppose that, in the contrary, for all $k$, there is $r_ke^{i\theta_k}\in K$ such that $T_k(u_0;r_k,\theta_k)\leq t$. Since $K$ is compact, there exists a subsequence $r_{k_j}e^{i\theta_{k_j}}$ and $re^{i\theta}\in K$ such that  $r_{k_j}e^{i\theta_{k_j}}\to re^{i\theta}$. Notice that $T(u_0;r,\theta)>t$. Since $K$ has a positive distance from the support of $\mu_{u_0}$, we apply the dominated convergence theorem and get
    \[t\geq \lim_{j\to\infty} T_k(u_0;r_{k_j},\theta_{k_j}) = T(u_0;r,\theta)>t,\]
    which is a contradiction. (The above limit for $T_k$ is also valid if $r = 0$. In this case $k_j(r_{k_j}^{2/k_j}-1)\to -\infty$, which is consistent to $T(u_0;0,\theta) = \infty$.) This shows for all $k$ large enough, $T_k(u_0;r,\theta)>t$ for all $re^{i\theta}\in K$.

    To complete the proof, we first observe that $T_k$ is continuous outside the support of $\mu_{u_0}$. Thus, for any $re^{i\theta}$ in the boundary of $\Sigma_k(u_0,t)$, either $re^{i\theta}$ is in the support of $\mu_{u_0}$, or $T(u_0;r,\theta) = t.$ By the preceding paragraph, $K$ does not intersect the support of $\mu_{u_0}$, and for any $re^{i\theta}\in K$, $T_k(u_0;r,\theta)>t$. Hence, any $re^{i\theta}\in K$ is not on the boundary of $\Sigma_k(u_0,t)$, completing the proof of $K\subset (\overline{\Sigma_k(u_0,t)})^c$ for large $k$.
    \end{proof}

\begin{corollary}
    Fix $t>0$. The closure of the domain $\Sigma_k(u_0,t)\setminus \overline{D}_k(a,t)$ converges to $\overline{\Sigma_\infty(u_0,t)}$ in the Hausdorff distance as $k\to\infty$.
\end{corollary}
\begin{proof}
    It suffices to show that $\overline{\Sigma_k(u_0,t)}$ converges to $\overline{\Sigma_\infty(u_0,t)}$ in the Hausdorff distance since $D_k(a,t)$ is empty for large $k$.

    Fix $R>0$ such that $\overline{\Sigma_\infty(u_0,t)}\subset \{z:\vert z\vert\leq R\}$. For all $k>0$ and $r>1$, $k(r^{2/k}-1)>\log r^2$, so that $T_k(u_0;r,\theta)>T(u_0;r,\theta)$. It follows that $\overline{\Sigma_k(u_0,t)}\subset\{z:\vert z\vert\leq R\}$ for all $k$.

    We denote by $K_\varepsilon = \{\lambda\in\mathbb{C}: \vert \lambda - w\vert<\varepsilon \textrm{ for some $w\in K$}\}$ for any set $K\subset \mathbb{C}$. Let $\varepsilon>0$ be arbitrary. We will show that $\overline{\Sigma_\infty(u_0,t)}\subset (\overline{\Sigma_k(u_0,t})_\varepsilon$ and $\overline{\Sigma_k(u_0,t)}\subset (\overline{\Sigma_\infty(u_0,t)})_\varepsilon$. The set $K = \overline{B(0, R)}\setminus(\overline{\Sigma_\infty(u_0,t)})_\varepsilon$ is a compact set. By Theorem \ref{thm.domain.conv} and the preceding paragraph, $K\subset (\overline{\Sigma_k(u_0,t)})^c$ for all large $k$. This shows $\overline{\Sigma_k(u_0,t)}\subset (\overline{\Sigma_\infty(u_0,t)})_\varepsilon$ for all large $k$. On the other hand, choose a compact set $K\subset \Sigma_\infty(u_0,t)$ such that $\overline{\Sigma_\infty(u_0,t)}\subset K_\varepsilon$. Then $K\subset \Sigma_k(u_0,t)$ for all large enough $k$; for these $k$, $\overline{\Sigma_\infty(u_0,t)}\subset K_\varepsilon\subset (\overline{\Sigma_k(u_0,t})_\varepsilon$. The proof is completed.
\end{proof}

\subsection{Convergence of the function $\eta_k$\label{sec.varepsilonk.conv}} \hfill

\medskip

In this section, we prove that $k^2\eta_k$ has a finite limit as $k\to\infty$. The function $\eta_k$ is defined through a subordination function, denoted by $\psi_k$ in Definition \ref{notation.psi.epsilon}, in free convolution. 

From now on, we denote $z^{1/k}$ to be the principal $k$-th root of $z$ with argument in $(-\pi/k,\pi/k]$. By specifying the probability measures in Theorem \ref{thm.subordination.general} to be $\mu_1 = \tilde\mu_{\vert Z-z^{1/k}\vert}$ and $\mu_2 = \tilde\mu_{\sqrt{t/k}\vert a\vert}$ as in Definition \ref{notation.psi.epsilon}(3), the value $\psi_k(t,z)$ can be classified as the Denjoy--Wolff point of 
\begin{equation}
    \label{eq.DWpoint.selfmap.psi}
    H_{\tilde\mu_{\sqrt{t/k}\vert a\vert}}\circ H_{\tilde\mu_{\vert Z-z^{1/k}\vert}}.
\end{equation}
We will then use results of Denjoy--Wolff points in \cite{Heins1941} (see also \cite{BelinschiBercoviciHo2022} in the context of free probability) to show that $k\psi_k(t,z)$ converges to a (finite) positive multiple of $i$ (Theorem \ref{thm.DWpoint.conv} and Corollary \ref{cor.DWpoint.conv.unif}).
\begin{remark}
    \label{rem.H.Nevanlinna}
    We first notice that $H_{\tilde\mu_{\sqrt{t/k}\vert a\vert}}(z) =\sqrt{\frac{t}{k}}H_{\mu_{\vert a\vert}}(\sqrt{k/t}z)$. By Proposition 2.2(b) of Maassen's paper, there is a probability measure  $\sigma_{\vert a\vert}$ that has total mass $1$ such that
        \[H_{\tilde\mu_{\sqrt{t/k}\vert a\vert}}(z) = \sqrt{\frac{t}{k}}\int\frac{\sigma_{\vert a\vert}(dx)}{x-\sqrt{k/t}z}.\]
\end{remark}

Lemma \ref{lem.Z.Cauchy} computes the Cauchy transform of $\tilde\mu_{\vert \lambda-Z\vert}$. In the following definition, we introduce a ``normalized'' version of the Cauchy transform and its related analytic self-maps on $\mathbb{C}^+$. The normalization, by a change of variable $\zeta = kz$, is to prepare for taking limit as $k\to\infty$.

\begin{notation}
    \label{notation.G.H.hat}
    For convenience, we employ the following notation.
    \begin{enumerate}
    \item For each $z\in\mathbb{C}$, we define an analytic self-map $\hat G_{z,k}$ on $\mathbb{C}^+$ by
          \begin{align*}
                \hat G_{z,k}(\zeta) &= \int_{(-\pi,\pi]} \sum_{-k/2<j\leq k/2}\frac{\zeta}{\zeta^2-k^2\vert z^{\frac{1}{k}}-e^{\frac{i(2\pi j+\alpha)}{k}}\vert^2}\,\mu_{u_0}(d\alpha) \\
                 &= \int_{(-\pi,\pi]} \sum_{-k/2<j\leq k/2}\frac{\zeta}{\zeta^2-a_k(\alpha;j,\rho,\theta)}\,\mu_{u_0}(d\alpha)
          \end{align*}
          where $z = e^{\rho}e^{i\theta}$ and 
        \begin{equation}
            \label{eq.ak.def}
            a_k(\alpha;j,\rho,\theta)=k^2(1-e^{\rho/k})^2+2e^{\rho/k}k^2\left[1-\cos\left(\frac{2\pi j+\alpha-\theta}{k}\right)\right].
        \end{equation}
    \item We also define, for each $z\in\mathbb{C}$, an analytic self-map $\hat H_{z,k}$ on $\mathbb{C}^+$ by
          \[\hat H_{z,k}(\zeta) = \frac{1}{\hat G_{z,k}(\zeta)}-\frac{\zeta}{k}.\]
    \item For each $z\in\mathbb{C}$, we write $z=e^{\rho}e^{i\theta}$ where $\theta\in(-\pi,\pi]$ and define analytic self-maps $\hat G_{z}$ and $\hat H_{z}$ on $\mathbb{C}^+$ by
          \[\hat G_{z}(\zeta) =\int_{(-\pi,\pi]}\sum_{j=-\infty}^\infty\frac{\zeta}{\zeta^2-\rho^2-(2\pi j-\theta+\alpha)^2}\,\mu_{u_0}(d\alpha)\]
          and $\hat H_{z} = 1/\hat G_{z}$.
    \item Finally, we define $\psi(t,z)$ to be the Denjoy--Wolff point of $-t\hat G_{z}$, and write
    \[\eta(t,z) = -\psi(t,z)^2.\]
    We will see in Corollary \ref{cor.DWpoint.conv.unif} that $\eta(t,z)>0$ if  $z\in \Sigma_\infty(u_0,t)$ and $\eta(t,z)=0$ if  $z\not\in \Sigma_\infty(u_0,t)$
    \end{enumerate}
\end{notation}
The values of $\hat G_{z, k}$ and $\hat H_{z, k}$ are independent of the choice of the $k$-th root of $z$. We use the principal $k$-th root $z^{1/k}$ for convenience.

We identity $k\psi_k(t,z)$ as the Denjoy--Wolff point of a function $\Psi_{z,k,t}$ defined in the following proposition.
\begin{proposition}
    \label{prop.psi.k}
    Let $z\in\mathbb{C}$ and $z^{1/k}$ be a $k$-th root of $z$. Then $k\psi_k(t,z)$ is the Denjoy--Wolff point of the function
    \[\Psi_{z, k, t}(\zeta) =  \int\frac{t\sigma_{\vert a\vert}(dx)}{\sqrt{t/k}x-\hat H_{z, k}(\zeta)},\quad \zeta\in \mathbb{C}_+,\]
    where $\sigma_{\vert a\vert}$ comes from Remark \ref{rem.H.Nevanlinna} and $\hat H_{z, k}$ is defined in Definition \ref{notation.G.H.hat}.
\end{proposition}
\begin{proof}
    The maps $H_{\tilde\mu_{\vert Z-z^{1/k}\vert}}$ and $\hat H_{z, k}$ are related by
    \[H_{\tilde\mu_{\vert Z-z^{1/k}\vert}}(\zeta) =\hat H_{z, k}(k\zeta).\]
    The value $\psi_k(t,z)$ is the Denjoy--Wolff point of $H_{\tilde\mu_{\sqrt{t/k}\vert a\vert}}\circ H_{\tilde\mu_{\vert Z-z^{1/k}\vert}}$. By Remark \ref{rem.H.Nevanlinna}, $w=\psi_k(t,z)$ is the unique solution in $\mathbb{C}^+$ of the equation
    \[w = \sqrt{\frac{t}{k}}\int\frac{\sigma_{\vert a\vert}(dx)}{x-\sqrt{k/t}H_{\tilde\mu_{\vert Z-z^{1/k}\vert}}(w)}.\]
    which can be written in terms of $\hat H_{z, k}$ by
    \begin{equation*}
        kw = \int\frac{t\sigma_{\vert a\vert}(dx)}{\sqrt{t/k}x-\hat H_{z, k}(kw)}.
    \end{equation*}
    The above equation says $k\psi_k(t,z)$ is a fixed point of $\Psi_{z,k,t}$. The conclusion follows from the fact that any analytic self-map on $\mathbb{C}^+$ has at most one fixed point in $\mathbb{C}^+$, and the unique fixed point is the Denjoy--Wolff point.
\end{proof}

Our strategy is to find a self-map $\Psi_z$ on $\mathbb{C}^+$ such that $\Psi_{z,k,t}$ converges pointwise, to $-t\hat G_{z}$, in $\mathbb{C}^+$. The next proposition shows that $\hat H_{z,k}$ has pointwise limit $\hat H_{z}$.
\begin{lemma}
\label{lem.Hhat.limit.u}
Let $z_k\in\mathbb{C}$ be a convergent sequence such that $z_k\to z$. We write $\lambda_k = z_k^{1/k}$ to be the $k$-th root of $z_k$ with argument in $(-\pi,\pi]$. Then
\[\hat H_{z_k,k}(\zeta) \to \hat H_{z}(\zeta)\]
for each $\zeta\in\mathbb{C}^+$.
\end{lemma}
\begin{proof}
Fixed $\zeta\in \mathbb{C}^+$. If suffices to show that $\hat G_{z_k,k}(\zeta)$ converges to $\hat G_z(\zeta)$. In this proof, we write $\zeta^2 = x+iy$ and $z_k^{1/k} = e^{\rho_k/k}e^{i\theta_k/k}$ where $\theta_k\in(-\pi,\pi]$. 

For all $-k/2<j\leq k/2$, $\left\vert\frac{2\pi j}{k}+\frac{\alpha-\theta_k}{k}\right\vert\leq \pi+\frac{\vert \alpha-\theta_k\vert}{k}$.
If $k$ is large enough, $\pi+(\alpha-\theta_k)/k$ is away from $2\pi$ for all $\alpha,\theta_k\in(-\pi,\pi]$. Hence, there exists $c>0$ such that for all $-k/2<j\leq k/2$, 
\[c<\left[1-\cos\left(\frac{2\pi j+\alpha-\theta_k}{k}\right)\right]\left(\frac{2\pi j+\alpha-\theta_k}{k}\right)^{-2},\]
so that the $a_k(\alpha;j,\rho_k,\theta_k)$ defined in  \eqref{eq.ak.def} satisfies the estimates: if $\rho\geq -k$, then
 \begin{equation}
        \label{eq.ak.estimate}
        a_k(\alpha;j,\rho,\theta)\geq c\cdot \min\{[2\pi(j+1)]^2,[2\pi(j-1)]^2\}, \quad \vert j\vert \geq 1;
    \end{equation}
    if $\rho< -k$, then
    $a_k(\alpha;j,\rho,\theta)$ in \eqref{eq.ak.def} satisfies the estimate 
    \begin{equation}
        \label{eq.ak.estimate.small}
        a_k(\alpha;j,\rho,\theta)\geq k^2(1-e^{-1})^2.
    \end{equation} 
    This means that if $\rho_k\geq -k$, then
\[
 \left\vert\frac{\zeta}{\zeta^2-a_k(\alpha;j,\rho_k,\theta_k)}\right\vert \leq \frac{\vert \zeta\vert}{c(\min\{[2\pi(j+1)]^2,[2\pi(j-1)]^2\})^2-x},\quad \vert j\vert\geq 1
\]
except for finitely many $j$'s where the denominator $2c(2\pi j+\alpha-\theta)^2-x$ is negative; if $\rho_k<-k$, then 
\[
 \left\vert\frac{\zeta}{\zeta^2-a_k(\alpha;j,\rho_k,\theta_k)}\right\vert \leq \frac{\vert \zeta\vert}{k^2(1-e^{-1})^2-x}.
\]
It then follows that the series
\[\sum_{\vert j\vert<k/4}\frac{\zeta}{\zeta^2-a_k(\alpha;j,\rho_k,\theta_k)}\]
is absolutely convergent. By the dominated convergence theorem,
\begin{align*}
    \lim_{k\to \infty}\hat G_{z_k,k}(\zeta) &= \lim_{k\to \infty}\int_{(\pi,\pi]}\sum_{-k/2< j\leq k/2}\frac{\zeta}{\zeta^2-a_k(\alpha;j,\rho_k,\theta_k)}\mu_{u_0}(d\alpha)\\
    &=\hat G_{z}(\zeta),
\end{align*}
where we have used the limit 
\[\lim_{k\to \infty}a_k(\alpha;j,\rho_k,\theta_k) = \rho^2+(2\pi j+\alpha-\theta)^2.\]
We remark that if $z=0$ (i.e. $\rho_k\to -\infty$), then $\hat G_{z}$ is the zero function.
\end{proof}

\begin{theorem}
    \label{thm.DWpoint.conv}
Let $z_k\in \mathbb{C}$ and $t_k>0$ be a convergent sequence such that $z_k\to z$ and $t_k\to t$. Assume that $t>0$. Write $z_k = e^{\rho_k} e^{i\theta_k}$, $z = e^{\rho}e^{\theta}$ where $\theta_k$ and $\theta$ are chosen in $(-\pi,\pi]$. Then
\[\psi(t,z) = \lim_{k\to\infty}k\psi_k(t_k,z_k)\]
exists and is the Denjoy--Wolff point of $-t\hat G_z$.
\end{theorem}
\begin{proof}
    Define $\Psi_{z_k,k,t_k}$ as in Proposition \ref{prop.psi.k}. Then, Proposition \ref{prop.psi.k} asserts that $k\psi_k(t_k,z_k)$ is the Denjoy--Wolff point of $\Psi_{z_k,k,t_k}$; that is $k\psi_k(t_k,z_k)$ is the unique fixed point of $\Psi_{z_k,k,t_k}$ in $\mathbb{C}^+$. We can see that, by Lemma \ref{lem.Hhat.limit.u} and the dominated convergence theorem, 
    \[\Psi_{z_k,k,t_k}\to -t \hat G_z\]
    pointwise. It follows from \cite{Heins1941} (see also \cite{BelinschiBercoviciHo2022} in the context of free probability), the Denjoy--Wolff points $k\psi_k(t,z_k)$ of $\Psi_{z,k,t}$ have a limit
    \[\psi(t,z) = \lim_{k\to\infty} k\psi_k(t,z_k),\]
    and this limit is the Denjoy--Wolff point of $-\hat G_{z}$. 
\end{proof}
\begin{corollary}
\label{cor.DWpoint.conv.unif}
The convergence of the limit
\[\psi(t,z) = \lim_{k\to\infty}k\psi_k(t_k,z_k)\]
is uniform for $(t_k,z_k)$ in any compact subset in $(0,\infty)\times \mathbb{C}$. Moreover, $\eta$ is a continuous function in $(0,\infty)\times \mathbb{C}$. 

If $z\in \Sigma_\infty(u_0,t)$, then $\psi(t,z)$ is a positive multiple of $i$ determined by the equation
    \begin{equation}
        \label{eq.DW.condition.u}
        \int_{(-\pi,\pi]}\sum_{j=-\infty}^\infty\frac{1}{\rho^2+(2\pi j+\alpha-\theta)^2-\psi(t,z)^2}\mu_{u_0}(d\alpha)= \frac{1}{t}.
    \end{equation}
    If $z\not\in\Sigma_\infty(u_0,t)$, $\psi(t,z) = 0$.
\end{corollary}
\begin{proof}
    If the limit is not uniform in some compact subset, there exists a convergent sequence $z_k$ such that $z_k\to z$, $t_k\to t$ and $k\psi_k(t_k,z_k)$ does not converge to the $\psi(t,z)$. But then this contradicts Theorem \ref{thm.DWpoint.conv}.

    To show that $\eta$ is a continuous function, let $(t_k,z_k)\to (t,z)$ in $(0,\infty)\times\mathbb{C}$ where $t>0$. By Theorem \ref{thm.DWpoint.conv}, $\eta(t_k,z_k)$ is the Denjoy--Wolff point of $-t_k\hat G_{z_k}$. One can show that $\hat G_{z_k}\to\hat G_{z}$ pointwise using a calculation similar to the proof of Lemma \ref{lem.Hhat.limit.u}; hence, the Denjoy--Wolff points $\eta(t_k,z_k)$ of $-t_k\hat G_{z_k}$ converge to the Denjoy--Wolff point $\psi(t,z)$ of $-t\hat G_z$.

    If $z\not\in \overline{\Sigma_\infty(u_0,t)}$, then $z\not\in \overline{\Sigma_k(u_0,t)}$ for all large $k$ by Theorem \ref{thm.domain.conv}. For all $k$ large enough, $D_k(a,t)$ is empty. By Proposition \ref{prop.Brown.add.support}(4), $\psi_k(z) = 0$ for all large enough $k$. Thus, 
    \[\psi(t,z) = \lim_{k\to\infty}k\psi_k(z) = 0.\] 
    
    If $z$ is on the boundary of $\Sigma_\infty(u_0,t)$, then $\psi(t,z) = 0$ by continuity of $\eta$.
    
    Now assume $z\in\Sigma_\infty(u_0,t)$. We will show that $-\hat G_z$ has a fixed point in the upper half plane; hence, the Denjoy--Wolff point $\psi(t,z)$ of $-\hat G_z$ must be equal to the fixed point. To do this, we will actually show that there is a fixed point of the form $iy$ for some $y>0$. By the statement and the proof of Lemma \ref{lem.domain.by.Tk}(2), if we write $z=e^{\rho}e^{i\theta}$, then 
    \begin{align*}
        T_k(u_0;r,\theta) &= \left(\frac{1}{k^2}\int_{-\pi}^{\pi}\sum_{-\frac{k}{2}<j\leq\frac{k}{2}}\frac{1}{\vert e^{2\pi i j/k} - r^{1/k}e^{i(\theta-\alpha)/k}\vert^2}\mu_{u_0}(d\alpha)\right)^{-1}\\
        &=\left(\int_{(-\pi,\pi]}\sum_{-\frac{k}{2}<j\leq \frac{k}{2}}\frac{1}{a_k(\alpha;j,\rho,\theta)}\mu_{u_0}(d\alpha)\right)^{-1}.
    \end{align*}
    We find an $\varepsilon>0$ such that $T(u_0;r,\theta)<t-\varepsilon$. Thus, $z\in \Sigma_\infty(u_0,t-\varepsilon)$. By Theorem \ref{thm.domain.conv}, $z\in\Sigma_k(u_0,t-\varepsilon)$ for all $k$ large enough. This shows
    \[\int_{(-\pi,\pi]}\sum_{-\frac{k}{2}<j\leq \frac{k}{2}}\frac{1}{a_k(\alpha;j,\rho,\theta)}\mu_{u_0}(d\alpha)>\frac{1}{t-\varepsilon}\]
    for all $k$ large enough. By the estimate \eqref{eq.ak.estimate}, we can apply the dominated convergence theorem to take limit as $k\to\infty$, so that
    \[\int_{(-\pi, \pi]}\sum_{j=-\infty}^\infty\frac{1}{\rho^2+(2\pi j+\alpha-\theta)^2}\,\mu_{u_0}(d\alpha)\geq\frac{1}{t-\varepsilon}>\frac{1}{t}.\]
    so there is a $y>0$ such that 
        \[\int_{(-\pi,\pi]}\sum_{j=-\infty}^\infty\frac{1}{y^2+\rho^2+(2\pi j+\alpha-\theta)^2}\,\mu_{u_0}(d\alpha)=\frac{1}{t}.\]
    Multiplying $iy$ on both sides gives $iy = -t\hat G_z(iy)$.
    Thus, the Denjoy--Wolff point of $-\hat G_z$ is the fixed point $iy$. The above display equation also shows that $\psi(t,z)$ is a positive multiple of $i$ satisfying \eqref{eq.DW.condition.u}.
\end{proof}

\subsection{Convergence of the density\label{sec.Brown.density.conv}} \hfill

\medskip

In this section, we prove the convergence of the density of the Brown measure of $u_0b_k(t)$ to the density of the Brown measure of $u_0b(t)$ as $k\to\infty$ in Theorem \ref{thm.density.limit}, uniformly in any compact subsets of $\Sigma_\infty(u_0,t)$. Corollary \ref{cor.Brown.weak.conv} shows that the Brown measure of $u_0b_k(t)$ converges weakly to the Brown measure of $u_0b(t)$ as $k\to\infty$.

For convenience, define
\[J(\zeta) = \int\frac{t\sigma_{\vert a\vert}(dx)}{\sqrt{t/k}x-\zeta}\quad \textrm{and}\quad L_k(z,\bar z, \zeta) = J(\hat H_{z,k}(\zeta))-\zeta,\]
so that, by Theorem \ref{thm.DWpoint.conv},
\[L_k(z,\bar z, k\psi_k(t,z)) = 0.\]
By Corollary \ref{cor.DWpoint.conv.unif}, the function $k\psi_k$ has a limit $\psi$, uniform on compact sets, determined by that $\psi(z)$ is a imaginary number with positive imaginary part and by the equation $L_\infty(\bar z, \psi(z)) = 0$ where
\[L_\infty(z,\bar z, \zeta) = \int_{(-\pi,\pi]}\sum_{j=-\infty}^\infty\frac{1}{\rho^2+(2\pi j+\alpha-\theta)^2-\zeta^2}\mu_{u_0}(d\alpha)-\frac{1}{t}.\]

The density of the Brown measure $u_0b_k(t)$ in Theorem \ref{thm.RW.Brown.density} involves a derivative $\bar z(\partial/\partial\bar z)$ of an expression involving the function $\eta_k$. The first step to establish our main result in this section is to prove, in Proposition \ref{prop.kpsi.derivative}, that a derivative of $k\psi_k$ converges to the derivative of $\psi$.

\begin{lemma}
    \label{lem.eta.derivative}
We have
\[\bar z\frac{\partial\psi(t,z)}{\partial\bar z} =\frac{ \displaystyle \int_{(-\pi,\pi]}\sum_{j=-\infty}^\infty\frac{\rho-i(2\pi j+\alpha-\theta)}{(\rho^2+(2\pi j+\alpha-\theta)^2+\eta(t,z))^2}\mu_{u_0}(d\alpha)}{\displaystyle\int_{(-\pi,\pi]}\sum_{j=-\infty}^\infty\frac{2\psi(t,z)}{(\rho^2+(2\pi j+\alpha-\theta)^2+\eta(t,z))^2}\mu_{u_0}(d\alpha)}.\]
\end{lemma}
\begin{proof}
Applying chain rule to the equation $L(z,\bar z, \psi(t,z))=0$, we have
\[\bar z\frac{\partial\psi(t,z)}{\partial\bar z} =\frac{ \displaystyle -\bar z\frac{\partial L_\infty}{\partial \bar z}(z,\bar z,\psi(t,z))}{ \displaystyle \frac{\partial L_\infty}{\partial\zeta}(z,\bar z,\psi(t,z))}.\]
It then remains to compute the partial derivatives. Noting that $\bar z\partial/\partial\bar z = (1/2)(\partial/\partial\rho+i\partial/\partial\theta)$, we calculate
\[-\bar z\frac{\partial L_\infty}{\partial\bar z}(z,\bar z,\psi(t,z)) = \int_{(-\pi,\pi]}\sum_{j=-\infty}^\infty\frac{\rho-i(2\pi j+\alpha-\theta)}{(\rho^2+(2\pi j+\alpha-\theta)^2+\eta(t,z))^2}\mu_{u_0}(d\alpha)\]
and
\[\frac{\partial L_\infty}{\partial\zeta}(z,\bar z,\psi(t,z)) = \int_{(-\pi,\pi]}\sum_{j=-\infty}^\infty\frac{2\psi(t,z)}{(\rho^2+(2\pi j+\alpha-\theta)^2+\eta(t,z))^2}\mu_{u_0}(d\alpha)\]
where we have used the notation $\eta$ from Definition \ref{notation.G.H.hat}(4).
\end{proof}

\begin{lemma}
    \label{lem.sum.order}
    Let $z = e^\rho e^{i\theta}$ where $\theta\in(-\pi,\pi]$ and $z^{1/k}$ be the principal $k$-th root. Assume that $z\in\Sigma_\infty(u_0,t)$. Then there exist constants $c, C>0$ such that
    \[c \min\{\vert j+1\vert,\vert j-1\vert\} \leq k\vert z^{1/k} - e^{\frac{i(2\pi j+\alpha)}{k}}\vert\leq  C\max\{\vert j+1\vert, \vert j-1\vert\}\]
    for all $\alpha\in(-\pi,\pi]$ and all $\vert j\vert \geq 1$. The constants $c$ and $C$ are independent of the choice of $z\in\Sigma_\infty(u_0,t)$.
\end{lemma}
\begin{proof}
    We note that $\sqrt{a_k(\alpha;j,\rho,\theta)} =k\vert z^{1/k} - e^{\frac{i(2\pi j+\alpha)}{k}}\vert $. The lower estimate is a restatement of the estimate \eqref{eq.ak.estimate}. The upper estimate follows from that
    \[\left[1-\cos\left(\frac{2\pi j+\alpha-\theta}{k}\right)\right]\left(\frac{2\pi j+\alpha-\theta}{k}\right)^{-2}\leq\frac{1}{2}\]
    and that $\Sigma_\infty(u_0,t)$ is a bounded domain so $\rho$ is bounded above and below.
\end{proof}

\begin{proposition}
\label{prop.kpsi.derivative}
For any compact subset $K\subset\Sigma_\infty(u_0,t)$, the limit
\[\lim_{k\to\infty}\bar z\frac{\partial }{\partial\bar z}k\psi_k(t,z) = \bar z\frac{\partial}{\partial\bar z}\psi(t,z)\]
converges uniformly in $K$.
\end{proposition}
\begin{proof}
By applying chain rule to $L_k(z,\bar z, k\psi_k(t,z))=0$,
\begin{equation}
    \label{eq.kpsi.chainrule}
    \bar z\frac{\partial}{\partial\bar z} k\psi_k(t,z)= \frac{\displaystyle -\bar z\frac{\partial L_k}{\partial \bar z}(z, \bar z, k\psi_k(t,z))}{\displaystyle \frac{\partial L_k}{\partial \zeta}(z,\bar z,k\psi_k(t,z))}.
\end{equation}
We then compute the derivatives
\[
-\bar z\frac{\partial L_k}{\partial \bar z}(z,\bar z,\zeta) = \frac{J'(\hat H_{z,k}(\zeta))}{\hat G_{z,k}(\zeta)^2}\int_{(\pi,\pi]}\sum_{-\frac{k}{2}<j\leq \frac{k}{2}}\frac{\bar z^{\frac{1}{k}}\zeta k(z^{\frac{1}{k}}-e^{\frac{i(2\pi j+\alpha)}{k}})}{(k^2\vert z^{\frac{1}{k}}-e^{\frac{i(2\pi j+\alpha)}{k}}\vert^2-\zeta^2)^2}\mu_{u_0}(d\alpha)\]
and
\begin{align*}
&\frac{\partial L_k}{\partial \zeta}(z,\bar z,\zeta) \\
&= J'(\hat H_{z,k}(\zeta))\left[\frac{1}{\hat G_{z,k}(\zeta)^2}\int_{(-\pi,\pi]}\sum_{-\frac{k}{2}<j\leq \frac{k}{2}}\frac{\zeta^2+k^2\vert z^{\frac{1}{k}}-e^{\frac{i(2\pi j+\alpha)}{k}}\vert^2}{(k^2\vert z^{\frac{1}{k}}-e^{\frac{2\pi i j}{k}}\vert^2-\zeta^2)^2}\mu_{u_0}(d\alpha)-\frac{1}{k}\right]-1.
\end{align*}
We now take limits. Recall that $k\psi_k$ converges uniformly to $\psi$ in compact sets. One can see, by Montel's theorem for example (because $\psi_k$ and $\psi$ are in the upper half plane), that 
\[J'(\hat H_{z,k}(k\psi_k(t,z)))\to t\hat G_z(\psi(t,z))^2.\]
Therefore, recalling the notation $\eta$ from Definition \ref{notation.G.H.hat}(4),
\[
\lim_{k\to\infty}-\bar z\frac{\partial L_k}{\partial \bar z}(z,\bar z, k\psi_k(t,z)) = \int_{(-\pi,\pi]}\sum_{-\frac{k}{2}<j\leq \frac{k}{2}}\frac{t\psi(t,z) (\rho-i(2\pi j+\alpha-\theta))}{(\rho^2+(2\pi j+\alpha-\theta)^2+\eta(t,z))^2}\mu_{u_0}(d\alpha)
\]
because the summand in $-\bar z\frac{\partial L_k}{\partial \bar z}(\bar z, k\psi_k(z))$ is of order $O(1/j^3)$ by Lemma \ref{lem.sum.order}, and
\begin{align*}
&\lim_{k\to\infty}\frac{\partial L_k}{\partial \zeta}(z,\bar z, k\psi_k(t,z))\\
& = t\int_{(-\pi,\pi]}\left(\sum_{j=\infty}^\infty\frac{-\eta(t,z)+\rho^2+(2\pi j+\alpha-\theta)^2}{(\rho^2+(2\pi j+\alpha-\theta)^2+\eta(t,z))^2}-\frac{1}{t}\right)\mu_{u_0}(d\alpha)\\
&=  \int_{(-\pi,\pi]}\sum_{j=-\infty}^\infty \frac{2t\psi(t,z)^2}{(\rho^2+(2\pi j+\alpha-\theta)^2+\eta(t,z))^2}\mu_{u_0}(d\alpha)
\end{align*}
because the summand in $\frac{\partial L_k}{\partial \alpha}(\bar z, k\psi_k(t,z))$ is of order $O(1/j^2)$ by Lemma \ref{lem.sum.order}. We have also used \eqref{eq.DW.condition.u} in the above display equation.

The conclusion then follows from taking limit as $k\to\infty$ in \eqref{eq.kpsi.chainrule} and comparing the limit to $\bar z\partial \psi(t,z)/\partial z$ Lemma \ref{lem.eta.derivative}.
\end{proof}

We can now proceed to calculate the limit of the density of the Brown measure $u_0b_k(t)$ in Theorem \ref{thm.RW.Brown.density}. We first take the limit of the derivative
    \[\bar z\frac{\partial}{\partial\bar z}\sum_{-\frac{k}{2}<j\leq \frac{k}{2}}\int_{(-\pi,\pi]}\frac{\bar z^{\frac{1}{k}}-e^{-i\frac{2\pi  j+\alpha}{k}}}{\vert z^{\frac{1}{k}}-e^{i\frac{2\pi j+\alpha}{k}}\vert^2+\eta_k(t,z)}\mu_{u_0}(d\alpha)\]
on an arbitrary compact set $K\subset\Sigma_\infty(u_0,t)$. We will again use $k\psi_k\to\psi$ uniformly on compact sets and $\psi(t,z)\in\mathbb{C}^+$ for $z\in \Sigma_\infty(u_0,t)$.

\begin{lemma}
    \label{lem.series.density}
Let $K\subset \Sigma_\infty(u_0,t)$ be compact. Then the limit
\begin{align*}
&\lim_{k\to\infty}\int_{(-\pi,\pi]}\sum_{-\frac{k}{2}<j\leq \frac{k}{2}}\bar z\frac{\partial}{\partial\bar z}\frac{1}{k}\frac{\bar z^{\frac{1}{k}}-e^{-\frac{i(2\pi j+\alpha)}{k}}}{\vert z^{\frac{1}{k}}-e^{\frac{i(2\pi j+\alpha)}{k}}\vert^2+\eta_k(t,z)}\mu_{u_0}(d\alpha)\\
&=\int_{(-\pi,\pi]}\sum_{j=-\infty}^\infty\bar z\frac{\partial}{\partial\bar z}\frac{\rho+i(2\pi j+\alpha-\theta)}{\rho^2+(2\pi j+\alpha-\theta)^2-\eta(t,z)}\mu_{u_0}(d\alpha).
\end{align*}
converges absolutely uniformly for $e^{\rho}e^{i\theta}\in K$.
\end{lemma}
\begin{proof}
    For convenience, we denote $c(j,k;\rho,\theta,\eta) = \vert z^{\frac{1}{k}}-e^{\frac{i(2\pi j+\alpha)}{k}}\vert^2+\eta_k(t,z)$.
    We then compute
    \begin{align*}
    &\bar z\frac{\partial}{\partial\bar z}\frac{1}{k}\frac{\bar z^{\frac{1}{k}}-e^{-\frac{i(2\pi  j+\alpha)}{k}}}{c(j,k;\rho,\theta,\eta)}\\
    &= \frac{1}{k^2}\frac{\bar z^{\frac{1}{k}}}{c(j,k;\rho,\theta,\eta)}-\frac{\bar z^{\frac{1}{k}}-e^{-\frac{i(2\pi j+\alpha)}{k}}}{k}\frac{(z^{\frac{1}{k}}-e^{\frac{i(2\pi j+\alpha)}{k}})\frac{1}{k}\bar z^{\frac{1}{k}}-\bar z\frac{\partial\psi_k(t,z)^2}{\partial \bar z}}{c(j,k;\rho,\theta,\eta)^2}\\
    &= \frac{1}{k^2}\frac{\bar z^{\frac{1}{k}}}{c(j,k;\rho,\theta,\eta)}-\frac{\vert z^{\frac{1}{k}}-e^{\frac{i(2\pi j+\alpha)}{k}}\vert ^2\bar z^{\frac{1}{k}}}{k^2c(j,k;\rho,\theta,\eta)^2}+\frac{\bar z\frac{\partial k^2\psi_k(t,z)^2}{\partial \bar z}(\bar z^{\frac{1}{k}}-e^{-\frac{i(2\pi j+\alpha)}{k}})}{k^3c(j,k;\rho,\theta,\eta)^2}.
    \end{align*}
    Observe that $\vert z^{\frac{1}{k}}-e^{\frac{2\pi i j}{k}}\vert^2\leq c(j,k;\rho,\theta,\psi)$. Hence,
    \begin{align*}
    &\left\vert \bar z\frac{\partial}{\partial\bar z}\frac{1}{k}\frac{\bar z^{\frac{1}{k}}-e^{-\frac{i(2\pi j+\alpha)}{k}}}{\vert z^{\frac{1}{k}}-e^{\frac{i(2\pi j+\alpha)}{k}}\vert^2-\psi^2}\right\vert\\
    &\leq \frac{2\vert\bar z^{\frac{1}{k}}\vert}{k^2 c(j,k;\rho,\theta,\psi)}+\left\vert\bar z\frac{\partial k^2\psi_k(t,z)^2}{\partial \bar z}\right\vert\left(\frac{1}{k^2 c(j,k;\rho,\theta,\psi)}\right)^{\frac{3}{2}}.
    \end{align*}
    The series absolute convergence then follows directly from the facts that $\bar z\partial k^2\psi^2/\partial\bar z$ converges uniformly to $\bar z\partial \eta^2/\partial\bar z$ (Proposition \ref{prop.kpsi.derivative}) and that $k^2 c(j,k;\rho,\theta,\psi)$ has order $O(j^2)$ uniformly on $K$ by Lemma \ref{lem.sum.order}.

    Recall that $\eta_k(t,z) = -\psi_k(t,z)^2$ and $\psi(t,z)= \lim_{k\to\infty} k\psi_k(t,z)$. Since the series converges absolutely, we can take the limit using the dominated convergence theorem and rewrite the limit as 
    \begin{align*}
        %&\sum_{j=-\infty}^\infty\left(\frac{1}{\rho^2+(2\pi j-\theta)^2-\eta^2}-\frac{\rho^2+(2\pi j-\theta)^2}{(\rho^2+(2\pi j-\theta)^2-\eta^2)^2}+\frac{\bar z\frac{\partial\eta^2}{\partial\bar z}(\rho+i(2\pi  j-\theta))}{(\rho^2+(2\pi j-\theta)^2-\eta^2)^2}\right)\\
        \int_{(-\pi,\pi]}\sum_{j=-\infty}^\infty\bar z\frac{\partial}{\partial\bar z}\frac{\rho+i(2\pi j+\alpha-\theta)}{\rho^2+(2\pi j +\alpha-\theta)^2-\eta^2}\mu_{u_0}(d\alpha)
    \end{align*}
    where the equality can be verified by direct calculations. 
\end{proof}

We need the following two lemmas before we compute the limit of the density of the Brown measure of $u_0b_k(t)$. Recall that, by \eqref{eq.Sigma.u.bdry.funct}, the domain $\Sigma_\infty(u_0,t)$ is described using a function $r_{u_0}$.
\begin{lemma}
    \label{lem.eta.by.bdry}
    Let $z = e^{\rho}e^{i\theta}\in\Sigma_\infty(u_0,t)$. Then $\eta(t,z)$ can be calculated by
    \[\rho^2+\eta(t,z) = \rho_{u_0}(t,\theta)^2\]
    where $\rho_{u_0}(t,\theta) = \log r_{u_0}(t,\theta)$.
\end{lemma}
\begin{proof}
    By \eqref{eq.DW.condition.u}, $\eta(t,z)$ is determined by the equation
    \[\int_{(-\pi,\pi]}\sum_{j=-\infty}^\infty\frac{1}{\rho^2+(2\pi j+\alpha-\theta)^2+\eta(t,z)}\mu_{u_0}(d\alpha)=\frac{1}{t}.\]
    Using Mathematica, we can compute the sum 
    \begin{align*}
        \int_{(-\pi,\pi]}\sum_{j=-\infty}^\infty\frac{1}{\rho^2+(2\pi j+\alpha-\theta)^2}\mu_{u_0}(d\alpha) &= \int_{(-\pi,\pi]} \frac{1}{\log(\vert z\vert^2)}\frac{\vert z\vert^2-1}{\vert z-e^{i\alpha}\vert^2}\mu_{u_0}(d\alpha)\\
        &= (T(u_0;\vert z\vert, \theta))^{-1}.
    \end{align*}
    
    Meanwhile, for any $\theta$, $r_{u_0}(t,\theta)e^{i\theta}$ is on the boundary of the domain $\Sigma_\infty(u_0,t)$. By Corollary \ref{cor.DWpoint.conv.unif}, $\eta(t,r_{u_0}(t,\theta)e^{i\theta}) = 0$. Thus, for any $z = e^{\rho}e^{i\theta}\in \Sigma_\infty(u_0,t)$, $\eta(t,z) = \rho_{u_0}(t,\theta)^2-\rho^2$ is the unique positive number such that 
    \[\int_{(-\pi,\pi]}\sum_{j=-\infty}^\infty\frac{1}{\rho^2+(2\pi j+\alpha-\theta)^2+\eta(t,z)}\mu_{u_0}(d\alpha)=(T(u_0; r_{u_0}(t,\theta),\theta)^{-1}=\frac{1}{t},\]
    completing the proof.
\end{proof}

\begin{lemma}
    \label{lem.series.derivative}
We can rewrite the series
\begin{align*}
    &\int_{(-\pi,\pi]}\sum_{j=-\infty}^\infty\bar z\frac{\partial}{\partial\bar z}\frac{\rho+i(2\pi j+\alpha-\theta)}{\rho^2+(2\pi j+\alpha-\theta)^2+\eta(t,z)}\mu_{u_0}(d\alpha) \\
    &= \frac{1}{2t}+\frac{1}{2}\frac{\partial}{\partial\theta}\int\frac{r_{u_0}(t,\theta)\sin(\theta-\alpha)}{r_{u_0}(t,\theta)^2+1-2r_{u_0}(t,\theta)\cos(\theta-\alpha)}\mu_{u_0}(d\alpha).
\end{align*}
%By page 10 of FreeLinearized, the derivative can be written as
%\begin{align*}
    %\frac{1}{t}+i\cdot \bar z\frac{\partial}{\partial \bar z}\left(\frac{\theta}{t}+\frac{1}{2}\frac{\sin\theta}{\cos\theta-\cosh\rho_t(\theta)}\right) &= \frac{1}{2t}-\frac{1}{4}\frac{\partial}{\partial\theta}\left(\frac{\sin\theta}{\cos\theta-\cosh\rho_t(\theta)}\right)\\
    %&= \frac{1}{2t}+\frac{1}{4}\frac{\partial}{\partial\theta}\frac{2r_t(\theta)\sin\theta}{r_t(\theta)^2+1-2r_t(\theta)\cos\theta}.
%\end{align*}
\end{lemma}
\begin{proof}
    The sum is absolutely convergent, by Lemma \ref{lem.series.density}. We combine the $j$ and $-j$ terms for $j>0$ to rewrite the series into
    \[\bar z\frac{\partial}{\partial \bar z}\left[\int_{(-\pi,\pi]}\left(f_0(\alpha)+\sum_{j=1}^\infty(f_j(\alpha)+f_{-j}(\alpha))\right)\mu_{u_0}(d\alpha)\right]\]
    where
    \[f_j(\alpha) =\frac{\rho+i(\alpha-\theta)+2\pi i j}{\rho^2+(2\pi j+\alpha-\theta)^2+\eta(t,z)}.\]
    %\[\bar z\frac{\partial}{\partial\bar z}\left[\frac{\rho+i(\alpha-\theta)}{\rho^2+(\alpha-\theta)^2+\eta(t,z)}+\sum_{j=1}^\infty\left(\frac{\rho+i(2\pi j+\alpha-\theta)}{\rho^2+(2\pi j+\alpha-\theta)^2+\eta(t,z)}+\frac{\rho+i(-2\pi j+\alpha-\theta)}{\rho^2+(2\pi j-\alpha+\theta)^2+\eta(t,z)}\right)\right].\]
    We are able to take the derivative out of the sum because the sum is now uniformly absolutely convergent in a neighborhood of $(\rho,\theta)$. Then we collect the terms with $\rho+i(\alpha-\theta)$ in the numerator and use the classifying equation \eqref{eq.DW.condition.u} for $\eta = -\psi^2$
    \[\sum_{j=-\infty}^\infty\frac{1}{\rho^2+(2\pi j-\theta)^2+\eta(t,z)}=\frac{1}{t}\]
    to write the sum as
     \begin{align}
     &\bar z\frac{\partial}{\partial \bar z}\left[\frac{\rho+i(\alpha-\theta)}{t}+i\int_{(-\pi,\pi]}\sum_{j=1}^\infty\left(g_j(\alpha)+g_{-j}(\alpha)\right)\mu_{u_0}(d\alpha)\right]\nonumber\\
     &=\frac{1}{t}+i\cdot \bar z\frac{\partial}{\partial \bar z}\left[\int_{(-\pi,\pi]}\sum_{j=1}^\infty\left(g_j(\alpha)+g_{-j}(\alpha)\right)\mu_{u_0}(d\alpha)\right]\label{eq.sum.to.differentiate}
     \end{align}
    %\[\bar z\frac{\partial}{\partial \bar z}\left[\frac{\rho+i(\alpha-\theta)}{t}+i\int_{(-\pi,\pi]}\sum_{j=1}^\infty\left(\frac{2\pi j}{\rho^2+(2\pi j+\alpha-\theta)^2-\eta^2}-\frac{2\pi j}{\rho^2+(2\pi j-\alpha+\theta)^2-\eta^2}\right)\mu_{u_0}(d\alpha)\right],\]
    where
    \[g_j(\alpha)=\frac{2\pi j}{\rho^2+(2\pi j+\alpha-\theta)^2+\eta(t,z)}.\]
    Now, by Lemma \ref{lem.eta.by.bdry}, $\rho^2+\eta(t,z)=\rho_{u_0}(t,\theta)$. We then use Mathematica to evaluate the infinite sum
    \begin{align*}
    &\int_{(-\pi,\pi]}\sum_{j=1}^\infty\left(g_j(\alpha)+g_{-j}(\alpha)\right)\mu_{u_0}(d\alpha)\\
    &= \int_{(-\pi,\pi]}\frac{\theta-\alpha}{t}-\frac{1}{2}\frac{2r_{u_0}(t,\theta)\sin(\theta-\alpha)}{r_{u_0}(t,\theta)^2+1-2r_{u_0}(t,\theta)\cos(\theta-\alpha)}\mu_{u_0}(d\alpha).
    \end{align*}
    The conclusion then follows from evaluating the derivative $\bar z\partial/\partial \bar z$ in \eqref{eq.sum.to.differentiate}. Remark that in \eqref{eq.sum.to.differentiate}, $\bar z\partial/\partial \bar z = (i/2)\partial/\partial\theta$ because the above display equation is independent of $\rho$.
\end{proof}

The limit of the density of the Brown measure of $u_0b_k(t)$ can be written in terms of the function $r_{u_0}$ om \eqref{eq.Sigma.u.bdry.funct}.
\begin{theorem}
    \label{thm.density.limit}
The limit of the density $\rho_k(t,\cdot)$ of the Brown measure of $u_0b_k(t)$ is
\[\frac{1}{4\pi\vert z\vert^2}\left( \frac{2}{t}+\frac{\partial}{\partial\theta}\int\frac{2r_{u_0}(t,\theta)\sin(\theta-\alpha)}{r_{u_0}(t,\theta)^2+1-2r_{u_0}(t,\theta)\cos(\theta-\alpha)}\mu_{u_0}(d\alpha)\right),\]
which is the density of the Brown measure of the free multiplicative Brownian motion $u_0b(t)$. The convergence is uniform in any compact subset of $\Sigma_\infty(u_0,t)$.
\end{theorem}
\begin{proof}
    By Theorem \ref{thm.RW.Brown.density}, the density $\rho_k(t,z)$ is given by
    \[\rho_k(t,z) = \frac{1}{\vert z\vert^{2-2/k}}\frac{\bar z^{1-1/k}}{k\pi}\frac{\partial}{\partial\bar z}\sum_{-\frac{k}{2}<j\leq \frac{k}{2}}\int_{(-\pi,\pi]}\frac{\bar z^{\frac{1}{k}}-e^{-i\frac{2\pi  j+\alpha}{k}}}{\vert z^{\frac{1}{k}}-e^{i\frac{2\pi j+\alpha}{k}}\vert^2+\eta_k(t,z)}\mu_{u_0}(d\alpha),\]
    where $z^{\frac{1}{k}}$ denotes the $k$-th root of $z$ with argument in $(-\pi/k,\pi/k]$. By Lemma \ref{lem.series.density} and Lemma \ref{lem.series.derivative},
    \begin{align*}
        \lim_{k\to\infty}\rho_k(t,z)& =  \frac{1}{\pi\vert z\vert^2}\left( \frac{1}{2t}+\frac{1}{2}\frac{\partial}{\partial\theta}\int\frac{r_{u_0}(t,\theta)\sin(\theta-\alpha)}{r_{u_0}(t,\theta)^2+1-2r_{u_0}(t,\theta)\cos(\theta-\alpha)}\mu_{u_0}(d\alpha)\right)\\
        &=\frac{1}{4\pi\vert z\vert^2}\left( \frac{2}{t}+\frac{\partial}{\partial\theta}\int\frac{2r_{u_0}(t,\theta)\sin(\theta-\alpha)}{r_{u_0}(t,\theta)^2+1-2r_{u_0}(t,\theta)\cos(\theta-\alpha)}\mu_{u_0}(d\alpha)\right).
    \end{align*}
    The convergence is uniform in any compact subset of $\Sigma_\infty(u_0,t)$ because the convergences in Lemma \ref{lem.series.density} and Lemma \ref{lem.series.derivative} are so.
\end{proof}

\begin{corollary}
    \label{cor.Brown.weak.conv}
    The Brown measure $u_0b_k(t)$ converges weakly to the Brown measure of $u_0b(t)$ as $k\to\infty$.
\end{corollary}
\begin{proof}
    Denote $\mu_k$ and $\mu$ be the Brown measures of $u_0b_k(t)$ and $u_0b(t)$ respectively. Let $\varepsilon>0$. We want to show that for any test function $f$,
    \begin{equation}
        \label{eq.weak.conv}
        \left\vert\int f\,d\mu_k-\int f\,d\mu\right\vert<\varepsilon.
    \end{equation}
    
    The measure $\mu$ is absolutely continuous, with full measure on $\Sigma_\infty(t)$. Choose open sets $V_1$ such that $V_1$ contains the boundary of $\Sigma_\infty(u_0,t)$ and 
    \[\mu(V_1)<\frac{\varepsilon}{2\|f\|_\infty},\]
    so that 
     \[\left\vert\int_{V_1}f\,d\mu_k-\int_{V_1} f\,d\mu\right\vert<\frac{\varepsilon}{2}.\]
     
    Since $\Sigma_\infty(u_0,t)$ is a bounded domain and $V_1$ contains its boundary, $\Sigma_\infty(u_0,t)\setminus V_1$ is a compact set. By Theorem \ref{thm.density.limit}, 
    \[\left\vert\int_{\Sigma_\infty(u_0,t)\setminus V_1}f\,d\mu_k-\int_{\Sigma_\infty(u_0,t)\setminus V_1} f\,d\mu\right\vert<\frac{\varepsilon}{2}.\]
    Finally, $\mathrm{supp}(f)\cap V_1^c$ is a compact subset of $(\Sigma_\infty(u_0,t))^c$. By Theorem \ref{thm.domain.conv}, 
    \[\int_{V_1^c}f\,d\mu_k = \int_{V_1^c}f\,d\mu = 0.\]
    Thus, \eqref{eq.weak.conv} holds.
\end{proof}

\begin{ack}
\phantomsection
\addcontentsline{toc}{section}{Acknowledgments}
We are grateful to Nick Cook, who made us aware of the powerful tools in \cite{RudelsonVershynin2014} and suggested the approach that led to Theorem \ref{thm.pseudospectrum}.  We also thank Hari Bercovici, Alice Guionnet, Roland Speicher, Ofer Zeitouni, and Ping Zhong for helpful conversations.
\end{ack}

\phantomsection
\bibliographystyle{acm}
\bibliography{Lima-Bean-Law}

\begin{thebibliography}{10}

\bibitem{Bai-Silverstein}
{\sc Bai, Z., and Silverstein, J.~W.}
\newblock {\em Spectral analysis of large dimensional random matrices},
  second~ed.
\newblock Springer Series in Statistics. Springer, New York, 2010.

\bibitem{BannaMai}
{\sc Banna, M., and Mai, T.}
\newblock H\"older continuity of cumulative distribution functions for
  noncommutative polynomials under finite free {F}isher information.
\newblock {\em J. Funct. Anal. 279}, 8 (2020), 108710, 45.

\bibitem{BaoES2019singlering}
{\sc Bao, Z., Erd\"{o}s, L., and Schnelli, K.}
\newblock {Local single ring theorem on optimal scale}.
\newblock {\em Ann. Probab. 47}, 3 (2019), 1270 -- 1334.

\bibitem{BPZ2019}
{\sc Basak, A., Paquette, E., and Zeitouni, O.}
\newblock Regularization of non-normal matrices by {G}aussian noise---the
  banded {T}oeplitz and twisted {T}oeplitz cases.
\newblock {\em Forum Math. Sigma 7\/} (2019), Paper No. e3, 72.

\bibitem{BPZ202}
{\sc Basak, A., Paquette, E., and Zeitouni, O.}
\newblock Spectrum of random perturbations of {T}oeplitz matrices with finite
  symbols.
\newblock {\em Trans. Amer. Math. Soc. 373}, 7 (2020), 4999--5023.

\bibitem{Belinschi2014}
{\sc Belinschi, S.}
\newblock {$L^\infty$ boundedness of density for free additive convolutions}.
\newblock {\em {Revue roumaine de math{\'e}matiques pures et appliqu{\'e}es}
  59}, 2 (June 2014), 173--184.
\newblock In the ''Special Issue Dedicated to Professor \c{S}erban
  Str\v{a}tip\v{a} on the Occasion of His $70^{\rm th}$ Birthday.''.

\bibitem{BelinschiBercoviciHo2022}
{\sc Belinschi, S., Bercovici, H., and Ho, C.-W.}
\newblock On the convergence of {D}enjoy-{W}olff points.
\newblock {\em arXiv:2203.16728\/} (2022).

\bibitem{Belinschi2008}
{\sc Belinschi, S.~T.}
\newblock The {L}ebesgue decomposition of the free additive convolution of two
  probability distributions.
\newblock {\em Probab. Theory Related Fields 142}, 1-2 (2008), 125--150.

\bibitem{BenaychGeorges2017}
{\sc Benaych-Georges, F.}
\newblock Local single ring theorem.
\newblock {\em Ann. Probab. 45}, 6A (2017), 3850--3885.

\bibitem{BGR2016}
{\sc Benaych-Georges, F., and Rochet, J.}
\newblock Outliers in the single ring theorem.
\newblock {\em Probab. Theory Related Fields 165}, 1-2 (2016), 313--363.

\bibitem{BGR2017}
{\sc Benaych-Georges, F., and Rochet, J.}
\newblock Fluctuations for analytic test functions in the single ring theorem.
\newblock {\em Indiana Univ. Math. J. 66}, 6 (2017), 1981--2013.

\bibitem{BercoviciVoiculescu1993}
{\sc Bercovici, H., and Voiculescu, D.}
\newblock Free convolution of measures with unbounded support.
\newblock {\em Indiana Univ. Math. J. 42}, 3 (1993), 733--773.

\bibitem{BercVoic1995}
{\sc Bercovici, H., and Voiculescu, D.}
\newblock Superconvergence to the central limit and failure of the {C}ram\'er
  theorem for free random variables.
\newblock {\em Probab. Theory Related Fields 103}, 2 (1995), 215--222.

\bibitem{BercoviciZhong2022Rdiag}
{\sc Bercovici, H., and Zhong, P.}
\newblock The {B}rown measure of a sum of two free random variables, one of
  which is {$R$}-diagonal.
\newblock {\em arXiv:2209.12379v1\/} (2022).

\bibitem{Berger}
{\sc Berger, M.~A.}
\newblock Central limit theorem for products of random matrices.
\newblock {\em Trans. Amer. Math. Soc. 285}, 2 (1984), 777--803.

\bibitem{Biane1998}
{\sc Biane, P.}
\newblock Processes with free increments.
\newblock {\em Math. Z. 227}, 1 (1998), 143--174.

\bibitem{BianeLehner2001}
{\sc Biane, P., and Lehner, F.}
\newblock Computation of some examples of {B}rown's spectral measure in free
  probability.
\newblock {\em Colloq. Math. 90}, 2 (2001), 181--211.

\bibitem{BianeSpeicher1998}
{\sc Biane, P., and Speicher, R.}
\newblock Stochastic calculus with respect to free {B}rownian motion and
  analysis on {W}igner space.
\newblock {\em Probab. Theory Related Fields 112}, 3 (1998), 373--409.

\bibitem{BS1998}
{\sc Biane, P., and Speicher, R.}
\newblock Stochastic calculus with respect to free {B}rownian motion and
  analysis on {W}igner space.
\newblock {\em Probability Theory Related Fields 112\/} (1998), 373--409.

\bibitem{Bordenave-Chafai-circular}
{\sc Bordenave, C., and Chafa\"{\i}, D.}
\newblock Around the circular law.
\newblock {\em Probab. Surv. 9\/} (2012), 1--89.

\bibitem{BKS}
{\sc Bo\.zejko, M., K\"ummerer, B., and Speicher, R.}
\newblock {$q$}-{G}aussian processes: non-commutative and classical aspects.
\newblock {\em Comm. Math. Phys. 185}, 1 (1997), 129--154.

\bibitem{Brown1986}
{\sc Brown, L.~G.}
\newblock Lidski\u{\i}'s theorem in the type {${\rm II}$} case.
\newblock In {\em Geometric methods in operator algebras ({K}yoto, 1983)},
  vol.~123 of {\em Pitman Res. Notes Math. Ser.} Longman Sci. Tech., Harlow,
  1986, pp.~1--35.

\bibitem{Cook}
{\sc Cook, N.~A., Guionnet, A., and Husson, J.}
\newblock Spectrum and pseudospectrum for quadratic polynomials in {G}inibre
  matrices.
\newblock {\em Ann. Inst. Henri Poincar\'e{} Probab. Stat. 58}, 4 (2022),
  2284--2320.

\bibitem{Donsker}
{\sc Donsker, M.~D.}
\newblock An invariance principle for certain probability limit theorems.
\newblock {\em Mem. Amer. Math. Soc. 6\/} (1951), 12.

\bibitem{DriverProbNotes}
{\sc Driver, B.}
\newblock Probability tools with examples.
\newblock unpublished, 2019.

\bibitem{DHKBrown}
{\sc Driver, B.~K., Hall, B.~C., and Kemp, T.}
\newblock The {B}rown measure of the free multiplicative {B}rownian motion.
\newblock {\em arXiv:1903.11015v2\/} (2019).

\bibitem{FPZ2015}
{\sc Feldheim, O.~N., Paquette, E., and Zeitouni, O.}
\newblock Regularization of non-normal matrices by {G}aussian noise.
\newblock {\em Int. Math. Res. Not. IMRN}, 18 (2015), 8724--8751.

\bibitem{Ginibre1965}
{\sc Ginibre, J.}
\newblock {Statistical Ensembles of Complex, Quaternion and Real Matrices}.
\newblock {\em J. Math. Phys. 6\/} (1965), 440--449.

\bibitem{Girko}
{\sc Girko, V.~L.}
\newblock The circular law.
\newblock {\em Teor. Veroyatnost. i Primenen. 29}, 4 (1984), 669--679.

\bibitem{GuionnetKZ-single-ring}
{\sc Guionnet, A., Krishnapur, M., and Zeitouni, O.}
\newblock The single ring theorem.
\newblock {\em Ann. of Math. (2) 174}, 2 (2011), 1189--1217.

\bibitem{GWZ2014}
{\sc Guionnet, A., Wood, P.~M., and Zeitouni, O.}
\newblock Convergence of the spectral measure of non-normal matrices.
\newblock {\em Proc. Amer. Math. Soc. 142}, 2 (2014), 667--679.

\bibitem{GuionnetZeitouni}
{\sc Guionnet, A., and Zeitouni, O.}
\newblock Support convergence in the single ring theorem.
\newblock {\em Probab. Theory Related Fields 154}, 3-4 (2012), 661--675.

\bibitem{HaagerupLarsen2000}
{\sc Haagerup, U., and Larsen, F.}
\newblock Brown's spectral distribution measure for {R}-diagonal elements in
  finite von {N}eumann algebras.
\newblock {\em J. Funct. Anal. 176}, 2 (2000), 331 -- 367.

\bibitem{HaagerupSchultz2007}
{\sc Haagerup, U., and Schultz, H.}
\newblock Brown measures of unbounded operators affiliated with a finite von
  {N}eumann algebra.
\newblock {\em Math. Scand. 100}, 2 (2007), 209--263.

\bibitem{HaagerupT2005Annals}
{\sc Haagerup, U., and Thorbj{\o}rnsen, S.}
\newblock A new application of random matrices: {${\rm Ext}(C^*_{\rm
  red}(F_2))$} is not a group.
\newblock {\em Ann. of Math. (2) 162}, 2 (2005), 711--775.

\bibitem{HallHo2025Spectrum}
{\sc Hall, B.~C., and Ho, C.-W.}
\newblock Spectral results for free random variables.
\newblock {\em (to submit to arXiv)\/} (2025).

\bibitem{Heins1941}
{\sc Heins, M.~H.}
\newblock On the iteration of functions which are analytic and single-valued in
  a given multiply-connected region.
\newblock {\em Amer. J. Math. 63\/} (1941), 461--480.

\bibitem{HoZhong2020Brown}
{\sc Ho, C.-W., and Zhong, P.}
\newblock Brown measures of free circular and multiplicative {B}rownian motions
  with self-adjoint and unitary initial conditions.
\newblock {\em J. Eur. Math. Soc. (JEMS) 25}, 6 (2023), 2163--2227.

\bibitem{HoZhong2025}
{\sc Ho, C.-W., and Zhong, P.}
\newblock Deformed single ring theorems.
\newblock {\em J. Funct. Anal. 288}, 5 (2025), Paper No. 110797, 42.

\bibitem{Kato1982}
{\sc Kato, T.}
\newblock {\em A short introduction to perturbation theory for linear
  operators}.
\newblock Springer-Verlag, New York-Berlin, 1982.

\bibitem{Kemp-RMT}
{\sc Kemp, T.}
\newblock Introduction to random matrix theory.
\newblock unpublished, 2013.

\bibitem{Kemp2016}
{\sc Kemp, T.}
\newblock The large-{$N$} limits of {B}rownian motions on {$\Bbb{GL}_N$}.
\newblock {\em Int. Math. Res. Not. IMRN}, 13 (2016), 4012--4057.

\bibitem{Kemp-JCTA}
{\sc Kemp, T., Mahlburg, K., Rattan, A., and Smyth, C.}
\newblock Enumeration of non-crossing pairings on bit strings.
\newblock {\em J. Combin. Theory Ser. A 118}, 1 (2011), 129--151.

\bibitem{MR2016}
{\sc Marinelli, C., and R\"ockner, M.}
\newblock On the maximal inequalities of {B}urkholder, {D}avis and {G}undy.
\newblock {\em Expositiones Mathematicae 34\/} (2016), 1--26.

\bibitem{McKean}
{\sc McKean, Jr., H.~P.}
\newblock {\em Stochastic integrals}, vol.~No. 5 of {\em Probability and
  Mathematical Statistics}.
\newblock Academic Press, New York-London, 1969.

\bibitem{MingoSpeicherBook}
{\sc Mingo, J.~A., and Speicher, R.}
\newblock {\em Free probability and random matrices}, vol.~35 of {\em Fields
  Institute Monographs}.
\newblock Springer, New York; Fields Institute for Research in Mathematical
  Sciences, Toronto, ON, 2017.

\bibitem{NShS}
{\sc Nica, A., Shlyakhtenko, D., and Speicher, R.}
\newblock Some minimization problems for the free analogue of the {F}isher
  information.
\newblock {\em Adv. Math. 141}, 2 (1999), 282--321.

\bibitem{Nica-Speicher-1998duke}
{\sc Nica, A., and Speicher, R.}
\newblock Commutators of free random variables.
\newblock {\em Duke Math. J. 92}, 3 (1998), 553--592.

\bibitem{SpeicherNicaBook}
{\sc Nica, A., and Speicher, R.}
\newblock {\em Lectures on the Combinatorics of Free Probability}.
\newblock London Mathematical Society Lecture Note Series. Cambridge University
  Press, 2006.

\bibitem{NikitopoulosIto}
{\sc Nikitopoulos, E.}
\newblock {I}t\^{o}'s formula for noncommutative ${C}^2$ functions of
  self-adjoint free {I}t\^{o} processes.
\newblock {\em Documenta Mathematica 27\/} (2022), 1447--1507.
\newblock Erratum: Documenta Mathematica \textbf{28} (2023), 1275--1277.

\bibitem{Parraud2022}
{\sc Parraud, F.}
\newblock On the operator norm of non-commutative polynomials in deterministic
  matrices and iid {H}aar unitary matrices.
\newblock {\em Probab. Theory Related Fields 182}, 3-4 (2022), 751--806.

\bibitem{RudelsonVershynin2014}
{\sc Rudelson, M., and Vershynin, R.}
\newblock Invertibility of random matrices: unitary and orthogonal
  perturbations.
\newblock {\em J. Amer. Math. Soc. 27}, 2 (2014), 293--338.

\bibitem{ShlySkou}
{\sc Shlyakhtenko, D., and Skoufranis, P.}
\newblock Freely independent random variables with non-atomic distributions.
\newblock {\em Trans. Amer. Math. Soc. 367}, 9 (2015), 6267--6291.

\bibitem{Sniady}
{\sc \'Sniady, P.}
\newblock Random regularization of {B}rown spectral measure.
\newblock {\em J. Funct. Anal. 193}, 2 (2002), 291--313.

\bibitem{Speicher1994}
{\sc Speicher, R.}
\newblock Multiplicative functions on the lattice of noncrossing partitions and
  free convolution.
\newblock {\em Math. Ann. 298}, 4 (1994), 611--628.

\bibitem{Trefethen}
{\sc Trefethen, L.~N., and Embree, M.}
\newblock {\em Spectra and pseudospectra}.
\newblock Princeton University Press, Princeton, NJ, 2005.
\newblock The behavior of nonnormal matrices and operators.

\bibitem{Voiculescu1986}
{\sc Voiculescu, D.}
\newblock Addition of certain noncommuting random variables.
\newblock {\em J. Funct. Anal. 66}, 3 (1986), 323--346.

\bibitem{Voiculescu1987}
{\sc Voiculescu, D.}
\newblock Multiplication of certain noncommuting random variables.
\newblock {\em J. Operator Theory 18}, 2 (1987), 223--235.

\bibitem{Voiculescu1991}
{\sc Voiculescu, D.}
\newblock Limit laws for random matrices and free products.
\newblock {\em Invent. Math. 104}, 1 (1991), 201--220.

\bibitem{Voiculescu1993}
{\sc Voiculescu, D.}
\newblock The analogues of entropy and of {F}isher's information measure in
  free probability theory. {I}.
\newblock {\em Comm. Math. Phys. 155}, 1 (1993), 71--92.

\bibitem{Watkins}
{\sc Watkins, J.~C.}
\newblock Functional central limit theorems and their associated large
  deviation principles for products of random matrices.
\newblock {\em Probab. Theory Related Fields 76}, 2 (1987), 133--166.

\bibitem{Wegner}
{\sc Wegner, F.}
\newblock Bounds on the density of states in disordered systems.
\newblock {\em Z. Phys. B 44}, 1-2 (1981), 9--15.

\bibitem{WongZakai1}
{\sc Wong, E., and Zakai, M.}
\newblock On the convergence of ordinary integrals to stochastic integrals.
\newblock {\em Ann. Math. Statist. 36\/} (1965), 1560--1564.

\bibitem{WongZakai2}
{\sc Wong, E., and Zakai, M.}
\newblock On the relation between ordinary and stochastic differential
  equations.
\newblock {\em Internat. J. Engrg. Sci. 3\/} (1965), 213--229.

\bibitem{Wood2016}
{\sc Wood, P.~M.}
\newblock Universality of the {ESD} for a fixed matrix plus small random noise:
  a stability approach.
\newblock {\em Ann. Inst. Henri Poincar\'e{} Probab. Stat. 52}, 4 (2016),
  1877--1896.

\bibitem{Zhong2021Brown}
{\sc Zhong, P.}
\newblock Brown measure of the sum of an elliptic operator and a free random
  variable in a finite von neumann algebra.
\newblock {\em arXiv:2108.09844v4, to appear in Amer. J. Math.\/} (2021).

\end{thebibliography}

\end{document}